\documentclass{amsart}
\usepackage[left=1.5in, right=1.5in]{geometry}
\usepackage[latin1]{inputenc}
\usepackage{amsfonts}
\usepackage[toc,page,title,titletoc,header]{appendix}
\usepackage{graphicx,psfrag,epsfig,multirow,caption,subcaption}
\usepackage{amssymb,amsmath,amscd,amsthm,amssymb,verbatim,setspace}
\numberwithin{equation}{section}
\usepackage{mathrsfs,wrapfig}
\usepackage{indentfirst}
\usepackage{extarrows}
\usepackage[colorlinks=true]{hyperref}

\newtheorem{Thm}{Theorem}
\newtheorem{Lm}{Lemma}[section]
\newtheorem{sublemma}[Lm]{Sublemma}
\newtheorem{Prop}[Lm]{Proposition}
\newtheorem{Cor}[Lm]{Corollary}
\newtheorem{Def}[Lm]{Definition}
\newtheorem{Rk}[Lm]{Remark}
\newtheorem{Not}[Lm]{Notation}
\newtheorem{Conjecture}{Conjecture}

\def\bdef{\begin{Def}}
\def\endef{\end{Def}}
\def\bthm{\begin{Thm}}
\def\ethm{\end{Thm}}
\def\bprop{\begin{Prop}}
\def\enprop{\end{Prop}}
\def\blm{\begin{Lm}}
\def\elm{\end{Lm}}
\def\beq{\begin{equation}}
\def\eeq{\end{equation}}
\def\bcor{\begin{Cor}}
\def\ecor{\end{Cor}}
\def\bfig{\begin{picture}}
\def\efig{\end{picture}}
\def\be{\begin{eqnarray}}
\def\ee{\end{eqnarray}}
\def\beal{\begin{aligned}}
\def\enal{\end{aligned}}
\newcommand{\bcon}{\begin{Conjecture}}
\newcommand{\econ}{\end{Conjecture}}

\newcommand\bL{\boldsymbol{L}}
\newcommand{\sign}{{{\rm sign}}}
\newcommand{\brc}{\bar c}

\newcommand{\gm}{\gamma}

\newcommand{\al}{\alpha}

\newcommand{\cR}{\mathcal{R}}

\newcommand{\cF}{\mathcal{F}}
\newcommand{\cD}{\mathcal{D}}
\newcommand{\cL}{\mathcal{L}}

\newcommand{\fh}{\mathfrak{h}}

\newcommand{\R}{\mathbb{R}}

\newcommand{\om}{\omega}

\newcommand{\bt}{\beta}

\newcommand{\dt}{\delta}
\newcommand{\lb}{\lambda}

\newcommand{\eps}{\varepsilon}
\newcommand{\T}{\mathbb{T}}

\newcommand{\Ger}{\mathbf{G}}

\newcommand{\Loc}{\mathbb{L}}
\newcommand{\Glob}{\mathbb{G}}
\newcommand{\Id}{\mathrm{Id}}
\newcommand{\Span}{\text{span}}
\newcommand{\Ker}{\text{Ker}}
\newcommand{\lin}{{\bf l}}
\newcommand{\brlin}{{\bar\lin}}
\newcommand{\brrlin}{{\bar\brlin}}
\def\({\left(}
\def\){\right)}

\def\Const{{\text{Const}}}

\newcommand\bx{\boldsymbol{x}}
\newcommand\by{\boldsymbol{y}}

\newcommand\bbV{{\mathbb{V}}}

\newcommand{\brE}{\bar E}

\newcommand{\bre}{\bar e}

\newcommand{\brG}{\bar G}
\newcommand{\brg}{\bar g}

\newcommand{\bv}{{\bf u}}

\newcommand{\brv}{\bar\bv}
\newcommand{\brrv}{\bar \brv}

\newcommand\brdelta{{\bar\delta}}

\newcommand\brtau{{\bar\tau}}
\newcommand\brchi{{\bar\chi}}
\newcommand{\cK}{\mathcal{K}}
\newcommand{\cP}{\mathcal{P}}

\newcommand\su{\mathsf u}
\newcommand\sv{\mathsf v}
\newcommand\sw{\mathsf w}

\def\hu{{\hat{u}}}

\def\hlin{{\hat{\lin}}}
\def\hGamma{{\hat{\Gamma}}}

\def\ta{\tilde{a}}
\def\tb{\tilde{b}}

\def\te{\tilde{e}}

\def\tQ{\tilde{Q}}

\def\tw{\tilde{w}}

\def\tGamma{\tilde{\Gamma}}
\def\tdelta{\tilde{\delta}}

\title{Noncollision Singularities in a Planar Four-body Problem}
\author[Jinxin Xue]{Jinxin Xue}
\address{Yau Mathematical Sciences Center \& Department of Mathematics, Tsinghua University, Beijing, China, 100084}
\email{jxue@tsinghua.edu.cn, jinxinxue@gmail.com}
\begin{document}
\maketitle
\begin{abstract}
In this paper, we show that there is a Cantor set of initial conditions in the planar four-body problem such that all four bodies escape to infinity in a finite time, avoiding collisions. This proves the Painlev\'{e} conjecture for the four-body case, and thus settles the last open case of the conjecture.
\end{abstract}
\tableofcontents

\renewcommand\contentsname{Index}


\section{Introduction}
Consider two large bodies $Q_1$ and $Q_2$ of masses $m_1=m_2=1$
located at distance $\chi\gg 1$ from each other initially, and
two small particles $Q_3$ and $Q_4$ of masses $m_3=m_4=\mu\ll 1$.
The $Q_i$s interact with each other via Newtonian potential. We denote the momentum of $Q_i$ by  $P_i.$
The Hamiltonian of this system can be written as
\begin{equation}
\label{eq: hammain}
H(Q_1,P_1;Q_2,P_2;Q_3,P_3;Q_4,P_4)=\frac{P_1^2}{2}+\frac{P_2^2}{2}+\frac{P_3^2}{2\mu}+\frac{P_4^2}{2\mu}
\end{equation}
$$-\frac{1}{|Q_1-Q_2|}-\frac{\mu}{|Q_1-Q_3|}-\frac{\mu}{|Q_1-Q_4|}-\frac{\mu}{|Q_2-Q_3|}-\frac{\mu}{|Q_2-Q_4|}-\frac{\mu^2}{|Q_3-Q_4|}.$$
We choose the mass center as the origin.

We want to study singular solutions of this system, that is solutions which cannot be continued for all positive times.
We will exhibit a rich variety of singular solutions. Fix a small $\eps_0.$ Let $\boldsymbol\omega=\{\omega_j\}_{j=1}^\infty$
be a sequence of 3s and 4s.

\begin{Def}\label{DefSing}
We say that $(Q_i(t), \dot Q_i(t)),\ i=1,2,3,4,$ is a {\bf singular solution with symbolic sequence $\boldsymbol\omega$} if there exists a positive
increasing sequence $\{t_j\}_{j=0}^\infty$ such that
\begin{itemize}
\item $t^*=\lim_{j\to\infty} t_j<\infty.$
\item $|Q_3-Q_2|(t_j)\leq \eps_0,$ $|Q_4-Q_2|(t_j)\leq \eps_0.$
\item For $t\in [t_{j-1}, t_j]$, $|Q_{7-\omega_j}-Q_2|(t)\leq \eps_0$ and
$\{Q_{\omega_j}(t)\}_{t\in [t_{j-1}, t_j]}$ leaves the $\eps_0$ neighborhood of $Q_2$, winds around $Q_1$ exactly once, then reenters the $\eps_0$ neighborhood of $Q_2$.
\item $\limsup_{t}|\dot{Q}_i(t)|,\ \limsup_{t}|{Q}_i(t)| \to \infty$  as $t\to t^*,$\ $i=1,2,3,4$.
\end{itemize}
\end{Def}
During the time interval $[t_{j-1}, t_j]$ we refer to $Q_{\omega_j}$ as the traveling particle and
to $Q_{7-\omega_j}$ as the captured particle. Thus $\omega_j$ prescribes which particle is
the traveler during the $j$th trip.

We denote by ${\Sigma_{\boldsymbol\omega}}$ the set of initial conditions of singular orbits with symbolic sequence $\boldsymbol\omega.$
\bthm
\label{ThMain}
There exists $\mu_*\ll 1$ such that for $\mu<\mu_*$ the set $\Sigma_{\boldsymbol\omega}\neq\emptyset.$

Moreover there is an open set $U$ on the zero energy level and zeroth angular momentum level, and a foliation of $U$ by two-dimensional surfaces such that for any leaf
$S$ of our foliation $\Sigma_{\boldsymbol\omega}\cap S$ is a Cantor set.
\ethm

We remark that the choice of the zero energy level is only for simplicity. Our construction holds for sufficiently small nonzero energy levels.

\subsection{Motivations and perspectives}
Our work is motivated by the following fundamental  problem in celestial mechanics.
{\it Describe the set of initial conditions  of the Newtonian N-body problem leading to global solutions.}
The complement to this set splits into the initial conditions leading to the collision and non-collision singularities.

It is clear that the set of initial conditions leading to collisions is non-empty for all $N>1$ and it is shown in
\cite{Sa1} that it has zero measure. Much less is known about the non-collision singularities. The main motivation
for our work is provided by the
following basic problems.

\bcon
\label{ConZM}
The set of non-collision singularities has zero measure for all $N>3.$
\econ

This conjecture can be found in the problem list \cite{Sim} as the first problem. This conjecture remains almost completely open. The only known result, by Saari \cite{Sa2}, is that the conjecture is true for $N=4$ .  To obtain the complete solution of this conjecture one needs to understand better
the structure of the non-collision singularities. Our Cantor set in Theorem \ref{ThMain} has zero measure and codimension 2 on the energy level, which is in favor of Conjecture \ref{ConZM}. As a first step, it is natural to conjecture the following.

\bcon[Painlev\'{e} Conjecture, 1897]
\label{ConNE}
The set of non-collision singularities is non-empty for all $N>3.$
\econ

There is a long history studying Conjecture \ref{ConNE}. There are some nice surveys, see for instance \cite{G3}. Conjecture \ref{ConNE} was explicitly mentioned
in Painlev\'{e}'s lectures \cite{Pa} where the author proved that for $N=3$ there are no non-collision singularities, using an argument based on the triangle inequality (see also \cite{G3} for the argument).
Soon after Painlev\'{e}, von Zeipel showed that if the system of $N$ bodies has a non-collision singularity, then some
particle should fly off to infinity in finite time. Thus non-collision singularities seem quite counterintuitive. The first landmark towards proving the conjecture came in 1975. In \cite{MM} Mather and McGehee constructed a system of four bodies on the line where the particles go to infinity
in finite time after an infinite number of binary collisions (it was known since the work of Sundman \cite{Su} that
binary collisions can be regularized so that the solutions can be extended beyond the collisions). Since the
Mather-McGehee example had collisions it did not solve Conjecture \ref{ConNE} but made it plausible.
Conjecture \ref{ConNE} was proved independently by Xia \cite{X} for the spatial five-body problem and by Gerver
\cite{G1} for the planar $3N$-body problem where $N$ is sufficiently large. It is a general belief that a non collision singularity in the $(N+1)$-body problem can be obtained by adding one more remote and light body to the $N$-body problem, to which the existence of non-collision singularities is known. The hardest case of the problem, $N=4$, still remained open. Our result proves the conjecture in the $N=4$ case. 

We believe the method used in this paper could also be used to construct
noncollision singularities for the general $N$-body problem, for any $N>3$.
We can put any number of bodies into our system sufficiently far from the mass center of our four bodies, orthogonal to the line passing through $Q_1$ and $Q_2$. This produces noncollision singularities in the $N$-body problem. We have not checked all the details in that case
but we do not expect any significant difficulties. Treating the general $N$ however would significantly
increase the length of the paper, so to simplify the exposition we concentrate here on the four-body case.

Since our technique is perturbative and it is necessary that $\mu\ll 1$, we ask the following questions.

\textbf{Question 1:} {\it Are there noncollision singularities for the four-body problem in which all the four bodies have comparable masses?}

In fact it is possible that the following stronger result holds.

\textbf{Question 2:} {\it Is it true that for any choice of positive masses $(m_1,m_2,m_3,m_4)\in \R \mathbb P^3 $
the corresponding four-body problem has noncollision singularities?}

We need to develop some nonperturbative techniques for the first question and we need to explore
the obstructions for the existence of noncollision singularities for the second.

\subsection{Sketch of the proof}
 The main idea of the proof is outlined in \cite{Xu}. The proof consists of the following three aspects: physical, mathematical and algorithmic aspects.
The physical aspect is an idealistic model constructed by Gerver \cite{G2} (see Section \ref{SGerver}), in which the hyperbolic Kepler motion of one light body can extract energy from the elliptic Kepler motion of the other light body. Moreover, after each cycle of energy extraction, the configuration is made self-similar to the beginning, so that the procedure of energy extraction can be iterated infinitely.

The mathematical aspect is a partially hyperbolic dynamics framework. We find that there are two strongly expanding directions that are invariant under iterates along our singular orbits. The strong expansions allow us to push the iteration to the future and synchronize the two light bodies. Namely, the two light bodies can be chosen to come to the correct place simultaneously in order to have a close encounter. One strong expansion is given by a close encounter between $Q_1$ and $Q_4$. This is the hyperbolicity created by scattering (hyperbolic Kepler motion). The other one is induced by shear coming from the elliptic Kepler motion, which seems quite new in celestial mechanics. 

The algorithmic aspect is a systematic toolbox that we develop to compute the derivative of the Poincar\'{e} map in detail. This toolbox includes symplectic coordinate systems and partition of the phase space (Section \ref{sct: symp} and Appendix \ref{appendixa}), integration of the variational equations (Section \ref{sct: var}) and boundary contributions (Section \ref{sct: boundary}), coordinate change between different pieces of the phase space (Section \ref{sct: switch}), collision exclusion (Section \ref{ssct: nocollision}) {\it etc}. Moreover, we develop new methods to regularize the double collision using hyperbolic Delaunay coordinates and extract $\mathscr C^1$ information of the near double collision from its singular limit, the elastic collision, using polar coordinates (Section \ref{sct: loc}). These new methods are more suitable to our framework than previously known methods such as Levi-Civita regularization, and hopefully have wider applications.

The paper is organized as follows. In Section \ref{sct: main}, we give the proof of the main Theorem \ref{ThMain}. In Section \ref{sct: hyp} we study the structure of the derivative of the local map and the global map. In Section \ref{sct: symp}, we perform several symplectic transformations to reduce the Hamiltonian system to a form suitable for doing calculations and estimates.
This section is purely algebraic without dynamics. Next, we state our estimates for the derivatives of the factor maps of the global map as Proposition \ref{Prop: main} in Section \ref{sct: prop}. The following Sections \ref{sct: eqmotion}, \ref{sct: var}, \ref{sct: boundary}, \ref{sct: switch} and \ref{sct: loc} are devoted to the proof of the proposition.
In Appendix \ref{appendixb}, we give the proof of our main estimate for the derivative of the global map,
Lemma \ref{Lm: glob}, based on Proposition \ref{Prop: main}. 
 Finally,
in Appendix \ref{appendixa}, we give an introduction to Delaunay variables including estimates of the various partial derivatives
which are used in our calculations, and in Appendix \ref{section: gerver}, we summarize the result of Gerver in \cite{G2}.



We use the following conventions for constants.

{\bf Convention for constants}:
{\it
\begin{itemize}
\item We use $C,c,\hat C,\tilde{C}$ (without subscript) to denote a constant whose value may be different in different contexts.
\item When we use subscript $1,3,4$, for instance $C_1,\ C_3,\ C_4$ {\it etc}, we mean the constant has fixed value throughout the paper specifically chosen for the first, third or fourth body.
\end{itemize}
}
\section{Proof of the main theorem}\label{sct: main}
\subsection{The coordinates}
We first introduce the set of coordinates needed to state our lemmas and prove our theorems. This set of coordinates is known as the Jacobi coordinates.

\begin{Def}[The coordinates]\label{def: coordR}

\begin{itemize}
\item We define the relative position of $Q_1,Q_3,Q_4$ to $Q_2$ as the new variables $q_1,q_3,q_4$
\begin{equation}\label{eq: qQ}
q_1=Q_1-Q_2,\quad q_3=Q_3-Q_2,\quad q_4=Q_4-Q_2,
\end{equation}
and the new momentum $p_1,p_3,p_4$
which are related to the old momentum $P_1,P_3,P_4$ by
\begin{equation}\label{eq: pP} P_1=\mu p_1,\quad P_3=\mu p_3,\quad P_4=\mu p_4.\end{equation}
\item Next, we define the new set of variables $(x_3,v_3;x_1,v_1;x_4,v_4)$ called Jacobi coordinates through
\begin{equation}\label{eq: right}
\begin{cases}
&v_3=p_3+\frac{\mu}{1+\mu}(p_4+p_1),\\
&v_1=p_1,\\
&v_4=p_4+\frac{\mu p_1}{1+2\mu},\\
\end{cases}\quad
\begin{cases}
&x_3=q_3,\\
&x_1=q_1-\frac{\mu(q_3+q_4)}{2\mu+1},\\
&x_4=q_4-\frac{\mu q_3}{1+\mu}.
\end{cases}
\end{equation}
One can easily check that this transformation is symplectic, i.e. the following symplectic form $\bar\omega$ is preserved
\begin{equation}\label{eq: baromega}
\bar\omega=\sum_{i=3,1,4}dp_i\wedge dq_i=\sum_{i=3,1,4}dv_i\wedge dx_i.
\end{equation}
\item The total angular momentum is
$$G_0:=\sum_{i=3,1,4}p_i\times q_i=\sum_{i=3,1,4}v_i\times x_i.$$
In this paper we assume the total angular momentum $G_0=0$.
\end{itemize}
\end{Def}
\begin{Rk}\label{RkCoord}
\begin{itemize}
\item This set of new coordinates $(x_3,v_3;x_1,v_1;x_4,v_4)$ look complicated. Heuristically,
the new coordinates have the same physical meanings as $(q_3,p_3;q_1,p_1;q_4,p_4)$, since the transformation between them is a $O(\mu)$ perturbation of $\Id$. We will study coordinate changes systematically in Section \ref{sct: symp}.
\item The rescaling \eqref{eq: pP} changes the meanings of some physical quantities. First, $v_3,v_4$ are close to the velocities of $Q_3$ and $Q_4$ respectively; however, $v_1$ is not close to the velocity of $Q_1$ but is close to $\mu^{-1}$ times the velocity of $Q_1$. Next, the angular momentum $G_0$ that we use here is actually $\mu^{-1}$ times the angular momentum defined using the original coordinates $P_i,Q_i,\ i=1,2,3,4$. Similarly, the energy is also $\mu^{-1}$ times the original energy.
\end{itemize}
\end{Rk}
We then use Appendix \ref{appendixa} to pass to Delaunay variables $(x_3,v_3)\to (L_3,\ell_3,G_3,g_3)$ and $(x_4,v_4)\to (L_4,\ell_4,G_4,g_4)$. For Kepler motion with Hamiltonian $H_2=\frac{|v|^2}{2}-\frac{1}{|x|},\ (x,v)\in \R^2\times \R^2$, the Delaunay variables have explicit geometric meanings. When $H_2<0$, the Kepler motion is elliptic. The quantity $L^2$ is the semimajor axis, $|LG|$ is the semi-minor axis, $g$ is the argument of apapsis, and $\ell$ is the mean anomaly indicating the position of the moving particle on the ellipse. When $H_2>0$, the Kepler motion is hyperbolic, in which case the Delaunay variables have similar geometric meanings. Details are provided in Appendix \ref{appendixa}.

To start we assume the energy $E_3$ of the subsystem $(x_3,v_3)$ is negative while the energy $E_4$ of the subsystem $(x_4,v_4)$ is positive. The energies and their relations to the Delaunay variables are given as follows
$$E_3:=\frac{|v_3|^2}{2m_3}-\frac{k_3}{|x_3|}=-\frac{m_3k_3^2}{2L_3^2},\quad\mathrm{and}\quad E_4:=\frac{|v_4|^2}{2m_4}-\frac{k_4}{|x_4|}=\frac{m_4k_4^2}{2L_4^2},$$
where the values of $m_{i}$ and $k_i$ are given explicitly in \eqref{Eqmk} below, and it is enough to know that $m_i,\,k_i=1+O(\mu),$ \ $i=3,4$.
The variable $G_i=v_i\times x_i$ means minus the angular momentum of the subsystem $(x_i,v_i),\ i=3,4$.

We fix the zero energy level so that we can eliminate $L_4$ from our list of variables, applying the implicit function theorem (Section \ref{SSReduction}). Next we pick a Poincar\'{e} section and treat $\ell_4$ as the new time (see Definition \ref{def: sct} below), so that we eliminate $\ell_4$ from our set of coordinates. So we get  $(L_3,\ell_3,G_3,g_3;x_1,v_1;G_4,g_4)\in \R^7\times \T^3$ as the set of coordinates that we use to do calculations. In this section, we use the energy $E_3$ instead of $L_3$, eccentricities $e_3,e_4$ instead of the negative angular momentum $G_3,G_4$. The new choice of coordinates are related to the old ones through $e_i=\sqrt{1+2G_i^2E_i},\ i=3,4$. We use the set of coordinates $(E_3,\ell_3,e_3,g_3;x_1,v_1;e_4,g_4)$ to give the proof of the main theorem since it is easier to study their
behavior under the rescaling.
Actually, our system still has total angular momentum conservation.
We could have fixed an angular momentum and eliminated two more variables. However, this would lead to more
complicated calculations.

\begin{Not}\label{NotVX}
\begin{itemize}
\item We refer to our set of variables as
$$\mathcal V=(\mathcal V_3;\mathcal V_1;\mathcal V_4)=(L_3,\ell_3,G_3,g_3;x_1,v_1; G_4,g_4).$$
\item We denote the Cartesian variables as
$$\mathcal X:=(\mathcal X_3;\mathcal X_1;\mathcal X_4)=(x_3,v_3;x_1,v_1;x_4,v_4).$$
\item In the following, when we use Cartesian coordinates such as $x,v$, each letter has two components. We will use the subscript $\parallel$ to denote the horizontal coordinate and subscript $\perp$ to denote the vertical coordinate. So we write $x=(x_\parallel,x_\perp)$ and $v=(v_\parallel,v_\perp)$ {\it etc}.
\end{itemize}
\end{Not}

\begin{figure}[ht]
\begin{center}
\includegraphics[width=0.8\textwidth]{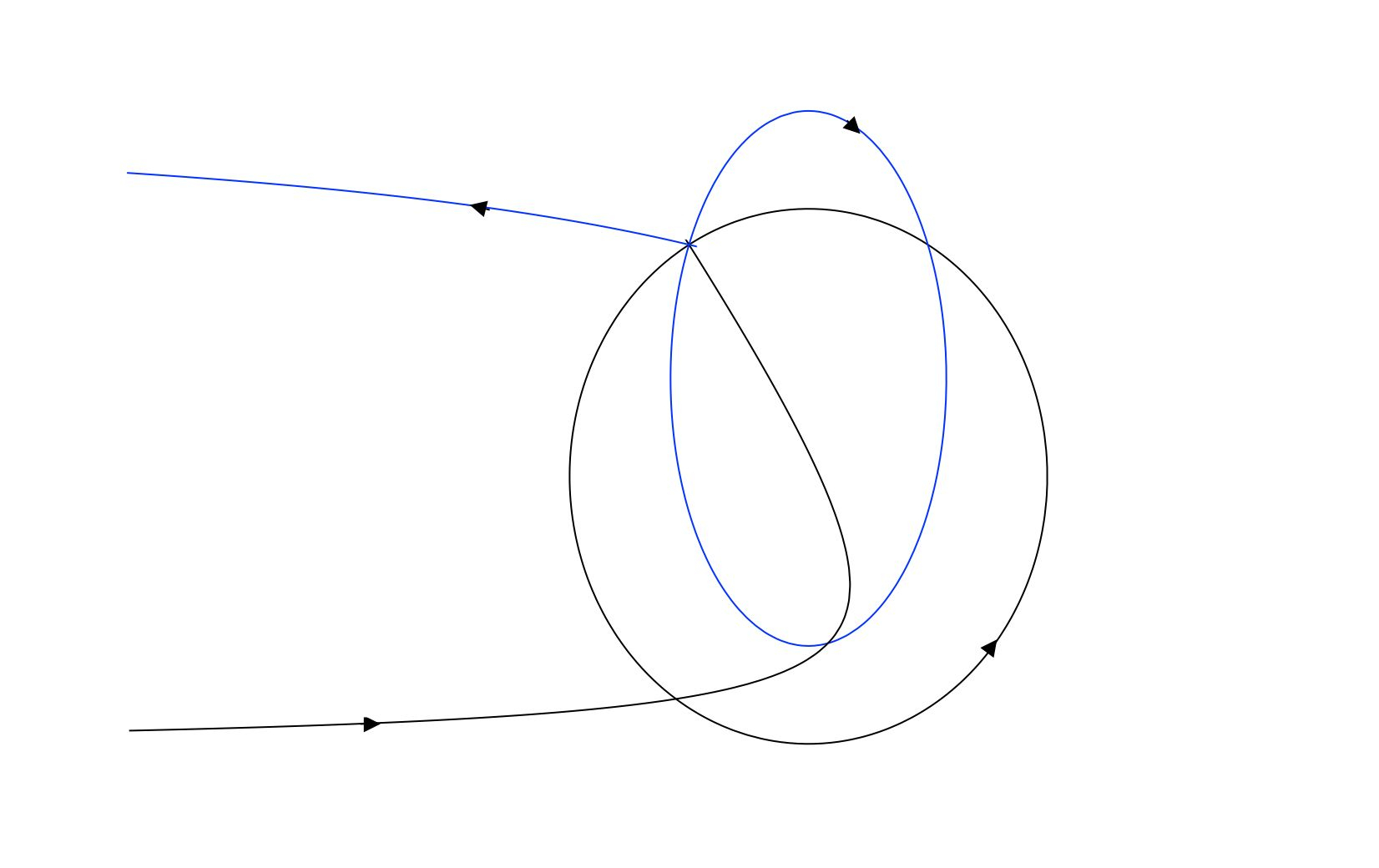}
\caption{Angular momentum transfer}
\end{center}
\label{fig:coll1}
\end{figure}

\begin{figure}[ht]
\begin{center}
\includegraphics[width=0.8\textwidth]{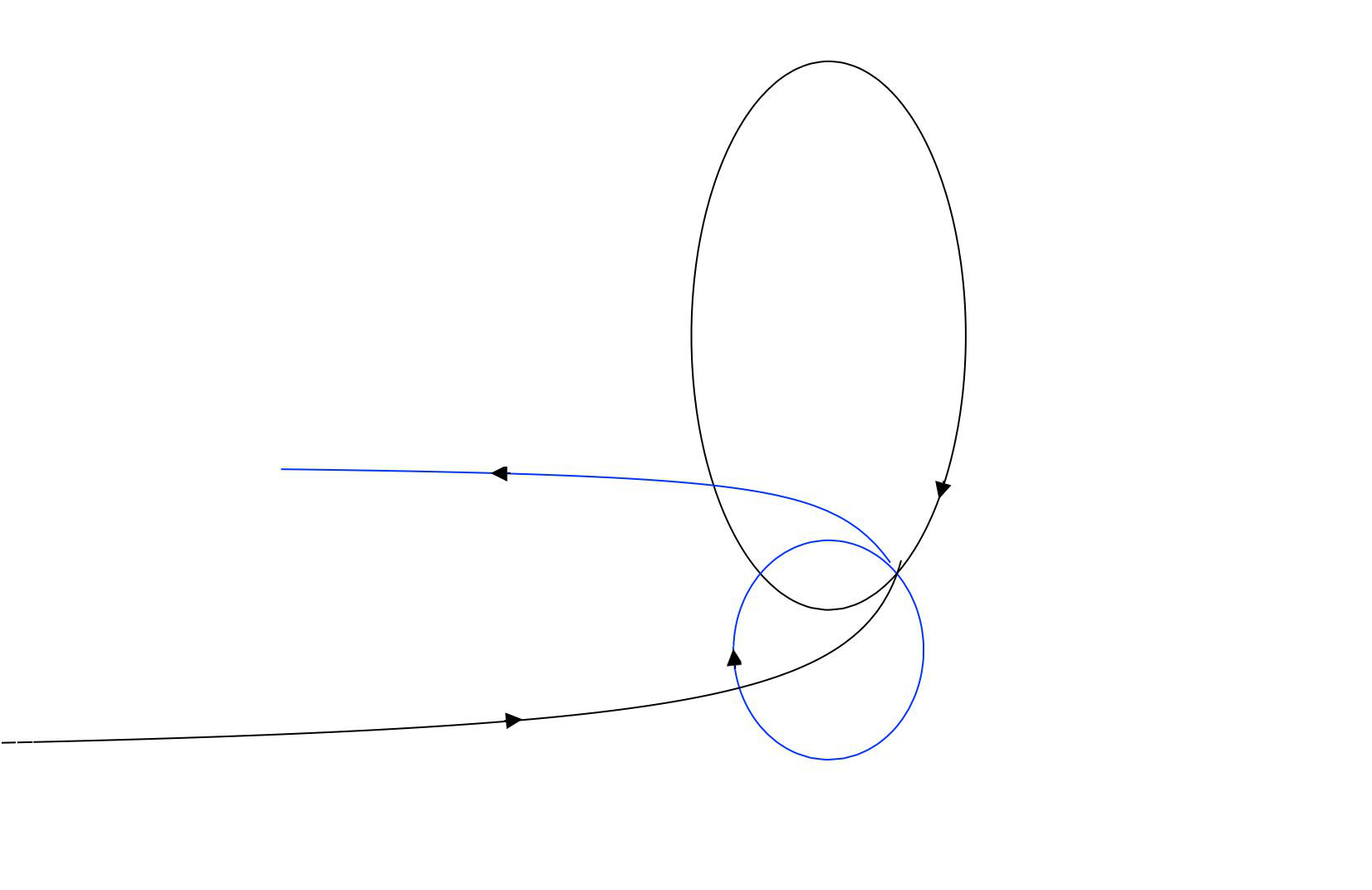}
\caption{Energy transfer}
\end{center}
\label{fig:coll1}
\end{figure}
\subsection{Gerver's model}\label{SGerver}
Following \cite{G2}, we discuss in this section the dynamics of the subsystem $Q_2,Q_3,Q_4$ in the limit case $\mu=0$ with $Q_1$ ignored.
We assume that
\begin{itemize}
\item $Q_3$ has elliptic motion and $Q_4$ has hyperbolic motion with focus $Q_2$;
\item $Q_3$ and $Q_4$ arrive at the correct intersection point of their orbits simultaneously (see Figure 1 and 2);
\item $Q_3$ and $Q_4$ do not interact unless they have an exact collision, and the collision is treated as elastic collision (energy and momentum are preserved).
\end{itemize}
The main conclusion is that
\begin{itemize}
\item the major axis of the elliptic motion is always kept vertical;
\item the incoming and outgoing asymptotes of the hyperbolic motion are always horizontal;
\item after two steps of the collision procedure, the ellipse has the same eccentricity as the ellipse before the first collision, but has a smaller semi-major axis (see Figure 1 and 2).
\end{itemize}

The interaction of $Q_3$ and $Q_4$ is desribed by the
elastic collision. That is,
velocities before ($-$) and after $(+)$ the collision are related by
\begin{equation}
\label{Elastic}
 v_3^+=\frac{v_3^-+v_4^-}{2}+\left|\frac{v_3^--v_4^-}{2}\right| n(\al), \quad
v_4^+=\frac{v_3^-+v_4^-}{2}-\left|\frac{v_3^--v_4^-}{2}\right| n(\al),
\end{equation}
where $n(\al)$ is a unit vector making angle $\al$ with $v_3^--v_4^-.$ The only free parameter $\al$ here is fixed by the condition that the outgoing asymptote of the traveling particle is horizontal.

We next introduce the {\it Gerver map} to formalize the above description. The Gerver map describes the parameters of the elliptic orbit change during the interaction of $Q_3$ and $Q_4.$
The orbits of $Q_3$ and $Q_4$ intersect in two points, of which we pick one (see Figure 1 and 2).
We use the subscript $j\in\{1, 2\}$ to describe the first or the second collision in Gerver's construction. Since $Q_1$ is ignored, we use only the orbit parameters $(E_3,\ell_3,e_3,g_3; e_4,g_4)$. The assumptions on the horizontal asymptotes of the traveler further remove $g_4$. Finally, at the intersection point of the elliptic and hyperbolic orbit, we get rid of one last variable $\ell_3$, so we only need to work with the variables $(E_3,e_3,g_3, e_4).$

With this in mind we proceed to define the Gerver map
$\Ger_{e_4, j, \omega}(E_3, e_3, g_3).$
This map depends on two discrete parameters $j\in \{1, 2\}$ and
$\omega\in\{3, 4\}.$ The role of $j$ has been explained above, and $\omega$ will tell us which particle will be the traveler
after the collision.
After colliding, the particles move independently. Thus $Q_3$ moves on an orbit with parameters $(\brE_3, \bre_3, \brg_3)$, and
$Q_4$ moves on an orbit with parameters
$(\brE_4, \bre_4, \brg_4).$

If $\omega=4$, we choose $\al$ so that after the exchange $Q_4$ moves on a hyperbolic
orbit with horizontal asymptote and let
$$ \Ger_{e_4, j, 4}(E_3, e_3, g_3)=(\brE_3, \bre_3, \brg_3).$$
If $\omega=3$ we choose $\al$ so that after the exchange $Q_3$ moves on a hyperbolic
orbit with horizontal asymptote and let
$$ \Ger_{e_4, j, 3}(E_3, e_3, g_3)=(\brE_4, \bre_4, \brg_4).$$

{\it In the following, to fix our notation, we always call the captured particle $Q_3$ and the traveler $Q_4$, i.e. we fix $\omega=4$.
}

We will denote the ideal orbit parameters in Gerver's paper \cite{G2} of $Q_3$ and $Q_4$ before the first (respectively second) collision with * (respectively **).
Thus, for example, $G_4^{**}$ will denote the negative angular momentum of $Q_4$ before the second collision.
The real values after the first (respectively, after the second) collisions are denoted with a $bar$ or $double\ bar$.


The following is the main result of \cite{G2} and plays a key role in constructing singular solutions.

\blm [\cite{G2}, Lemma 2.2 of \cite{DX}]
\label{LmGer}
Assume that the total energy of the $Q_2,Q_3,Q_4$ system is zero, i.e. $E_3+E_4=0$,
and fix the incoming and outgoing asymptotes of the hyperbola to be horizontal.
\begin{itemize}
\item[(a)]For $E_3^*=-\frac{1}{2}, g_3^*=\frac{\pi}{2}$ and for any
  $\ e_3^*\in (0,\frac{\sqrt 2}{2})$, there exist $e^*_4, e^{**}_4,\lb_0>1$ such that
$$ (e_3, g_3, E_3)^{**}=\Ger_{e_4^*, 1, 4}\left(e_3, g_3, E_3 \right)^*, \quad (e_3, -g_3, \lambda_0 E_3)^*=\Ger_{e_4^{**}, 2, 4}\left(e_3, g_3, E_3 \right)^{**}, $$
where $E_3^{**}=E_3^*=-\frac{1}{2},\ g_3^{**}=g_3^*=\frac{\pi}{2}$ and $e_3^{**}=\sqrt{1-e_3^{*2}}$.
\item[(b)] There is a constant $\brdelta$ such that if $(e_3,g_3,E_3)$
  lie in a $\brdelta$ neighborhood of $(e_3^*,g_3^*,E_3^*),$  then there exist smooth functions $e_4'(e_3, g_3),$ $e_4''(e_3, g_3),$ and $\lambda(e_3, g_3,E_3)$ such that
$$ e_4'(e_3^*, g_3^*)=e_4^*, \quad e_4''(e_3^*, g_3^*)=e_4^{**}, \quad \lambda(e^*_3, g^*_3,E^*_3)=\lb_0,$$
\begin{equation}\nonumber
\begin{aligned}
(\bre_3, \brg_3, \brE_3)&=\Ger_{e_4'(e_3, g_3), 1, 4}\left(e_3, g_3, E_3 \right),\\
(e^*_3, -g^*_3, \lb(e_3,g_3,E_3)E_3^*)&=\Ger_{e_4''(e_3, g_3), 2, 4}\left(\bre_3, \brg_3, \brE_3 \right). \end{aligned}\end{equation}
\item[(c)] 1-homogeneity in $E_3$: for any $\lambda >0$ and $(e_3,g_3,E_3)$ such that $\|(e_3,g_3,E_3/\lambda,e_4)-(e_3,g_3,E_3,e_4)^\dagger\|_{\infty}<\bar\dt$, with $\dagger=*,**$, we have
$$\pi_{E_3} \Ger_{e_4,j,4}(e_3,g_3, E_3)=\lambda\cdot\pi_{E_3} \Ger_{e_4,j,4}(e_3,g_3,E_3/\lambda),$$
where $\pi_{E_3}$ means the projection to the $E_3$ component, and $j=1,2$ corresponds to $*$, $**$.
\end{itemize}
\elm

Part (a) is the main content of \cite{G2}, which gives a two-step procedure to decrease the energy of the elliptic Kepler motion and maintain the self-similar structure (See Figure 1 and 2). We call the collision points in part (a) the {\it Gerver's collision points}, whose exact coordinates can be found in Appendix \ref{section: gerver}.
The results are summarized in Appendix \ref{section: gerver} with orbit parameters given explicitly. Part (b) says that once the ellipse gets deformed slightly away from the standard case in Figure 1 after the first collision, we can correct it by changing the phase of $Q_3$ slightly at the next collision to guarantee that the ellipse that we get after the second collision is standard. 

The notion of {\it angle of asymptote} above is clear since we only deal with the Kepler motion. We next introduce the explicit definition of angles of asymptotes, which are used in place of $g_4$ sometimes even when we deal with perturbed Kepler motion.
\begin{Not}[Angles of asymptotes]\label{NotAsymp}
In the following, we use $$\theta_4^-:=g_4-\arctan \frac{G_4}{L_4}$$
for the incoming $($superscript $-)$ asymptote of the $(x_4,v_4)$ motion and
\begin{equation}\label{EqAsymp}\theta_4^+:=\pi+g_4+\arctan \frac{G_4}{L_4}\end{equation}
for the outgoing $($superscript $+)$ asymptote. In Lemma \ref{LmGer}, we always have $\theta_4^-=0$ and $\theta_4^+=\pi.$ Geometrically, the angle is formed by the asymptote pointing to the direction of $x_4$'s motion and the positive $x_\parallel$ axis.  See Appendix \ref{appendixa} for a detailed discussion of the choice of the sign in front of $\arctan \frac{G_4}{L_4}$.
\end{Not}
\subsection{The local and global map, the renormalization and the domain}
\subsubsection{The Poincar\'{e} section and the Poincar\'{e} map}
\begin{Def}[The Poincar\'{e} section, the local map, the global map and the Poincar\'{e} map]\label{def: sct}
 We define a section
 $\{x_{4,\parallel}=-2\}$ on the zeroth energy level.
\begin{itemize}
 \item Following the Hamiltonian flow, to the right of this section, we define the local map
 \[\Loc:\ \{x_{4,\parallel}=-2,\ v_{4,\parallel}>0\}\to \{x_{4,\parallel}=-2,\ v_{4,\parallel}<0\},\]
 \item and to the left we define the global map \[\Glob:\ \{x_{4,\parallel}=-2,\ v_{4,\parallel}<0\}\to \{x_{4,\parallel}=-2,\ v_{4,\parallel}>0\}.\]
\item Finally, we define the Poincar\'{e} return map \[\mathcal P=\Glob\circ\Loc:\  \{x_{4,\parallel}=-2,\ v_{4,\parallel}>0\}\to \{x_{4,\parallel}=-2,\ v_{4,\parallel}>0\}.\]
   \end{itemize}
   \end{Def}
These maps $\Glob,\Loc, \cP$ are defined by the standard procedure following the Hamiltonian flow. Once we find one orbit going from one section to another, the corresponding map can be defined in a neighborhood of this orbit. The existence of a returning orbit  follows from Lemma \ref{LmChooseAM}.

\subsubsection{The renormalization map}
Next, we define the renormalization map $\cR$, which will be applied after two applications of the Poincar\'{e} map.
 We first fix a large number $\chi\gg 1$ which can be thought as a typical distance between the heavy
bodies $Q_1$ and $Q_2.$

\begin{Def}[The renormalization map]
\label{Def: renorm}
We define the renormalization map $\mathcal R$ in several steps as follows.

\begin{itemize}
\item Given a point $\bx$, called the base point, on  the section $\{x_{4,\parallel}=-2, $ $v_{4,\parallel}>0\}$, we denote by $\mathsf C(\bx)$ a cube of size $\frac{1}{2\sqrt{\lambda\chi}}$ centered at $\bx$, where 
 $\lb=-2E_3$ is measured at $\bx$. 

 Let $\beta=-\arctan\frac{x_{1,\perp}}{x_{1,\parallel}}$ evaluated at $\bx$, and denote by $\mathrm{Rot}(\beta)$ the rotation of the plane by angle $\beta$ around the origin. 

\item We push forward the cube $\mathsf C(\bx)$ to the section $$\{ (\mathrm{Rot}(\beta)^{-1}\cdot x_{4})_{\parallel}=\cos\beta x_{4,\parallel}+\sin \beta x_{4,\perp}=-2/\lb,\quad v_{4,\parallel}>0\}$$ along the Hamiltonian flow.  We define $$\tilde\Glob:\quad \{x_{4,\parallel}=-2,\ v_{4,\parallel}>0\}\to\{(\mathrm{Rot}(\beta)^{-1}\cdot x_{4})_{\parallel}=-2/\lb, \ v_{4,\parallel}>0\}$$and apply the following procedure to $\tilde\Glob(\mathsf C(\bx))$.
\item Rotation: we rotate the $x_\parallel$-axis around the origin by angle $\beta$,
so that for the center point in each cube, we have that $x_{1,\perp}$ is nearly zero $($to be estimated as $|x_{1,\perp}|=O(\mu/\chi)$, with the error caused by $\tilde\Glob)$. Now the section $\{(\mathrm{Rot}(\beta)^{-1}\cdot x_{4})_{\parallel}=-2/\lambda, \ v_{4,\parallel}>0\}$ becomes $\{x_{4,\parallel}=-2/\lambda, \ v_{4,\parallel}>0\}$.

\item Rescaling: we zoom in on the configuration space by $\lb>1$.
Simultaneously, we also slow down the velocities by dividing by $\sqrt \lambda.$ Now the section $\{x_{4,\parallel}=-2/\lambda, \ v_{4,\parallel}>0\} $ becomes $\{x_{4,\parallel}=-2, \ v_{4,\parallel}>0\} $.
\item Reflection: we reflect the whole system along the $x$-axis. 
\item Finally, we reset $\chi$ to be equal to the value $\lambda |x_{1,\parallel}|$ evaluated at $\bx$.
\end{itemize}

We have \[\cR:\  \tilde\Glob(\mathsf C(\bx))\left(\subset\{(\mathrm{Rot}(-\beta)\cdot x_{4})_{\parallel}=-2/\lambda,\ v_{4,\parallel}>0\}\right)\to  \left\{ x_{4,\parallel}=-2, \ v_{4,\parallel}>0\right\},\]
\begin{equation}\label{EqRenorm}
\begin{aligned}&
\cR(E_3,\ell_3,e_3,g_3;x_1,v_1;e_4,g_4)
=\\&\left( \frac{E_3}{\lb}, \ell_3,e_3,-(g_3-\beta);\lb \left[\begin{array}{cc}
1&0\\
0&-1
\end{array}\right] \mathrm{Rot}(\bt) x_1,\left[\begin{array}{cc}
1&0\\
0&-1
\end{array}\right] \frac{\mathrm{Rot}(\bt) v_1}{\sqrt{\lb}};e_4,-(g_4-\beta)\right). \end{aligned}\end{equation}
The renormalization also sends time $t$ to $\lb^{3/2}t$ and the Poincar\'e-Cartan invariant gets multiplied by $\lb^{1/2}$.
\end{Def}
\begin{Rk}
The primary goal of the definition of the renormalization map is to rescale the lower ellipse in Figure 2 to the size of the lower ellipse in Figure 1. The reflection is needed since the motions on the two ellipses have opposite orientations $($compare the arrows in Figure 1 and Figure 2$)$. The rotation is needed since we want to put $x_1$ on the horizontal axis, however,  $x_1$ has some angular momentum relative to $0$, hence $v_1$ forms an angle with $x_1$, which moves $x_1$ away from the horizontal axis.
\end{Rk}


We will iterate the map $ \cR\circ\tilde\Glob\circ (\Glob\circ \Loc)^2$: $\left\{ x_{4,\parallel}=-2, \ v_{4,\parallel}>0\right\}\to \left\{ x_{4,\parallel}=-2, \ v_{4,\parallel}>0\right\}$.


We shall show that for orbits of interest $\cR$ sends $\chi$ to $\lb\chi(1+O(\mu))$.
Thus $\chi$ will grow to infinity exponentially under iteration. Hence $\beta=O(\chi^{-1/2})$ decays exponentially to zero. 
 Without loss of generality we always assume in our estimates
that $1/\chi\ll \mu.$


\subsection{Asymptotics of the local and global map}


\subsubsection{The standing assumptions}

To simplify the presentation, we list standard assumptions that we will impose on the initial or final values of the local and global map respectively.

We introduce
$$K:=\sup\max\left\{\|d\Ger_{e_4,1,4}(e_3,g_3,E_3)\|_{\infty}+\left\|\frac{\partial\Ger_{e_4,1,4}}{\partial e_4}(e_3,g_3,E_3)\right\|_{\infty},\|d(e_4',e_4'')(e_3,g_3)\|_{\infty}\right\}+1$$
where the sup is taken over $\dagger = *,**$, and over all $(e_3,g_3,E_3,e_4)$ in a $\bar\delta$-neighborhood of $(e_3,g_3,E_3,e_4)^\dagger$, the maps $\Ger$ and $e_4',e_4''$ are in Lemma \ref{LmGer}, and the $\|\cdot\|_\infty$ norm for a linear map $M:\ \R^n\to \R^m$ is defined as $\sup\|Mv\|_{\infty}$, where the sup is taken among all $v\in \R^n$ with $\|v\|_{\infty}=1$.

We consider $0<\dt<\bar\dt/K^2$ and fix some large numbers $C_0, C_0'.$ For $\hat\lambda=1$ or $\lambda_0$ in Lemma \ref{LmGer}, we use the following standing assumption for the local map.

\textbf{AL}$(\hat\lambda)$:
{\it \begin{enumerate}
\item[(AL.3)] Initially on the section $\{x_{4,\parallel}=-2,\ v_{4,\parallel}>0\}$ we have
$$ \left\|(e_3,g_3-\sigma(\hat\lambda)\cdot \pi,E_3/\hat\lambda)-(e_3,g_3,E_3)^{\dagger}\right\|_{\infty}<K^\dagger\delta;$$
\item[(AL.1)] the initial values of $(x_1,v_1)$ satisfy\begin{equation*}
x_{1,\parallel} \leq-\chi  ,\quad |x_{1,\perp}|\leq C_0\mu,\quad |v_{1,\perp}|\leq C_0/\chi,\quad 0< -v_{1,\parallel}<C_0;\end{equation*}

\item[(AL.4)] the incoming and outgoing asymptotes of the nearly hyperbolic motion of $x_4,v_4$ satisfy
\[|\theta_4^-|\leq C_0\mu,\ \quad |\bar{\theta}_4^+-\pi|\leq \tilde\theta,\]
and the initial value of $e_4$ satisfies  $ |e_4-e_4^{\dagger}|<K^\dagger \delta,$\\
where \begin{itemize}
\item $\dagger =*,**$ and $K^*=1,\ K^{**}=K$;
\item $\tilde\theta\ll 1$ is a constant independent of $\chi,\mu$;
\item $\sigma:\ \{1,\lambda_0\}\to\{0,1\}$ is defined as $\sigma(1)=0$ and $\sigma(\lambda_0)=1$.
\end{itemize}
\end{enumerate}}

We use the following standing assumption for the global map.

\textbf{AG}$(\hat\lambda)$:
{\it \begin{enumerate}
\item[(AG.3)] Initially on the section $\{x_{4,\parallel}=-2,\ v_{4,\parallel}<0\}$, we have
\[\left\|(e_3,g_3-\sigma(\hat\lambda)\cdot \pi,E_3/\hat\lambda)-\Ger_{e_4^\dagger,i,4}(e_3,g_3,E_3)^\dagger\right\|_{\infty}<KK^\dagger\dt, \]
where $ \dagger=*,**$ and $i=1,2$ correspond to the first and second collisions;
\item[(AG.1)] the initial conditions of $x_1,v_1$ satisfy \begin{equation*}\label{Eqx1v1}
-1.1\chi\leq x_{1,\parallel} \leq  -\chi,\quad |x_{1,\perp}|\leq C'_0\mu,\quad |v_{1,\perp}|\leq C'_0/\chi,\quad \frac{1}{C_0'}< -v_{1,\parallel}<C'_0;\end{equation*}
\item[(AG.4)] on the section $\{x_{4,\parallel}=-2\}$, we have $|x_{4,\perp}|< C_0'$ holds both at initial and final moments.
\end{enumerate}}
If $\hat \lambda=1$, we abbreviate \textbf{AL}=\textbf{AL}$(1)$ and \textbf{AG}=\textbf{AG}$(1)$.

We stress that in both {\bf AL}$(\hat\lambda)$ and {\bf AG}$(\hat\lambda)$, we consider only orbits on the zeroth energy level and the zeroth total angular momentum level of the Hamiltonian \eqref{eq: hammain}.

\begin{Rk}
\begin{itemize}
\item In {\bf AL}$(\hat\lambda)$, we ask the initial values of $(x_3,v_3),(x_4,v_4)$ to be close to Gerver's value in Lemma \ref{LmGer}. The assumption on $(x_1,v_1)$ requires $Q_1$ to be far away and not to have too much energy. We also require the outgoing asymptote to be almost horizontal, which forces $Q_3$ and $Q_4$ to have a close encounter since otherwise $Q_4$ moves on a slightly perturbed hyperbola whose outgoing asymptote will not be nearly horizontal.
\item In {\bf AG}$(\hat\lambda)$, the main requirement is $(AG.4)$ where we require $|x_{4,\perp}|$ to be bounded at both the initial and final moments. This will force the motion of $(x_4,v_4)$ to be close to a horizontal free motion for most of the time.
\end{itemize}
\end{Rk}
\subsubsection{The asymptotes of the local, global maps}

In the next two lemmas, our notations are such that $\Loc,\Glob$ send {\it unbarred} variables to {\it barred} variables.

The next lemma shows that the real local map $\Loc$ is well approximated by the Gerver map $\Ger$ in the $\mathscr{C}^0$ sense. Its proof will be given in Section \ref{SSC0Loc}.
\blm
\label{LmLMC0}
Assume {\bf AL}$(\hat\lambda)$ with $\hat\lambda=1$ or $\lambda_0$.
Then after the application of $\Loc$, the following asymptotics hold uniformly
\[(\brE_3, \bre_3, \brg_3) =\Ger_{e_4}(E_3,e_3,g_3)+o(1)\]
as $1/\chi\ll\mu\to 0$ and $\tilde\theta\to 0$.
\elm

The next lemma deals with the $\mathscr{C}^0$ estimates for the global map $\Glob$.
\blm
\label{LmGMC0}
Assume {\bf AG}$(\hat\lambda)$ with $\hat\lambda=1$ or $\lambda_0$. Then there exist constants $C_3$ and $C_4$ such that after the application of $\Glob$ and $\tilde\Glob\circ\Glob$ the following estimates hold uniformly in $\chi, \mu$ as $1/\chi\ll\mu\to 0$
\begin{enumerate}
\item[(a)]$\left|\frac{\brE_3}{E_3}-1\right|\leq C_3\mu, \quad \left|\frac{\brG_3}{G_3}-1\right|\leq C_3\mu, \quad |\brg_3-g_3|\leq C_3\mu;$
\item[(b)]$|\theta_4^+-\pi|\leq C_4\mu,\quad |\bar\theta_4^-|\leq C_4\mu;$
\item[(c)] the return times defining $\Glob$ and $\tilde\Glob\circ\Glob$ are bounded by $3\chi$.
\end{enumerate}
\elm
The proof of this lemma is given in Section \ref{SSPfGMPM}. From now on we choose the constant $C_0$ in {\bf AL} to be larger than $C_4$ in Lemma \ref{LmGMC0}.
\subsubsection{Dynamics of $(x_1,v_1)$ under the renormalized Poincar\'e map}
The next lemma deals with the $\mathscr{C}^0$ estimates of $(x_1,v_1)$.  The proof is also in Section \ref{SSPfGMPM}.
\begin{Lm}\label{LmPMC0}Fix $\hat\lambda=1$, there exist constants $C_0, C_0', c_1,\bar c_1, C_1>0$ with $\bar c_1<C_0$,  such that the following holds.
Consider an orbit with initial condition $\bx$ satisfying
\begin{itemize}
\item[(i)] {\rm(AL.3)} and {\rm(AL.4)} satisfied when applying $\Loc$ for the first time, and {\rm(AG.4)} satisfied when applying $\Glob$ for the first time;
\item[(ii)]  initially on the section $\{x_{4,\parallel}=-2,\ v_{4,\parallel}>0\}$\begin{equation}\label{eq: x1v1}
G_0=0,\quad -\chi-\frac{1}{\sqrt{\chi}}\leq x_{1,\parallel}(0)\leq -\chi,\quad |x_{1,\perp}(0)|\leq \frac{1}{\sqrt{\chi}},
\quad -\bar c_1\leq v_{1,\parallel}(0)\leq -c_1.\end{equation}

\end{itemize}
Then we have
\begin{itemize}
\item[(a)] after the application of $\cP,$ ${\rm(AL.1)}$ is satisfied for $(x_1,v_1)$;
\item[(b)] after the application of $\Loc$ and $\Loc\circ \cP$ $($whenever the second $\Loc$ is defined$)$, ${\rm(AG.1)}$ is satisfied for $(x_1,v_1)$.
\end{itemize}
Assume ${\rm(i}')$ and ${\rm(ii)}$ in place of ${\rm(i)}$ and ${\rm(ii)}$ above, where
\begin{itemize}
\item[(i$'$)] {\rm(AL.3)} and {\rm(AL.4)} are satisfied when applying $\Loc$ for both the first and the second times, and {\rm(AG.4)} is satisfied when applying $\Glob$ for both the first and the second times.
\end{itemize}
Then we have
\begin{itemize}
\item[(c)]after the application of $\cR\circ\tilde\Glob\circ\cP^2$, where $\cR$ is based at the point $\cP^2(\bx)$, we get that the renormalized $\chi$, denoted by $\tilde \chi$, satisfies
$\lb(1+ C^{-1}_1\mu)\chi\leq \tilde\chi\leq \lb(1+ C_1\mu)\chi$, and the renormalized orbit parameters $G_0,  x_1,  v_1$ satisfy \eqref{eq: x1v1} with $\chi$ replaced by $\tilde\chi$.
\end{itemize}
\end{Lm}
\begin{Rk}\label{Rkx1v1} We explain the physical meaning of the lemma. The assumption implies that both $v_4$ and $v_3$ are of order $1$. By \eqref{eq: x1v1}, $v_1$ is also of order 1 and $v_{1,\perp}$ is bounded by $C/\chi$. In Remark \ref{RkCoord} we have stressed that $\mu v_1$ instead of $v_1$ is close to the velocity of $Q_1$. So $Q_1$ moves to the left with a velocity of order $\mu$ having a tiny vertical component. It takes $Q_4$ time of order $\chi$ to complete a return and during this time, $Q_1$ moves a distance of order $\mu\chi$. This gives the estimates of $x_{4,\parallel}$ and $\tilde\chi$ after renormalization. The energy exchange between $Q_1$ and $Q_4$ will change $v_{4,\parallel}$ significantly, but the renormalization will slow down $v_{4,\parallel}$ to the interval $[-\bar c_1, -c_1]$. The rotation in the renormalization controls $x_{4,\perp}$. 
\end{Rk}
\subsection{The tangent dynamics and the strong expansion}\label{SSProofMain}
\begin{Def}
Given $\delta<\frac{\bar\delta}{K^2}$ where $\bar\dt$ is in Lemma \ref{LmGer}, we define the following open sets in the section $\{x_{4,\parallel}=-2,\ v_{4,\parallel}>0\}$ on the zeroth energy level by
\begin{equation*}
\begin{aligned}
 U_1(\delta)&=\left\{
{\bf AL}, {\rm\ except\ the\ }\ \bar\theta_4^+ {\rm\ assumption\ therein},{\rm\ holds\ with\ }\dagger=*\right\}, \\
 U_2(\delta)&=\left\{{\bf AL}, {\rm\ except\ the\ }\ \bar\theta_4^+ {\rm\ assumption\ therein},{\rm\ holds\ with\ }\dagger=**\right\}, \\
 U_0(\delta)&=\left\{{\bf AL}(\lambda_0), {\rm\ except\ the\ }\ \bar\theta_4^+ {\rm\ assumption\ therein},{\rm\ holds\ with\ }\dagger=*\right\} .
 \end{aligned}
 \end{equation*}
\end{Def}
\begin{Rk}
\begin{enumerate}
\item The sets $U_j(\dt),\ j=1,2$ are neighborhoods of Gerver's collision points in Lemma \ref{LmGer}. The set $U_0$ is introduced to study the dynamics without the renormalization. 
\item In the definition
we do not restrict $\ell_3$ since $\ell_3$ can take any value in $[0,2\pi)$. We do not restrict $v_{1,\perp}$, since it can be bounded by $C/\chi$ by the information in \eqref{eq: x1v1}. We also get rid of the assumption on the final value $\bar\theta_4^+$ in {\bf AL}. 
\end{enumerate}
\end{Rk}
\subsubsection{The invariant cone fields}
We introduce the following cone fields.
\begin{Def}[Cone fields]\label{DefKone}
Let \begin{enumerate}
\item $\cK_1\subset T_{U_1(\dt)}(\R^7\times \T^3)$  be the set of vectors forming an angle less than a small number $\eta$ with span$(d\cR w_2,d\cR\tilde w)$,
\item $\cK_0\subset T_{ U_0(\dt)}(\R^7\times \T^3)$ be the set of vectors forming an angle less than a small number $\eta$ with span$(w_2,\tilde w)$, and
\item $\cK_2\subset T_{U_2(\dt)}(\R^7\times \T^3)$ be the set of vectors forming an angle less than $\eta$ with span$(w_1,\tilde w)$,
\end{enumerate} where
$$\tilde w=\frac{\partial}{\partial\ell_3},\quad\mathrm{and}\quad w_j=\left(\frac{\sqrt{e_4^2-1}}{L_3e_4}\frac{\partial }{\partial e_4}-\frac{1}{L_3e_4^2}\frac{\partial }{\partial g_4}\right),\quad j=1,2.$$
We choose our parameters to be $0<1/\chi\ll\mu\ll\dt\ll\eta\ll1$.
\end{Def}

The next lemma establishes the (partial) hyperbolicity of the Poincar\'{e} map.
\blm
\label{LmKone}
There exists a constant $c$
such that
for all $\bx\in U_1(\delta)$ satisfying $\mathcal{P}(\bx)\in U_2(\dt)$,
and for all $\bx\in U_2(\delta)$ satisfying $\mathcal{P}(\bx)\in  U_0(\dt)$, we have
\begin{itemize}
\item[(a)] $($Invariance$)$ $d\cP (\cK_1)\subset \cK_2$, $d\cP(\cK_2)\subset \cK_0$, $d(\cR\circ\tilde\Glob\circ \cP) (\cK_2)\subset \cK_1$, where the base point defining $\cR$ can be chosen to be any point in $U_0(\dt)$ since $\dt\ll\eta$.
\item[(b)] $($Expansion$)$ If  $v\in\cK_1$, then $\Vert d\cP(v)\Vert\geq c\chi \Vert v\Vert$.\\
If $v\in\cK_2$, then $\Vert d \cP(v)\Vert\geq c\chi \Vert v\Vert$ and $\Vert d(\cR\circ\tilde\Glob\circ \cP)(v)\Vert\geq c\chi \Vert v\Vert.$
\end{itemize}
\elm
We give the proof in Section \ref{sct: hyp}.
The next lemma follows directly from Definition \ref{DefKone}.
\blm
\label{LmAdm}
\begin{itemize}
\item[(a)] The vector $\tw=\frac{\partial}{\partial \ell_3}$ is in $\cK_i.$
\item[(b)] For any plane $\Pi$ in $\cK_i$ the projection map
$ \pi_{e_4, \ell_3}=(de_4, d\ell_3):\Pi\to\R^2$
is one-to-one.
\end{itemize}
\elm

\subsubsection{The admissible surfaces}
\begin{Def}[admissible surfaces]\label{DefAdm}
We call a two-dimensional $\mathscr{C}^1$ surface $S\subset U_i(\delta)$
{\rm admissible} if $TS\subset \cK_i$, $i=0,1,2$.
\end{Def}

Since Poincar\'e maps send admissible surfaces to admissible surfaces if the images lie in $U_j(\dt),\ j=1,2,$ by Lemma \ref{LmKone} and Lemma \ref{LmAdm}, we can restrict the Poincar\'e maps to admissible surfaces to obtain two-dimensional maps. The reduction is done as follows. We introduce two cylinder sets
$$ \mathcal C_0(\dt)=\mathcal C_1(\dt):=(e^*_4-\dt,e^*_4+\dt)\times \T^1,\quad \mathcal C_2(\dt)=(e^{**}_4-K\dt,e^{**}_4+K\dt)\times \T^1.$$
By Lemma \ref{LmAdm}, we get that each piece of admissible surface $S\subset U_i(\dt)$ is the graph of a function $\mathcal S$ defined on $\mathcal C_i(\dt)$, $i=0,1,2$. So we get that $\cP\circ\mathcal S$ is a function of two variables $(e_4,\ell_3)$. However, for most points in $\mathcal C_i(\dt)$, the map $\cP\circ\mathcal S$ is not defined since the orbit might not return. 

Given a piece of admissible surface $S\subset U_j(\dt)$, we next introduce the maps $\mathcal Q_1$,$\mathcal Q_2$ and $\mathcal Q_0$ from a subset of $\mathcal C_{1}(\dt)$ to $\mathcal C_{2}(\dt),$ a subset of $\mathcal C_{2}(\dt)$ to $\mathcal C_{0}(\dt),$ and a subset of $\mathcal C_{0}(\dt)$ to $\mathcal C_{2}(\dt)$ respectively:
$$\mathcal Q_1:=\pi_{e_4,\ell_3}\cP(\mathcal S(\cdot, \cdot)),\quad \mathcal Q_2:=\pi_{e_4,\ell_3}\mathcal \cP(\mathcal S(\cdot, \cdot)),\quad \mathcal Q_0:=\pi_{e_4,\ell_3}\cP\cR\tilde\Glob(\mathcal S(\cdot, \cdot)).$$
where the base point of $\cR$ in $ \mathcal Q_0$ will be specified below.
The domain of $\mathcal Q_1$ can be found by taking $\mathcal Q_1^{-1}(\mathcal C_{2}(\dt))\cap \mathcal C_{1}(\dt)$, and similarly for $\mathcal Q_2$ and $\mathcal Q_0$.
This completes the reduction of the Poincar\'e maps to two-dimensional maps. 
\begin{Def}[Essential admissible surfaces]\label{DefEssAdm}
For $\dt'<\dt$, we call an admissible surface $S\subset U_j(\dt)$ $\dt'$-{\rm essential} if $\pi_{e_4, \ell_3}S$ contains $\mathcal C_j(\dt')$, $j=1,2$.
\end{Def}

\blm
\label{LmChooseAM}
Given $0<\dt'<\dt\leq \bar\dt/K^2$, we have the following for $\chi$ sufficiently large.
\begin{itemize}
\item[(a)] Given a $\dt'$-essential admissible surface $S\subset U_1(\delta)$, and $\te_4\in (e^*_4-\dt'+\frac{1}{\chi}, e^*_4+\dt'-\frac{1}{\chi})$ there exists $\tilde{\ell}_3$, such that $\pi_{e_4}\cP\mathcal S(\te_4,\tilde\ell_3)=e_4^{**}$. Moreover, there exists a neighborhood  $V_1(\te_4)(\subset \mathcal C_1(\dt'))$ of $(\tilde e_4,\tilde \ell_3)$ of diameter $O(1/\chi)$, such that $\mathcal Q_1$ maps $V_1(\te_4)$ surjectively to $\mathcal C_2(\dt)$. 
\item[(b)] Given a $\dt'$-essential admissible surface $S\subset U_2(\delta)$  and $\te_4\in (e^{**}_4-K\dt'+\frac{1}{\chi}, e^{**}_4+K\dt'-\frac{1}{\chi})$ there exists $\tilde{\ell}_3$, such that $\pi_{e_4}\cP\mathcal S(\te_4,\tilde\ell_3)=e_4^{*}$. Moreover, there exists a neighborhood  $V_2(\te_4)(\subset \mathcal C_2(\dt'))$ of $(\tilde e_4,\tilde \ell_3)$ of diameter $O(1/\chi)$, such that $\mathcal Q_2$ maps $V_2(\te_4)$ surjectively to $\mathcal C_0(\dt)$.
\item[(c)] Given a $\dt'$-essential admissible surface $S\subset U_0(\delta)$ and $\te_4\in (e^*_4-\dt'+\frac{1}{\chi}, e^*_4+\dt'-\frac{1}{\chi})$ there exists $\tilde{\ell}_3$, such that $\pi_{e_4}\cP\mathcal S(\te_4,\tilde\ell_3)=e_4^{**}$. Moreover, defining the renormalization $\cR$ based at the point $\mathcal S(\te_4,\tilde\ell_3)$, there exists a neighborhood $V_0(\te_4)(\subset \mathcal C_0(\dt'))$ of $(\tilde e_4,\tilde \ell_3)$ of diameter $O(1/\chi)$, such that $\mathcal Q_0$ maps $V_0(\te_4)$ surjectively to $\mathcal C_2(\dt)$. 
\item[(d)]For points in $V_i(\te_4)$ from parts $(a)$ and $(b)$ $(i=0,1,2)$, there exist constants $c,\mu_0,\chi_0$, such that for $\mu<\mu_0$ and $\chi>\chi_0$, we have that the particles $Q_3$ and $Q_4$ avoid collisions before the next return and the minimal distance $d$ between $Q_3$ and $Q_4$ satisfies $c\mu\leq d \leq \frac{\mu}{c}.$ Moreover, $Q_1$ and $Q_4$ do not collide.
\end{itemize}
\elm
Part (a), (b) and (c) of the lemma are proved in Section \ref{SSAngMom}. Part (d) is given in Section \ref{ssct: nocollision} as well as Lemma \ref{Lm: landau}(b).
\subsection{Proof of the main Theorem \ref{ThMain}}

{\bf Step 1, Concatenating Lemma \ref{LmLMC0}, \ref{LmGMC0} and \ref{LmPMC0}.}

We will iterate $\cR\circ\tilde\Glob\circ\Glob\circ\Loc\circ\Glob\circ\Loc$. Suppose we have a point $\bx\in U_1(\dt)$ whose images $\cP(\bx)\in U_2(\dt)$ and $\cR\tilde\Glob\cP^2(\bx)\in U_1(\dt)$, where $\cR$ is defined with the base point $\cP^2(\bx)$. We assume in addition that \eqref{eq: x1v1} is satisfied for $\bx$. Leaving the existence of such a point to be addressed later, we first show how the assumptions of Lemma \ref{LmLMC0}, \ref{LmGMC0} and \ref{LmPMC0} are satisfied.

The assumption {\bf AL} (except the $\theta^+_4$ assumption therein) for Lemma \ref{LmLMC0} is satisfied since $\bx\in U_1(\dt)$.
To proceed, we pick some small $\tilde\theta$ and assume $|\bar\theta^+_4-\pi|<\tilde\theta$ in (AL.4) is satisfied.

The conclusion of Lemma \ref{LmLMC0} combined with Lemma \ref{LmGer} implies (AG.3) by choosing $\mu$ and $\tilde\theta$ sufficiently small, and the assumption that $\bx\in U_1(\dt)$ and $\cP(\bx)\in U_2(\dt)$ implies (AG.4).  Next the assumptions of Lemma \ref{LmPMC0} for the first application of $\cP$ are satisfied, so we get (AG.1). Now the assumption {\bf AG} is satisfied.

 Now we apply Lemma \ref{LmGMC0} to conclude that $E_3,G_3,g_3$ have $O(\mu)$-oscillations and the initial and final angles of asymptotes are $O(\mu)$ close to $0$ and $\pi$ respectively. By choosing $\mu$ small, we see that the $\theta_4^+$ assumption in (AL.4) is automatically satisfied. That is to say, if $\Glob$ is applicable after the application of $\Loc$, then the $\theta_4^+$ assumption in (AL.4) is redundant.

 Next we consider the second application of $\Loc$. In the first application of $\Loc$, we see that $\Loc$ is approximated by $\Ger$ by Lemma \ref{LmGer}. Next, the application of $\Glob$ gives only a $\mu$-oscillation to the values of $E_3,g_3,e_3$, so applying Lemma \ref{LmGer}, we see that (AL.3) is satisfied for the second application of $\Loc$.
 The $\theta_4^-$ and $e_4$ parts of  (AL.4) are satisfied since we have $\cP(\bx)\in U_2(\dt)$.
 Lemma \ref{LmPMC0} implies that (AL.1) is satisfied. The only missing assumption in (AL.4) is the assumption on the outgoing angle of asymptote $\bar\theta_4^+$, which is again redundant under the assumption $\cR\tilde\Glob\cP^2(\bx)\in U_1(\dt)$. 

 Now we can apply Lemma \ref{LmLMC0} for the second time. Similarly, we verify the assumption for the second application of $\Glob$. After $\cR\tilde\Glob\cP^2$, the assumption (AL.1)  and \eqref{eq: x1v1} are provided by part (c) of Lemma \ref{LmPMC0}. The assumption (AL.3) and $e_4$ part of (AL.4) is provided by Lemma \ref{LmGer} and the renormalization applied to $E_3$. The assumptions on the angles of asymptotes in (AL.4) are again given by the existence of returning orbits, to be addressed below. So we can apply Lemma \ref{LmLMC0} for the third time.




{\bf Step 2, choosing the initial piece of admissible surface}

We choose a number $\dt'<\dt/K^2$. 
 Let $S_0$ be a $\dt'$-essential admissible surface in $U_1(\dt)$;
then by Definition \ref{DefKone} and Definition \ref{DefAdm} on $S_0$ we have
\begin{equation}\label{EqOsc3}
|E_3-E^*_3|<\dt'+\eta\dt,\quad |e_3-e^*_3|<\dt'+\eta\dt, \quad |g_3-g^*_3|<\dt'+\eta\dt,\end{equation}
where $\eta$ is the small number in Definition \ref{DefKone}. Here we choose $\eta$ so small that $\dt'+\eta\dt<\dt$. Such a piece of $\dt'$-essential admissible surface $S_0$ exists by explicit construction as follows. We first take an integral curve in $ U_1(\dt')$ along the vector field $w_1$ in Definition \ref{DefKone} such that its $e_4$ component is the interval $(e^*_4-\dt',e^*_4+\dt')$. Then the surface $S_0$ can be chosen as the product of the curve with $\T^1(\ni\ell_3)$.

{\bf Step 3, Noncollision singularities.}

We wish to construct a singular orbit with initial value in $S_0$.  We define $S_i$ inductively so that $S_1$ is a $\delta'$-essential component of ${\mathcal P}(S_0)\cap U_2(\delta)$, and for $i\ge 2$, $S_i$ is a $\delta ^\prime$-essential component of $({\mathcal P\mathcal R}\tilde {\mathbb G}{\mathcal P})(S_{i-1})\cap U_2(\delta)$ (we shall show below that such components exist).  Given a $\delta ^\prime$-essential admissible surface $S_i\subset U_2(\delta)$, choose $\tilde e_4\in (e_4^{**}-K\delta^{\prime}+\frac 1{\chi},e_4^{**}+K\delta^{\prime}-\frac 1{\chi})$.  Then the hypothesis of Lemma 2.21(b) is satisfied, so there exist $\tilde \ell _3$ and $V_{2,i}(\tilde e_4)$ satisfying Lemma 2.21(b).  In particular, $V_{2,i}(\tilde e_4)$ is a subset of ${\mathcal C}_2(\delta^{\prime})$ with diameter $O(\mu /\chi)$, and $(\tilde e_4,\tilde \ell _3)\in V_{2,i}(\tilde e_4)$.  It follows that for every $(e_4,\ell _3)\in V_{2,i}(\tilde e_4)$, we have $e_4\in (e_4^{**}-K\delta^{\prime},e_4^{**}+K\delta^{\prime})$.  In fact this is true for every $(e_4,\ell _3)$ in $\bar V_{2,i}(\tilde e_4)$, the closure of $V_{2,i}(\tilde e_4)$.  Thus $\bar V_{2,i}(\tilde e_4)\subset {\mathcal C}_2(\delta ^{\prime})$, and ${\mathcal S}_i$ is defined on $\bar V_{2,i}(\tilde e_4)$.  Let $\hat S_i={\mathcal S}_i(\bar V_{2,i}(\tilde e_4))$.  Then, because ${\mathcal S}_i$ is a continuous bijection, $\hat S_i$ is closed.  Also, because $S_i={\mathcal S}_i({\mathcal C}_2(\delta^{\prime}))$, we have $\hat S_i\subset S_i$.  Likewise, $({\mathcal P\mathcal R}\tilde {\mathbb G}{\mathcal P})^{-1}(\hat S_i)$ is a closed subset of $({\mathcal P\mathcal R}\tilde {\mathbb G}{\mathcal P})^{-1}(S_i)$.  We shall show below that $({\mathcal P\mathcal R}\tilde {\mathbb G}{\mathcal P})^{-1}(S_{i+1})\subset \hat S_i$.  It follows by induction on $i$ that \begin{equation}
\bigl \{ {\mathcal P}^{-1}({\mathcal P\cR}\tilde {\mathbb G}{\mathcal P})^{-i}\hat S_{i+1}\bigr \} _{i=0}^{\infty}
\nonumber
\end{equation}
is a family of nested non-empty sets, whose intersection $X$ is therefore non-empty.  Choose any $\boldsymbol{x}\in X$.  (In fact, $X$ has only one element, but we do not need to use that fact.) We claim that $\bx$ has a singular orbit. We define $t_i$ as the time of the orbit's $2i$-th visit to the section $\{x_{4,\parallel}=-2,\ x_{4,\parallel}>0\}.$
By Lemma \ref{LmGer} and Lemma \ref{LmLMC0}, the rescaled energy is close to Gerver's values in Lemma \ref{LmGer} and the rescaling factor $\lambda_0+\tdelta\geq \lambda\geq \lambda_0-\tdelta>1$ where $\lambda_0$ is in Lemma \ref{LmGer} and $\tdelta=\dt'+\eta\dt$, so  the unrescaled energy of $(x_3,v_3)$ satisfies $$\frac{1}{2}(\lambda_0-\tdelta)^{i-1}\leq -E_3(t_i)\leq \frac{1}{2}(\lambda_0+\tdelta)^{i-1}.$$
According to Lemma \ref{LmGMC0}, Lemma \ref{LmPMC0} and the total energy conservation (see Lemma \ref{LmHamR} below for the Hamiltonian), we get that the velocity $|v_4|$ during the $i$-th iteration is bounded from below by $c\sqrt{-E_3(t_i)}\geq c (\lambda_0-\tdelta)^{(i-1)/2}.$
Note that in Step 1
the initial conditions for $x_1,v_1$ are chosen to satisfy the assumption \eqref{eq: x1v1}.
Lemma \ref{LmPMC0} then shows that the assumptions on $x_1,v_1$ are always satisfied.
Thus we can iterate Lemma \ref{LmPMC0} for arbitrarily many steps.

Now let us look at the orbit in the physical space without doing any renormalization. Inductively, we have
\[x_{1,\parallel}(t_i)\in \left[-(1+\mu C_1)^{i-\frac12}\chi_0, -(1+\mu  C^{-1}_1)^{i-\frac12} \chi_0\right] \]
after the $i$-th iteration using part (c) of Lemma \ref{LmPMC0}, where $\chi_0$ is the initial value for $\chi$. Therefore, $x_{1,\parallel}\to -\infty$ as $n\to \infty$. The value of $\chi$ used during each step of $\cP\cR\tilde\Glob\cP$, denoted by $\chi_{i}(=\frac12 x_{1,\parallel}(t_i)/E_3(t_i))$, is estimated as
$$(\lb_0-\tilde\dt)^{i-1}(1+\mu  C^{-1}_1)^{i-\frac12} \chi_0\leq \chi_{i}\leq (\lb_0+\tilde\dt)^{i-1}(1+\mu  C_1)^{i-\frac12} \chi_0.$$

Next, for each application of $\Loc$, the total time is bounded by a uniform constant. For each application of $\Glob$, the return time is bounded by $3\chi_{i}$ by Lemma \ref{LmGMC0}(c). So without the renormalization, the time difference
 \[|t_{i+1}-t_i|\leq C(\lambda_0-\tdelta)^{-3i/2}\cdot (\lb_0+\tilde\dt)^{i}(1+\mu  C_1)^i \chi_0\]
 where the constant $C$ absorbs finite powers of $(\lambda_0\pm\tilde\dt)$ and $(1+\mu C_1)$,
so the total time $t_*=\lim_{i\to\infty} t_i$ is bounded as needed.
This shows that infinitely many steps complete within finite time and $x_1$ goes to infinity. Since $\mu$ is small and in $U_{j}(\dt),\ j=1,2$ both $x_3$ and $x_4$ are bounded, from \eqref{eq: right} we see that $q_1$ also goes to infinity. This implies that both $Q_1$ and $Q_2$ escape to infinity since $q_1=Q_1-Q_2$ and the mass center is fixed. We also have that $Q_3$ escapes to infinity since $Q_3$ is always close to $Q_2$, i.e. $q_3$ is bounded. Finally, $Q_4$ travels between $Q_1$ and $Q_2$. To see that no collision occurs during the whole process, we only examine the $Q_3$-$Q_4$ and $Q_1$-$Q_4$ close encounters whose collisions are excluded by part (d) of Lemma \ref{LmChooseAM} (See Section \ref{ssct: nocollision} for more details).

The symbolic dynamics in the statement of the main theorem is due to the fact that we can switch the roles of $Q_3$ and $Q_4$ after their close encounter. For elastic collisions, such a switch is done by replacing $\al$ by $\pi-\al$ in \eqref{Elastic}. Both cases ($\al$ and $\pi-\al$) can be shadowed by Kepler hyperbolic motion when $\mu>0$. See \cite{G3} for more details. In the above we have been fixing the discrete parameter $\omega=4$ in the definition of Gerver's map $\Ger_{e_4,j,\omega}$, i.e. we have been choosing $Q_4$ as the traveler. In this case, the global map $\Glob$ sends points in $U_j(\dt)$ to points in $U_{3-j}(\dt)$, $j=1,2$, so we see that $Q_4$ winds around $Q_1$ once in the sense of Definition \ref{DefSing}. In general it is the traveler $Q_{\omega}$ that winds around $Q_1$.

{\bf Step 4, The induction steps. }

It remains to show that we can find a $\dt'$-essential component of $S_{i}$ inside $U_2(\delta)$. 

We proceed inductively, and assume $S_{i}\subset U_2(\dt)$ for $i\geq 2$ is a $\dt'$-essential admissible surface after application of $(\cP\cR\tilde\Glob\cP)^{i-1}$ to a subset of $S_1$. The fact that the $\dt'$-essential admissible surface $S_1$ exists follows from Lemma \ref{LmChooseAM}(a) applied to $S_0$. Indeed, $S_1=\mathcal P\mathcal S(V_1(\tilde e_4)).$



We next apply Lemma \ref{LmChooseAM}(b). For given $\tilde e_4\in (e^{**}_4-K\dt'+\frac{1}{\chi}, e^{**}_4+K\dt'-\frac{1}{\chi})$, there exists $\tilde\ell_3\in \T^1$ such that $\pi_{e_4} \cP(\mathcal S_{i}(\tilde e_4,\tilde \ell_3))=e_4^{*}$. 
Moreover,
there exists  a neighborhood $V_{2,i}(\tilde e_4)$ of $(\tilde e_4,\tilde \ell_3)$ such that  $\mathcal Q_2$ maps $V_{2,i}(\tilde e_4)$  surjectively onto $\mathcal C_0(\dt')$. We denote by $S_{i+1/2}$ the image $\cP(\mathcal S_i(V_{2,i}(\tilde e_4)))$, which is admissible by Lemma \ref{LmAdm} and $\dt'$-essential. Moreover, every point $\bx\in S_{i+1/2}$ satisfies (AL.3) by Lemma \ref{LmLMC0} and Lemma \ref{LmGMC0}(a), (AL.1) by Lemma \ref{LmPMC0} and (AL.4) by Lemma \ref{LmGMC0}(b) and the $\dt'$-essentiality implies that $\pi_{e_4}S_{i+1/2}=(e_4^*-\dt', e_4^*+\dt'); $ thus $S_{i+1/2}\subset U_0(\dt)$.

We next apply  Lemma \ref{LmChooseAM}(c)\ 
 to find a point $(e^*, \tilde{\tilde\ell}_3)$  such that $\pi_{e_4}\cP\mathcal S_{i+1/2}(e^*_4,\tilde{\tilde \ell}_3)=e^{**}_4$. Now we introduce the renormalization $\cR$ based at the point $\mathcal S_{i+1/2}(e^*_4,\tilde{\tilde \ell}_3)$ and define $\mathcal Q_0$. Again by Lemma \ref{LmChooseAM}(c), there exists a neighborhood $V_{0,i+1/2}(e^*_4)\subset \mathcal C_0(\dt)$ such that $\mathcal Q_0$ maps $V_{0,i+1/2}(e^{*}_4)$ surjectively onto $\mathcal C_2(\dt')$, and we have that the diameter of $V_{0,i+1/2}(e^*_4)$ is $O(\mu/\chi)$, which is much smaller than $1/\sqrt\chi$, the size of the domain of $\cR$, so $\cR$ is well-defined on $\tilde \Glob(V_{0,i+1/2}(e^*_4))$ ($d\tilde \Glob$ is bounded).



 We simply define $$S_{i+1}:=\{\bx\in\cP\cR\tilde \Glob( \mathcal S_{i+\frac12}(V_{0,i+1/2}( e_4^*)))\ :\  \pi_{e_4,\ell_3}\bx\in \mathcal C_2(\dt')\}.$$
By Lemma \ref{LmKone}(a), we know that $S_{i+1}$ is admissible. 

By Definition \ref{DefEssAdm}, we know that $S_{i+1}$ is $\dt'$-essential. On $S_{i+1}$, we always have \eqref{EqOsc3} for the variables $E_3,e_3, g_3$ by Definition \ref{DefKone} and Definition \ref{DefAdm}.  The variables $(x_1,v_1)$ are always controlled by part (c) of Lemma \ref{LmPMC0} and the angle $g_4$ is controlled by part (b) Lemma \ref{LmGMC0} using the asymptotes of the hyperbolic motion (The bounds on $\theta_4^-$ and $\bar\theta_4^+$ require that $Q_4$ has a near collision with $Q_3$, which constrains $g_4$). So we only need to deal with the variables $(e_4,\ell_3)$ on a $\dt'$-essential admissible surface.

Finally, to prove that $({\mathcal P\cR}\tilde {\mathbb G}{\mathcal P})^{-1}(S_{i+1})\subset \hat S_i$, we note that
\begin{equation*}
\begin{aligned}
({\mathcal P\cR}\tilde {\mathbb G}{\mathcal P})^{-1}(S_{i+1})&= \bigl \{ \boldsymbol{x}\in {\mathcal P}^{-1}({\mathcal S}_{i+1/2}(V_{0,i+1/2}(e_4^*))):\pi _{e_4,\ell _3}({\mathcal P\cR}\tilde {\mathbb G}{\mathcal P})(\boldsymbol{x})\in {\mathcal C}_2(\delta ^{\prime})\bigr \} \nonumber \\
&\subset {\mathcal P}^{-1}({\mathcal S}_{i+1/2}(V_{0,i+1/2}(e_4^*)))\subset {\mathcal P}^{-1}({\mathcal S}_{i+1/2}({\mathcal C}_2(\delta ^{\prime}))) \nonumber \\
&={\mathcal P}^{-1}(S_{i+1/2})={\mathcal S}_i(V_{2,i}(\tilde e_4))\subset {\mathcal S}_i(\bar V_{2,i}(\tilde e_4))=\hat S_i\mbox. \nonumber
\end{aligned}
\end{equation*}


\qed


\section{The hyperbolicity of the Poincar\'{e} map}\label{sct: hyp}
In this section, we consider the hyperbolicity of the Poincar\'{e} map by studying the derivative of the local and global maps.
\subsection{The structure of the derivative of the global map and local map}
\begin{Lm}\label{Lm: loc}
Suppose the initial condition $\bx_\mu\in U_j(\delta)$ satisfies {\bf AL} and let $\hat \bx$ be the closest point to $\bx_\mu$ in the set $U_j(\delta)$ with $\mu=0$, which leads to a collision between bodies 3 and 4. 
Then
\begin{enumerate}
\item[(a)]  there exist a constant $C$ independent of $\mu$,
continuous vector-valued functions $\lin_j(\hat\bx)$ and $\mathbf u_j(\hat\bx,\bar\theta_4^+)$ and a continuous matrix-valued function $B_j(\hat\bx,\bar\theta_4^+)$, where $\bar\theta_4^+$ is the angle of the outgoing asymptote, such that $\|\mathbf u_j\|, \|\mathbf \lin_j\|,\|B_j\|<C$ and as $1/\chi\ll\mu\to 0$
$$ d\Loc(\bx_\mu)=\frac{1}{\mu} (\bv_j(\hat\bx,\bar\theta_4^+)+o(1))\otimes (\lin_j(\hat\bx) +o(1))+B_j(\hat\bx,\bar\theta_4^+)+o(1); $$
\item[(b)] moreover, there exist vectors $\hlin_j$ and $\hu_j$, and a matrix $\hat B_j$ such that
$$ \lin_j(\hat\bx)\to \hlin_{j},  \quad \bv_j(\hat\bx, \bar\theta_4^+)\to \hu_{j}, \quad B_j(\hat\bx, \bar\theta_4^+)\to \hat B_j, \text{ as } \delta,\tilde\theta\to 0,
$$
\end{enumerate}
where $j=1,2$, meaning the first or the second collision.
\end{Lm}
The proof is given in Section \ref{SSLocC1}.
\begin{Lm}\label{Lm: glob}
Let $\bx$ and $\by=\Glob(\bx)$ be the initial and final values of the global map $\Glob$ and satisfy {\bf AG}.
 Then
 \begin{enumerate}
\item[(a)]   there exist continuous linear functionals $\brlin_j(\bx)$ and $\brrlin_j(\bx)$, continuous vectorfields $\brv_j(\by)$ and $\brrv_j(\by)$, and nonvanishing constants $c_1$ and $c_2$ such that as $1/\chi\ll\mu\to 0$
$$ d\Glob(\bx)=c_1\chi^2 (\brv_j(\by)+o(1)) \otimes(\brlin_j(\bx)+o(1))+c_2\chi (\brrv_j(\by)+o(1)) \otimes(\brrlin_j(\bx)+o(1))+O(\mu\chi). $$
\item[(b)] Moreover, we have the following explicit expressions for the above vectors and functionals in Delaunay coordinates $($see $\mathcal V$ in Notation \ref{NotVX}$)$ 
\[\brrlin_j=(1,0_{1\times 9}),\ \brlin_j=\left(-\frac{\tilde G_{4,j}/\tilde L_{4,j}}{\tilde L_{4,j}^2+\tilde G_{4,j}^2}, 0_{1\times 7},\frac{1}{\tilde L_{4,j}^2+\tilde G_{4,j}^2} , \frac{1}{\tilde L_{4,j}}\right),\]
\[ \brrv_j=\tilde w:=(0,1,0_{1\times 8})^T,\quad \brv_j=w_j+c\tw,\quad  w_j=\left(0_{1\times 8};1,\frac{\hat L_{4,j}}{\hat L_{4,j}^2+\hat G_{4,j}^2}\right)^T,\] where $c$ is some constant, $\tilde L_{4,j}$ and $\tilde G_{4,j}$ stand for the initial values of $L_{4,j},G_{4,j}$ and $\hat L_{4,j},\hat G_{4,j}$ stand for the final values, and $\tilde w$ and $ w_j$ appeared in Definition \ref{DefKone} with $L_4=L_3$ when $\mu=0$ and $e_4=\sqrt{1+(G_4/L_3)^2}$.
\end{enumerate}
\end{Lm}
\begin{Not} We denote by $\hat{\brrlin}_j$ and $\hat\brlin_j$ the vectors corresponding to $\brrlin_j$ and $\brlin_j$ respectively when $L_{4,j},G_{4,j}$ are evaluated at Gerver's collision points.
\end{Not}
The proof of Lemma \ref{Lm: glob} is in Appendix \ref{appendixb}. The leading terms in the above two lemmas have clear physical meanings explained in \cite{Xu} by simple arguments.
\subsection{The nondegeneracy condition}
\blm
\label{LmNonDeg}
The following non degeneracy conditions are satisfied for $E_3^*=-\frac{1}{2}$, $e^*_3=\frac{1}{2}$ and $g_3^*=\frac{\pi}{2}$.
\begin{itemize}
\item[(a1)] $ \Span(\hu_{1}, \hat B_1(\hlin_{1}(\tw) d\cR w_2-\hlin_{1}(d\cR w_2)\tw))$ is transversal to $\Ker(\hat\brlin_{1})\cap \Ker(\hat\brrlin_{1}).$
\item[(a2)] $de_4(d\cR w_2)\neq 0.$
\item[(b1)]
$ \Span(\hu_{2}, \hat B_2(\hlin_{2}(\tw) w_1-\hlin_{2}(w_1)\tw))$ is transversal to $\Ker(\hat\brlin_{2})\cap \Ker(\hat\brrlin_{2}).$
\item[(b2)] $de_4(w_1)\neq 0.$
\item[(c)] $\hat{\mathbf l}_j\cdot\tilde w\neq 0$,\quad $\hat{\mathbf l}_2\cdot w_{1}\neq0,$\quad $\hat{\mathbf l}_1\cdot d\cR w_{2}\neq0,$\quad $\hat\brlin_j\cdot \hat{\mathbf u}_j\neq 0$, $j=1,2.$
\end{itemize}
\elm
This lemma is proved in Section \ref{SSTrans} and \ref{subsection: local3}.
\subsection{Proof of Lemma \ref{LmKone}, the expanding cones}

Consider for example the case where $\bx\in U_2(\delta).$
We claim that if $\delta, \mu$ are small enough then
$d\Loc(\Span(w_1, \tw))$ is transversal to $\Ker \brlin_{2}\cap \Ker \brrlin_{2}.$ Indeed take
$\Gamma$ such that $\lin_2 (\Gamma)=0.$ If $\Gamma=a w_1+\ta \tw$ then
$a \lin_2 (w_1)+\ta \lin_2 (\tw)=0.$
It follows that the direction of $\Gamma$ is close to the direction of
$\hGamma=\hlin_2 (\tw) w_1-\hlin_2(w_1) \tw. $ Next take $\tGamma=bw_1+\tb\tw$ where
$b \lin_2 (w_1)+\tb \lin_2 (\tw)\neq 0.$
Then
the direction of  $d\Loc \tGamma $ is close to $\hat{\mathbf{u}}_{2}$ and the direction of
$d\Loc(\Gamma)$ is close to $B_2(\hGamma)$, so our claim follows from Lemma \ref{LmNonDeg}.

Thus for any plane $\Pi$ close to $\Span(w_1, \tw)$ we have that
$d\Loc(\Pi)$ is transversal to $\Ker \brlin_{2}\cap \Ker \brrlin_{2}.$
Take any $Y\in \cK_2.$ Then either
$Y$ and $w_1$ are linearly independent,  or $Y$ and $\tw$ are linearly independent.
Hence $d\Loc(\Span(Y, w_1))$ or $d\Loc(\Span(Y, \tw))$ is transversal to $\Ker \brlin_{2}\cap \Ker \brrlin_{2}.$
Accordingly either $\brlin_{2}(d\Loc(Y))\neq 0$
or $\brrlin_{2}(d\Loc(Y))\neq 0.$ If $\brlin_{2}(d\Loc(Y))\neq 0$ then the direction of $d(\Glob\circ \Loc)(Y)$
is close to $\brv.$ If $\brlin_{2}(d\Loc(Y))=0$ then the direction of $d(\Glob\circ \Loc)(Y)$
is close to $\brrv.$ Next, estimating $d\tilde{\mathbb G}$, we have $d\tilde{\mathbb G}\cdot \mathrm{span}(\brv,\brrv)=\mathrm{span}(\brv,\brrv)+o(1)$ (see Lemma \ref{Lm: 2pieces}).
Next, by Definition \ref{DefKone}, we get that $d\cR \mathrm{span}(\brv,\brrv)\subset \cK_1$.

So in either of the two cases above, we have $d(\cR\circ\tilde\Glob\circ \Glob\circ\Loc)(Y)\in \cK_{1}$ and
$\Vert d(\cR\circ\tilde\Glob\circ\Glob\circ\Loc)(Y)\Vert \geq c\chi \Vert Y\Vert .$
This completes the proof in the case $\bx\in U_2(\delta).$ The case of $\bx\in U_1(\delta)$ is similar.
\qed

\section{Symplectic transformations and Poincar\'{e} sections}\label{sct: symp}

In this section we define several Poincar\'{e} sections and perform symplectic transformations in
the regions between the consecutive  sections to make the Hamiltonian system suitable for doing calculations.
\subsection{The Poincar\'e coordinates}
 We start with the Hamiltonian \eqref{eq: hammain}. The translation invariance enables us to remove one body in the Hamiltonian. We choose $Q_2$ as this body.
 We start with the symplectic form
\begin{equation}\begin{aligned}
\om=&\sum_{i=1}^4 dP_i\wedge dQ_i
=d(P_1+P_2+P_3+P_4)\wedge dQ_2+d P_1\wedge d(Q_1-Q_2)\\
&+d P_3\wedge d(Q_3-Q_2)+d P_4\wedge d(Q_4-Q_2)\\
=&d(P_1+P_2+P_3+P_4)\wedge dQ_2+ dP_1\wedge dq_1+dP_3\wedge dq_3+dP_4\wedge dq_4,
\end{aligned}\nonumber\end{equation}
where we have used \eqref{eq: qQ}.
If we choose the mass center of the four bodies as the origin, then $P_1+P_2+P_3+P_4=0.$
Now the Hamiltonian becomes
\begin{equation}\nonumber
\begin{aligned}
H(q_1,P_1;q_3,P_3;q_4,P_4)
&=P_1^2+\frac{1}{2}\left(1+\frac{1}{\mu}\right)(P_3^2+P_4^2)+(\langle P_1,P_3\rangle+\langle P_1,P_4\rangle+\langle P_3,P_4\rangle)\\
&-\frac{1}{|q_1|}-\frac{\mu}{|q_3|}-\frac{\mu}{|q_4|}-\frac{\mu}{|q_1-q_3|}-\frac{\mu}{|q_1-q_4|}-\frac{\mu^2}{|q_3-q_4|}.
\end{aligned}
\end{equation}
Restricted to the subspace 
where $P_1+P_2+P_3+P_4=0$, up to a factor $\mu$, the symplectic form is $\bar\omega$ defined in \eqref{eq: baromega}.
We divide the whole Hamiltonian by $\mu$ to get
\begin{equation}\label{eq: hamglob}
\begin{aligned}
&H(q_1,p_1;q_3,p_3;q_4,p_4)=
\mu p_1^2+\frac{\mu}{2}\left(1+\frac{1}{\mu}\right)(p_3^2+p_4^2)
+\mu(\langle p_1,p_3\rangle+\langle p_1,p_4\rangle+\langle p_3,p_4\rangle)\\
&-\frac{1}{\mu |q_1|}-\frac{1}{|q_3|}-\frac{1}{|q_4|}-\frac{1}{|q_1-q_3|}-\frac{1}{|q_1-q_4|}-\frac{\mu}{|q_3-q_4|}.\\
\end{aligned}
\end{equation}
It can be checked that the Poincar\'e-Cartan one-form is multiplied by a factor $\mu$ due to the coordinate change and the Hamiltonian canonical equation holds true in the new coordinates.

In the new coordinates the total angular momentum equals
$$ G=Q_1 \times P_1-Q_2 \times (P_1+P_3+P_4)+ Q_3 \times P_3+Q_4 \times P_4=
q_1\times P_1+q_3\times P_3+q_4 \times P_4.$$
Therefore the conservation of angular momentum takes the form
$$ \sum_{j=3,1,4} q_j \times p_j=\Const. $$

\subsection{More Poincar\'{e} sections}
When $Q_4$ is closer to $Q_i, i=1,2$,
we treat its motion as a hyperbolic Kepler motion with focus at $Q_i$ and perturbed by $Q_3,$ $Q_{3-i}$. \begin{Def}\label{def: coordL}
We introduce one more set of coordinates
\begin{equation}\label{eq: left}
\begin{cases}
&v_3=p_3+\frac{\mu}{1+\mu}(p_1+p_4),\\
&v_1=p_1+p_4,\\
&v_4=\frac{1}{1+\mu}p_4-\frac{\mu}{1+\mu}p_1,\\
\end{cases}\quad
\begin{cases}
&x_3=q_3,\\
&x_1=\frac{1}{1+\mu}q_1-\frac{\mu}{1+\mu}q_3+\frac{\mu}{1+\mu}q_4,\\
&x_4=q_4-q_1.\\
\end{cases}
\end{equation}
\end{Def}
One can check that the transformation \eqref{eq: left} is symplectic with respect to the symplectic form $\bar\omega$.
\begin{Not}
To distinguish the new set of coordinates from those of Definition \ref{def: coordR}, we use superscript $R$
(meaning {\it right}) and write $(x_3,v_3;x_1,v_1;x_4,v_4)^R=(x_3^R,v_3^R;x_1^R,v_1^R;x_4^R,v_4^R)$ for the coordinates
from Definition \ref{def: coordR} and use superscript $L$ (meaning {\it left}) and write
$(x_3,v_3;x_1,v_1;x_4,v_4)^L=(x^L_3,v^L_3;x^L_1,v^L_1;x^L_4,v^L_4)$ for the coordinates
from Definition~\ref{def: coordL}. Notice that $(x_3,v_3)^R=(x_3,v_3)^L$, so we omit this superscript for simplicity.
\end{Not}

\begin{Def}[Further Poincar\'{e} sections]\label{DefSection}
We further define two more sections to cut the global map into three pieces (see Figure 2).
\begin{itemize}
\item Map $(I)$ is the Poincar\'{e} map between the sections
$$\left\{x^R_{4,\parallel}=-2,\ v^R_{4,\parallel}<0\right\},\quad\mathrm{ and}\quad
\left\{x^R_{4,\parallel}=-\frac{\chi}{2},\ v^R_{4,\parallel}<0\right\}.$$
\item Map $(III)$ is the Poincar\'{e} map between the sections $$\left\{x^R_{4,\parallel}=-\frac{\chi}{2},\ v^R_{4,\parallel}<0\right\},\quad
\mathrm{and}\quad  \left\{x^L_{4,\parallel}=\frac{\chi}{2},\ v^L_{4,\parallel}>0\right\}.$$
\item Map $(V)$ is the Poincar\'{e} map between the sections $$\left\{x^L_{4,\parallel}=\frac{\chi}{2},\ v^L_{4,\parallel}>0\right\},\quad \mathrm{and}\quad
\left\{x^R_{4,\parallel}=-2,\ v^R_{4,\parallel}>0\right\}.$$
\item We also introduce map $(II)$ to change coordinates from right to the left on the section
$\left\{x^R_{4,\parallel}=-\frac{\chi}{2},\ v^R_{4,\parallel}<0\right\}$ and map $(IV)$ to change coordinates from left to the right on the section $\left\{x^L_{4,\parallel}=\frac{\chi}{2},\ v^L_{4,\parallel}>0\right\}.$
\end{itemize}
\end{Def}
\begin{Rk}\begin{enumerate}
\item 
The two sections $\left\{x^R_{4,\parallel}=-\frac{\chi}{2},\ v^R_{4,\parallel}<0\right\}$ and $\left\{x^L_{4,\parallel}=\frac{\chi}{2},\ v^L_{4,\parallel}>0\right\}$ lie almost at the midpoint of $Q_1$ and $Q_2$.

\item In Section 4.3, 4.4, we will treat the equations of motion as three Kepler motions $(x_i,v_i)^{R,L},\ i=3,1,4$ with perturbations. When perturbation is neglected the orbit of $x_4^R$ is a hyperbola focused at the origin and opening to the left while the orbit of $x_4^L$ is a hyperbola focused at origin and opening to the right.
\end{enumerate}
\end{Rk}

\begin{figure}
\begin{center}
\includegraphics[width=0.8\textwidth]{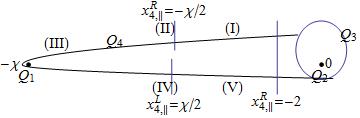}
\caption{Poincar\'{e} sections}
\end{center}\label{figureSection}
\end{figure}
In the following subsections we describe the suitable changes of variables adapted to maps $(I), (III), (V)$
as well as the local map $\Loc.$
\subsection{Hamiltonian of the right case, when $Q_4$ is closer to $Q_2$}
 We write the Hamiltonian in terms of three Kepler motions with perturbations.

\begin{Lm}\label{LmHamR} In the coordinates \eqref{eq: right}, the Hamiltonian for maps $(I)$ and $(V)$ has the form
\begin{equation}\label{eq: hamR}
\begin{aligned}
&H(x_3,v_3;x_1,v_1;x_4,v_4)=\left[\frac{\mu(1+\mu)}{1+2\mu} v_1^2-\frac{2\mu+1}{\mu|x_1|}\right]+\left[\frac{1+\mu}{2}v_3^2-\frac{1}{|x_3|}\right]+\left[\frac{1+2\mu}{2(1+\mu)}v_4^2
-\frac{1+\mu}{|x_4|}\right]+U^R,
\end{aligned}
\end{equation}
where
\begin{equation}\label{eq: UR}
\begin{aligned}
 U^R(x_3,x_1,x_4)=&\frac{2\mu+1}{\mu|x_1|}-\left(\frac{1}{\mu\left|x_1+\frac{\mu}{1+2\mu}x_4+\frac{\mu}{1+\mu} x_3\right|}+
\frac{1}{\left|x_1+\frac{\mu}{1+2\mu}x_4-\frac{1}{1+\mu}x_3\right|}
+\right.\\
& \left.\frac{1}{\left|x_1-\frac{1+\mu}{1+2\mu}x_4\right|}\right)
+\frac{1+\mu}{|x_4|}-\left(\frac{1}{\left|x_4+\frac{\mu x_3}{1+\mu}\right|}
+\frac{\mu}{\left|\frac{x_3}{1+\mu}-x_4\right|}\right).
\end{aligned}
\end{equation}

And in Delaunay variables we have
\begin{equation}\nonumber
\begin{aligned}
H(&L_3,\ell_3,G_3,g_3;x_1,v_1;L_4,\ell_4, G_4,g_4)=\left(\frac{v_1^2}{2m_{1R}}-\frac{k_{1R}}{ |x_1|}\right)-\frac{m_{3R}k^2_{3R}}{2L_3^2}+\frac{m_{4R}k_{4R}^2}{2L_4^2}+U^R
\end{aligned}
\end{equation}
where
\begin{equation}\label{Eqmk}
m_{1R}=\frac{1+2\mu}{2\mu(1+\mu)},\ m_{3R}=\frac{1}{1+\mu},\ m_{4R}=\frac{1+\mu}{1+2\mu},\ k_{1R}=\frac{1+2\mu}{\mu},\ k_{3R}=1 ,\ k_{4R}=1+\mu.\end{equation}
\end{Lm}

\subsection{Hamiltonian of the left case, when $Q_4$ is closer to $Q_1$}
In this section, we explain the choice of the Jacobi coordinates \eqref{eq: left} and derive the corresponding Hamiltonian.
When $Q_4$ is moving between the sections $\{x^R_{4,\parallel}=-\frac{\chi}{2}\}$ and $\{x^L_{4,\parallel}=\frac{\chi}{2}\}$ and turns around $Q_1$, 
we treat $Q_4$'s motion as an approximate hyperbola with focus at $Q_1$.

\begin{Lm}\label{LmHamL}In the coordinates \eqref{eq: left}, the Hamiltonian for map $(III)$ becomes
\begin{equation}\label{eq: hamL}
\begin{aligned}
&H(x_3,v_4;x_1,v_1;x_4,v_4)=\frac{\mu}{1+\mu} v_1^2-\frac{(1+\mu)^2}{\mu|x_1|}+\frac{\mu}{2}\left(1+\frac{1}{\mu}\right)(v_3^2+v_4^2)-\frac{1}{|x_3|}-\frac{1}{|x_{4}|}+U^L
\end{aligned}
\end{equation}
where
\begin{equation}\label{eq: UL}
\begin{aligned}
 &U^L(x_3,x_1,x_4)=\frac{(1+\mu)^2}{\mu|x_1|}-\left(\frac{1}{\mu \left|x_1+\frac{\mu}{1+\mu}x_3-\frac{\mu}{1+\mu}x_4\right|}+\frac{1}{\left|x_1+\frac{\mu}{1+\mu}x_3+\frac{1}{1+\mu}x_4\right|}+\right.\\
&\left.\frac{1}{\left|x_1-\frac{1}{1+\mu}x_3-\frac{\mu}{1+\mu}x_4\right|}
+\frac{\mu}{\left|x_1-\frac{1}{1+\mu}x_3+\frac{1}{1+\mu}x_4\right|}\right).
\end{aligned}
\end{equation}
In Delaunay coordinates, we have
\begin{equation}\nonumber
\begin{aligned}
H(&L_3,\ell_3,G_3,g_3;x_1,v_1;L_4,\ell_4, G_4,g_4)=\left(\frac{v_1^2}{2m_{1L}}-\frac{k_{1L}}{ |x_1|}\right)-\frac{m_{3L}k^2_{3L}}{2L_3^2}+\frac{m_{4L}k_{4L}^2}{2L_4^2}+U^L,
\end{aligned}
\end{equation}
\[\mathrm{where\quad }m_{1L}=\frac{1+\mu}{2\mu},\ m_{3L}=m_{4L}=\frac{1}{1+\mu},\ \quad k_{1L}=\frac{(1+\mu)^2}{\mu},\ k_{3L}=k_{4L}=1.\]
\end{Lm}

\subsection{Hamiltonian of the local map, away from close encounter}
We cut the local map into three pieces by introducing a new section $|q_3-q_4|=\mu^\kappa, \ 1/3<\kappa<1/2$. The
restriction $\kappa< 1/2$ comes from the proof of Lemma \ref{Lm: 2pieces} where we need $\mu^{1-2\kappa}$ to be small, and the restriction $\kappa> 1/3$ comes from the proof of Lemma \ref{Lm: landau} where we need $\mu^{3\kappa-1}$ to be small. 

When $Q_3,Q_4$ are moving outside the circle $|q_3-q_4|=\mu^\kappa$, we use the same transformation as \eqref{eq: right} but different ways of grouping terms. So we get the following from the first equality of \eqref{eq: hamR}.
\begin{Lm} \label{LmRightOut} In the coordinates \eqref{eq: right},
assume $|x_3|,|x_4|<2$ and $|x_1|\sim \chi$. Then the Hamiltonian can be written in the following form
\begin{equation}\label{eq: hamRR}
\begin{aligned}
H( x_1, v_1;x_3,v_3;x_4,v_4)^R&=
\left[\frac{\mu(1+\mu)}{1+2\mu} v_1^2-\frac{2\mu+1}{\mu|x_1|}\right]+\left[\frac{1+\mu}{2}v_3^2-\frac{1}{|x_3|}\right]+\left[\frac{1+2\mu}{2(1+\mu)}v_4^2-\frac{1+\mu}{|x_4|}\right]\\
&-\frac{\mu}{\left|\frac{x_3}{1+\mu}-x_4\right|}+V_{out}(x_3,x_1,x_4),\\
\end{aligned}
\end{equation}
where \begin{equation*}\begin{aligned}V_{out}(x_3,x_1,x_4)&=\left(\frac{1+\mu}{|x_4|}-\frac{1}{\left|x_4+\frac{\mu x_3}{1+\mu}\right|}
\right)+\left(\frac{2\mu+1}{\mu|x_1|}-
\frac{1}{\mu\left|x_1+\frac{\mu}{1+2\mu}x_4+\frac{\mu}{1+\mu} x_3\right|}\right.\\
&\left.-\frac{1}{\left|x_1+\frac{\mu}{1+2\mu}x_4-\frac{1}{1+\mu}x_3\right|}-\frac{1}{\left|x_1-\frac{1+\mu}{1+2\mu}x_4\right|}\right).\end{aligned}\end{equation*}
In Delaunay coordinates we have
\begin{equation}\label{eq: hamRRDel}
\begin{aligned}
&H( x_1, v_1;L_3,\ell_3,G_3,g_3;L_4,\ell_4,G_4,g_4)^R\\
&=\left(\frac{v_1^2}{2m_{1R}} -\frac{k_{1R}}{ |x_1|}\right)-\frac{m_{3R}k_{3R}^2}{2L_3^2}+\frac{m_{4R}k_{4R}^2}{2L_4^2}
-\frac{\mu}{\left|\frac{x_3}{1+\mu}-x_4\right|}+V_{out}.
\end{aligned}
\end{equation}
\end{Lm}

\subsection{Hamiltonian of the local map, close encounter}
When $Q_3,Q_4$ are moving inside the circle $|q_3-q_4|=\mu^\kappa$, we derive the Hamiltonian system describing the relative motion of $Q_3,Q_4$. We start with the Hamiltonian
\eqref{eq: hamglob} and make the following symplectic changes to convert to the coordinates of relative motion and motion of the center of mass  \begin{equation}\label{eq: relcm}\begin{cases}
&q_-=\frac{1}{2}(q_3-q_4)\\
&q_+=\frac{1}{2}(q_3+q_4)\\
&q_1=q_1
\end{cases},\quad \begin{cases}
&p_-=p_3-p_4\\
&p_+=p_3+p_4\\
&p_1=p_1
\end{cases}.\end{equation}
The symplectic form is now\[\bar\omega=d p_1\wedge d  q_1+d p_+\wedge d \mathbf{q}+d p_-\wedge d q_-.\]

\begin{Lm}\label{LmRightIn}
In the coordinates \eqref{eq: relcm},  we have $|q_-|\leq 2\mu^\kappa$
inside the circle $|q_-|=2\mu^\kappa$ and the following
expression of the Hamiltonian \eqref{eq: hamglob}
\begin{equation*}\label{eq: hamrel}
\begin{aligned}
H( q_1, p_1;q_-,p_-; q_+,p_+)&=\left[\mu p_1^2-\frac{1+2\mu}{\mu |q_1|}\right]+\left[\frac{1+2\mu}{4}p_+^2-\frac{2}{|q_+|}\right]+\left[\frac{1}{4}p_-^2-\frac{\mu}{2|q_-|}\right]+V_{in},
\end{aligned}
\end{equation*}
where
\begin{equation}
\begin{aligned}
V_{in}&=\mu\langle p_1,p_+\rangle-\frac{1}{|q_+-q_-|}
 -\frac{1}{|q_++q_-|}-\frac{1}{|q_1- q_++q_-|}-\frac{1}{|q_1- q_+-q_-|}+\frac{2}{|q_1|}+\frac{2}{|q_+|}\\
&=\mu\langle p_1,p_+\rangle- \frac{3\langle q_+,q_-\rangle^2}{2|q_+|^5}+
\frac{|q_-|^2}{|q_+|^3}-\frac{\langle q_1,q_+\rangle}{|q_1|^3}+O\left(|q_-|^3+\frac{1}{|q_1|^3}\right).
\end{aligned}
\end{equation}
The $O$ expression holds as $|q_-|\to 0$ and $|q_1|\to\infty.$
\end{Lm}
We can convert the term $\frac{1}{4}p_-^2-\frac{\mu}{2|q_-|}$ into Delaunay coordinates as $\frac{\mu^2}{4L_-^2}$.

\section{Statement of the main technical proposition}\label{sct: prop}
In this section, we give the statement of our calculation of matrices needed in the proof of the global map.
We use the coordinate system $(L_3,\ell_3,G_3,g_3;x_1,v_1;G_4,g_4)$ to do the calculation.
In the following, the superscript $i$ means ``initial" and $f$ means ``final".

\begin{Not}\label{Def: sim} To avoid many $O$ notations in our estimates, we introduce the following conventions. \begin{itemize}
\item
We use the notation $a\lesssim b$ if $a=O(b)$ or equivalently $|a|\leq C|b|$ for some constant $C$ independent of $\chi,\mu$,  and the notation $a\sim b$ if both $a\lesssim b$ and $b\lesssim a$ hold.
\item We also generalize this notation to vectors and matrices. For two vectors $A,\ B\in \R^{n}$, we write $A\lesssim B$ if $A_i\lesssim B_i$ holds for each entry $A_i, B_i$ of $A$ and $B$ respectively, and write $A\sim B$ if $A\lesssim B$ and  $B\lesssim A$ hold. Similarly for matrices.
\item For a matrix $[\sharp]$, we refer to its blocks as $\left[\begin{array}{c|c|c}
\sharp_{33}&\sharp_{31}&\sharp_{34}\\
\hline
\sharp_{13}&\sharp_{11}&\sharp_{14}\\
\hline
\sharp_{43}&\sharp_{41}&\sharp_{44}
\end{array}\right]$, and its $(i,j)$-th entry as $[\sharp](i,j),\ i,j=1,2,\ldots,10.$
\item Moreover, when we use ``$\lesssim$", there may be some entries in the vector or matrix, for which we have an estimate in the sense of $\sim$. Those entries will be important to show that the $\chi^2$ and $\chi$ terms
in Lemma \ref{Lm: glob} do not vanish. For those entries, we use {\bf bold} font.
    \end{itemize}
\end{Not}

\begin{Prop}\label{Prop: main}
Under the assumption {\bf AG}, we have the following:
$(a)$\begin{itemize}
\item[(a.1)] The derivative of the global map is the product of five $10\times 10$ matrices $d\mathbb G=(V)(IV)(III)(II)(I)$
having the following form
\begin{equation}\nonumber
\begin{aligned}
(I)=&(\Id_{10}+\chi u_1^f\otimes l_1^f)N_1(\Id_{10}-u_1^i\otimes l_1^i)
\lesssim(\Id_{10}+\chi u\otimes l)N_1(\Id_{10}+u_1^i\otimes l_1^i),\\
(II)=&(\chi u_{iii}\otimes l_{iii}+A)L\cdot R^{-1}(\chi u_{i}\otimes l_{i}+C ),\\
(III)=&(\Id_{10}+\chi u_3^f\otimes l_3^f)N_3(\Id_{10}-\chi u_3^i\otimes l_3^i)\lesssim(\Id_{10}+\chi u\otimes l)N_3(\Id_{10}+\chi u\otimes l'),\\
(IV)=&(\chi u_{iii'}\otimes l_{iii'}+A)R\cdot L^{-1}(\chi u_{i'}\otimes l_{i'}+C)\\
(V)=&(\Id_{10}+u_5^f\otimes l_5^f)N_5(\Id_{10}-\chi u_5^i\otimes l_5^i)\lesssim(\Id_{10}+u_1^i\otimes l_1^i)N_5(\Id_{10}+\chi u\otimes l'),
\end{aligned}
\end{equation}
\item[(a.2)] where
 \begin{equation}\nonumber
\begin{aligned}
 u_1^f,\ u_3^f,\ u^i_3,\ u_5^i\lesssim u&:=\left(\frac{1}{\chi^3},-\mathbf 1, \frac{1}{\chi^3},\frac{1}{\chi^3};\mu,\frac{\mu}{\chi},\frac{1}{\mu\chi^2},\frac{1}{\chi^3} ;\frac{1}{\chi^2},\frac{1}{\chi^2} \right)^T_{10\times 1},\\
l_1^f,\ l_3^f \lesssim l&:=\left(\mathbf 1, \frac{1}{\chi^3},\frac{1}{\chi^3},\frac{1}{\chi^3};\frac{1}{\mu\chi^2},\frac{1}{\chi^3}, \mu,\frac{\mu}{\chi};\frac{1}{\chi}, \frac{1}{\chi}\right)_{1\times 10},\\
-l_3^i,\  -l_5^i\lesssim  l'&:=\left(\mathbf 1, \frac{1}{\chi^3},\frac{1}{\chi^3},\frac{1}{\chi^3};\frac{1}{\chi},\frac{1}{\chi^3}, \mu,\frac{\mu}{\chi};\frac{1}{\chi}, \frac{1}{\chi}\right)_{1\times 10},\\
u_{iii},\ u_{iii'}&\sim\left(0_{1\times 8};\mathbf 1,\mathbf 1\right)^T_{10\times 1},\\
 l_{iii},\ l_{iii'}&\lesssim\left(0_{1\times 8};\frac{1}{\chi^2}, \frac{\mathbf 1}{\chi}, \frac{1}{\chi},\mathbf 1\right)_{1\times 12},\\
u_i&=\left(0_{1\times 9},\frac{L_4}{2m_{4R}^2k_{4R}^2},0,\frac{\mathbf 1}{\chi}\right)^T_{12\times 1},\\
u_{i'}&=\left(0_{1\times 9},\frac{L_4}{2m_{4L}^2k_{4L}^2},0,\frac{\mathbf 1}{\chi}\right)^T_{12\times 1},\\
l_{i}&\lesssim\left(\mathbf 1,\frac{1}{\chi^3},\frac{1}{\chi^3},\frac{1}{\chi^3};\frac{1}{\mu\chi^2},\frac{1}{\chi^3},\mu,\frac{\mu}{\chi};\mathbf 1,\mathbf 1\right)_{1\times 10},\\
l_{i'}&\lesssim\left(\frac{1}{\chi},\frac{1}{\chi^4},\frac{1}{\chi^4},\frac{1}{\chi^4};\frac{1}{\mu\chi^3},\frac{1}{\chi^4},\frac{\mu}{\chi},\frac{\mu^2}{\chi^2};\mathbf 1,\mathbf 1\right)_{1\times 10}.
\end{aligned}
\end{equation}
\begin{equation}\nonumber
\begin{aligned}
u_5^f,\  u_1^i&\lesssim \left(\mu, 1, \mu,\mu;\mu,\frac{\mu}{\chi},\frac{1}{\mu\chi^2},\frac{1}{\chi^3}; \mu,\mu \right)^T_{10\times 1},\\
 l_5^f,\ l_1^i&\lesssim \left( 1, \mu,\mu,\mu;\frac{1}{\mu\chi^2},\frac{1}{\chi^3}, \mu,\frac{\mu}{\chi}; 1,  1\right)_{1\times 10};
\end{aligned}
\end{equation}
\item[(a.3)] and $A=\left[\begin{array}{c|cccc}
\mathrm{Id}_{8\times 8}&0_{8\times 1}&0_{8\times 1}&0_{8\times 1}&0_{8\times 1}\\
\hline
0_{1\times 8}&0&0&0&0\\
0_{1\times 8}&O\left(\frac{1}{\chi^2}\right)&O\left(\frac{1}{\chi^2}\right)& O(1)&O(1)
\end{array}\right]_{10\times 12},$

\[C= \left[\begin{array}{c|c}
\mathrm{Id}_{8\times 8} &0_{8\times 2}\\
\hline
0_{1\times 8} & 0_{1\times 2}\\
0_{1\times 8} &O(1)_{1\times 2}\\
\breve{l}_i &O\left(\frac{1}{\chi}\right)_{1\times 2}\\
0_{1\times 8} &0_{1\times 2}
\end{array}\right]_{12\times 10}\mathrm{with\ }
\breve{l}_i\lesssim\(1,\frac{1}{\chi^3},\frac{1}{\chi^3},\frac{1}{\chi^3};\frac{1}{\mu\chi^2},\frac{1}{\chi^3},\mu,\frac{\mu}{\chi}\)_{1\times 8};
\]
\item[(a.4)] matrices $R$ and $L$ are the derivative
matrices of the transformations \eqref{eq: right} and \eqref{eq: left} respectively,
and they have the following expression in the coordinates $(x_3,v_3;x_1,v_1;x_4,v_4)$
\begin{equation}
\label{DerLR}
R\cdot L^{-1},\ L\cdot R^{-1}=\left[\begin{array}{ccccc}
\Id_4&0&0&0&0\\
0&m_\pm\Id_2&0&\mp\frac{2\mu}{1+2\mu}\Id_2&0\\
0&0&m_\mp\Id_2&0&\mp \Id_2\\
0&\pm \Id_2&0 &m_\mp\Id_2&0\\
0&0&\pm\frac{2\mu}{1+2\mu}\Id_2&0&m_\pm\Id_2
\end{array}
\right]_{12\times 12},
\end{equation}
where $m_+=\frac{1+\mu}{1+2\mu}$ and $m_-=\frac{1}{1+\mu}$, and we choose the upper sign for $R\cdot L^{-1}$ and the lower sign for $L\cdot R^{-1}$ when we need to make a choice in $\pm$ or $\mp$;
\item[(a.5)] the following estimates hold
\begin{equation}N_1-\Id_{10}\lesssim\left[\begin{array}{llll|llll|ll}
\mu &\mu&\mu&\mu &\frac{\mu}{\chi}&\frac{\mu}{\chi}&\mu^2&\frac{\mu}{\chi}&\mu&\mu\\
\mu\chi &\mu^2\chi&\mu^2\chi&\mu^2\chi &\mu^2&\mu^2&\mu\chi&\mu&\mu^2\chi&\mu^2\chi\\
\mu &\mu&\mu&\mu&\frac{\mu}{\chi}&\frac{\mu}{\chi}&\mu^2&\frac{\mu}{\chi}&\mu&\mu\\
\mu &\mu&\mu&\mu&\frac{\mu}{\chi}&\frac{\mu}{\chi}&\mu^2&\frac{\mu}{\chi}&\mu&\mu\\
\hline
\mu\chi&\mu^2\chi&\mu^2\chi&\mu^2\chi&\mu^2&\mu^2&\mu\chi&\mu^2& \mu^2\chi&\mu^2\chi\\
\mu&\mu^2&\mu^2&\mu^2&\frac{\mu}{\chi}&\frac{1}{\chi}&\mu^2&\mu\chi&\mu&\mu\\
\frac{1}{\mu\chi}&\frac{1}{\chi}&\frac{1}{\chi}&\frac{1}{\chi}&\frac{1}{\mu\chi^2}&\frac{1}{\chi^2}&\frac{1}{\chi}&\frac{\mu}{\chi}&\frac{1}{\chi}&\frac{1}{\chi}\\
\frac{1}{\chi}&\frac{\mu}{\chi}&\frac{\mu}{\chi}&\frac{\mu}{\chi}&\frac{1}{\chi^2}&\frac{1}{\mu\chi^2}&\frac{\mu}{\chi}&\frac{1}{\chi}&\frac{1}{\chi}&\frac{1}{\chi}\\
\hline
1 &\mu&\mu&\mu&\frac{1}{\chi}&\frac{1}{\chi}&\mu&\mu&\mathbf{1}&\mathbf{1}\\
1 &\mu&\mu&\mu&\frac{1}{\chi}&\frac{1}{\chi}&\mu&\mu&\mathbf{1}&\mathbf{1}
\end{array}\right]_{10\times 10},\nonumber
\end{equation}
\begin{equation}N_3-\Id_{10}\lesssim\left[\begin{array}{llll|llll|ll}
\frac{\mu}{\chi}&\frac{1}{\chi^2}&\frac{1}{\chi^2}&\frac{1}{\chi^2}&\frac{1}{\mu\chi^3}&\frac{1}{\chi^3}&\frac{\mu}{\chi} &\frac{\mu}{\chi^2}&\frac{1}{\chi^2}&\frac{1}{\chi^2}\\
\mu\chi&\frac{\mu}{\chi}&\frac{\mu}{\chi}&\frac{\mu}{\chi}&\frac{1}{\mu\chi}&\frac{1}{\chi}&\mu\chi&\mu&1&1\\
\frac{\mu}{\chi}&\frac{1}{\chi^2}&\frac{1}{\chi^2}&\frac{1}{\chi^2}&\frac{1}{\mu\chi^3}&\frac{1}{\chi^3}&\frac{\mu}{\chi}&\frac{\mu}{\chi^2}&\frac{1}{\chi^2}&\frac{1}{\chi^2}\\
\frac{\mu}{\chi}&\frac{1}{\chi^2}&\frac{1}{\chi^2}&\frac{1}{\chi^2}&\frac{1}{\mu\chi^3}&\frac{1}{\chi^3}&\frac{\mu}{\chi}&\frac{\mu}{\chi^2}&\frac{1}{\chi^2}&\frac{1}{\chi^2}\\
\hline
\mu\chi&\frac{\mu}{\chi}&\frac{\mu}{\chi}&\frac{\mu}{\chi}&\frac{1}{\chi}&\frac{\mu}{\chi}&\mu\chi&\mu^2&\mu&\mu\\
\mu&\frac{\mu}{\chi^2}&\frac{\mu}{\chi^2}&\frac{\mu}{\chi^2}&\frac{\mu}{\chi}&\frac{1}{\chi}&\mu^2&\mu\chi&\mu&\mu\\
\frac{1}{\mu\chi}&\frac{1}{\mu\chi^3}&\frac{1}{\mu\chi^3}&\frac{1}{\mu\chi^3}&\frac{1}{\mu\chi^2}&\frac{1}{\chi^2}&\frac{1}{\chi}&\frac{\mu}{\chi}&\frac{1}{\chi}&\frac{1}{\chi}\\
\frac{1}{\chi}&\frac{1}{\chi^3}&\frac{1}{\chi^3}&\frac{1}{\chi^3}&\frac{1}{\chi^2}&\frac{1}{\mu\chi^2}&\frac{\mu}{\chi}&\frac{1}{\chi}&\frac{1}{\chi}&\frac{1}{\chi}\\
\hline
1&\frac{1}{\chi^2}&\frac{1}{\chi^2}&\frac{1}{\chi^2}&\frac{1}{\chi}&\frac{1}{\chi}&\mu&\mu&\mathbf{1}&\mathbf{1}\\
1&\frac{1}{\chi^2}&\frac{1}{\chi^2}&\frac{1}{\chi^2}&\frac{1}{\chi}&\frac{1}{\chi}&\mu&\mu&\mathbf{1}&\mathbf{1}
\end{array}\right]_{10\times 10},\nonumber
\end{equation}
\begin{equation}N_5-\Id_{10}\lesssim\left[\begin{array}{llll|llll|ll}
\mu^2\chi &\mu&\mu&\mu  &\frac{1}{\chi}&\frac{\mu}{\chi}&\mu^2\chi& \mu^2 &\mu&\mu\\
\mu\chi &\mu&\mu&\mu &\frac{1}{\mu\chi}&\frac{1}{\chi}&\mu\chi&\mu&1&1\\
\mu^2\chi &\mu&\mu&\mu &\frac{1}{\chi }&\frac{\mu}{\chi}&\mu^2\chi&\mu^2&\mu&\mu\\
\mu^2\chi &\mu&\mu&\mu &\frac{1}{\chi }&\frac{\mu}{\chi}&\mu^2\chi&\mu^2&\mu&\mu\\
\hline
\mu\chi&\mu^2&\mu^2&\mu^2&\frac{1}{\chi}&\frac{\mu}{\chi}&\mu\chi&\mu^2& \mu&\mu\\
\mu&\frac{\mu^2}{\chi}&\frac{\mu^2}{\chi}&\frac{\mu^2}{\chi}&\frac{\mu}{\chi}&\frac{1}{\chi}&\mu^2&\mu\chi&\mu&\mu\\
\mu^2&\frac{\mu}{\chi}&\frac{\mu}{\chi}&\frac{\mu}{\chi}&\frac{1}{\mu\chi^2}&\frac{1}{\chi^2}&\mu^2&\frac{\mu}{\chi}&\frac{1}{\chi}&\frac{1}{\chi}\\
\mu^2&\frac{\mu}{\chi}&\frac{\mu}{\chi}&\frac{\mu}{\chi}&\frac{1}{\chi^2}&\frac{1}{\mu\chi^2}&\mu^2&\frac{1}{\chi}&\frac{1}{\chi}&\frac{1}{\chi}\\
\hline
\mu^2\chi &\mu&\mu&\mu &\frac{1}{\chi }&\frac{1 }{\chi }&\mu^2\chi&\mu&\mathbf{1}&\mathbf{1}\\
\mu^2\chi &\mu&\mu&\mu &\frac{1}{\chi }&\frac{1 }{\chi }&\mu^2\chi&\mu&\mathbf{1}&\mathbf{1}
\end{array}\right]_{10\times 10}.\nonumber
\end{equation}
\end{itemize}
$(b)$ Moreover, for the $\mathbf 1$ entries in $(a.2)$, we have the following exact estimates
\begin{itemize}
\item[(b.1)] As $1/\chi\ll\mu\to 0$, we have
\begin{equation}\nonumber
\begin{aligned}
u_{1,3}^{f},\ u_{3,5}^i,\ u&\to (0,1,0_{1\times 8})_{10\times 1}^T=\tw,\\
u_{iii}&\to \left(0_{1\times 8};1,\frac{1}{\tilde L_{4,j}}\right)^T_{10\times 1},\ u_{iii'}\to \left(0_{1\times 8};1,\frac{\hat L_{4,j}}{\hat G_{4,j}^2+\hat L_{4,j}^2}\right)^T_{10\times 1}=w_j,\\
l^f_{1,3},\ l^i_{3,5},\ l,\ l' &\to (1,0_{1\times 9})_{1\times 10}=\hat\brrlin_j,\\
l_i&\to \left(\frac{\tilde G_{4,j}/\tilde L_{4,j}}{\tilde L_{4,j}^2+\tilde G_{4,j}^2}, 0_{1\times 7},-\frac{1}{\tilde L_{4,j}^2+\tilde G_{4,j}^2} , -\frac{1}{\tilde L_{4,j}}\right)_{1\times 10}=\hat\brlin_j.
\end{aligned}
\end{equation}
Here $j=1,2$ means the first and second collisions in Gerver's construction. $\tilde L_4$ and $\tilde G_4$ are the values of the Delaunay coordinates
at the initial point for the global map and $\hat L_4$ and $\hat G_4$ are the values of the Delaunay coordinates at the final point. 
\item[(b.2)]In addition, we have as $1/\chi\ll\mu\to 0$,
\[l_{i'}\to  \left( 0_{1\times 8},\frac{1}{\tilde L_{4,j}^2} , -\frac{1}{\tilde L_{4,j}}\right)_{1\times 10},\quad
 u_{i'}=\left(0_{1\times 9},\frac{L_4}{2}+O(\mu),0,O\left(\frac{1}{\chi}\right)\right)^T_{12\times 1},\]
\[ l_{iii}=-l_{iii'}=
\left(0_{1\times 8}; O\(\frac{\mu}{\chi^2}\), -\frac{1+O(\mu)}{\chi L_4}, O\(\frac{\mu}{\chi}\), -\frac{1}{2}\)_{1\times 12}.
\]
\item[(b.3)]  The $O(1)$ blocks in $N_1,N_3,N_5$ have exact estimates as follows,

\begin{equation}\nonumber
\begin{aligned}
&(N_1)_{44}\simeq
\left[\begin{array}{cc}
1-\frac{\tilde L^2_{4,j}}{2 (\tilde L_{4,j}^2+\tilde G_{4,j}^2)} & -\frac{\tilde L_{4,j}}{2}\\
 \frac{\tilde L_{4,j}^3}{2 (\tilde L_{4,j}^2+\tilde G_{4,j}^2)^2}& 1+\frac{\tilde L_{4,j}^2}{ 2(\tilde L_{4,j}^2+\tilde G_{4,j}^2)} \\
\end{array}
\right], \ (N_3)_{44}\simeq\left[\begin{array}{cc}
\frac{1}{2} & -\frac{ \tilde L_{4,j}}{2}\\
 \frac{3}{2\tilde L_{4,j}}& \frac{1}{2} \\
 \end{array}
\right] ,\\
&(N_5)_{44}\simeq\left[\begin{array}{cc}
1+\frac{1/2\hat L^2_{4,j}}{\hat L_4^2+\hat G_{4,j}^2} &-\hat L_{4,j}/2\\
 \frac{1/2\hat L_{4,j}^3}{(\hat L_4^2+\hat G_{4,j}^2)^2}& 1-\frac{1/2\hat L_{4,j}^2}{\hat L_{4,j}^2+\hat G_{4,j}^2} \\
\end{array}
\right].  \\ 
\end{aligned}\label{eq: twist}
\end{equation}
where the notation $\simeq$ means up to $O(\mu)$ relative error.
\item[(b.4)] Finally, the derivative of the renormalization map is
\[d\cR=\mathrm{diag}\left\{ \sqrt{\lb}, 1,-\sqrt{\lb},-1;\lb \left[\begin{array}{cc}
1&0\\
0&-1
\end{array}\right] \mathrm{Rot}(\bt) ,\left[\begin{array}{cc}
1&0\\
0&-1
\end{array}\right] \frac{\mathrm{Rot}(\bt)}{\sqrt{\lb}};-\sqrt{\lb},-1\right\}. \]
\end{itemize}
\end{Prop}

 In part $(a.1)$ of the proposition, each of the five matrices is a product of three matrices. For the matrices $(I),(III),(V)$, we use the formula
for the derivative of the Poincar\'{e} map
(see equation \eqref{eq: formald4} and Section \ref{sct: boundary}). The matrices $N_1,N_3,N_5$ are solutions of the variational equations and the two remaining matrices are boundary contributions coming from the fact that different orbits take different time to travel between two consecutive sections. For $(II)$, we first convert from Delaunay variables to Cartesian variables in the right, then we use $L\cdot R^{-1}$ to convert $(x_3,v_3;x_1,v_1;x_4,v_4)^R\to (x_3,v_3;x_1,v_1;x_4,v_4)^L$, and finally we convert from Cartesian in the left to Delaunay variables. The matrix $(IV)$ is similar but in the opposite direction.

The plan of the proof of Proposition \ref{Prop: main} is as follows.
\begin{itemize}
\item In Section \ref{sct: eqmotion} and \ref{sct: var}, we write down the equations of motion, the variational equations and estimate their solutions. This gives us the matrices $N_1, N_3$ and $N_5$ in Proposition \ref{Prop: main}.
\item In Section \ref{sct: boundary}, we study the boundary contribution to the derivative of the Poincar\'{e} map. We get all the $u$'s and $l$'s with various sub- and super-scripts in $(I),(III),(V)$. Together with $N_1,N_3,N_5$, the estimates of the boundary contributions complete the estimates of $(I),(III),(V)$.
\item In Section \ref{sct: switch}, we study the transformation of coordinates from the left to the right and that from the right to the left. This gives us the matrices $(II),(IV)$ stated in Proposition \ref{Prop: main}.
\item The derivative of the renormalization map follows immediately from its definition in Definition \ref{Def: renorm}.
\end{itemize}
We now compute the matrices $R\cdot L^{-1}$ and $L\cdot R^{-1}$
based on Definitions \ref{def: coordR} and \ref{def: coordL}.
\begin{proof}[Proof of \eqref{DerLR}]
To get $R\cdot L^{-1}=\frac{\partial (x_3,v_3;x_1,v_1;x_4,v_4)^R}{\partial (x_3,v_3;x_1,v_1;x_4,v_4)^L}$, we first use
\eqref{eq: left} to compute the matrix $L^{-1}:=\frac{\partial (q_3,p_3;q_1,p_1;q_4,p_4)}{\partial (x_3,v_3;x_1,v_1;x_4,v_4)^L}$,
then we use \eqref{eq: right} to compute $R:=\frac{\partial (x_3,v_3;x_1,v_1;x_4,v_4)^R}{\partial (q_3,p_3;q_1,p_1;q_4,p_4)}$.
The composition of the two gives us $R\cdot L^{-1}$. Similarly we get $L\cdot R^{-1}=\frac{\partial (x_3,v_3;x_1,v_1;x_4,v_4)^L}{\partial (x_3,v_3;x_1,v_1;x_4,v_4)^R}=(R\cdot L^{-1})^{-1}$.
\end{proof}

\section{Equations of motion, $\mathscr C^0$ control of the global map}\label{sct: eqmotion}
\subsection{The Hamiltonian equations}\label{SSReduction}
For the Hamiltonians in Lemma \ref{LmHamR} and Lemma \ref{LmHamL}, we suppress the super- or sub-scripts $R$ and $L$ to express the Hamiltonian in one single expression in Delaunay coordinates
$$H(L_3,\ell_3,G_3,g_3; x_1,v_1;L_4,\ell_4,G_4,g_4)=E_1-\frac{m_3k_3^2}{2L_3^2}+\frac{m_4k_4^2}{2L_4^2}+U,\mathrm{\ where\ }E_1=\frac{|v_1|^2}{2m_1}-\frac{k_1}{|x_1|}.$$
 Next we perform the energy reduction to get rid of the variables $L_4,\ell_4$.

We solve for $L_4$ using energy conservation. Suppose the total energy of the system is zero; we get
\begin{equation*}
\frac{1}{L_4^2}=\frac{m_3k_3^2}{m_4k_4^2L_3^2}\left(1-\frac{2L_3^2}{m_3k_3^2}(E_1+U)
\right),
\end{equation*}
hence
\begin{equation}
L_4=L_3\frac{m_4^{1/2}k_4}{m_3^{1/2}k_3}\left(1+\frac{L_3^2}{m_3k_3^2}(E_1+U)+h.o.t.\right),
\label{eq: L4}
\end{equation}
where the $higher$ $order$ $terms$ are in $E_1+U$.
We treat $\ell_4$ as the new time. So we divide the Hamiltonian equations by the equation
$\frac{d\ell_4}{dt}=-\frac{m_4k_4^2}{L_4^3}+\frac{\partial U}{\partial L_4},$
whose reciprocal is
$\frac{dt}{d\ell_4}=-\frac{L_4^3}{m_4k_4^2}\left(1+\frac{L_4^3}{m_4k_4^2}\frac{\partial U}{\partial L_4}+O(U^2)\right).$

Eliminating $L_4$ using \eqref{eq: L4} we get
\begin{equation}
\begin{aligned}\label{eq: dt/dl4}
\frac{dt}{d\ell_4}&=-\frac{\left(\frac{m_4^{1/2}k_4}{m_3^{1/2}k_3}\right)^3}{m_4k_4^2}L_3^3\left(1+\frac{3L_3^2}{m_3k_3^2}(E_1+U)\right)-\frac{\left(\frac{m_4^{1/2}k_4}{m_3^{1/2}k_3}\right)^6}{m_4^2k_4^4}L_3^6\frac{\partial U}{\partial L_4}+h.o.t.\\
&=-(1+O(\mu))L_3^3\left(1+3(1+O(\mu))L_3^2(E_1+U)\right)-(1+O(\mu))L_3^6\frac{\partial U}{\partial L_4}+h.o.t,\\
\end{aligned}
\end{equation}
where in the last equality, we use the fact that $k_{3,4},m_{3,4}=1+O(\mu)$.

Now we write the equations of motion as follows:
\begin{equation}\label{eq: ham}
\begin{cases}
&\frac{dL_3}{d\ell_4}=-\frac{dt}{d\ell_4}\frac{\partial U}{\partial \ell_3},\\
&\frac{dG_3}{d\ell_4}=-\frac{dt}{d\ell_4}\frac{\partial U}{\partial g_3},\\
&\frac{dx_1}{d\ell_4}=\frac{dt}{d\ell_4}\frac{v_1}{m_1},\\
&\frac{dG_4}{d\ell_4}=-\frac{dt}{d\ell_4}\frac{\partial U}{\partial g_4},
\end{cases}\quad
\begin{cases}
&\frac{d\ell_3}{d\ell_4}=\frac{dt}{d\ell_4}\left(\frac{m_3k_3^2}{L_3^3}+\frac{\partial U}{\partial L_3}\right),\\
&\frac{d g_3}{d\ell_4}=\frac{dt}{d\ell_4}\left(\frac{\partial U}{\partial G_3}\right),\\
&\frac{dv_1}{d\ell_4}=-\frac{dt}{d\ell_4}\left(\frac{k_1 x_1}{|x_1|^3}+\frac{\partial U}{\partial x_1}\right),\\
&\frac{d g_4}{d\ell_4}=\frac{dt}{d\ell_4}\left(\frac{\partial U}{\partial G_4}\right).
\end{cases}
\end{equation}
\begin{Not}
 We denote the RHS of \eqref{eq: ham} by $\mathcal F=(\mathcal F_3;\mathcal F_1;\mathcal F_4)$.
Thus \eqref{eq: ham} takes the form
$\frac{d}{d\ell_4}\mathcal V_i=\mathcal F_i,$ $i=3,1,4$. \end{Not}

\subsection{Estimates of the Hamiltonian equations}
\label{SSApriori}
\subsubsection{Estimates of the positions}
The next step in our analysis is an important {\it a priori} bound.

We make the following standing assumptions:

We first introduce a rectangle to which $x_4,x_1$ are confined.
\begin{Def}
We let $\mathcal S_{\hat C}$ be the strip bounded by two horizontal lines, $x_\perp=\pm \hat C$ and two vertical lines $x_\parallel=-2$ and $x_\parallel=-2\chi$.
\end{Def}

\subsubsection{Estimate of the derivatives of the potential}
We make the following standing assumption
\begin{equation}\label{eq: assumprough}
|x_1|\geq0.9 \chi, \ |x_4|\leq 0.8\chi,\ |x_3|<2,\ |x_3-x_4|>\dt>0.
\end{equation}
\begin{Lm}\label{Lm: derU}
Define $\su(\ell_4)=\frac{1}{\chi^3}+\frac{\mu}{|\ell_4|^3+1}.$ Suppose we have \eqref{eq: assumprough} for both the left and right cases, and in the right case in addition that $1/C<\frac{|x_4^R(\ell_4)|}{|\ell_4|}<C$.
\begin{itemize}
\item[(a)]  Then we have the following estimates for the first order derivatives
\begin{equation}
\begin{aligned}
&\frac{\partial U^R}{\partial x_3}\lesssim \su(\ell_4),\  \frac{\partial U^R}{\partial x_4}\lesssim \frac{1}{\chi^2}+\frac{\mu }{\ell_4^4+1},\  \frac{\partial U^{R,L}}{\partial x_1}\lesssim \frac{1}{\chi^2},\ \frac{\partial U^L}{\partial x_3}\lesssim \frac{1}{\chi^3},\ \frac{\partial U^L}{\partial x_4}\lesssim \frac{1}{\chi^2}.\\
\end{aligned}\nonumber
\end{equation}
\item[(b)] the second order derivatives satisfy the following estimates
\begin{equation}
\begin{aligned}
&\frac{\partial^2 U^R}{\partial x_3^2}\lesssim \su(\ell_4),\quad \frac{\partial^2 U^R}{\partial x_3\partial x_4}\lesssim \frac{\mu}{\chi^4}+\frac{\mu }{|\ell_4|^4+1},\quad \frac{\partial^2 U^R}{\partial x_4^2}\lesssim \frac{1}{\chi^3}+\frac{\mu }{|\ell_4|^5+1},
\\
&  \frac{\partial^2 U^{R,L}}{\partial x_3\partial x_1}\lesssim\frac{1}{\chi^4},\quad \frac{\partial^2 U^L}{\partial x_3^2}\lesssim \frac{1}{\chi^3},\quad \frac{\partial^2 U^L}{\partial x_3\partial x_4}\lesssim \frac{\mu}{\chi^4}.
\end{aligned}\nonumber
\end{equation}
\item[(c)] If we assume furthermore that $x_1\in \mathcal S_{\mu\hat C}$ and $-x_4^L, x^R_4\in \mathcal S_{\hat C}$, then we have
$$\frac{\partial^2 U^{R,L}}{\partial x_1^2},\ \frac{\partial^2 U^{R,L}}{\partial x_4\partial x_1},\ \frac{\partial^2 U^L}{\partial x_4^2}\lesssim \frac{\Id_2}{\chi^3}+\frac{(\chi,1)^{\otimes 2}}{\chi^5}.$$
\end{itemize}
\end{Lm}
\begin{proof}
First,
let $X=c_1x_1+c_4x_4+c_3x_3$; then we have
 \begin{equation*}
 \begin{aligned}
\frac{\partial }{\partial x_i}\frac{1}{|X|}&=-c_i\frac{X}{|X|^3},\quad
\frac{\partial^2 }{\partial x_i\partial x_j}\frac{1}{|X|}=c_ic_j\left(\frac{-\mathrm{Id}_2}{|X|^3}+3\frac{X\otimes X}{|X|^5}\right).
 \end{aligned}
 \end{equation*}
 This is enough to give us the estimates $  \frac{\partial U^L}{\partial x_4},\frac{\partial U^L}{\partial x_3}$, $ \frac{\partial^2 U^{R,L}}{\partial x_3\partial x_1}$, $ \frac{\partial^2 U^L}{\partial x_3^2},$$\frac{\partial^2 U^L}{\partial x_3\partial x_4}$, $\frac{\partial^2 U^{R,L}}{\partial x_4\partial x_1},\frac{\partial^2 U^L}{\partial x_4^2}$.

 Second, for the estimates $\frac{\partial U^{R,L}}{\partial x_1}$ and $ \frac{\partial^2 U^{R,L}}{\partial x_1^2}$, we need to utilize the cancelation due to the Kepler potentials $\frac{1}{\mu |x_1|}$. To see the cancelation, we next introduce $f(t)=\frac{1}{|a- tb|}$ so $\frac{1}{|a-b|}=f(1)$. Thus we get $f(1)=f(0)+f'(\xi)$ for some $\xi\in[0,1]$, hence $\frac{1}{|a-b|}=\frac{1}{|a|}-\frac{a-\xi b}{|a-\xi b|^3}\cdot b$ for some $\xi\in[0,1]$ by the mean value theorem. From this, we get $O(1/\chi)$ instead of the $O(\frac{1}{\mu\chi})$ estimate for the term $\frac{2\mu+1}{\mu|x_1|}-\frac{1}{\mu\left|x_1+\frac{\mu}{1+2\mu}x_4+\frac{\mu}{1+\mu} x_3\right|}$ in $U^R$ and the term $\frac{(\mu+1)^2}{\mu|x_1|}-\frac{1}{\mu\left|x_1-\frac{\mu}{1+\mu}x_4+\frac{\mu}{1+\mu} x_3\right|}$ in $U^L$. To get the first order derivative $\frac{\partial U^{R,L}}{\partial x_1}$, we use the following: $-\frac{a-b}{|a-b|^3}=-\frac{a}{|a|^3}+ (-\frac{\mathrm{Id}_2}{|a-\xi b|^3}+ 3\frac{(a-\xi b)\otimes (a-\xi b)}{|a-\xi b|^5})\cdot b$ for some $\xi\in [0,1]$. Similarly, we get the second order derivative $\frac{\partial^2 U^{R,L}}{\partial x_1^2}$ using
 \begin{equation*}
 \begin{aligned}
 &\left(-\frac{\mathrm{Id}_2}{|a-b|^3}+ 3\frac{(a-b)\otimes (a-b)}{|a-b|^5}\right)-\left(-\frac{\mathrm{Id}_2}{|a|^3}+ 3\frac{a\otimes a}{|a|^5}\right) \\
 =&-3\frac{\mathrm{Id}_2}{|a-\xi b|^5}\langle a-\xi b,b\rangle- 3\frac{b\otimes (a-\xi b)+(a-\xi b)\otimes b}{|a-\xi b|^5}+15\frac{(a-\xi b)\otimes (a-\xi b)}{|a-\xi b|^7} \langle a-\xi b,b\rangle
 \end{aligned}
 \end{equation*}
 and the fact that $x_4$ and $x_1$ are almost parallel due to the assumptions $x_1\in \mathcal S_{\mu \hat C}$ and $-x_4^L, x^R_4\in \mathcal S_{\hat C}$.

Finally, for all the remaining estimates, we use the expansion $\frac{1}{|a+b|}=\frac{1}{|a|}\frac{1}{\sqrt{1+z}}=\frac{1}{|a|}(1-\frac12z+\sum_{n=2}^\infty c_nz^n)$ where $z=\frac{2\langle a,b\rangle}{|a|^2}+\frac{|b|^2}{|a|^2}$. We apply the expansion to the term $U_{34}:=\frac{1+\mu}{|x_4|}-\left(\frac{1}{\left|x_4+\frac{\mu x_3}{1+\mu}\right|}
+\frac{\mu}{\left|\frac{x_3}{1+\mu}-x_4\right|}\right)$ in $U^R$ such that $$\frac{1}{\left|x_4+\frac{\mu x_3}{1+\mu}\right|}=\frac{1}{|x_4|}(1-\frac{1}{2}z_1+\sum_{n=2}^\infty c_nz_1^n),\quad z_1=\frac{1}{|x_4|^2}(2 \mu\langle \frac{x_3}{1+\mu},x_4 \rangle+\frac{\mu^2}{(1+\mu)^2}|x_3|^2)$$ and $$\frac{\mu}{\left|\frac{x_3}{1+\mu}-x_4\right|}=\frac{\mu}{|x_4|}(1-\frac{1}{2}z_2+\sum_{n=2}^\infty c_nz_2^n),\quad z_2=\frac{1}{|x_4|^2}(-2\langle x_4,\frac{x_3}{1+\mu}\rangle+\frac{1}{(1+\mu)^2}|x_3|^2).$$
 It can be verified that the $O(1/|x_4|)$ and $O(1/|x_4|^2)$ terms in $U_{34}$ are canceled. So we get $$-U_{34}=\frac{-\mu}{2(1+\mu)}\frac{|x_3|^2}{|x_4|^3}+\frac{1}{|x_4|}\sum_{n=2}^\infty c_n(z_1^n+\mu z_2^n)$$ and we have the estimates $z_1=O(\frac{\mu}{|\ell_4|+1})$ and $z_2=O(\frac{1}{|\ell_4|+1})$. The exponential convergence in $z_i$ allows us to take derivatives term by term. For instance, we have $$-\frac{\partial U_{34}}{\partial x_3}=\frac{-\mu}{1+\mu}\frac{x_3}{|x_4|^3}+\frac{1}{|x_4|}\left(\frac{\partial z_1}{\partial x_3}(\sum_{n=2}^\infty n c_n  z_1^{n-1})+\mu\frac{\partial z_2}{\partial x_3}(\sum_{n=2}^\infty n c_n z_2^{n-1})\right)=O(\frac{\mu}{|\ell_4|^3+1})$$
 using the estimates $\frac{\partial z_1}{\partial x_3}=O(\frac{\mu}{|\ell_4|+1})$ and $\frac{\partial z_2}{\partial x_3}=O(\frac{1}{|\ell_4|+1})$. We apply the same procedure to all the remaining estimates above.
 \end{proof}

\begin{Lm}\label{LmEnergyCons}
Suppose \eqref{eq: assumprough} and in addition $|v_1|<C.$
Then on the zero energy level, we have
$$E_4=-E_3+O(\mu),\quad \mathrm{as}\quad 1/\sqrt\chi\leq\mu\to 0. $$
\end{Lm}
\begin{proof}
This lemma follows directly from the total energy conservation. We write the Hamiltonian as
$0=E_1+E_3+E_4+U.$
We  estimate the potential $U=O(\mu)$ by the assumption. Next $E_1=\frac{1}{2m_1}|v_1|^2-\frac{k_1}{|x_1|}$ where $m_1\simeq \frac{1}{2\mu}$ and $k_1\simeq 1/\mu$ in both the left and right cases. This gives us that $E_1=O(\mu).$

\end{proof}
\begin{Lm}\label{Lm: ham}
Suppose we have \eqref{eq: assumprough}.
\begin{itemize}
\item[(a.1)] Suppose in the right case in addition that $$1/C<|G_3|\leq|L_3|<C,\quad 1/C<\frac{|x_4(\ell_4)|}{|\ell_4|}<C, \quad |v_1|<C,$$ then
$$\left|\frac{dt}{d\ell_4}\right|\sim 1,\quad \frac{d}{d\ell_4}(L_3,\ell_3,G_3,g_3)=(0,-1,0,0)+O(\su(\ell_4),\mu, \su(\ell_4),\su(\ell_4)).$$
\item[(a.2)] Suppose in the left case, in addition to \eqref{eq: assumprough} that $$1/C<|G_3|\leq |L_3|<C,\quad |v_1|<C,$$ then we have
$$ \left|\frac{dt}{d\ell_4}\right|\sim 1,\quad \frac{d}{d\ell_4}(L_3,\ell_3,G_3,g_3)=(0,-1,0,0)+O\left(\frac{1}{\chi^3},\mu, \frac{1}{\chi^3},\frac{1}{\chi^3}\right).$$
\item[(b.1)]Suppose in the right case in addition to \eqref{eq: assumprough} that
$$1/C<|L_3|<C,\ 1/C<\frac{|x_4(\ell_4)|}{|\ell_4|}<C,\ |v_1|<C,\ |G_4|<C, \ x_4\in \mathcal S_{\hat C},\ x_1\in \mathcal S_{\mu\hat C}.$$
Then $\frac{d}{d\ell_4}(G_4,g_4)=(\sv(\ell_4))_{1\times 2}$, where  $\sv(\ell_4):=\frac{1}{\chi^2}+\frac{\mu}{|\ell_4|^3+1}$.
\item[(b.2)] Suppose in the left case in addition to \eqref{eq: assumprough} that $$1/C<|L_3|<C,\quad |G_4|<C,\quad |v_1|<C,\quad -x_4\in \mathcal S_{\hat C},\quad x_1\in \mathcal S_{\mu\hat C},$$
 then $\frac{d}{d\ell_4}(G_4,g_4)=O\left(\frac{1}{\chi^2}\right)_{1\times 2}.$
\item[(c.1)] In both the right and the left cases we have
$$\dfrac{dx_1}{dt}=\frac{v_1}{m_1},\quad \frac{dv_1}{dt}=-\frac{k_1 x_1}{|x_1|^3}+h.o.t.=O\left(\frac{1}{\mu \chi^2}\right).$$
\item[(c.2)] If we assume in addition that $x_4^R,\ -x_4^L\in \mathcal S_{\hat C},\quad x_1\in \mathcal S_{\mu\hat C}$, then we have
$$\dfrac{dx_1}{dt}=\frac{v_1}{m_1},\quad \frac{dv_1}{dt}=-\frac{k_1 x_1}{|x_1|^3}+h.o.t.=O\left(\frac{1}{\mu \chi^2}, \frac{1}{ \chi^3}\right).$$

\end{itemize}
\end{Lm}

\begin{proof}
We first prove part (a.1) and (a.2). We apply \eqref{eq: dt/dl4} and the assumptions to conclude that $\frac{1}{C}\leq \left|\frac{dt}{d\ell_4}\right|\leq C$. Next we consider the Hamiltonian equation \eqref{eq: ham}. We have $\frac{\partial U}{\partial \mathcal V_3}=\frac{\partial U}{\partial x_3}\frac{\partial x_3}{\partial \mathcal V_3}$ where $\mathcal V_3=(L_3,\ell_3,G_3,g_3)$. Because of the boundedness of $L_3$, we get that $\left|\frac{\partial x_3}{\partial \mathcal V_3}\right|<C$. Now part (a.1) and (a.2) are proved by applying Lemma \ref{Lm: derU}(a).

Part (c.1) and (c.2) follow directly from the Hamiltonian equation. In fact the estimates for $\frac{dv_1}{dt}$ are given by the Kepler motion.

We next prove part (b.1) and (b.2). We need  the estimates
$\frac{\partial x_4}{\partial g_4}\cdot x_4=0$ and $\frac{\partial x_4}{\partial G_4}\cdot x_4=O(\ell_4)$ from part (c) of Lemma \ref{Lm: dx/dDe}
in Appendix \ref{appendixa}.
For example in the right case in \eqref{eq: hamR}, we consider the derivative of the term
\[\frac{\partial }{\partial G_4}\frac{1}{\mu\left|x_1+\frac{\mu}{1+2\mu}x_4+\frac{\mu}{1+\mu} x_3\right|}=\frac{\left(x_1+\frac{\mu}{1+2\mu}x_4+\frac{\mu}{1+\mu} x_3\right) }{(1+2\mu)\left|x_1+\frac{\mu}{1+2\mu}x_4+\frac{\mu}{1+\mu} x_3\right|^3}\cdot \frac{\partial x_4}{\partial G_4}\]
We claim that the above expression is $O(1/\chi^2).$ Indeed, the denominator is of order $\chi^3.$ The main contributions to the numerator come from
$\left\langle x_4,\frac{\partial x_4}{\partial G_4}\right\rangle$ which is  $O(\ell_4)$ due to part (c) of
Lemma \ref{Lm: dx/dDe}, and from $\left\langle x_1,\frac{\partial x_4}{\partial G_4}\right\rangle.$
To estimate the later product we write
$$ x_1=\frac{|x_1|}{|x_4|} \cos \alpha \; x_4+|x_1| \sin \alpha \; \mathbf{e} $$
where $\alpha=\angle(x_4,x_1)$ and $\mathbf{e}$ is the unit vector perpendicular to $x_4.$
We note that  the assumptions $x_4\in \mathcal S_{\hat C}$ and $x_1\in \mathcal S_{\mu\hat C},$
imply $\alpha=O(1/\ell_4)$. This gives
$$\left\langle x_1,\frac{\partial x_4}{\partial G_4}\right\rangle=O\left(\frac{|x_1|}{|x_4|}\right)
\left\langle x_4, \frac{\partial x_4}{\partial G_4}
\right\rangle+|x_1| O(\alpha) O\left(\left|\frac{\partial x_4}{\partial G_4}\right|\right)=
O(\chi)$$
where the last estimate comes from Lemma \ref{Lm: dx/dDe}(c).

The other derivatives are estimated similarly and result in the estimates of the lemma.
In particular,
the $O\left(\frac{\mu}{|\ell_4|^3+1}\right)$ part in our bound for $\cF_4$ comes from differentiating the terms in $U^R$ which do not
contain $x_1.$ This bound
is obtained by multiplying the $O\left(\frac{\mu}{|\ell_4|^4+1}\right)$ term in the estimate of
$\frac{\partial U^R}{\partial x_4}$ in part (a) of Lemma \ref{Lm: derU} by the $O(\ell_4)$ bound on
$\frac{\partial x_4}{\partial G_4}$ from Lemma \ref{Lm: dx/dDe}(c).

\end{proof}

In the next lemma, we show that the assumption $\mathbf{AG}$, which is only on the initial and final conditions, gives control of the dynamics of $x_1,v_1,x_3,v_3$ for all time.
\begin{Lm}\label{Lm: position}
Assume $\mathbf{AG}$ and in addition $|x_4|<0.8\chi$ for all time $t\in [0,T], \ T\leq 10\chi,$ when $\Glob$ is defined.

 Then we have as $1/\chi\ll\mu\to 0$ that
\begin{itemize}
\item[(a)] $L_3(t)-L_3(0)=O(\mu),\quad G_3(t)-G_3(0)=O(\mu),\quad g_3(t)-g_3(0)=O(\mu),\quad t\in [0,T]$,
\item[(b)] $|x_4^{R,L}|=(L_3^2(0)+O(\mu)+o_{|\ell_4|\to\infty}(1))|\ell^{R,L}_4|.$
\item[(c)] $x^R_{1,\perp}(t)=O(\mu),\ x^R_{1,\parallel}(t)-x^R_{1,\parallel}(0)=O(\mu\chi), $ $v_{1,\perp}^R(t)=O(1/\chi)$, $v_{1,\parallel}^R(t)-v_{1,\parallel}^R(0)=O(\frac{1}{\mu\chi})$
for all $t$ defining piece $(I)$.
\end{itemize}
\end{Lm}

\begin{proof}
We first consider the piece of orbit going from the section $\{x_{4,\parallel}^R=-2\} $ to the section $\{x_{4,\parallel}^R=-\chi/2\} $ that we called piece $(I)$.
Let $[0,\tau]$ be the maximal time interval such that
\begin{equation}\label{eq: assump}
\begin{aligned}
&\left|\frac{G_3(t)}{G_3(0)}\right|,\,\left|\frac{L_3(t)}{L_3(0)}\right|\,
\in \left[\frac{3}{4},\frac{4}{3}\right],\\
& |x_1^R(t)|\geq 0.95\chi, \quad |v_{1}^R(t)|\leq C_0'+1.
  \end{aligned}\end{equation}
  During time $[0,\tau]$, we have \eqref{eq: assumprough} satisfied with the help of the additional assumption $|x_4|<0.8\chi$.

We always have $|x_4|\geq 2$ since $x_4$ is to the left of the section $\{x_4=-2\}$. So we get $L_4(t)=L_3(t)+O(\mu)$ for $t\in [0,\tau]$ using Lemma \ref{LmEnergyCons} and the bound on $|v_{1}^R(t)|$ in \eqref{eq: assump}. Then using formula \eqref{eq: Q4} and $e_4=\sqrt{1+\frac{G_4^2}{L_4^2}}$, we find
\begin{equation}\label{eq: Ol4}
\begin{aligned}
|x_4|&=\frac{1}{m_4k_4}L_4\sqrt{L_4^2(\cosh u-e_4)^2+G_4^2\sinh^2 u}\\
&=\frac{1}{m_4k_4}L_4\sqrt{L_4^2(\cosh^2 u-2e_4\cosh u+e_4^2)+(L^2_4e_4^2-L_4^2)\sinh^2 u}\\
&=\frac{1}{m_4k_4}L_4^2\sqrt{1-2e_4\cosh u+e_4^2+e_4^2\sinh^2 u}\\
&=\frac{1}{m_4k_4}L_4^2\sqrt{1-2e_4\cosh u+e_4^2+e_4^2(\cosh^2u-1 )}\\
&=\frac{1}{m_4k_4}L_4^2\sqrt{(1-e_4\cosh u)^2}=L_4^2(e_4\cosh u-1).
\end{aligned}
\end{equation}
We always have $e_4\geq 1$, so we get $|\ell-u|\geq|\sinh u|\geq \frac{e^{|u|}-1}{2}$ from \eqref{eq: hypul}, so that $u=o(\ell)$ as $|\ell|\to\infty$. Continuing \eqref{eq: Ol4}, we have \[e_4\cosh u\simeq e_4|\sinh u|=|\ell-u|=(1+o(1))|\ell|.\]
So we obtain \begin{equation}\label{EqO(l4)}|x_4|= L_4^2(1+o(1))|\ell_4|,\quad \mathrm{as}\ |\ell_4|\to\infty.\end{equation}
By assumption \eqref{eq: assump}, we get that $\frac{1}{C}<\frac{|x_4(\ell_4)|}{|\ell_4|}<C$ for some constant $C$ for the time interval $[0,\tau]$. So Lemma \ref{Lm: ham}(a.1) is applicable. Over time $O(\chi)$, we get $$L_3(t)-L_3(0)=O(\mu),\ G_3(t)-G_3(0)=O(\mu),\ g_3(t)-g_3(0)=O(\mu).$$

From the Hamiltonian equation, we get
\begin{equation}\label{Eqx1v1}\dot x_1=\frac{1}{m_1} v_1,\quad\dot v_1=\frac{x_1}{\mu |x_1|^3} +O\left(\frac{1}{\chi^2}\right)\end{equation} where the $O(1/\chi^2)$ estimate is from Lemma \ref{Lm: derU}(a). We get that the assumptions on $x_1,v_1$ in \eqref{eq: assump} are satisfied over time $O(\chi)$. This proves the estimate in item (a) for piece $(I)$. For the estimate in item (a) for piece $(III)$ and $(V)$, in order to repeat the above argument, we only need to show that $v_1$ is bounded so that Lemma \ref{LmEnergyCons} applies.  From equation \eqref{DerLR}, we get
\begin{equation}\label{Eqleftright}x_1^L=\frac{1}{1+\mu} x_1^R+\frac{2\mu}{1+2\mu} x_4^R,\quad v_1^L=\frac{1+\mu}{1+2\mu} v_1^R+ v_4^R.\end{equation}
Since we have $L^R_4(t)=L_3(t)+O(\mu)=L_3(0)+O(\mu)$ on the section $\{x_{4,\parallel}^R=-\chi/2\} $, we get $v_4^R$ hence $v_1^L$ is bounded on the section. Now we can repeat the previous case $(I)$ argument to establish the estimate in item (a) for piece $(III)$. Similarly for piece $(V)$.

We next work on item (c) for piece $(I)$ only. We assume $\tau$ is the maximal time such that the following holds:
\begin{equation}\label{EqAssymp1}-1.2\chi\leq x_{1,\parallel}^R(t)\leq -0.95\chi,\quad |x_{1,\perp}^R(t)|\leq 1,\quad |v_1^R(t)|\leq C_0'+1.\end{equation}
On the time interval $[0,\tau]$, we have $\dot v_{1,\perp}^R =O(\frac{1}{\mu\chi^3}+\frac{1}{\chi^2})$ from \eqref{Eqx1v1}, hence the oscillation of $v_{1,\perp}^R$ is bounded by $O\left(\frac{1}{\mu\chi^2}+\frac{1}{\chi}\right)$ and the oscillation of $x^R_{1,\perp}$ is $O(\mu)$ using the equation $\dot x_{1,\perp}^R=\frac{1}{m_1} v^R_{1,\perp}$ and $m_1\sim1/\mu$. Therefore on the time interval $[0,\tau]$, we always have $|x^R_{1,\perp}|\leq 1$ and obtain the estimate $x^R_{1,\perp}=O(\mu)$. Similarly, we have $x^R_{1,\parallel}(t)-x^R_{1,\parallel}(0)=O(\mu\chi)$, $v_{1,\perp}^R(t)=O(1/\chi)$  and $v_{1,\parallel}^R(t)-v_{1,\parallel}^R(0)=O(\frac{1}{\mu\chi})$. This implies that the assumption \eqref{EqAssymp1} holds for the entire piece $(I)$ and we have proved item (c).
\end{proof}

\subsection{Justification of the assumptions of Lemma \ref{Lm: ham}}
We demonstrate that the orbits satisfying \textbf{AG}
satisfy the assumptions of Lemma \ref{Lm: ham}. In \textbf{AG} we make assumptions on the initial and final values of $x_4,v_4$. However, in the assumptions of Lemma \ref{Lm: ham}, we require that the orbit of $x_4$ to be bounded in $\mathcal S_{\hat C}.$
\begin{Lm}
\label{Lm: strip}
Assume {\bf AG} for an orbit defined on the time interval $[0,T]$ such that $x^R_{4,\parallel}(0)=x^R_{4,\parallel}(T)=-2$ and $Q_4$ turns around $Q_1$ once in the sense of Definition \ref{DefSing}. Then there exist constants $\hat C, \mu_0$ such that for $\mu\leq\mu_0$ we have \[|G_4^{L,R}(t)|<\hat C,\quad\mathrm{and}\quad x^R_4,-x_4^L\in\mathcal S_{\hat C},\quad x_1^{R,L}\in\mathcal S_{\mu\hat C}\mathrm{\ for\ all\ } t\in [0, T].\]
\end{Lm}
To prove this result, we first need the following sublemma.
\begin{sublemma}\label{KeepDirection}
Given small $\tilde\theta>0$ there exist $\mu_0, $ $\chi_0$ such that under the assumptions of Lemma \ref{Lm: strip} if
$\mu\leq \mu_0,$ $\chi\geq \chi_0$ for all $t\in [0,  T]$ then

\begin{itemize}
\item[(a)] for all $t$ when the orbit is moving to the right of the sections $\{x^{R}_{4,\parallel}=\chi/2\}$ and $\{x^{L}_{4,\parallel}=-\chi/2\}$, we have \begin{equation}
\label{EqMiss}
|\pi-\theta^+_4(t)|<\tilde\theta,\quad |\theta^-_4(t)|<\tilde\theta,
\end{equation}
where $\theta^+_4$ $($respectively $\theta^-_4$$)$ is the angle of the outgoing $($respectively incoming$)$ asymptote of $x_4$ $($see Notation \ref{NotAsymp}$)$.
\item[(b)] for all $t$ when the orbit is moving to the left of the sections  $\{x^{R}_{4,\parallel}=\chi/2\}$ and $\{x^{L}_{4,\parallel}=-\chi/2\}$, we have $|\theta_4(t)-\pi|<\tilde\theta$ for the piece with $u<0$ and $|\theta_4(t)|<\tilde\theta$ for the piece with $u>0$. $($See Appendix \ref{subsection: hyp} for the convention of $u).$ 
\end{itemize}
\end{sublemma}
\begin{proof} 

Pick a large $D$ and let $\tau^*$ be the first time when $|x_4^R(\tau^*)|=D$ and let $\bar \tau$ be the first time when the orbit intersects the section $\{x_{4,\parallel}^R=-\chi/2\}$.  It is enough to consider below the times $t\geq \tau^*$. Indeed, $\theta_4^+$ changes by $O(D\mu)$ on the time segment $[0, \tau^*]$ since we have $\theta^+_4=\pi+g_4+\arctan\frac{G_4}{L_4}$ by \eqref{EqAsymp}, hence $\dot \theta_4^+=O(\mu)$ by the Hamiltonian equation, and we know that $\tau^*$ and $D$ are constants independent of $\mu$. Next,
$$ \theta^+(\tau^*)=\arctan\left(\frac{v_{4,\perp}^R} {v_{4,\parallel}^R}\right)(\tau^*)
+o_{D\to\infty}(1). $$
To fix our idea we suppose that $\tilde\theta\leq \left|\arctan\left(\frac{v_{4,\perp}^R} {v_{4,\parallel}^R}\right)(\tau^*)\right|\leq \frac{\pi}{4}.$ This implies that $v_4$ has a horizontal component that is bounded away from zero, therefore it takes time $O(\chi)$ to travel between two consecutive sections. Under this assumption, we get that $|x_4|\leq (\frac{\sqrt 2}{2}+O(\mu))\chi\leq 0.8\chi$ so that the assumptions of Lemma \ref{Lm: position} are satisfied.

Let $\tau^\dagger$ be the first time when $|v_4^R(\tau^\dagger)-v_4^R(\tau^*)|>0.01.$ For
$t\leq\min(\brtau, \tau^\dagger)$ we have
$$ D+c(t-\tau^*)<|x_4^R(t)|<D+C(t-\tau^*). $$
On the other hand, the Hamiltonian equations give
\[\dot v_4^R=-(1+O(\mu))\frac{x^R_4+O(\mu x_3)}{|x^R_4|^3}+O\left(\frac{x^R_4}{\chi^3}\right),\]
where $x_3$ is bounded by Lemma \ref{Lm: position}(a).
Integrating this estimate we get
\[|v^R_4(t)-v^R_4(\tau^*)|\leq 1.1 \int_{\tau^*}^t \frac{1}{|D+c(s-\tau^*)|^2}+O\left(\frac{D+C(s-\tau^*)}{\chi^3}\right)ds=
\frac{1.1}{cD}+O(t^2/\chi^3).\]
Thus, the oscillation of $v^R_4$ is smaller than $\frac{2}{cD}$ if $t\leq \tau^\dagger$ and
$t=O(\chi).$ It follows that $\brtau=O(\chi)$ and $\tau^\dagger>\brtau.$

Next we change the coordinates to the left variables.
From \eqref{DerLR}, we get that
\begin{equation}\label{EqR->L}
x_4^L=-x_1^R+\frac{1+\mu}{1+2\mu} x_4^R,\quad v_4^L=-\frac{2\mu}{1+2\mu} v_1^R+\frac{1}{1+\mu} v_4^R,\end{equation}
from which we obtain on the section $\{x^R_{4,\parallel}=-\chi/2\}$ that
\begin{equation*}
\left|\arctan\frac{v^L_{4,\perp}}{v^L_{4,\parallel}}(\brtau)-
\arctan\frac{v^R_{4,\perp}}{v^R_{4,\parallel}}(\tau^*)\right|\leq
\frac{3}{c^2D}+O(\mu),\quad |x^L_{4,\perp}(\bar\tau)|\geq \frac{\chi}{2}\frac{2}{3}\tilde\theta,\end{equation*}
by choosing $D$ large such that $\frac{4}{c^2D}<\frac{\tilde\theta}{3}$.
We apply a similar estimate to the left piece of orbit to show that for the orbit between the two sections $\{x_{4,\parallel}^R=-\chi/2\}$ and $\{x_{4,\parallel}^L=D\}$ the motion of $x_4^L$ is almost linear with the estimate
$\left|\arctan\frac{v^L_{4,\perp}}{v^L_{4,\parallel}}\right|>\frac{\tilde\theta}{3},\  |x^L_{4,\perp}|\geq \chi\frac{1}{3}\tilde\theta$ when arriving at the section $\{x_{4,\parallel}^L=D\}$.

 On the other hand, by definition we have $x^L_4=q_4-q_1=Q_4-Q_1$. In order to have a returning orbit to the section $\{x_{4,\parallel}^R=-2\}$, the two bodies $Q_4$ and $Q_1$ have to have a close encounter. This contradicts our estimate of $x^L_4$ at the end of the previous paragraph. This proves that the slope of the initial outgoing asymptote satisfies $|\theta_4^+-\pi|<\tilde\theta.$ Similarly, we get that the final incoming asymptote satisfies $|\theta_4^-|<\tilde\theta$ by repeating the above argument with the time reversed. For the estimate in part (b), we apply the same argument above treating the orbits as starting from a neighborhood of $Q_1$ moving towards $Q_2$.

\end{proof}

\begin{proof}[Proof of Lemma \ref{Lm: strip}]
The proof is a refinement of that of Sublemma \ref{KeepDirection} with the same general idea.

The fact that $x^R_1\in \mathcal S_{\mu\hat C}$ is given by Lemma \ref{Lm: position}(c).

{\bf Step 1, the boundedness of $G_4^{R,L}$.}

Without any assumption on $G_4$, we have that $\dot G_4=\frac{\partial U}{\partial x_4}\cdot\frac{\partial x_4}{\partial g_4}$ is $O(1/\chi)$ in the left case and is $O(\frac{1}{\chi}+\frac{\mu}{\ell_4^3+1})$ in the right case, directly from the estimate of $\frac{\partial U}{\partial x_4}$ in Lemma \ref{Lm: derU}, the bound on $L_3,$ and  the fact that $\left|\frac{\partial x_4}{\partial g_4}\right|=|x_4|$. This implies the oscillation of $G_4$ is $O(1)$ over time $O(\chi)$. By Sublemma \ref{KeepDirection}(b)
 we see that in the left case the slope of asymptotes of $x_4^L$ is bounded by $\tilde\theta$, so $\left|\frac{G_4}{L_4}\right|\leq 2\tilde\theta$. Next by Lemma \ref{Lm: position} and \ref{LmEnergyCons}, we get that $1/C<L_4<C$. Therefore $G_4^L=O(1)$ when $Q_4$ comes close to $Q_1$. We also assumed that $G_4^R=O(1)$ on the section $\{x^{R}_{4,\parallel}=-2\}$. We get $G_{4}^L,G_4^R=O(1)$ for all the time when they are defined, in particular,
when evaluated on the sections $\{x^{R}_{4,\parallel}=-\frac{\chi}{2}\}$ and $\{x^{L}_{4,\parallel}=\frac{\chi}{2}\}$.

{\bf Step 2, the estimate of $x^R_{4,\perp}$ and $v^R_{4,\perp}$.}

We use \eqref{EqR->L} to get the relation for angular momentum\beq\label{eq: GLGR}
\beal
G_4^L&=v_4^L\times x_4^L=\left(\frac{1}{1+\mu}v^R_4-\frac{2\mu}{1+2\mu}v_1^R\right)\times\left(\frac{x_4^R(1+\mu)}{1+2\mu}-x_1^R\right)\\
&=\frac{G_4^R}{1+2\mu}-\frac{1}{1+\mu}v^R_4\times x_1^R-\frac{2\mu(1+\mu)}{(1+2\mu)^2}v_1^R\times x_4^R+\frac{2\mu}{1+2\mu}v_1^R\times x_1^R.
\enal
\eeq
Using the estimates on $v^R_1=O(1,\frac{1}{\chi}),x^R_{1,\parallel}\leq -\chi, x^R_{1,\perp}=O(\mu)$ from Lemma \ref{Lm: position}(c), we get that on the section $\{x^{R}_{4,\parallel}=-\frac{\chi}{2}\}$
\[O(1)=G_{4}^L-\frac{G_4^R}{1+2\mu}=(1+O(\mu))[v^R_{4,\perp}\chi+O(\mu)]+O(\mu)[O(1)+x^R_{4,\perp}]+O(\mu)[O(1)].\]
This implies \begin{equation}\label{Eqv4Rperp}v^R_{4,\perp}\chi=O(1)+O(\mu)x^R_{4,\perp}.\end{equation} Next, we have \begin{equation}\label{EqX4RPerp}O(1)=G_{4}^R=v^R_{4,\perp}x^R_{4,\parallel}-x^R_{4,\perp}v^R_{4,\parallel}=-v^R_{4,\perp}\frac{\chi}{2}-x^R_{4,\perp}v^R_{4,\parallel}.\end{equation}
Substituting \eqref{Eqv4Rperp} into \eqref{EqX4RPerp} and using the lower bound on $v_{4,\parallel}^R$
we get $x^R_{4,\perp}=O(1)$. We next substitute the $x^R_{4,\perp}$ estimate back into \eqref{Eqv4Rperp} to get $v^{R}_{4,\perp}=O(1/\chi)$. We then obtain $x^L_{4,\perp}=O(1)$ and $v^{L}_{4,\perp}=O(1/\chi)$ using \eqref{EqR->L}. Remember that these estimates are only established so far on the sections $\{x^{R}_{4,\parallel}=-\frac{\chi}{2}\}$ and $\{x^{L}_{4,\parallel}=\frac{\chi}{2}\}$.

\textbf{Step 3, bounding the right piece of the orbit $x^R_{4}$.}

We next bound the orbit between the sections $\{x^{R}_{4,\parallel}=-D\}$ and $\{x^{R}_{4,\parallel}=-\frac{\chi}{2}\}$ for some large constant $D$ independent of $\chi,\mu$. Suppose the orbit intersects the section $\{x^{R}_{4,\parallel}=-\frac{\chi}{2}\}$ at time $t_0$ and the section $\{x^{R}_{4,\parallel}=-D\}$ at time $t_1$.  We have $|x_{4,\perp}(t_1)|\leq C$ for some constant $C$ due to the continuity of the flow and the boundedness of the initial conditions on the section $\{x^{R}_{4,\parallel}=-2\}$ as assumed. We have
\begin{equation}\label{EqInt} x_{4,\perp}(t)=x_{4,\perp}(t_0)+v_{4,\perp}(t_0)(t-t_0)+
\int_{t_0}^t \int_{t_0}^u\ddot{x}_{4,\perp}(s) dsdu, \end{equation}
where for piece $(I)$ we have $t_1<s<t_0$ and
$$\ddot{x}_{4,\perp}(s)=O\left(\frac{x_{4,\perp}(s)}{|x_4(s)|^3}+\frac{|x_4|}{\chi^3}+\frac{\mu |x_3|}{|x_4(s)|^3}\)=O\(\frac{x_{4,\perp}(s)+\mu}{(\chi/2+(s-t_0))^3}+\frac{|x_4|}{\chi^3}\).$$ By Step 2, we have $x_{4,\perp}(t_0)=O(1)$ and $v_{4,\perp}(t_0)=O(1/\chi)$. We bound the double integral of the term $\frac{|x_4|}{\chi^3}$ in $\ddot{x}_{4,\perp}$ by a constant. So we get
$$|x_{4,\perp}(t)|\leq C+C\sup_{t_0<s<t_1}|x_{4,\perp}(s)|\int_{t_0}^t \int_{t_0}^u\frac{1}{(\chi/2+(s-t_0))^3}dsdu\leq C+\frac{C}{D}\sup_{t_0<s<t_1}|x_{4,\perp}(s)|.$$
Choosing $D>C$, this shows that $x_4^R\in \mathcal S_{\hat C}$ for some large $\hat C$ for the piece of orbit in consideration. For piece $(V)$, $t_0 < s < t_1$ and the denominator of the integrand becomes $(\chi/2 - (s - t_0))^3$.

\textbf{Step 4, bounding the left piece of the orbit $x^L_{4}$ and the returning orbit.}

We have $x_{4,\perp}^L=O(1)$ on the section $\{x^{L}_{4,\parallel}=D\}$ by Sublemma \ref{KeepDirection}(b).
We apply the same argument as in Step 3 to both halves of the left piece between the section $\{x^{L}_{4,\parallel}=D\}$
and section $\{x^{R}_{4,\parallel}=-\frac{\chi}{2}\}$. This shows that $-x_4^L\in \mathcal S_{\hat C}$ for the piece of orbit in consideration. The fact that $x_1^L\in \mathcal S_{\mu\hat C}$ follows from the $v_{4,\perp}, x_{4,\perp}$ estimate in Step 2, Lemma \ref{Lm: position}(c) and equation \eqref{Eqx1v1}, \eqref{Eqleftright}.


\end{proof}
\begin{Rk}\label{RkStrip}  From the proof, we see that Lemma \ref{Lm: strip} still holds if instead of {\bf AG}, we assume $(AG.3)$, $(AG.1)$ and $|x^R_{4,\perp}(0)|<2$, $x^R_{4,\parallel}(0)=-2$ for the initial condition and $x_4^L(T)=0$ for the final condition. The final condition implies a collision between $Q_1$ and $Q_4$. The crucial ingredient in the proof is that $Q_1$ must come close to $Q_4$ in order to have a return orbit.
\end{Rk}
All the assumptions of Lemma \ref{Lm: ham} are implied by {\bf AG} due to Lemma \ref{Lm: strip} and Lemma \ref{Lm: position}, so we have the following.
\begin{Cor}\label{Cor: ham}
Assume {\bf AG}, then \begin{itemize}
\item[(a)] in the right case, we have
\begin{equation}
\begin{aligned}
\mathcal F^R=(0,1,0_{1\times 8})+O\left(\su(\ell_4),\mu,\su(\ell_4),\su(\ell_4);\mu,\frac{\mu}{\chi},\frac{1}{\mu\chi^2},\frac{1}{\chi^3};\sv(\ell_4),\sv(\ell_4)\right);
\end{aligned}
\nonumber
\end{equation}
\item[(b)] in the left case, we have
\begin{equation}
\begin{aligned}
\mathcal F^L=(0,1,0_{1\times 8})+O\left(\frac{1}{\chi^3},\mu,\frac{1}{\chi^3},\frac{1}{\chi^3};\mu,\frac{\mu}{\chi},\frac{1}{\mu\chi^2},\frac{1}{\chi^3}; \frac{1}{\chi^2},\frac{1}{\chi^2}\right).
\end{aligned}
\nonumber
\end{equation}
\end{itemize}
\end{Cor}
\begin{Lm}\label{Lm: tilt}
Assume $\mathbf{AG}$, then we have
\begin{itemize}
\item[(a)] when $x_4$ is moving to the right of the sections $\left\{x_{4,\parallel}^{R}=-\frac{\chi}{2}\right\}$ and $\left\{x_{4,\parallel}^{L}=\frac{\chi}{2}\right\}$, we have \[\tan g_4=-\mathrm{sign}(u)\frac{G_4}{L_4}+O\left(\frac{\mu}{|\ell_4|^2+1}+\frac{1}{\chi}\right), \quad\mathrm{as\ } |\ell_4|\to\infty,\ 1/\chi\ll\mu\to 0.\]
\item[(b)] When $x_4$ is moving to the left of the sections $\left\{x_{4,\parallel}^{R}=-\frac{\chi}{2}\right\}$ and $\left\{x_{4,\parallel}^{L}=\frac{\chi}{2}\right\}$, then $G_4,g_4=O(1/\chi)$ as $1/\chi\ll\mu\to 0$.
\end{itemize}
\end{Lm}

\begin{proof} The proof is to integrate the estimates of $\frac{d}{d\ell_4}(G_4,g_4)$ in Lemma \ref{Lm: ham}$(b.1)$ and $(b.2)$.

\textbf{Step 1.} We prove part $(b)$. Integrating
the Hamiltonian equation for $G_4^L,g_4^L$ in Lemma \ref{Lm: ham}(b.2) starting from $\ell_4=0$ we get $(G_4^L,g_4^L)(\ell_4)=(G_4^L,g_4^L)(0)+O(1/\chi)$ when arriving at the sections $\left\{x_{4,\parallel}^{R}=-\frac{\chi}{2}\right\}$ and $\left\{x_{4,\parallel}^{L}=\frac{\chi}{2}\right\}$. To conclude part (b), we need to show that the initial conditions $G_4^L(0),g_4^L(0)$ are bounded by $O(1/\chi)$.
Using \eqref{eq: Q4l} (we omit the superscript $L$ and subscript $4$), we have on the sections $\left\{x_{4,\parallel}^{R}=-\frac{\chi}{2}\right\}$ and $\left\{x_{4,\parallel}^{L}=\frac{\chi}{2}\right\}$
\begin{equation*}\begin{aligned}&x_{4,\perp}=\frac{1}{mk}(\sin g L^2(\cosh u-e)+\cos g LG\sinh u)\\
&=\frac{1}{mk}(\sin g(0) L^2(\cosh u-e)+\cos g(0) LG(0)\sinh u)+O(1).\end{aligned}\end{equation*}
Note that this holds for both large positive and large negative $u$ and that on both sections $\cosh u$ and
$|\sinh u|$ are of order $\chi.$
By  Lemma \ref{Lm: strip}, we have  $-x^L_4\in \mathcal S_{\hat C}$, which shows that $|x^L_{4,\perp}|\leq 2\hat C$ on the sections $\left\{x_{4,\parallel}^{R}=-\frac{\chi}{2}\right\}$ and $\left\{x_{4,\parallel}^{L}=\frac{\chi}{2}\right\}$. Next, we apply Lemma \ref{LmEnergyCons} and Lemma \ref{Lm: position} to get that $1/C<|L_4|<C$. Next, arguing as in Lemma \ref{Lm: strip} Step 1, we have that $|G_4|<C$, this implies that $e_4<C$. This implies that $|g(0)|,|G(0)|=O(1/\chi)$.

\textbf{Step 2.} Then we use the matrix $R\cdot L^{-1}$ in Proposition \ref{Prop: main} to convert the left variables to the right to obtain $v_4^R=O(\mu)v_1^L\pm (1+O(\mu))v_4^L$.
From Step 1 and \eqref{eq: Q4l}, we get that the slope of $v_4^L$ is $g_{4}^L-\arctan\frac{G_{4}^L}{L_{4}^L}+O(1/\chi^2)=O(1/\chi)$, and from assumption (AG.1) and part (c.1) of Lemma \ref{Lm: ham} that the slope of $v_1^L$ is $O(1/\chi)$. So the slope of $v_4^R$ is $g_{4}^R-\arctan\frac{G_{4}^R}{L_{4}^R}+O(1/\chi^2)=O(1/\chi)$ on the sections $\left\{x_{4,\parallel}^{R}=-\frac{\chi}{2}\right\}$ and $\left\{x_{4,\parallel}^{L}=\frac{\chi}{2}\right\}$ due to \eqref{eq: Q4}.

\textbf{Step 3.} To prove part $(a)$, we use the $O(1/\chi)$ estimates of the slope of $v_4^R$ in Step 2 as our initial condition. We get that the oscillation of $G_4^R,\ g_4^R$ is $O\left(\frac{\mu}{|\ell_4|^2+1}+\frac{1}{\chi}\right)$ from $\ell_4=O(\chi)$ to $\ell_4$ by integrating the $\frac{d G^R_4}{d\ell_4},\frac{d g^R_4}{d\ell_4}$ estimates in Lemma \ref{Lm: ham}.
\end{proof}

\subsection{Collision exclusion}\label{ssct: nocollision}

The following lemma excludes the possibility of collisions between $Q_1$ and $Q_4$.

\begin{Lm}\label{Lm: nocollision}
If we assume  $(AG.3)$, $(AG.1)$ and $|x^R_{4,\perp}(0)|<2$, $x^R_{4,\parallel}(0)=-2$ for the initial condition and $x_4^L(T)=0$ for the final condition $($collision between $Q_4$ and $Q_1)$, then there is an orbit bouncing back from the $Q_1$-$Q_4$ collision lying entirely in a strip $\mathcal S_{\hat C}$ for some constant $\hat C$. Moreover we have
$\brG_4^R+G^R_4=O(\mu)$ when evaluated on the section $\{x_{4,\parallel}^R=-2\}$, where $G_4^R$ and $\brG^R_4$ are the angular momentum of $(x_4,v_4)^R$ before and after the application of the global map respectively.
\end{Lm}
\begin{proof}
By Remark \ref{RkStrip}, the assumption implies the assumptions of Lemma \ref{Lm: position} according to Lemma \ref{Lm: strip} so we can use the conclusions of Lemma \ref{Lm: position}.

Suppose we have a collision. We compare the bouncing back orbit (subscript $out$) with the time reversal of the incoming orbit
(subscript $in$). We will show that the orbits are close and so the values of $G_4$ will be close when evaluated on the section $\{x_{4,\parallel}^R=-2\}$.

\textbf{Step 1, Comparing orbits to the left of the line $x_{4, \parallel}^R=-\frac{\chi}{2}.$}

For the collisional orbit and the bouncing back orbit to the left of the section $x_{4, \parallel}^R=-\frac{\chi}{2},$ the value $\ell_4=0$ corresponds to the collision. When $\ell_4=0$, all the values of $(L_3,g_3,x_1,G_4,g_4)$ are the same for the two orbits and the variables $(G_3,\ell_3, v_1)$ have opposite signs for the two orbits. Moreover, since the variables $\mathbf Y=(G_4,g_4)$ are constants of motion when the potential $U$ is neglected, we get that  for small $|\ell_4|>0$ the values of $\mathbf Y$ for the two orbits will stay close.

Let $\mathbf{F}$ be the RHS of
the corresponding Hamiltonian equations \eqref{eq: ham} for $\mathbf Y$.
We denote $\dt\mathbf Y=\mathbf Y_{in}-\mathbf Y_{out}$. Taking the difference of the Hamiltonian equations for $\mathbf Y_{in}$ and $\mathbf Y_{out}$, we have
\[\dfrac{d}{d\ell_4}\dt\mathbf Y=\frac{\partial \mathbf F}{\partial \mathbf Y}(\mathbf Y_{in},\hat{\mathbf Y}_{in})
\dt\mathbf Y+O\left(|\dt\mathbf Y|^2\right)+
[\mathbf F(\mathbf Y_{in}, \hat{\mathbf Y}_{in})-\mathbf F(\mathbf Y_{in}, \hat{\mathbf Y}_{out})],\]
where we denote $\hat{\mathbf Y}:=(L_3,\ell_3,G_3,g_3;x_1,v_1)$ and in the bracketed term we fix $\mathbf Y_{in}$ since the difference $\mathbf Y_{in}-\mathbf Y_{out}$ is considered in the $\dt\mathbf Y$ and $|\dt\mathbf Y|^2$ terms.

We trace the orbit back to the section $\{x_{4, \parallel}^R=-\frac{\chi}{2}\}.$ During the $O(\chi)$ time, the oscillation of $x_1$ and $v_1$ are estimated as $$(\dt x_1,\dt v_1)=(O(\mu\chi),O(\mu); O(1), O(1/\chi))$$ from Lemma \ref{Lm: ham}(c.2) as well as the sign change of the initial condition.

The term $[\mathbf F(\mathbf Y_{in}, \hat{\mathbf Y}_{in})-\mathbf F(\mathbf Y_{in}, \hat{\mathbf Y}_{out})]$ can be estimated as $O\left(\frac{1}{\chi^3}+\frac{\mu}{\chi^2}\right)$, where the estimate $\frac{1}{\chi^3}$ is given by $\frac{\partial\mathbf F}{\partial x_3}$ due to different values of $\ell_3$ for the two orbits, and the estimate $O\left(\frac{\mu}{\chi^2}\right)=\frac{\partial \mathbf F}{\partial x_1}\dt x_1$ is due to  different $x_1$ for the two orbits (see Lemma \ref{Lm: derU}).  The variable $v_1$ enters through $\frac{dt}{d\ell_4}$ in \eqref{eq: ham} and causes a difference in $\mathbf F$ that is much smaller than the above two cases for $\ell_3, x_1$.

We denote by $\ell_4^f$ the time when the time reversed
incoming orbit hits $\{x_{4, \parallel}^R=-\frac{\chi}{2}\}.$

Note the initial condition $(\mathbf Y_{in}-\mathbf Y_{out})(0)=0$ and that  the fundamental solution of the variational equation
$Z'=\frac{\partial \mathbf F}{\partial \mathbf Y} Z$ is $O(1)$ (the fundamental solution is given by the matrix $N_3$ in
Proposition \ref{Prop: main}. Here we pick only the rows and columns corresponding to variables in $\mathbf Y$. In fact we have the estimate $\frac{\partial \mathbf F}{\partial \mathbf Y}=O(1/\chi)$ in Lemma \ref{Lm: var}(b) below).
Since  we have by DuHamel's principle that
$$\int_{0}^{\ell_4^f}O(1) O\(\frac{\mu}{\chi^2}+\frac{1}{\chi^3}\)d\ell_4=
O\left(\frac{\mu}{\chi}\right),$$
the Gronwall inequality gives $ \mathbf Y_{in}-\mathbf Y_{out}=O\left(\frac{\mu}{\chi}\right)$ at time $\ell_4^f$.

\textbf{Step 2, Cartesian coordinates.} We already had in Step 1 the estimate of $\dt(x_1^L,v_1^L)$.
We need to control the change of $(x_4,v_4)^L$ as well.
We have
\[\dt (x_4,v_4)^L=\left[\dt L_4\frac{\partial }{\partial L_4}+\dt G_4\frac{\partial }{\partial G_4}+\dt g_4\frac{\partial }{\partial g_4}\right]^L(x_4,v_4)^L.\]
Note that here we do not have $\dt \ell_4\frac{\partial }{\partial \ell_4}$ since we have the same $\ell_4=\ell_4^f$ for the two orbits, so $\dt\ell_4=0$.
We use Lemma \ref{Lm: dx/dDe} in the appendix to get the partial derivatives
$$\frac{\partial v_4}{\partial *}=O(1),\quad  \frac{\partial x_4}{\partial *}=O(\chi),\quad *=L_4,G_4,g_4,$$
and in particular $$\frac{\partial v_{4,\perp}}{\partial L_4}=C\frac{G_4}{L_4(G_4^2+L_4^2)}=O(1/\chi),\quad \frac{\partial x_{4,\perp}}{\partial L_4}=C\frac{G_4L_4^2\ell_4}{(G_4^2+L_4^2)}=O(1)$$ since $G_4=O(1/\chi)$ by Lemma \ref{Lm: tilt} (b).

The estimates for $(\dt G_4, \dt g_4)^L=O(\mu/\chi)$ are obtained in Step 1. The estimate of $\dt L_4$ is obtained from \eqref{eq: L4} and $\dt L_4=\frac{\partial L_4}{\partial \mathbf Y}\dt \mathbf Y+L_{4}(\mathbf Y_{in},\hat{ \mathbf Y}_{in})-L_{4}(\mathbf Y_{in},\hat{ \mathbf Y}_{out})$. We have $$\frac{\partial L_4}{\partial \mathbf Y}=O(1),\quad \dt \mathbf Y=O(\mu/\chi),\quad L_{4}(\mathbf Y_{in},\hat{ \mathbf Y}_{in})-L_{4}(\mathbf Y_{in},\hat{ \mathbf Y}_{out})=O(\mu),$$ where the main contribution to the last $\mu$ estimate is given by $\frac{\partial L_4}{\partial v_1}=O(\mu), \dt v_1=O(1)$. So we get $\dt L_4=O(\mu)$.

This tells us that
$$\dt (x_4,v_4)^L=O(\mu\chi, \mu; \mu, \mu/\chi). $$
We also have $(x_4,v_4)^L=O(\chi,1; 1, 1/\chi)$ and $(x_1,v_1)^L=O(\chi,\mu, 1, 1/\chi)$ at time $\ell_4^f$ (see Step 2 of Lemma \ref{Lm: strip} for the estimate of $x_{4,\perp}$ and $v_{4,\perp}$).

\textbf{Step 3, Comparing angular momenta.} Using the relation
\beq
\beal
&G_4^R=v_4^R\times x_4^R=\left(\frac{1+\mu}{1+2\mu}v^L_4+\frac{2\mu}{1+2\mu}v_1^L\right)\times\left(\frac{x_4^L}{1+\mu}+x_1^L\right)\\
&=\frac{G_4^L}{1+2\mu}+\frac{1+\mu}{1+2\mu}v^L_4\times x_1^L+\frac{2\mu}{(1+2\mu)(1+\mu)}v_1^L\times x_4^L+\frac{2\mu}{1+2\mu}v_1^L\times x_1^L
\enal
\eeq
and the results of Step 2 we get $\delta G_4^R=O(\mu)$ at time $\ell^f_4.$

{\bf Step 4, Oscillation of $G^R_4$ to the right of $\{x_{4, \parallel}^R=-\frac{\chi}{2}\}$}.
Now we consider the right pieces of orbits. For the collisional orbit the oscillations of $G^R_4$
are $O(\mu)$ by integrating the estimate in Corollary \ref{Cor: ham}. In order to apply the same estimate to the bouncing back orbit, we need to show that the bouncing back orbit lies in the strip $\mathcal S_{\hat C}$.

We first get $$\dt (x_4,v_4)^R=O(\mu\chi, \mu; \mu, \mu/\chi)$$
using $R\cdot L^{-1}$ from \eqref{DerLR} of Proposition \ref{Prop: main} and the estimates of $\dt (x_4,v_4,x_1,v_1)^L$ above. At the time $\ell_4^f$, the collisional orbit is on the section $\{x_{4,\parallel}^R=-\chi/2\}$, but the bouncing back orbit might be $O(\mu\chi)$ distance away. Tracing the bouncing back orbit over time $O(\mu\chi)$ so that it is also on the section $\{x_{4,\parallel}^R=-\chi/2\}$, since $\dot v_4=-\frac{k_4 x_4}{|x_4|^3}+h.o.t.$, we see that $v_4$ gains a new oscillation $O(\mu/\chi)$. Comparing the two orbits on the section $\{x_{4,\parallel}^R=-\chi/2\}$ we get $\dt (x_4,v_4)^R=O(0, \mu; \mu, \mu/\chi)$. Applying \eqref{EqInt}, we see that the $x_{4,\perp}^R$ components for the two orbits stay $O(D\mu)$-close to each other when traveling between $\{x_{4,\parallel}^R=-\chi/2\}$ and $\{x_{4,\parallel}^R=-D\}$ for some large $D$ independent of $\mu,\chi$. This shows that the bouncing back orbit also lies in the strip $\mathcal S_{\hat C}$ since the collisional orbit does.

Now we apply Corollary \ref{Cor: ham} to get that the oscillation of $G_4^R$ for the bouncing back orbit is also $O(\mu)$ when traveling between the sections $\{x^R_{4,\parallel}=-\chi/2\}$ and $\{x^R_{4,\parallel}=-2\}$.

Steps 1--4 show that difference between the angular momenta of the reversed incoming orbit and the bouncing back orbit is $O(\mu).$
Without the time reversal
we have $\bar G^R_4+G^R_4=O(\mu)$ as claimed.
\end{proof}
The possibility of collision between $Q_4$ and $Q_1$ is excluded since in Gerver's construction,
$\bar G_4+G_4$ is always bounded away from zero independent of $\mu$. Now we exclude the possibility of collisions between $Q_3$ and $Q_4$.
Note that $Q_3$ and $Q_4$ have two potential collision points corresponding to two intersections of the ellipse of $Q_3$
and the branch of the hyperbola utilized by $Q_4.$ See Fig 1 and 2.
Now it follows from Lemma~\ref{Lm: landau}(b) that $Q_3$ and $Q_4$ do not collide near the intersection where they have
the close encounter. We need also to rule out the collision near the second intersection point.
This was done by Gerver in \cite{G2}. Namely he
shows that the times for $Q_3$ and $Q_4$ to move from one crossing point to the other are different.
As a result, if $Q_3$ and $Q_4$ come to the correct intersection points nearly simultaneously,
they do not collide at the wrong points. In the setting of our paper ($\mu>0$), the travel times for $Q_3$ and $Q_4$ to move from one crossing point to the other are $O(\mu)$ perturbations of that computed in \cite{G2}. So it is impossible to have a collision at a wrong intersection point.
\subsection{Proofs of Lemma \ref{LmGMC0} and \ref{LmPMC0}}\label{SSPfGMPM}
In this section, we prove of Lemma \ref{LmGMC0} and \ref{LmPMC0}. 
Now we prove Lemma \ref{LmGMC0}.
\begin{proof}[Proof of Lemma \ref{LmGMC0}]
We first prove part (c). We first get that $x_{1,\parallel}(t)=(1+O(\mu))x_{1,\parallel}(0)=(1+O(\mu))\chi$ by integrating part (c) of Lemma \ref{Lm: ham} for $t\in [0,100\chi]$. Next by Lemma \ref{LmEnergyCons} and Lemma \ref{Lm: position} we get that $E_4(t)=-E_3(t)+O(\mu)=\frac{1}{2}+O(\dt+\mu)$ for $t\in [0,100\chi]$. This implies that $|v_4|\geq 1-O(\dt+\mu)$. Next, by Sublemma \ref{KeepDirection}, we get that $|v_{4,\parallel}|>\frac{3}{4}$ by choosing $\dt,\mu,\tilde\theta$ small. So we get that the total return time $T\leq 2(1+O(\mu))\chi/(3/4)<3\chi(<100\chi)$.

Since {\bf AG} implies the assumptions of Lemma \ref{Lm: ham}, combined with Lemma \ref{Lm: position}, we get part (a) of Lemma \ref{LmGMC0} from Lemma \ref{Lm: position}, and part (b) from Lemma \ref{Lm: tilt} using Notation \ref{NotAsymp}.
\end{proof}
Now we are ready to prove Lemma \ref{LmPMC0}.
\begin{proof}[Proof of Lemma \ref{LmPMC0}]
 The idea of the proof is to integrate the equations $\dot x_1$ and $\dot v_1$ for the pieces $(I),(III)$ and $(V)$, and apply the coordinate changes $(II)$ and $(V)$, to keep track of the change of $x_1$ and $v_1$. The main idea was sketched in Remark \ref{Rkx1v1}. 

{\bf Step 0, preparations. }

We use Lemma \ref{Lm: ham}(c.1) and (c.2) to get that $v_1=v_1(0)+O(1/\mu\chi, 1/\chi^2) $ during time $O(\chi)$. It follows from $\frac{dx_{1,\parallel}}{dt}=v_{1,\parallel}/m_1$ that over time $O(\chi)$, the horizontal component $x_{1,\parallel}$ can move only distance $O(\mu\chi)$. Moreover, the local map takes only $O(1)$ time as $1/\chi\ll \mu\to 0.$ 

Initially, we have angular momentum conservation $G_1+G_3+G_4=0.$ Also from the initial conditions in the assumption $(i)$ and the total energy conservation (Lemma \ref{LmEnergyCons}), we estimate
\[|G_3|,\ |G_4|\leq 2C_0'+1.\]
We get from the definition of angular momentum, Lemma \ref{Lm: position}(c) and assumption $(ii)$ that
 \begin{equation}\nonumber| v_{1,\perp}(0)|\leq \left|v_{1,\parallel}\frac{ x_{1,\perp}}{ x_{1,\parallel}}\right|+\left|\frac{ G_1}{x_{1,\parallel}}\right|\leq O\left(\frac{\mu}{\chi}\right)+\frac{2(2C_0'+1)}{\chi}\leq \frac{4(C_0'+1)}{\chi}.\end{equation}

{\bf Step 1, piece $(I)$ composed with the local map.}

We integrate the $\frac{dv_{1,\parallel}}{dt}$ estimates from the section $\{x^R_{4,\parallel}=-2,\ v^R_{4,\parallel}>0\}$ to the section $\{x^R_{4,\parallel}=-\frac{\chi}{2}, v^R_{4,\parallel}<0\}$ (note that the local map is included).
The total traveling time is $<3\chi$ by Lemma \ref{LmGMC0}(c). In the following we use the notation $a=O_+(b)$ if $b>0,\ a=O(b)$ and $\frac{a}{b}>c>0$ for some constant $c$. Using Lemma \ref{Lm: ham}(c) we have as $1/\chi\ll \mu\to 0$, \begin{equation}\nonumber
\begin{aligned}
& x^R_{1,\parallel}-x^R_{1,\parallel}(0)=O_+(\mu\chi)v^R_{1,\parallel}(0),\quad x^R_{1,\perp}-x^R_{1,\perp}(0)=O(\mu\chi)v^R_{1,\perp}(0)=O(\mu),\\
&v^R_{1,\parallel}\in[-\bar c_1,-c_1]+O(1/\mu\chi),\quad v_{1,\perp}^R=O(1/\chi).
\end{aligned}
\end{equation}
on the section $\{x^R_{4,\parallel}=-\frac{\chi}{2}, v^R_{4,\parallel}<0\}$. On the same section, we also have
\[v^R_{4,\parallel}=-\sqrt{2E_4}+O(1/\chi)=-\sqrt{-2E_3}+O(\mu),\quad v_{4,\perp}^R=O(1/\chi)\]
by Lemma \ref{LmEnergyCons}, Lemma \ref{Lm: tilt} and equation \eqref{eq: Q4} (see also Step 2 of the proof of Lemma \ref{Lm: strip}).

{\bf Step 2, piece $(III).$}

We use Lemma \ref{Lm: strip} to get $x_4^R\in \mathcal{S}_{\hat C}$, i.e. $|x_{4,\parallel}^R|\leq 2\chi,\ |x_{4,\perp}^R|\leq \hat C$. Then we use \eqref{DerLR} to get, on the section $\{x^R_{4,\parallel}=-\frac{\chi}{2}, v^R_{4,\parallel}<0\}$, that
\begin{equation}
\begin{aligned}
& x^L_{1,\parallel}=\frac{1}{1+\mu} x^R_{1,\parallel}(0)+O_+(\mu\chi)v^R_{1,\parallel}(0)-\frac{\mu\chi}{1+2\mu}=\frac{1}{1+\mu} x^R_{1,\parallel}(0)-\frac{\mu\chi}{1+2\mu}-O_+(\mu\chi),\\
& x^L_{1,\perp}=\frac{1}{1+\mu}x^R_{1,\perp}(0)+O(\mu\chi)v^R_{1,\perp}(0)+O(\mu/\chi)=O(\mu),\\
&v^L_{1,\parallel}\in \frac{1+\mu}{1+2\mu}[-\bar c_1,-c_1]-\sqrt{-2E_3}+O(\mu),\quad v^L_{1,\perp}=O(1/\chi),\quad  \mathrm{as}\ 1/\chi\ll \mu\to 0.\\
\end{aligned}
\end{equation}
We integrate $\frac{dv_1}{dt}$ again over time $O(\chi)$ to get
\begin{equation}
\begin{aligned}
& x^L_{1,\parallel}=\frac{1}{1+\mu} x^R_{1,\parallel}(0)-\frac{\mu\chi}{1+2\mu}-O_+(\mu\chi),\quad x^L_{1,\perp}=O(\mu),\\
&v^L_{1,\parallel}\in \frac{1+\mu}{1+2\mu}[-\bar c_1,-c_1]-\sqrt{-2E_3}+O(\mu),\quad  v^L_{1,\perp}=O(1/\chi), \text{ as } \frac{1}{\chi}\ll \mu\to 0
\end{aligned}
\end{equation}
when arriving at the section $\{x^L_{4,\parallel}=\frac{\chi}{2}, v_{4,\parallel}^L>0\}$ where the $-O_+(\mu\chi)$ term in $x_{1,\parallel}^L$ has absorbed a new $-O_+(\mu\chi)$ contribution since $v_{1,\parallel}^L< 0$. Again it follows from Lemma \ref{Lm: tilt} and the energy conservation that
\[v^L_{4,\parallel}=\sqrt{2E_4}+O(1/\chi)=\sqrt{-2E_3}+O(\mu),\quad v_{4,\perp}^R=O(1/\chi).\]

{\bf Step 3, piece $(V).$}

We then apply \eqref{DerLR} and $-x_4^L\in\mathcal{S}_{\hat C}$ (Lemma \ref{Lm: strip}) to get that on the section $\{x^L_{4,\parallel}=\frac{\chi}{2}\}$,
\begin{equation}\label{Eqx1v1right}
\begin{aligned}
x^R_{1,\parallel}&=\frac{1+\mu}{1+2\mu}\left(\frac{x^R_{1,\parallel}(0)}{1+\mu}-\frac{\mu\chi}{1+2\mu}\right)-\frac{\mu\chi}{1+2\mu}-O_+(\mu\chi)\\
&=\frac{x^R_{1,\parallel}(0)}{1+2\mu}-\frac{\mu(2+3\mu)\chi}{(1+2\mu)^2}-O_+(\mu\chi),\\
x^R_{1,\perp}&=O(\mu),\quad v^R_{1,\perp}=O(1/\chi),\\
v^R_{1,\parallel}&\in\frac{1}{1+2\mu}[-\bar c_1,-c_1]-2\sqrt{-2E_3^*}+O(\dt+\mu),\\
\end{aligned}
\end{equation}
 as $1/\chi\ll \mu\to 0$, where the extra $O(\mu)$ in $v_{1,\parallel}^R$ comes from the oscillation of $E_3$ established in Lemma \ref{LmGMC0}(a), 
 and $O(\dt)$ is the deviation of the initial value $E_3$ from Gerver's value $E^*_3$, which is bounded by $C_3\dt$. Finally, we get the same estimate as \eqref{Eqx1v1right} when arriving at the section $\{x^R_{4,\parallel}=-2,\ v^R_{4,\parallel}>0\}$ with a new $-O_+(\mu\chi)$ added to $x_{1,\parallel}^R$. This completes one application of $\cP$.
The information that we need from $x_{1,\parallel}^R$ is that $x_{1,\parallel}^R<x^R_{1,\parallel}(0)$ after one application of $\cP.$
Indeed, it follows from the first row of \eqref{Eqx1v1right} and the assumption on $x_{1,\parallel}^R$ that
\begin{equation}
\label{Eqx1R//}
x^R_{1,\parallel}-x^R_{1,\parallel}(0)=\frac{2\mu\chi}{1+2\mu}-\frac{\mu(2+3\mu)\chi}{(1+2\mu)^2}-O_+(\mu\chi)=-O_+(\mu^2\chi)-O_+(\mu\chi)<0.
\end{equation}
{\bf Step 4, renormalization. }

One period in Gerver's construction consists of $\cR\circ\tilde\Glob\circ \cP^2$. We repeat the above procedure to get after $\cP^2$ (we use {\it double bar} for the orbit parameters),
\begin{equation}\label{Eqx1v1Renorm}
\begin{aligned}
& \bar{\bar x}^R_{1,\parallel}-x^R_{1,\parallel}(0)=-O_+(\mu\chi)<0,\quad \bar{\bar x}^R_{1,\perp}=O(\mu),\quad  \bar{\bar v}^R_{1,\perp}=O(1/\chi),\\
&\bar{\bar v}^R_{1,\parallel}\in\frac{1}{(1+2\mu)^2}[-\bar c_1,-c_1]-2\sqrt{-2E_3^*}-2\sqrt{-2E_3^{**}}+O(\dt+\mu),
\end{aligned}
\end{equation}
as $1/\chi\ll \mu\ll \dt\to 0$. The last step is to apply the renormalization $\cR$. Let us forget about the rotation by $\beta$ in Definition \ref{Def: renorm} for a moment and consider only the rescaling. We expect that
\[\cR(\bar{\bar v}^R_{1,\parallel})=\frac{1}{\sqrt{\lb}}\bar{\bar v}^R_{1,\parallel}\in[-\bar c_1,-c_1],\]  which is implied by
\[\bar{\bar v}^R_{1,\parallel}\in\frac{1}{(1+2\mu)^2}[-\bar c_1,-c_1]-2\sqrt{-2E_3^*}-2\sqrt{-2E_3^{**}}+O(\dt+\mu)\subset \sqrt{\lb}[-\bar c_1,-c_1],\]where $\lb$ is the renormalization factor in Definition \ref{Def: renorm}.
This implies \[c_1+\tilde c(\dt+\mu)\leq \frac{2}{\sqrt{\lb}-1}(\sqrt{-2E_3^*}+\sqrt{-2E_3^{**}})\leq \bar c_1-\tilde c(\dt+\mu)\] for some constant $\tilde c$ bounding the $O$ in the above estimates. We choose \[\bar c_1= \frac{4}{\sqrt{\lb}-1}(\sqrt{-2E_3^*}+\sqrt{-2E_3^{**}}), \quad c_1=\frac{1}{\sqrt{\lb}-1}(\sqrt{-2E_3^*}+\sqrt{-2E_3^{**}})\] so that the above inequality is satisfied uniformly for all sufficiently small $\mu,\dt,1/\chi$.
This completes the proof for $\cR(\bar{\bar v}^R_{1,\parallel})$.

{\bf Step 5, part (b), the estimates of $\cR(x_1).$}

This estimate follows by iterating \eqref{Eqx1R//} twice and applying the renormalization map $\cR.$


Now let us take care of the rotation $\beta$ of $\cR(\bar{\bar v}_{1,\parallel})$ which is $\arctan\frac{\bar{\bar x}_{1,\perp}}{\bar{\bar x}_{1,\parallel}}=O(\mu/\chi)$ by definition. This produces an error of $O(\mu/\chi)$ to $\bar{\bar v}_{1,\parallel}$, which can be absorbed into the $O$ part of the estimate of $\bar{\bar v}_{1,\parallel}$ in \eqref{Eqx1v1Renorm} so that we leave our choice of $c_1,\ \bar c_1$ unchanged.

If $\tilde\Glob$ in the definition of $\cR$ were the identity, then the rotation Rot$(\beta)$ would set $x_{1,\perp}$ to zero. Applying $\tilde\Glob$ causes an error $O(\mu/\chi)$ to $x_{1,\perp}$ obtained by integrating the estimate of $\dot x_{1,\perp}=O(\frac{\mu}{\chi})$
over time $O(1)$. Since $\mu/\chi\ll \frac{1}{2\sqrt{\lambda \chi}}$, we get the estimate for $\cR( x_{1,\perp})$ in the statement.
The $\cR(x_{1,\parallel})$ estimate comes from the definition of $\tilde\chi=\lambda\chi$ and the cube.

\textbf{Step 6, bounding the angular momentum and vertical component of the velocity $ \bar{\bar v}_{1,\perp}$.}

 After $\cP^2$ and $\mathcal R\circ\tilde\Glob\circ \mathcal P^2$, we have angular momentum conservation \[\bar{\bar{G}}_1+\bar{\bar{G}}_3+\bar{\bar{G}}_4=0,\ \mathrm{and}\ \cR(\bar{\bar{G}}_1)+\cR(\bar{\bar{G}}_3)+\cR(\bar{\bar{G}}_4)=0.\] After renormalization $\cR(\bar{\bar E}_3)$ is now $-1/2+O(1/\sqrt\chi)$, and $|\bar{\bar x}_{4,\perp}|\leq 2$. The energy conservation shows that $|\cR(\bar{\bar v}_4)|\leq 1+O(\mu)$, so that we have \[|\cR(\bar{\bar{G}}_3)|,\ |\cR(\bar{\bar{G}}_4)|\leq 2C_0'+1.\] 

We get from the definition of angular momentum and \eqref{Eqx1v1Renorm} that
\begin{equation}\nonumber |\cR(\bar{\bar{ v}}_{1,\perp})|\leq \left|\cR(\bar{\bar{v}}^R_{1,\parallel})\frac{ \cR(\bar{\bar{x}}^R_{1,\perp})}{\cR(\bar{\bar{ x}}^R_{1,\parallel})}\right|+\left|\frac{ \cR(\bar{\bar{G}}_1)}{ \cR(\bar{\bar{x}}^R_{1,\parallel})}\right|\leq O\left(\frac{\mu}{\chi}\right)+\frac{2(2C_0'+1)}{\tilde\chi}\leq \frac{4(C_0'+1)}{\tilde\chi}.\end{equation}
This completes the proof of the $\cR(\bar{\bar v}_{1,\perp})$ estimate in part (c) by defining $C_1:=4(C_0'+1)$.
\end{proof}

\subsection{Choosing angular momentum, proof of Lemma \ref{LmChooseAM}}
\label{SSAngMom}
In this section, we prove Lemma \ref{LmChooseAM}. We first need two auxiliary results.
\begin{sublemma}
\label{LmHitTarget}
Let $S\subset U_1(\dt')$ and $\te_4$ be as in part $(a)$ of Lemma \ref{LmChooseAM}. Then there exists $\tilde{\ell}_3$ such that  $\pi_{e_4}\cP(\mathcal S(\te_4, \tilde{\ell}_3))=e_4^{**}.$ There are analogous statements for $S\subset U_2(\dt')$ and $S\subset U_0(\dt')$ as in part $(b)$ and part $(c)$ of Lemma \ref{LmChooseAM}.
\end{sublemma}
We give the proof of this sub lemma immediately we complete the proof of Lemma \ref{LmChooseAM}. The next sublemma is easy to prove.
\begin{sublemma}
\label{LmCover}
Let $F$ be a map from $\R^2$ to $\R^2$ such that
\begin{enumerate}
\item $F(a^*)=b^*$, for some $a^*,b^*\in \R^2$;
\item if $|F(z)-b^*|<R$ for some $R>0,$ then $\Vert dF(z)(X)\Vert \geq \brchi \Vert X\Vert,$ for all vectors $X\in T_z\R^2$ and for some $\ \bar\chi>1.$
\end{enumerate}
 Then for each $b$ such that
$|b-b^*|<R$ there exists $z$ such that $|z|<R/\brchi$ and $F(z)=b.$
\end{sublemma}
With the help of the two sub lemmas, we finish the proof of Lemma \ref{LmChooseAM}.
\begin{proof}[Proof of Lemma \ref{LmChooseAM}]
We consider part (a) first.
Pick a piece of $\dt'$-admissible surface $S\subset U_1(\dt)$ for $\dt'<\dt$.

Choose any ${\tilde e}_4\in (e^*_4-\delta ^{\prime}+\frac 1{\chi},e^*_4+\delta ^{\prime}-\frac 1{\chi})$.  By Sublemma \ref{LmHitTarget}, there exists ${\tilde \ell}_3$ such that $\pi_{e_4}{\mathcal P}({\mathcal S}({\tilde e}_4,{\tilde \ell}_3))=e_4^{**}$.  Let $\ell_3^{\prime}=\pi_{\ell_3}{\mathcal P}({\mathcal S}({\tilde e}_4,{\tilde \ell}_3))$.  Then ${\mathcal Q}_1({\tilde e}_4,{\tilde \ell}_3)=(e_4^{**},\ell_3^{\prime})$. Our coordinates allow us to treat ${\mathcal Q}_1$ as a map from ${\mathbb R}\times {\mathbb T}\to {\mathbb R}\times {\mathbb T}$.  Let ${\tilde {\mathcal Q}}_1:{\mathbb R}^2\to {\mathbb R}^2$ be the covering map of ${\mathcal Q}_1$.  We now apply Sublemma \ref{LmCover}, using $a^*=({\tilde e}_4,{\tilde \ell}_3)$, $b^*=(e_4^{**},\ell_3^{\prime})$, $F={\tilde {\mathcal Q}}_1$, $\bar \chi = c\chi$ (from Lemma \ref{LmKone}) and letting $R$ be the distance from $b^*$ to the boundary of ${\mathcal C}_2(\delta )$.  This gives us a surjective map, satisfying the expansion condition, from a subset of ${\mathcal C}_1(\delta^{\prime})$ to the open disk of radius $R$ around $b^*$.  To extend this map to other parts of ${\mathcal C}_2(\delta )$, we can apply Sublemma \ref{LmCover} again, using a different $b^*$, choosing each $b^*$ from a region to which the map has already been extended.  Because ${\mathcal C}_2(\delta )$ is open and connected, we can eventually extend our surjective map to all of ${\mathcal C}_2(\delta )$, although we might have to apply Sublemma \ref{LmCover} an infinite number of times.  Because of the expansion condition, and the fact that the diameter of ${\mathcal C}_2(\delta )$ is $O(1)$, the diameter of the pre-image $V_1({\tilde e}_4)$ is $O(1/\chi)$.  This establishes Lemma \ref{LmChooseAM}(a).

Part (b) is similar to part (a).  For part (c), we first apply Sublemma \ref{LmHitTarget} to find ${\tilde \ell}_3$ for each ${\tilde e}_4$ such that $\pi_{e_4}{\mathcal P}({\mathcal S}({\tilde e}_4,{\tilde \ell}_3))=e_4^{**}$ for a given admissible surface $S\subset U_0(\delta ^{\prime})$. This gives the first statement in part (c).  We next introduce the renormalization ${\mathcal R}$ based at the point ${\mathcal S}({\tilde e}_4,{\tilde \ell}_3)$ and obtain ${\mathcal P}{\mathcal R}\tilde {\mathbb G}$ satisfying $\pi_{e_4}{\mathcal P}{\mathcal R}\tilde {\mathbb G}({\mathcal S}({\tilde e}_4,{\tilde \ell}_3))=e_4^{**}+O(\mu)$. (We get from ${\mathcal S}({\tilde e}_4,{\tilde \ell}_3)$ to ${\mathcal P}({\mathcal S}({\tilde e}_4,{\tilde \ell}_3))$ by following the Hamiltonian flow, and we get to $\tilde {\mathbb G}({\mathcal S}({\tilde e}_4,{\tilde \ell}_3))$ by following the first part of that same flow.  ${\mathcal P}{\mathcal R}$ continues the flow on a rotated, rescaled, reflected orbit, but those transformations do not change $e_4$ or $\ell_3$, so if we continued the flow to the same section as ${\mathcal P}({\mathcal S}({\tilde e}_4,{\tilde \ell}_3))$, but rotated, rescaled, and reflected, we would still get $e_4^{**}$.  Instead, we continue a distance of $O(1)$ to a new section, mostly because of the rescaling, but also from the rotation, and this changes $e_4$ by $O(\mu)$, because of the interaction between $Q_3$ and $Q_4$.) Because $\delta^{\prime}\gg \mu$, we still have ${\mathcal Q}_0({\tilde e}_4,{\tilde \ell}_3)\in {\mathcal C}_2(\delta^{\prime})$. By Lemma \ref{LmKone} and Sublemma \ref{LmCover} again, we get a neighborhood $V_0(\tilde e_4)$ such that ${\mathcal Q}_0$ maps $V_0(\tilde e_4)$ surjectively to ${\mathcal C}_2(\delta)$. By Lemma \ref{LmKone}, the weakest expansion rate of $d{\mathcal P}$ restricted to the cone fields is a constant times $\chi$, hence the diameter of $V_0(\tilde e_4)$ is $O(1/\chi)$.  Since we have $O(1/\chi)\ll 1/\sqrt {\chi}$, we get that ${\mathcal R}$ is well-defined in a neighborhood of $V_0(\tilde e_4)$.

Part (d) is given in Lemma  \ref{Lm: landau} (b).

\end{proof}

\begin{proof}[Proof of Sublemma~\ref{LmHitTarget}.]

The idea is to apply the strong expansion of the Poincar\'{e} map in a neighborhood of the collisional orbit studied in Lemma \ref{Lm: nocollision}. Note that Delaunay coordinates regularize double collisions in the sense that none of the variables blows up at a double collision, so that our estimate of $d\Glob$ holds also for collisional orbits. We give the proof only for initial conditions on an admissible surface $S\subset U_1(\dt')$. The other cases are similar.

{\bf Step 1.}
We first show that there is a collisional orbit satisfying $x_4^L(t)=0$ at some time $t$ as $\ell_3$ varies.

We apply the local map $\Loc$ to the admissible surface $S$ with $\tilde e_4$ fixed.
Sublemma \ref{KeepDirection} and its proof shows that if after the application of the local map we have
$\theta_4^+(0)=\pi-\bar\theta$, $0<\bar\theta<(\tilde \theta$ in Lemma \ref{Lm: loc}), then the $x^L_{4,\perp}$ coordinate is a large positive number of order $\bar\theta\chi$ when the orbit hits $\{x^L_{4,\parallel}=0\}$. Similarly, if $\theta_4^+(0)=\pi+\bar\theta$ then the orbit hits the line $\{x^L_{4,\parallel}=0\}$ so that its $x^L_{4,\perp}$ coordinate is a large negative number.
By the Intermediate Value Theorem there has to be an outgoing angle $\theta_4^+(0)$ leading to a collisional orbit with $x_4^L=0$. So it suffices to show that our admissible surface $S$ contains points $\bx_1, \bx_2$ such that
$\theta_4^+(\bx_1)=\pi-\bar\theta,$ $\theta_4^+(\bx_2)=\pi+\bar\theta.$

We have the expression $\theta^+_4=\pi+\tilde g_4+\arctan\frac{\tilde G_4}{\tilde L_4}$ (see \eqref{EqAsymp} for the formula and see Lemma \ref{Lm: glob} for the {\it tilde} notation). By direct calculation we find $d\theta_4^+=\tilde L_4\hat\brlin$ (see Lemma \ref{Lm: glob} for $\hat{\bar{\mathbf l}}$ and Notation \ref{NotAsymp} for $\theta^+$). Since $TS\subset \mathcal K_1$ and the cone $\mathcal K_1$ is centered at the plane $span\{w_1,\tilde w\}$ where $\tw=\frac{\partial}{\partial\ell_3}$ (Definition \ref{DefKone}). We get \[d\theta^+\cdot\left(d\Loc \frac{\partial}{\partial\ell_3}\right)=\tilde L_4\hat\brlin_1\cdot\left(\frac{1}{\mu}(\hat{\mathbf u}_1 (\hat\lin_1 \cdot\tw)+o(1))+O(1)\right)=c(\bx)/\mu,\]
where $c(\bx)\neq 0$ by Lemma \ref{LmNonDeg}(c).
So it is enough to vary $\ell_3$ in a $O(\mu)$ neighborhood of a point whose outgoing asymptote satisfies the assumption of Lemma \ref{Lm: loc} in order to get angles of outgoing asymptotes $\pi\pm\bar\theta$. Thus we get a collisional orbit for some point denoted by $(\tilde e_4,\hat\ell_3)$.

{\bf Step 2.} We next show that there exists $\ell_3$ such that $\pi_{e_4}(\cP(\mathcal S(\te_4,\ell_3)))$ is close to $e_4^{**}$ for fixed $\tilde e_4$. Now the function $\pi_{e_4}(\cP(\mathcal S(\te_4,\cdot)))$ is a function of one variable $\ell_3$ defined in a neighborhood of $\hat \ell_3$. 

Since $e_4=\sqrt{1+(G_4/L_4)^2}$ is not an injective function of $G_4$, we use $ G_4$ instead of $e_4$ and study the function $G_4(\cP(\mathcal S(\te_4,\ell_3)))$.


 Next we compute
 \begin{equation}\label{EqExpansion}
 \frac{d}{d\ell_3}G_4(\cP(\mathcal S(\te_4,\ell_3)))=dG_4 d\Glob d\Loc \frac{\partial}{\partial \ell_3}=\chi^2  (dG_4 w_1) \hat\brlin_1\cdot\left(d\Loc \frac{\partial}{\partial \ell_3}\right)+O(\chi)=\bar c(\bx)\frac{\chi^2}{\mu}+O(\chi),\end{equation}
 where $ \hat\brlin_1\cdot\left(d\Loc \frac{\partial}{\partial \ell_3}\right)$ is calculated in Step 1 and $dG_4 \tilde w=0$, $dG_4  w_1=1$. This derivative calculation holds provided {\bf AG} and {\bf AL} are satisfied so that we can apply Lemma \ref{Lm: loc} and \ref{Lm: glob}.

For the collisional orbit, its bouncing back orbit will intersect the section $\{x_{4,\parallel}^R=-2\}$ at a point that is within $O(\mu)$ distance from the initial point. To see this, we apply to \eqref{eq: Q4} the estimate of the difference of $G_4$ for two orbits in Lemma \ref{Lm: nocollision}, and the $O(\mu)$ estimate of angle of asymptotes and oscillations of $L_3,G_3,g_3$ in Lemma \ref{LmGMC0}(b).



So we consider the image of an $\epsilon$ interval centered at $\hat\ell_3$ under the map $G_4(\cP(\mathcal S(\te_4,\cdot)))$. By increasing $\epsilon$ from zero, we see that the assumptions {\bf AG} and {\bf AL} are all satisfied provided the returning orbit has $|x^R_{4,\perp}|<C_0'$ on the section $\{x_{4,\parallel}^R=-2\}$. So we always have the estimate \eqref{EqExpansion} and we can keep increasing $\epsilon$ until the inequality $|x^R_{4,\perp}|<C_0'$ is violated.

Thus it follows from the strong expansion of the map $G_4(\cP(\mathcal S(\te_4,\cdot)))$ and Sublemma \ref{LmCover} that an $R$-neighborhood of
$G_4^{**}$ (corresponding to $e_4^{**}$) is covered if $\ell_3$ varies in an $\frac{R\mu}{\brc\chi^2}$-neighborhood of $\hat\ell_3$. Then we use the intermediate value theorem to find $\tilde\ell_3$ such that $\pi_{e_4}\cP(\mathcal S(\te_4,\tilde\ell_3))=e_4^{**}.$ This completes the proof. 

\end{proof}

\section{The variational equation and its solution}\label{sct: var}
In this section, we first derive a formula for estimating the derivatives of $(I),(III),(V)$. This formula will reduce the derivative computation to the fundamental solution of the variational equation and two boundary terms, where the latter takes care of the issue that different orbits might take different time to travel between two consecutive sections. The rest of this section is devoted to estimating the variational equations and their fundamental solutions. This will give the estimates of $N_1,N_5,M$ in Proposition \ref{Prop: main}.

Let us first recall the notations. We use $\mathcal V=(\mathcal V_3;\mathcal V_1;\mathcal V_4)=(L_3,\ell_3,G_3,g_3;x_1,v_1;G_4,g_4)$ to denote the Delaunay coordinates. We use $\mathcal X=(\mathcal X_3;\mathcal X_1;\mathcal X_4)=(x_3,v_3;x_1,v_1;x_4,v_4)$ to denote the Cartesian coordinates. We use $\mathcal F=(\mathcal F_3,\mathcal F_1,\mathcal F_4)$ to denote the RHS of the Hamiltonian equation in Delaunay coordinates, i.e. $\frac{d\mathcal V}{d\ell_4}=\mathcal F$.

\subsection{Derivation of the formula for the boundary contribution}
Suppose that we want to compute the derivative of the Poincar\'{e} map between the sections $S^i$ and $S^f.$ We use $\mathcal V^i$ to denote the values of variables $\mathcal V$ restricted to the {\it initial} section $S^i$, while $\mathcal V^f$ means values of $\mathcal V$ on the {\it final} section $S^f$.  $\ell_4^i$ means the initial time and $\ell_4^f$ means the final time. We want to compute the derivative
$\cD$ of the Poincar\'{e} map along the orbit starting from $(\mathcal V^i_*, \ell_*^i)$ and ending at $(\mathcal V^f_*, \ell_*^f).$ We have
$\cD=dF_3 dF_2 dF_1$ where
$F_1$ is the Poincar\'{e} map between $S^i$ and $\{\ell_4=\ell_*^i\},$
$F_2$ is the flow map between the times $\ell_*^i$ and $\ell_*^f,$ and
$F_3$ is the Poincar\'{e} map between $\{\ell_4=\ell_*^f\}$ and $S^f.$ We have
$F_1=\Phi(\mathcal V^i, \ell_4(\mathcal V^i), \ell_*^i)$ where $\Phi(\mathcal V, a, b)$ denotes the flow map
starting from $\mathcal V$ at time $a$ and ending at time $b.$ Since
\[ \frac{\partial \Phi}{\partial \mathcal V}(\mathcal V_*^i, \ell_*^i, \ell_*^i)=\Id, \quad
\frac{\partial \Phi}{\partial a}=-\mathcal{F} \]
we have
$dF_1=\Id-\mathcal F(\ell_4^i)\otimes \frac{D \ell_4^i}{D \mathcal V^i}.$
Inverting the time we get
$dF_3=\left(\Id-\mathcal F(\ell_4^f)\otimes \frac{D \ell_4^f}{D \mathcal V^f}\right)^{-1}.$
Finally $dF_2=\frac{D\mathcal V(\ell_*^f)}{D\mathcal V(\ell_*^i)}$ is just the fundamental solution of the variational equation between the times
$\ell_*^i$ and $\ell_*^f.$ Thus we get
\begin{equation}
\cD
=\left(\Id-\mathcal F(\ell_4^f)\otimes \frac{D \ell_4^f}{D\mathcal V^f}\right)^{-1}\frac{D\mathcal V(\ell^f_4)}{D\mathcal V(\ell^i_4)}\left(\Id-\mathcal F(\ell_4^i)\otimes \frac{D \ell_4^i}{D \mathcal V^i}\right).
\label{eq: formald4}
\end{equation}
To invert $\Id-\mathcal F(\ell_4^f)\otimes \frac{D \ell_4^f}{D\mathcal V^f}$, we need $\left|\frac{D \ell_4^f}{D\mathcal V^f}\cdot \mathcal F(\ell_4^f)\right|<1$. Suppose this inequality is satisfied. Indeed for all the cases that we encounter in this paper, the inner product is at most $O(\mu)$. We use
\begin{equation}
\label{EqInversion}
\begin{aligned}
\left(\Id-\mathcal F(\ell_4^f)\otimes \frac{D \ell_4^f}{D\mathcal V^f}\right)^{-1}
&= \mathrm{Id}+\frac{1}{1-\frac{D \ell_4^f}{D\mathcal V^f}\cdot \mathcal F(\ell_4^f)}\mathcal F(\ell_4^f)\otimes \frac{D \ell_4^f}{D\mathcal V^f}.
\end{aligned}\end{equation}
\begin{Def}\label{DefBoundary}
We call the two terms $dF_1,\ dF_3$ the {\rm boundary contributions. }
\end{Def}

\subsection{Estimates of the variational equation}
Recall $\su(\ell_4),\ \sv(\ell_4)$ defined in Lemma \ref{Lm: derU} and Lemma \ref{Lm: ham} respectively and define further
$$\begin{cases}
&\su^R=\frac{1}{\chi^3}+\frac{\mu}{|\ell_4|^3+1},\\
&\su^L= \frac{1}{\chi^3},
\end{cases}\quad
\begin{cases}
&\sv^R=\frac{1}{\chi^2}+\frac{\mu}{|\ell_4|^3+1},\\
& \sv^L=\frac{1}{\chi^2},
\end{cases}
\quad
\begin{cases}
&\sw^R=\frac{1}{\chi}+\frac{\mu}{|\ell_4|^3+1},\\
&\sw^L=\frac{1}{\chi}
\end{cases}.
 $$
 When the super-scripts $R,L$ are omitted, we use $\su,\sv,\sw$ to represent either the right or left case depending on the context.

The lemmas in this and the next two sections will be under the standard assumption {\bf AG}. For the same reason as Corollary \ref{Cor: ham}, we can use all the conclusions of Lemma \ref{Lm: ham}, since its assumptions are implied by {\bf AG} due to Lemma \ref{Lm: strip} and Lemma \ref{Lm: position}.

We start with the following auxiliary estimate. We use the notation $\partial$ to denote the partial derivative with respect to the 12 Delaunay variables and by $\nabla_{\mathcal V}$ the covariant derivative with respect to the ten variables $\mathcal V$, where $L_4$ is solved for on the zero energy level. In particular, we have $\nabla_{\mathcal V}=\frac{\partial}{\partial \mathcal V}+\frac{\partial}{\partial L_4}\otimes \nabla_{\mathcal V} L_4 $.  Note that for the covariant derivative, we think of $\ell_4$ as time and do not take derivatives with respect to it.

\begin{Lm} \label{Lm: dL4} Assume {\bf AG}, then
we have the following estimates:
\begin{itemize}
\item[(a)]  $\nabla_{\mathcal V}L_4 =(1,0_{1\times 9})+O\left(\mu, \su,\su,\su;\frac{1}{\mu\chi^2},\frac{1}{\chi^3}, \mu,\frac{\mu}{\chi}; \sv,\sv\right),$\\
\item[(b)] $\nabla_{\mathcal V}(dt/d\ell_4) =-3L_3^2(1,0_{1\times 9})+O\left(\mu,\su,\su,\su;\frac{1}{\mu\chi^2},\frac{1}{\chi^3}, \mu,\frac{\mu}{\chi}; \sw,\sw\right).$\\
\item[(c)] Corollary \ref{Cor: ham} can be simplified as $$\mathcal F=(0,-1,0_{1\times 8})+O\left(\su,\mu,\su,\su;\mu,\frac{\mu}{\chi},\frac{1}{\mu\chi^2},\frac{1}{\chi^3}; \sv,\sv\right).$$
 \end{itemize}
\end{Lm}
\begin{proof}

We use the same argument as in the proof of Lemma \ref{Lm: ham} to estimate the angles among the vectors $x_4,x_1,\frac{\partial x_4}{\partial G_4},$ {\it etc}. Part (a) follows directly from equation \eqref{eq: L4}. That is, for each of the ten variables $\mathcal V$, we apply $\nabla_{\mathcal V}$ to both sides of equation (6.1), and then apply the above formula for $\nabla_{\mathcal V}$ to $\nabla_{\mathcal V}(E_1+U)$.  This results in a term with $\nabla_{\mathcal V}L_4$ on the right side of the equation, so we bring that term to the left side and solve for $\nabla_{\mathcal V}L_4$.
Part (b) follows from equation \eqref{eq: dt/dl4} and equation \eqref{eq: L4}.

Part (a) and (b) differ only in their $G_4$ and
$g_4$ components. The estimates of the $G_4,g_4$ components
in part (a) is the same as that of Lemma \ref{Lm: ham}.
However, for part (b) we have a $\frac{\partial U}{\partial L_4}$ term in \eqref{eq: dt/dl4}. As a result we do not have the almost orthogonality of $\frac{\partial x_4}{\partial G_4}$ with $x_1$ as we did in the proof of Lemma \ref{Lm: ham}. 
\end{proof}

We also need to figure out the order of magnitude of each entry of the RHS of the variational equation.

\begin{Lm}\label{Lm: var}
Assume {\bf AG}, then\begin{itemize}
\item[(a)]
\begin{equation}\nabla_{\mathcal V}\mathcal F \lesssim\left[\begin{array}{llll|llll|ll}
\su&\su&\su&\su&\frac{\su}{\chi}&\frac{\su}{\chi}&\mu \su&\frac{\su}{\chi}&\su&\su\\
\mu&\su&\su&\su&\frac{1}{\mu\chi^2}&\frac{1}{\chi^3}&\mu&\frac{\mu}{\chi}&\sw&\sw\\
\su&\su&\su&\su&\frac{\su}{\chi}&\frac{\su}{\chi}&\mu \su&\frac{\su}{\chi}&\su&\su\\
\su&\su&\su&\su&\frac{\su}{\chi}&\frac{\su}{\chi}&\mu \su&\frac{\su}{\chi}&\su&\su\\
\hline
\mu&\su\mu&\su\mu&\su\mu&\frac{1}{\chi^2}&\frac{\mu}{\chi^3}&\mu&\frac{\mu^2}{\chi}& \mu \sw&\mu\sw\\
\frac{\mu}{\chi}&\frac{\su\mu}{\chi}&\frac{\su\mu}{\chi}&\frac{\su\mu}{\chi}&\frac{1}{\chi^3}&\frac{\mu}{\chi^4}&\frac{\mu^2}{\chi}&\mu&\frac{\mu  \sw}{\chi}&\frac{\mu  \sw}{\chi}\\
\frac{1}{\mu\chi^2}&\frac{\su}{\chi}&\frac{\su}{\chi}&\frac{\su}{\chi}&\frac{1}{\mu\chi^3}&\frac{1}{\chi^4}&\frac{1}{\chi^2}&\frac{1}{\chi^3}&\frac{1}{\chi^2}&\frac{1}{\chi^2}\\
\frac{1}{\chi^2}&\frac{\su}{\chi}&\frac{\su}{\chi}&\frac{\su}{\chi}&\frac{1}{\mu\chi^4}&\frac{1}{\mu\chi^3}&\frac{\mu}{\chi^2}&\frac{\mu}{\chi^3}&\frac{1}{\chi^2}&\frac{1}{\chi^2}\\
\hline
\sw&\su&\su&\su&\frac{\sw}{\chi} &\frac{\sw}{\chi}&\mu \sw&\frac{ \sw}{\chi}&\sw&\sw\\
\sw&\su&\su&\su&\frac{\sw}{\chi} &\frac{\sw}{\chi}&\mu \sw&\frac{\sw}{\chi}&\sw&\sw
\end{array}\right]\nonumber
\end{equation}
\item[(b)]
In addition we have in the right case
\[\nabla_{\mathcal V^R_4}\mathcal{F}^R_4 =
\frac{1}{\chi}\frac{\xi}{(1-\xi)^3} \left[
\begin{array}{cc}
-\frac{L^4\sign(v_{4,\parallel})}{(G^2+L^2)}&L^3\\
\frac{- L^5 }{(G^2+L^2)^2}& \frac{ L^4\sign(v_{4,\parallel})}{(G^2+L^2)}\\
\end{array}
\right]+O\left(\frac{\mu}{\chi}+\frac{\mu}{|Q_4|^2}\right),
\]
where $\xi=\frac{|x_4|}{\chi}.$\\
and in the left case
\[\nabla_{\mathcal V^L_4} \mathcal{F}^L_4=
-\frac{1}{\chi} \frac{\xi}{(1-\xi)^3}\left[\begin{array}{cc}
L^2 \sign(v_{4,\parallel})& L^3\\
- L& - L^2\sign(v_{4,\parallel})\\
\end{array}
\right]+O\left(\frac{\mu}{\chi}\right),
\mathrm{\ where\ } \xi=\frac{|x_4|}{\chi}.\]

\end{itemize}
\end{Lm}
\begin{proof}
The proof is organized as follows. We start with a formal computation which gives the formula for estimating each block of the matrices. Next we work on the estimate of each block.  In the left case, a near collision may occur. Since we treat $\ell_4$ as the new time, we do not take derivatives with respect to $\ell_4$ when deriving the variational equation, so in the Hamiltonian equations as well as the variational equation we need only $C^0$ dependence on $\ell_4$, which is always true (see \eqref{eq: delaunay4} and \eqref{eq: hypul}) even when we have a collision between $Q_1$ and $Q_4$. In the estimate of the rows related to $G_4$ and $g_4$ in the left case, we will need Lemma \ref{LmSmallu}.

$\bullet$ \textbf{A formal computation.}

Using our notation, the two matrices are $\frac{d \mathcal F}{d \mathcal V}$.
They are the coefficient matrices of the variational equations $\frac{d}{d\ell_4} \dt \mathcal V=\frac{d \mathcal F}{d \mathcal V}\dt \mathcal V$.
We split each of the two matrices into nine blocks corresponding to $\frac{\partial \mathcal F_i}{\partial \mathcal V_j}$, where $i,j=3,1,4$.

We have $\frac{d \mathcal F}{d \mathcal V}=\nabla_{\mathcal V}\mathcal F $. Notice that $\mathcal F=\frac{dt}{d\ell_4}J\frac{\partial H}{\partial {\mathcal V}}$ where $J$ is the standard symplectic matrix. Then we get the formal expression to calculate the two matrices:
\begin{equation}\label{eq: formalvar}  \nabla_{\mathcal V_j} \mathcal F_i =J_i\frac{\partial H}{\partial \mathcal V_i}\otimes \left(\nabla_{\mathcal V_j}\frac{dt}{d\ell_4}\right)+\frac{dt}{d\ell_4}\left(\nabla_{ \mathcal V_j}J_i\frac{\partial H}{\partial \mathcal V_i}\right).\end{equation}

Note that $\nabla_{\mathcal V}\frac{dt}{d\ell_4}$ is done in Lemma \ref{Lm: dL4} and $J_i\frac{\partial H}{\partial \mathcal V_i}$ is done in Corollary \ref{Cor: ham},
the term $\frac{dt}{d\ell_4}=O(1)$ and the new term we need to consider is $\nabla_{ \mathcal V_j}J_i\frac{\partial H}{\partial \mathcal V_i}$. For $i,j=3,1,4$, we have
\begin{equation}\label{eq: formalvar2}
\begin{aligned}
\nabla_{ \mathcal V_j}J_i\frac{\partial H}{\partial \mathcal V_i}&=\frac{\partial }{\partial \mathcal V_j}J_i\frac{\partial H}{\partial \mathcal V_i}+\frac{\partial }{\partial L_4}J_i\frac{\partial H}{\partial \mathcal V_i}\otimes \nabla_{\mathcal V_j} L_4 \\
&=\frac{\partial \mathcal X_j}{\partial \mathcal V_j}\frac{\partial }{\partial \mathcal X_j}\left(J_i\frac{\partial H}{\partial \mathcal X_i} \frac{\partial \mathcal X_i}{\partial \mathcal V_i}\right)+\frac{\partial }{\partial L_4}\left(J_i\frac{\partial H}{\partial \mathcal X_i} \frac{\partial \mathcal X_i}{\partial \mathcal V_i}\right)\otimes \nabla_{\mathcal V_j} L_4 \\
&=\frac{\partial \mathcal X_j}{\partial \mathcal V_j}J_i\frac{\partial^2H}{\partial \mathcal X_j \partial \mathcal X_i}\frac{\partial \mathcal X_i}{\partial \mathcal V_i}+J_i\frac{\partial H}{\partial \mathcal X_i}\frac{\partial^2 \mathcal X_i}{\partial \mathcal V_j\partial \mathcal V_i}\\
&+\left(\frac{\partial \mathcal X_4}{\partial  L_4}J_i\frac{\partial^2H}{\partial \mathcal X_4 \partial \mathcal X_i}\frac{\partial \mathcal X_i}{\partial \mathcal V_i}+J_i\frac{\partial H}{\partial \mathcal X_i}\frac{\partial^2 \mathcal X_i}{\partial L_4\partial \mathcal V_i}\right)\otimes \nabla_{\mathcal V_j}L_4 , \\
\end{aligned}\end{equation}
where $J_i$ is the standard symplectic matrix in the $i$ component.
We know by Lemma \ref{Lm: position} that $\frac{\partial \mathcal X_3}{\partial \mathcal V_3}=O(1)$,
$\frac{\partial \mathcal X_1}{\partial \mathcal V_1}=\mathrm{Id}_{4}$
and $\frac{\partial \mathcal X_4}{\partial L_4},\frac{\partial \mathcal X_4}{\partial \mathcal V_4}=O(\ell_4)$ according to Lemma \ref{Lm: dx/dDe}.
Moreover, $\frac{\partial H}{\partial \mathcal X_j}$, $\frac{\partial^2H}{\partial \mathcal X_i \partial \mathcal X_j}$ are done in Lemma \ref{Lm: derU}, $\frac{\partial^2 \mathcal X_j}{\partial \mathcal V_i\partial \mathcal V_j}$ and $\frac{\partial^2 \mathcal X_j}{\partial L_4\partial \mathcal V_j}$ are done in Lemma \ref{Lm: 2ndderivative}, and finally, $\frac{\partial L_4}{\partial \mathcal V_i}$ is done in Lemma \ref{Lm: dL4}. Now every term in \eqref{eq: formalvar} and \eqref{eq: formalvar2} is already estimated. What we need to do below is to find the leading term for each matrix entry among all the terms above.

In the following we analyze the two matrices blockwise. We will handle the left and right cases simultaneously. 

$\bullet \nabla_{\mathcal V_3} \mathcal F_3 $.

For this block the $(2,1)$ entry is special. All the remaining entries are done together.

Using the Hamiltonian equations \eqref{eq: ham}, we see the $(2,1)$ entry is
\[\nabla_{L_3}\left(\frac{dt}{d\ell_4}\left(\frac{m_3k_3^2}{L_3^3}+\frac{\partial U}{\partial L_3}\right)\right)\]\[=-\nabla_{L_3}\left(m_3k_3^2\left(1+L_3^2\left(\frac{v_1^2}{2m_{1}} -\frac{k_{1}}{ |x_1|}\right)+3UL_3^2\right)-L_3^6\frac{\partial U}{\partial L_4}+L_3^3\frac{\partial U}{\partial L_3}+h.o.t.\right)\]
The leading term is $\nabla_{L_3} L_3^2 \cdot\left(\frac{v_1^2}{2m_{1}} -\frac{k_{1}}{ |x_1|}\right)=O(\mu)$
since
$$m_1=O(1/\mu),\ L_3=O(1),\ v_1=O(1),\ |x_1|=O(\chi).$$
All the other subleading terms involve derivatives of $U$, which are at most $\sw$ coming from $\nabla_{L_3}\frac{\partial U}{\partial L_4}$. This completes the estimate of the $(2,1)$ entry. The other three entries in the second row are estimated by $\su$ mainly contributed by $\frac{\partial U}{\partial x_3}$ (Lemma \ref{Lm: derU}).

For the first, third and fourth rows, we use formula \eqref{eq: formalvar}. The first summand in \eqref{eq: formalvar} contributes $(\su,\su,\su)\otimes (\su,1,\su,\su)$ to the three rows.
The second summand is given by \eqref{eq: formalvar2}.
The first and second terms after the third equal sign in \eqref{eq: formalvar2} have the same estimates,
$\frac{\partial U}{\partial x_3}$, $\frac{\partial^2 U}{\partial x_3^2}\lesssim
\su$, as we get in Lemma \ref{Lm: derU},
since $\frac{\partial \mathcal X_3}{\partial \mathcal V_3},\ \frac{\partial^2 \mathcal X_3}{\partial \mathcal V_3^2}=O(1)$ and $ \nabla_{\mathcal V_3} L_4 =O(1)$ using Lemma \ref{Lm: dL4}. The summand $J_3\frac{\partial H}{\partial \mathcal X_3}\frac{\partial^2 \mathcal X_3}{\partial L_4\partial \mathcal V_3}\otimes \nabla_{\mathcal V_3}L_4 =0$ since  $\frac{\partial^2 \mathcal X_3}{\partial L_4\partial \mathcal V_3}=0$.
The third term in \eqref{eq: formalvar2}
is estimated as $\frac{\mu}{|\ell_4|^3+1}$ in the right case and $\frac{\mu}{\chi^3}$ in the left case using the estimate $\frac{\partial^2 U}{\partial x_3\partial x_4}$ in Lemma \ref{Lm: derU} and the fact that $\frac{\partial \mathcal X_4}{\partial L_4}=O(\ell_4)$. So in summary, all the entries in the three rows are bounded by $\su.$


$\bullet \nabla_{\mathcal V_1}\mathcal F_3$.

The first summand in \eqref{eq: formalvar} gives $(\su,1,\su,\su)\otimes (\frac{1}{\mu\chi^2},\frac{1}{\chi^3}, \mu,\frac{\mu}{\chi})$. The second summand in \eqref{eq: formalvar}  is now reduced to $J_3\frac{\partial^2U}{\partial x_1 \partial x_3}\frac{\partial x_3}{\partial \mathcal V_3}+\frac{\partial x_4}{\partial  L_4}J_3\frac{\partial^2U}{\partial x_4 \partial x_3}\frac{\partial x_3}{\partial \mathcal V_3}\otimes \nabla_{\mathcal V_1} L_4 $ since $\frac{\partial \mathcal X_1}{\partial \mathcal V_1}=Id$ and $\frac{\partial^2\mathcal X_3}{\partial\mathcal V_1\partial\mathcal V_3}=\frac{\partial^2 \mathcal X_3}{\partial L_4\partial \mathcal V_3} =0$. The two terms are estimated using
$\frac{\partial^2U}{\partial x_1 \partial x_3}\lesssim
\frac{1}{\chi^4}$ given by Lemma \ref{Lm: derU},
$\nabla_{\mathcal V_1}L_4 \lesssim \left(\frac{1}{\mu\chi^2},\frac{1}{\chi^3}, \mu,\frac{\mu}{\chi}\right)$
by Lemma \ref{Lm: dL4} and $\left(\frac{\partial x_4}{\partial  L_4}J_j\frac{\partial^2U}{\partial x_4 \partial x_3}\frac{\partial x_3}{\partial \mathcal V_3}\right)\lesssim \su$ since we have
$\frac{\partial \mathcal X_4}{\partial L_4}=O(\ell_4)$, and
$\frac{\partial^2H}{\partial \mathcal X_4 \partial \mathcal X_3}\lesssim \frac{\mu}{\ell_4^4+1}$ in the right case and $\frac{\mu}{\chi^4}$ in the left case using Lemma \ref{Lm: derU}. Therefore the estimate of the block is given by the max of $(\su,1,\su,\su)\otimes (\frac{1}{\mu\chi^2},\frac{1}{\chi^3}, \mu,\frac{\mu}{\chi})$ and $1/\chi^4$.

$\bullet  \nabla_{\mathcal V_4}\mathcal F_3 $.

The second row is handled in a similar manner to the $ \nabla_{ \mathcal V_3}\mathcal F_3$ block.
Thus
\begin{equation}\nonumber\begin{aligned}
& \nabla_{\mathcal V_4}\left(\frac{dt}{d\ell_4}\left(\frac{m_3k_3^2}{L_3^3}+\frac{\partial U}{\partial L_3}\right)\right)=- \nabla_{\mathcal V_4}\left(m_3k_3^2\left(1+3L_3^2\left(\frac{v_1^2}{2m_{1}} -\frac{k_{1}}{ |x_1|}\right)+3UL_3^2\right)\right.\\
&\left.-L_3^6\frac{\partial U}{\partial L_4}+L_3^3\frac{\partial U}{\partial L_3}+h.o.t.\right)=- \nabla_{\mathcal V_4}\left(3UL_3^2-L_3^6\frac{\partial U}{\partial L_4}+L_3^3\frac{\partial U}{\partial L_3}+h.o.t.\right)
\end{aligned}\end{equation}
The leading term is given by $ \nabla_{\mathcal V_4}\frac{\partial U}{\partial L_4}=
\left(\frac{\partial \mathcal X_4}{\partial  L_4}J_4\frac{\partial^2U}{\partial x_4^2 }
 \nabla_{\mathcal V_4}x_4 \right)\lesssim \sw$ using Lemma \ref{Lm: derU}.

Next, the first summand in \eqref{eq: formalvar} gives the estimate $(\su,\su,\su)\otimes \left(\sw,\sw\right)$ for the first, third and fourth row.
This is smaller than what we have stated in the lemma. It remains to consider \eqref{eq: formalvar2} which is reduced to
$\frac{\partial x_4}{\partial \mathcal V_4}J_3\frac{\partial^2U}{\partial x_4 \partial x_3}\frac{\partial x_3}{\partial \mathcal V_3}+\frac{\partial x_4}{\partial  L_4}J_3\frac{\partial^2U}{\partial x_4 \partial x_3}\frac{\partial x_3}{\partial \mathcal V_3}\otimes \nabla_{\mathcal V_4} L_4 $
 since $\frac{\partial x_3}{\partial \mathcal V_4}=\frac{\partial x_3}{\partial  L_4}=0$. The first summand has the estimate $\frac{\mu}{\ell_4^3+1}$ in the right case and $\frac{\mu}{\chi^3}$ in the left case since
we have the estimate $\frac{\partial x_4}{\partial \mathcal V_4}=O(\ell_4)$ and $\frac{\partial^2H}{\partial \mathcal X_4 \partial \mathcal X_3}\lesssim \frac{\mu}{\ell_4^4+1}$ in the right case and $\frac{\mu}{\chi^4}$ in the left case in Lemma \ref{Lm: derU}. So we use $\su$ to bound this first summand as stated in the lemma. For the second summand in \eqref{eq: formalvar2}, we estimate $\nabla_{\mathcal V_4}L_4 $
as $(\sv,\sv)$ using Lemma \ref{Lm: dL4}, and $\left(\frac{\partial x_4}{\partial  L_4}J_4\frac{\partial^2U}{\partial x_4 \partial x_3}
\frac{\partial x_3}{\partial \mathcal V_3}\right)$ has the same estimate as the first summand.
So the second summand is much smaller than the first summand.

{\it For the next three blocks $\nabla_{ \mathcal V_j} \mathcal F_1,\ j=3,1,4$, using $\frac{\partial \mathcal X_1}{\partial\mathcal V_1}=\Id$, \eqref{eq: formalvar2} is reduced to
\begin{equation}\label{Eq1j}\frac{\partial \mathcal X_j}{\partial \mathcal V_j}J_1\frac{\partial^2H}{\partial \mathcal X_j \partial \mathcal X_1}+\frac{\partial x_4}{\partial  L_4}J_1\frac{\partial^2H}{\partial x_4 \partial \mathcal X_1}\otimes \nabla_{\mathcal V_j} L_4 . \end{equation}}
$\bullet  \nabla_{\mathcal V_3}\mathcal F_1 $.

The first summand in \eqref{eq: formalvar} gives $ \left(\mu,\frac{\mu}{\chi},\frac{1}{\mu\chi^2},\frac{1}{\chi^3}\right)\otimes (1,\su,\su,\su)$.
The first summand in \eqref{Eq1j} has the same estimate as
$\frac{\partial^2H}{ \partial x_3\partial x_1}=\frac{\partial^2U}{ \partial x_3\partial x_1}\lesssim \frac{1}{\chi^4},\ \frac{\partial^2H}{\partial \mathcal X_3\partial v_1}=0$ using Lemma \ref{Lm: derU}, whose contribution to the current block is $\left[\begin{array}{c}
0_{2\times 4}\\
(1/\chi^4)_{2\times 4}\end{array}\right]$.

  Next, we consider the second summand in \eqref{Eq1j}. We have
$\nabla_{\mathcal V_3} L_4 \lesssim (1,\su,\su,\su)$ using Lemma \ref{Lm: dL4}. Next we consider $\frac{\partial x_4}{\partial  L_4}J_1\frac{\partial^2H}{\partial x_4 \partial \mathcal X_1}$. This is a vector of four entries whose first two entries are $\frac{\partial^2H}{\partial  L_4 \partial v_1}=\frac{\partial v_1}{\partial  L_4}=0$ and whose last two entries are bounded by $\frac{1}{\chi^2}$, since we have $\frac{\partial x_4}{\partial L_4}=O(\ell_4)$, and $\frac{\partial^2H}{\partial x_4 \partial x_1}\lesssim\frac{1}{\chi^3}$ using Lemma \ref{Lm: derU}. So we use $ \left(\mu,\frac{\mu}{\chi}\right)\otimes (1,\su,\su,\su)$ as the estimate for the first two rows and use the max of $\left(\frac{1}{\mu\chi^2},\frac{1}{\chi^3}\right)\otimes (1,(\su)_{1\times 3})$, $(1/\chi^4)_{2\times 4}$ and $(\frac{1}{\chi^2},\frac{1}{\chi^2})\otimes(1,\su,\su,\su)$ as the estimate for the last two rows.

$\bullet  \nabla_{\mathcal V_1}\mathcal F_1 $.

The first summand in \eqref{eq: formalvar} gives $ \left(\mu,\frac{\mu}{\chi},\frac{1}{\mu\chi^2},\frac{1}{\chi^3}\right)\otimes \left(\frac{1}{\mu\chi^2},\frac{1}{\chi^3}, \mu,\frac{\mu}{\chi}\right)$. The first summand in \eqref{Eq1j} is $\left[\begin{array}{cc}\frac{\partial^2 H}{\partial x_1\partial v_1}&\frac{\partial^2 H}{\partial v_1\partial v_1}\\
-\frac{\partial^2 H}{\partial x_1^2}&-\frac{\partial^2 H}{\partial x_1\partial v_1}\end{array}\right]=\left[\begin{array}{cc}0&\frac{1}{m_1}\mathrm{Id}\\
\left(\frac{3k_1 x_1\otimes x_1}{|x_1|^5}\right)-\frac{k_1 \mathrm{Id}}{|x_1|^3}&0\end{array}\right]$,
where $\frac{\partial^2 H}{\partial x_1^2}$ is given by $\frac{\partial^2}{\partial x_1^2}\frac{k_1}{|x_1|}$.
We compare the two matrices using $x_1=O(\chi,\mu)$ to get the first three rows.
Finally we consider the second summand in \eqref{Eq1j}. We notice that
$\nabla_{\mathcal V_1} L_4 \lesssim
\left(\frac{1}{\mu\chi^2},\frac{1}{\chi^3}, \mu,\frac{\mu}{\chi}\right)$. The term
$\left(\frac{\partial x_4}{\partial  L_4}J_1\frac{\partial^2H}{\partial x_4 \partial \mathcal X_1}\right)\lesssim \left(0,0,\frac{1}{\chi^2},\frac{1}{\chi^2}\right)$ as in the previous paragraph.
This gives us the last row.

$\bullet  \nabla_{\mathcal V_4}\mathcal F_1 $.

The first summand in \eqref{eq: formalvar} gives $\left(\mu,\frac{\mu}{\chi},\frac{1}{\mu\chi^2},\frac{1}{\chi^3}\right)\otimes \left(\sw,\sw\right)$ using Lemma \ref{Lm: dL4} and Corollary \ref{Cor: ham}. This gives the first two rows in both matrices.
The first two rows of \eqref{Eq1j}  are zero because
$\frac{\partial^2 H}{\partial  x_4\partial v_1}=0$. For the last two rows in \eqref{Eq1j} we first have $\left(\frac{\partial x_4}{\partial  \mathcal V_4}\right)\left(J_1\frac{\partial^2U}{\partial x_4 \partial x_1}\right)
\lesssim 1/\chi^2$ and $\frac{\partial x_4}{\partial  L_4}J_1\frac{\partial^2U}{\partial x_4 \partial x_1}\lesssim1/\chi^2$
as in the previous paragraph, so the tensor part is
$O\left(\sv/\chi^2\right)$. So for the last two rows we use the estimate $1/\chi^2$. 

$\bullet  \nabla_{\mathcal V_3}\mathcal F_4 $.

The first summand in \eqref{eq: formalvar} gives $ \left(\sw,\sw\right)\otimes(1,\su,\su,\su)$.  The first summand in \eqref{eq: formalvar2} is given by
$\frac{\partial x_3}{\partial  \mathcal V_3}\left(J_3\frac{\partial^2U}{\partial x_4 \partial x_3}\right)\frac{\partial x_4}{\partial  \mathcal V_4} \lesssim \su$ using Lemma \ref{Lm: derU}. This gives the second, third and fourth columns of the block.

Next, the second summand in \eqref{eq: formalvar2} vanishes since $\frac{\partial ^2x_4}{\partial \mathcal V_3\partial \mathcal V_4}=0$. 

It remains to consider the third summand in \eqref{eq: formalvar2}. We have
$ \nabla_{\mathcal V_3}L_4 \lesssim (1,(\su)_{1\times 3})$. Next $J_4\frac{\partial U}{\partial x_4}\frac{\partial x_4}{\partial L_4\partial  \mathcal V_4}\lesssim \ell_4\left(\frac{1}{\chi^2}+\frac{\mu}{\ell_4^4+1}\right)\lesssim \sw$ using Lemma \ref{Lm: derU} and \ref{Lm: 2ndderivative}.
Next, $\frac{\partial x_4}{\partial  L_4}\left(J_4\frac{\partial^2U}{\partial x_4 \partial x_4}\right)\frac{\partial x_4}{\partial  \mathcal V_4}\lesssim \sw$. So the third summand in \eqref{eq: formalvar2} has the estimates $(\sw)_{1\times 2}\otimes(1,(\su)_{1\times 3})$ and contributes to the first column of the block along with the the first summand of \eqref{eq: formalvar}.

$\bullet  \nabla_{\mathcal V_1}\mathcal F_4 $.

The first summand in \eqref{eq: formalvar} gives $ \left(\sw,\sw\right)\otimes\left(\frac{1}{\mu\chi^2},\frac{1}{\chi^3}, \mu,\frac{\mu}{\chi}\right)$.   The first summand in \eqref{eq: formalvar2}
$\left(J_4\frac{\partial^2U}{\partial x_4 \partial x_1}\right)\frac{\partial \mathcal X_4}{\partial  \mathcal V_4} \lesssim \frac{1}{\chi^2}$
 using Lemma \ref{Lm: derU}. The second summand in  \eqref{eq: formalvar2} vanishes since $\frac{\partial^2 x_4}{\partial \mathcal V_1\partial\mathcal V_4}=0$.
Finally, we consider the third summand in \eqref{eq: formalvar2}.
We have $\nabla_{\mathcal V_1} L_4 \lesssim \left(\frac{1}{\mu\chi^2},\frac{1}{\chi^3}, \mu,\frac{\mu}{\chi}\right)$.
The estimate of  $\frac{\partial x_4}{\partial  L_4}\left(J_4\frac{\partial^2U}{\partial x_4 \partial x_4}\right)\frac{\partial x_4}{\partial  \mathcal V_4}+J_4\frac{\partial U}{\partial x_4}\frac{\partial x_4}{\partial L_4\partial  \mathcal V_4}\lesssim \sw$ was done in the previous block.

So the estimate of this block is to take the larger between $(1/\chi^2)_{2\times 4}$ and $\left(\sw,\sw\right)\otimes\left(\frac{1}{\mu\chi^2},\frac{1}{\chi^3}, \mu,\frac{\mu}{\chi}\right)$.

$\bullet  \nabla_{\mathcal V_4}\mathcal F_4 $.

The leading contribution is given by the first and second summands in \eqref{eq: formalvar2},
\[\frac{\partial x_4}{\partial \mathcal V_4}J_4\frac{\partial^2U}{\partial x_4 \partial x_4}\frac{\partial x_4}{\partial \mathcal V_4}+J_4\frac{\partial U}{\partial x_4}\frac{\partial^2 x_4}{\partial \mathcal V_4\partial \mathcal V_4}\lesssim\sw\]
using Lemma \ref{Lm: derU}. We next show the other summands are small.
The first summand in \eqref{eq: formalvar} gives $ \left(\sw,\sw\right)\otimes\left(\sw,\sw\right)$. Next we consider the second summand in \eqref{eq: formalvar2}.
We have $\nabla_{\mathcal V_4} L_4 \lesssim \left(\sv,\sv\right)$.
Next,
$\frac{\partial x_4}{\partial  L_4}\left(J_4\frac{\partial^2U}{\partial x_4 \partial x_4}\right)\frac{\partial x_4}{\partial  \mathcal V_4}+\left(J_4\frac{\partial U}{\partial x_4}\right)\frac{\partial^2 x_4}{\partial  L_4\partial\mathcal V_4}\lesssim \sw$ is
estimated as before. So the first and third summands in \eqref{eq: formalvar2} are much smaller than $\sw$.

{\bf Part (a) is complete now. Finally, we show part (b). }

 According to the last bullet point, the leading terms in $\nabla_{\mathcal V_4}\mathcal F_4$ come from \[\frac{dt}{d\ell_4}\left(\frac{\partial \mathcal X_4}{\partial  \mathcal V_4}\left(J_j\frac{\partial^2H}{\partial \mathcal X_4 \partial \mathcal X_4}\right)\frac{\partial \mathcal X_4}{\partial  \mathcal V_4}+J_4\frac{\partial H}{\partial \mathcal X_4}\frac{\partial^2 \mathcal X_4}{\partial \mathcal V_4\partial \mathcal V_4}\right)=L_4^3\left(\frac{\partial }{\partial G_4}, \frac{\partial }{\partial g_4}\right)\left[\begin{array}{c}
\frac{\partial x_4}{\partial g_4}\cdot\frac{\partial U}{\partial x_4}\\
-\frac{\partial x_4}{\partial G_4}\cdot\frac{\partial U}{\partial x_4}
\end{array}
\right].\]
Let us now look at $U^R$ in \eqref{eq: hamR}.  Only those terms in $U^R$ containing both $x_4$ and $x_1$ can be as large as $O(1/\chi)$ according to Lemma \ref{Lm: derU}. So we only need to consider the following three terms in $U^R$, \[-\left(\frac{1}{\mu\left|x_1+\frac{\mu}{1+2\mu}x_4+\frac{\mu}{1+\mu} x_3\right|}+
\frac{1}{\left|x_1+\frac{\mu}{1+2\mu}x_4-\frac{1}{1+\mu}x_3\right|}
+\frac{1}{\left|x_1-\frac{1+\mu}{1+2\mu}x_4\right|}\right).\] When we take two derivatives with respect to $x_4$, a
$\mu^2$ factor will multiply the first two terms, so that the first two terms would be $O(\mu)$ compared to the third term.  So the leading contribution to $\frac{\partial^2U^R}{\partial x_4^2}$ is given by $\frac{\partial^2}{\partial x_4^2}\frac{-1}{\left|x_1-\frac{1+\mu}{1+2\mu}x_4\right|}$. The same analysis for $U^L$ in \eqref{eq: hamL} shows the leading contribution to $\frac{\partial^2U^L}{\partial x_4^2}$ is given by $\frac{\partial^2}{\partial x_4^2}\frac{-1}{\left|x_1+\frac{\mu}{1+\mu}x_3+\frac{1+\mu}{1+2\mu}x_4\right|}$.

Consider the $(9,9)$ entry. The main contribution to this entry comes from
\begin{equation}\label{Key55R}
L_4^3\frac{\partial }{\partial G_4}\left(\frac{\partial x_4}{\partial g_4}\cdot\frac{\partial U}{\partial x_4}\right)=(1+O(\mu))L_4^3\frac{\partial }{\partial G_4}\left(\frac{\frac{\partial x_4}{\partial g_4} \cdot (x_4-(1+O(\mu)x_1)}{|x_4-(1+O(\mu)x_1)|^3}\right) .
\end{equation}
The numerator on the RHS equals
\[(1+O(\mu))L_4^3\frac{\partial}{\partial G_4} \left( \frac{\partial x_4}{\partial g_4} \cdot x_4\right)-(1+O(\mu))L_4^3
\frac{\partial^2 x_4}{\partial G_4 \partial g_4} \cdot x_1. \]
The first term is $O(\chi)$ due to Lemma \ref{Lm: dx/dDe}(c) so the main contribution comes from the second term which is $O(\chi^2)$ using Lemma \ref{Lm: 2ndderivative}.
We use the same argument for the other entries to get \begin{equation}\label{EqTwistDiff}\nabla_{\mathcal V_4^R}\mathcal{F}^R_4 =-L_3^3
\left[\begin{array}{cc}
\frac{\partial^2 x_4}{\partial G\partial g}\cdot \frac{x_1}{|x_4-x_1|^3}& \frac{\partial^2 x_4}{\partial g^2}\cdot \frac{x_1}{|x_4-x_1|^3}\\
-\frac{\partial^2 x_4}{\partial G^2}\cdot \frac{x_1}{|x_4-x_1|^3}&-\frac{\partial^2 x_4}{\partial G\partial g}\cdot \frac{x_1}{|x_4-x_1|^3}\\
\end{array}
\right]+O\left(\frac{\mu}{\chi}+\frac{\mu}{|x_4|^3}\right).
\end{equation}

Using Lemma \ref{Lm: 2ndderivative}
we see that the
$(9,9)$ entry equals
\[-\frac{L_4^5}{\sqrt{L_4^2+G_4^2}}\frac{\chi \sinh u}
{|x_4-x_1|^3}
+O\left(\frac{\mu}{\chi}+\frac{\mu}{|x_4|^3}\right)
. \]
Recall that $L_3=L_4(1+o(1))$ (due to \eqref{eq: L4}) and $\sinh u=\sign (u)\frac{|\ell_4| L_4}{\sqrt{L_4^2+G_4^2}}$
(due to \eqref{eq: hypul}).
Since Lemma \ref{Lm: position} implies that
$|x_4|=|\ell_4|/L_4^2 (1+o(1))$
we get that the $O(1/\chi)$-term in $(9,9)$ is asymptotic to
\[ -\frac{L^4 \sign(u)}{L^2+G^2} \frac{\chi |x_4|}{(\chi-|x_4|)^3} .\]
Since $u$ and $v_{4,\parallel}$ have opposite signs we obtain the asymptotics of the $O(1/\chi)$-term
claimed in part (b) of Lemma \ref{Lm: var} for the $(9,9)$ entry.
The analysis of other entries of $\nabla_{\mathcal V_4^R} \mathcal{F}^R_4 $ is similar.

Next, we consider the left case. The argument is the same except for the following differences. First, the error term in \eqref{Key55R} is now $O(\mu/\chi)$ since $\mu/|x_4|^3$ should be replaced by $1/\chi^3$ as usual. Next $U^L$ is roughly $\frac{1}{|x_4+x_1|}$ up to some $\mu$ error, which differs from $U^R$ by a ``-" sign.
Then we have
that the asymptotic expression of \eqref{EqTwistDiff}  follows directly from Lemma \ref{Lm: 2ndderivative}(c).

\end{proof}

\subsection{Estimates of the solution of the variational equations}\label{subsct: SolnOfVar}
In this section, we give the proof of the estimates of matrices $N_1,\, N_3,\, N_5$ and the $(I)_{44},(III)_{44},(V)_{44}$ blocks in Proposition \ref{Prop: main}.

From one Poincar\'{e} section to the next, it takes time of order $O(\chi)$. The main body of the proof is to show that {\it the fundamental solutions of the variational equations are estimated by {\bf three steps} of Picard iterations.} 
\subsubsection{The asymptotics of the $(N_1)_{44},(N_3)_{44},(N_5)_{44}$  blocks in Proposition \ref{Prop: main} (b).}
The blocks are obtained by integrating the leading terms in the estimates of $\nabla_{\mathcal V_4}\mathcal F_4 $ in part (b) of Lemma \ref{Lm: var}. After a rescaling of time, the problem is reduced to finding the fundamental solution of a linear ODE system defined by constant $2\times 2$ matrices. 
\subsubsection{The matrix $N_3$ for the piece $(III)$.}

Let us first explain how to get the matrix $N_3$. Since the right matrix of Lemma \ref{Lm: var} has constant entries, which we denote by $K$ temporarily, $N_3$ can be estimated by
the fundamental solution of the ODE $X'=K\cdot X,$ that is, by
$X(\chi)=e^{K\chi}=\sum_{n=0}^\infty \frac{1}{n!}(K\chi)^n$.
Note that $K$ has positive entries. We claim that in fact
\begin{equation}
\label{ConstCoef2Terms}
e^{K\chi}-\Id_{10}=O(K \chi+\frac{1}{2}(K\chi)^2+(K\chi)^3).
\end{equation}
Indeed a brute force
force calculation shows that
$(K\chi)^4\leq C_3 (K\chi+(K\chi)^2+(K\chi)^3).$ This allows us to get inductively that
\begin{equation}
\label{KnK2+K1}
(K\chi)^n\leq C_n (K\chi+(K\chi)^2+(K\chi)^3) \text{ where }C_n=C_3(1+C_3)^{n-4}.
\end{equation}
Summing the series for $e^{K\chi}$ we obtain \eqref{ConstCoef2Terms}. All entry except (6,5) and (7,8) appear in $K\chi+(K\chi)^2$.

We remark that the computation can be done either by computer or by hand. Note that the 1st, 3rd and 4th rows, the 9th and 10th rows, the 2nd, 3rd and 4th columns and the 9th and 10th columns are the same respectively, so we can reduce the size of the matrices from $10\times 10$ to $7 \times 7.$


\subsubsection{The matrices $N_1, N_5$. }
We first explain the strategy of reducing the estimate of the fundamental solution to three steps of Picard iterations.

{\bf Step 1, The strategy.}

Denote the ODE by $\frac{dY}{dt}=\Lambda(t) Y$
with the initial condition $Y(0)=\mathrm{Id}_{10}$.
Using the Picard iteration, the solution is
\begin{equation}\label{eq: picard}
\begin{aligned}
Y(t)&=\mathrm{Id}+\int_{0}^{t}\Lambda\cdot Y(s )\,ds=\mathrm{Id}+\int_{0}^{t}\Lambda dt+\int_{0}^{t }\Lambda\(\int_{0}^{s}\Lambda(\tau)d\tau\)\,ds+\cdots \\
:&=\Id+I_1(t)+I_2(t)+\cdots
\end{aligned}
\end{equation}
where $I_i$ is the $i$-th iterated integral. We will show that $Z(t)=\mathrm{Id}+c_1I_1(t)+c_2I_2(t)+\cdots+c_nI_n(t)$ with properly chosen $c_n=c_{n-1}=\cdots=c_3>c_2>c_1>1$, satisfies the inequality $Z'\geq \Lambda(t) Z(t)$ or equivalently \begin{equation*}
\begin{aligned}\mathrm{Id}+c_1I_1(t)+\cdots +c_nI_n(t)&=Z(t)\geq \mathrm{Id}+\int_0^t\Lambda(s) Z(s)\,ds\\ &\geq \mathrm{Id}+I_1(t)+c_1I_2(t)+c_2I_3(t)+\cdots +c_nI_{n+1}(t),
\end{aligned}\end{equation*}
\begin{equation}\label{EqCheck}
c_n I_{n+1}(t)\leq (c_1-1)I_1(t)+(c_2-c_1) I_2(t)+(c_3-c_2) I_3(t).\end{equation} Then by the Gronwall inequality, we get that $Y(t)\leq  Z(t)$.

 {\bf Step 2, Checking \eqref{EqCheck}. }

We next show how to compute the matrix products.
The following observations allow us to reduce \eqref{EqCheck} to computing products of constant matrices, which simplifies
the calculation significantly. In $\su,\sv,\sw$, we replace $\frac{\mu}{\ell_4^3+1}$ by $\frac{\mu}{|\ell_4|^3}$ with $\ell_4$ lying between $1$ and $O(\chi)$. Recall that $\frac{\mu}{|\ell_4|^3}$ is the correct bound of terms of the form $\frac{\mu |x_3|}{|x_4|^3}$ in Lemma \ref{Lm: position}, while $\frac{\mu}{\ell_4^3+1}$ was used to show that the denominator cannot be zero.

For $N_1$, we pick a small constant $\epsilon_0$ which is independent of $\mu,\chi$, so that $\int_{\ell_4^i}^{\ell_4^i+\epsilon_0} \frac{\mu}{\ell^3_4}d\ell_4=\frac{\mu\epsilon_0}{(\ell^i_4)^3}+O(\epsilon_0^2)$ where $\ell_4^i=O(1)\neq 0$ is the initial $\ell_4.$ Inequality \eqref{EqCheck} holds for $\ell_4\in[\ell_4^i,\ell_4^i+\epsilon_0]$ for $\epsilon_0$ small enough. For $\ell_4\geq \ell_4^i+\epsilon_0$, we have $\int_{\ell_4^i}^{\ell_4} \frac{\mu}{s^3}ds=\frac{\mu}{2(\ell^i_4)^2}-\frac{\mu}{2(\ell_4)^2}=O(\mu),$ as $\ell_4\to\infty,\ \mu\to 0$. So we replace all the integrals $\int_{\ell_4^i}^{\ell_4} \frac{\mu}{s^3}ds$ by $\mu$ in the sense of ``$\sim$".
After integration in \eqref{EqCheck}, there are no terms of the form $1/\ell_4^k,\ k>0$.

Notice that we can decompose right matrix in Lemma \ref{Lm: var}
as $K+\frac{\mu}{\ell^3_4}B$, where $K$ and $B$ do not depend on $\ell_4,$
and $K$ is exactly the same as in \eqref{ConstCoef2Terms}.
We have
\begin{equation*}
\begin{aligned}
I_1(\ell_4)&\lesssim \ell_4 K+ \mu B,\quad I_2(\ell_4)\lesssim \ell_4^2 K^2+\mu\ell_4 KB+\mu BK+\mu^2 B^2,\\
I_3(\ell_4)&\lesssim \ell_4^3K^3+\mu\ell_4^2K^2B+\mu\ell_4 KBK+\mu^2\ell_4KB^2\\
&+\mu\ln\ell_4  BK^2+\mu^2(BKB+B^2K)+\mu^3 B^3.
\end{aligned}
\end{equation*}
To simplify the the proof, we note that $I_3(\ell_4)$ is bounded by $\frac{1}{3!}\ell_4^3K^3+\eps I_1(\ell_4)+\eps I_2(\ell_4)$ where $\eps $ can be chosen to be arbitrarily small, provided $\mu$ and $1/\chi$ are small enough. So $I_4(\ell_4)$ is bounded by $$\int_1^{\ell_4} (K+\frac{\mu}{t^3} B)(t^3K^3+\eps I_1(t)+\eps I_2(t))\,dt\leq \frac{1}{4!} \ell_4^4K^4+\mu \ell_4 BK^3 +( \eps I_2(\ell_4)+\eps I_3(\ell_4)).$$
It turns out that we have $\mu \chi BK^3\ll \chi^4K^4$, so we get $I_4(\ell_4)\leq \frac{1+\eps}{4!} \ell_4^4K^4 +(\eps  I_2(\ell_4)+\eps  I_3(\ell_4)).$
Inductively, we have
 $$I_n(\ell_4)\leq C_n \ell_4^nK^n+ (\eps  I_{n-2}(\ell_4)+\eps  I_{n-1}(\ell_4))$$
 where $C_n$ satisfies $C_{n+1}\leq (\frac{1}{n+1}+\frac{\eps}{n-2}) C_n$. We can bound $C_{n}\leq C( 1+2\eps)^n/n!$. Moreover, by $\chi^4K^4\leq C\chi^3K^3$, we get $\chi^nK^n\leq C^{n-3}\chi^3K^3$ where $C$ is independent of $\mu$ or $\chi$. So we get
 \begin{equation}\label{EqIn}
I_n(\chi)\leq C\frac{(C(1+2\eps))^n}{n!}\chi^3K^3+ (\eps I_{n-2}(\chi)+\eps I_{n-1}(\chi)).\end{equation}
By applying \eqref{EqIn} recursively, we can bound $$I_n(\chi)\leq C\frac{(C(1+2\eps))^n}{n!}\chi^3K^3+O(\eps)\chi^3K^3+O(\eps) I_1(\chi)+O(\eps) I_2(\chi)$$
where the $O$s depend on $n$ but not on $\mu,\chi$.
We choose $n$ (independent of $\mu$ and $\chi$) so that $C\frac{(C(1+2\eps))^n}{n!}$ is smaller than $1/100$.  We then choose $\mu$, without changing $n$, so that all the $O(\eps)s$ are smaller than $1/100$.  We can replace $\chi^3K^3$ by $I_3(\chi)$, because $\frac 16\chi^3K^3<I_3(\chi)$.  Then we see that \eqref{EqCheck} is satisfied if we choose $c_1=2,c_2=3,c_3=4$.

\vskip 0.1in

 For $N_5$, we integrate $\ell_4$ from $O(\chi)$ to $O(1)$, using only $1/\chi$ in $\sw(\ell_4)$ when doing integration since its integral dominates the other term. Again  we can decompose the right matrix in Lemma \ref{Lm: var} as $K+\frac{\mu}{\ell^3_4}B$, whose integration for $\ell_4$ from $\frac{\chi}{2}$ to 1 is $\sim (\frac{\chi}{2}-\ell_4)K+ \frac{\mu}{\ell_4^2}B.$
\begin{equation}\nonumber
\begin{aligned}
I_1(\ell_4)&\lesssim(\chi-\ell_4)K+\frac{\mu}{\ell_4^2}B,\quad I_2(\ell_4)\lesssim(\chi-\ell_4)^2K^2+\frac{\mu}{\ell_4}KB+\mu \frac{(\chi-\ell_4)}{\ell_4^2} BK+\frac{\mu^2}{\ell_4^4}B^2,\\
I_3(\ell_4)&\lesssim(\chi-\ell_4)^3K^3+\mu\ln \frac{\ell_4}{\chi} K^2B+\mu(\frac{\chi}{\ell_4}-\log\ell_4)KBK+\frac{\mu^2}{\ell_4^3}KB^2+\mu \frac{(\chi-\ell_4)^2}{\ell_4^2}BK^2\\
&+\frac{\mu^2 }{\ell_4^3}BKB+\mu^2\frac{\chi-\ell_4}{\ell_4^4}B^2K+\frac{\mu^3}{\ell_4^6}B^3.
\end{aligned}
\end{equation}
We next prove that three steps of Picard iterations give the correct estimate of the fundamental solution.
It can be verified that $$I_3(\ell_4)\leq \frac{1}{3!}(\chi-\ell_4)^3K^3+\mu\frac{(\chi-\ell_4)^2}{\ell_4^2}BK^2+\eps(I_1(\ell_4)+I_2(\ell_4)).$$ In other words, the new contributions from $I_3$ come mainly from $K^3$ and $BK^2$. Moreover, we have  \begin{equation}\label{EqIterationN5}
\begin{aligned} \frac{\mu^2}{\ell_4^4} B^2+
\mu\frac{(\chi-\ell_4)}{\ell_4}KBK\leq \eps((\chi-\ell_4) K+\frac{\mu}{\ell_4^2} B+(\chi-\ell_4)^2 K^2).\end{aligned}
\end{equation}
This inequality allows us to remove higher powers of $B$ and to keep only the $K^j$ and $BK^j$ terms.

Next by splitting the integral into integrals over $[1,\chi/2]$ and $[\chi/2,\chi]$, we have for $t\in [1,\chi]$
$$\int_{t}^\chi\frac{(\chi-s)^n}{s^k}\,ds \leq  100\cdot 2^nt^{-k+1}(\chi-t)^n$$
for $n\geq 1$ and $k= 2,3,4,5. $

Let us now consider $I_4$:
\begin{equation*}
\begin{aligned}
I_4(\ell_4)&\leq \int^{\chi}_{\ell_4} \left(K+\frac{\mu}{t^3}B\right)\left(\frac{1}{3!}(\chi-t)^3K^3+\mu\frac{(\chi-t)^2}{t^2}BK^2+\eps(I_1(t)+I_2(t))\right)dt\\
&\leq \frac{1}{4!}(\chi-\ell_4)^4K^4+ 100\cdot 2^3\frac{1}{3!}\mu\frac{(\chi-\ell_4)^3}{\ell_4^2}BK^3+100\cdot 2^2\mu\frac{(\chi-\ell_4)^2}{\ell_4} KBK^2\\
&+100\cdot 2^2\mu^2\frac{(\chi-\ell_4)^2}{\ell_4^4}B^2K^2
+\eps (I_2(\ell_4)+I_3(\ell_4))\\
&\leq  \frac{1}{4!}(\chi-\ell_4)^4K^4+ 100\cdot 2^3\frac{1}{3!}\mu\frac{(\chi-\ell_4)^3}{\ell_4^2}BK^3+\eps (I_2(\ell_4)+I_3(\ell_4))\\
&+100\cdot 2^3\eps\left[(\chi-\ell_4)^2K^2+(\chi-\ell_4)^3K^3+(\chi-\ell_4)^4K^4+\mu \frac{(\chi-\ell_4)}{\ell_4^2}BK+\mu\frac{(\chi-\ell_4)^2}{\ell_4^2}BK^2\right]
\end{aligned}
\end{equation*}
where in the last $\leq$ we use \eqref{EqIterationN5}.
Inductively, we get
\begin{equation*}
\begin{aligned}
I_n(\ell_4)&\leq \frac{C^n}{n!}(\chi-\ell_4)^nK^n+\frac{C^n}{(n-1)!}\mu\frac{(\chi-\ell_4)^{n-1}}{\ell_4^2}BK^{n-1}+\eps(I_{n-2}(\ell_4)+I_{n-1}(\ell_4))\\
&+\eps C^n\left(\sum_{j=1}^n(\chi-\ell_4)^jK^j\right)+(C\eps)^{n-2}\mu \frac{B}{\ell_4^2}\left(\sum_{j=1}^{n-2}(\chi-\ell_4)^jK^j\right),
\end{aligned}
\end{equation*}
where $(C\eps)^{n-2}$ appears since each application of \eqref{EqIterationN5} gives rise to a multiple of $C\eps$.
Further argument is similar to the $N_1$ case. \qed

\section{Estimates of the boundary contribution}\label{sct: boundary} In this section, we work on all the boundary contributions (see Definition \ref{DefBoundary} and equation \eqref{eq: formald4}) for the maps $(I),\ (III),\ (V)$.
\subsection{Boundary contribution for $(I)$}\label{SSbd1}
\begin{proof}[Computation of matrix $(I)$ in Proposition \ref{Prop: main}]
By \eqref{eq: formald4}, $(I)$ is a product of three matrices \eqref{eq: formald4} and we already know the matrix $N_1$, i.e. the solution of the variational equation. It remains to work out the two matrices for boundary contributions. The expression for $x^R_{4,\parallel}$ is the following (see Appendix \ref{appendixa})
\begin{equation}\label{eq: section}x^R_{4,\parallel}=-\cos g_4 L_4^2(\cosh u_4-e_4)+\sin g_4(L_4G_4\sinh u_4).\end{equation}
For fixed $x^R_{4,\parallel}=-\frac{\chi}{2}$ or $-2$, we can solve $\ell_4$ as a function of $L_4,G_4,g_4$.
The bounds for $L_4,G_4$ have been obtained in Lemma \ref{LmEnergyCons}, \ref{Lm: position}(a) and \ref{Lm: strip}.
So we get the following using the implicit function theorem and Lemma \ref{Lm: dx/dDe}.
\begin{equation}\label{eq: dell4}
\begin{aligned}
\left(\frac{\partial \ell_4}{\partial L_4}, \frac{\partial \ell_4}{\partial G_4}, \frac{\partial \ell_4}{\partial g_4}\right)&=-\frac{1}{\frac{\partial x_{4,\parallel}}{\partial \ell_4}}\left(\frac{\partial x_{4,\parallel}}{\partial L_4},\frac{\partial x_{4,\parallel}}{\partial G_4},\frac{\partial x_{4,\parallel}}{\partial g_4}\right).\\
\left(\frac{\partial \ell_4}{\partial L_4}, \frac{\partial \ell_4}{\partial G_4}, \frac{\partial \ell_4}{\partial g_4}\right)^R\Big|_{x_{4,\parallel}^R=-\frac{\chi}{2}}&\lesssim
(\chi, 1, 1),\quad \left(\frac{\partial \ell_4}{\partial L_4}, \frac{\partial \ell_4}{\partial G_4}, \frac{\partial \ell_4}{\partial g_4}\right)^R\Big|_{x_{4,\parallel}^R=-2}\lesssim
(1, 1, 1).
\end{aligned}\end{equation}
Note that $\ell_4$ depends on all other variables including $L_4$, so using Corollary \ref{Cor: ham}, and Lemma \ref{Lm: dL4}, we obtain for the section, $x_{4,\parallel}^R=-2$,
\begin{equation}\label{eq: bd-2}
 \begin{aligned}
\left(\frac{\partial \ell_4}{\partial \mathcal V}\right)^R\Big|_{x_{4,\parallel}^R=-2}&\lesssim\left(1, \mu,\mu,\mu;\frac{1}{\mu\chi^2},\frac{1}{\chi^3}, \mu,\frac{\mu}{\chi};1, 1\right)_{1\times 10},\\
\mathcal{F}^R\Big|_{x_{4,\parallel}^R=-2}&\lesssim\left(\mu,1,\mu,\mu;\mu,\frac{\mu}{\chi},\frac{1}{\mu\chi^2},\frac{1}{\chi^3}; \mu,\mu \right)^T_{1\times 10}
\end{aligned}\end{equation}
and for the section $x_{4,\parallel}^R=-\frac{\chi}{2}$,
\begin{equation}\label{eq: bd-x/2RI}
\begin{aligned}
\left(\frac{\partial \ell_4}{\partial \mathcal V}\right)^R\Big|_{x_{4,\parallel}^R=-\frac{\chi}{2}}&\lesssim\left(\boldsymbol \chi, \frac{1}{\chi^2},\frac{1}{\chi^2},\frac{1}{\chi^2};\frac{1}{\mu\chi},\frac{1}{\chi^2}, \mu\chi,\mu;1, 1\right)_{1\times 10},\\
 \mathcal{F}^R\Big|_{x_{4,\parallel}^R=-\frac{\chi}{2}}&\lesssim\left(\frac{1}{\chi^3},\mathbf 1, \frac{1}{\chi^3},\frac{1}{\chi^3};\mu,\frac{\mu}{\chi},\frac{1}{\mu\chi^2},\frac{1}{\chi^3} ;\frac{1}{\chi^2},\frac{1}{\chi^2} \right)^T_{1\times 10},
 \end{aligned}
 \end{equation}
 where the two entries in bold font are estimates in the sense of $\sim$ rather than $O$. The $\mathbf 1$ entry in $\mathcal F^R\Big|_{x_{4,\parallel}^R=-\frac{\chi}{2}}$ is already established in Corollary \ref{Cor: ham}. To get the $\boldsymbol \chi$ entry in
 $\left(\frac{\partial \ell_4}{\partial \mathcal V}\right)^R\Big|_{x_{4,\parallel}^R=-\frac{\chi}{2}}$, we use the $(1,1),(1,2)$ entries in
 $\mathcal D$ from  \eqref{eq: d/dDe}. The result is $-\frac{\frac{\partial x_{4,\parallel}}{\partial L_4}}{\frac{\partial x_{4,\parallel}}{\partial \ell_4}}\simeq -2\ell_4/L_4\simeq \chi/L_4^3$, where the last equality
 is obtained by setting $Q_{\parallel}=-\frac{\chi}{2}$ in \eqref{eq: Q4}. In this case $u>0,\ell_4<0$.
Denote
\begin{equation}\label{eqULDef}
\begin{aligned}
&l:=\left(\mathbf 1, \frac{1}{\chi^3},\frac{1}{\chi^3},\frac{1}{\chi^3};\frac{1}{\mu\chi^2},\frac{1}{\chi^3}, \mu,\frac{\mu}{\chi};\frac{1}{\chi}, \frac{1}{\chi}\right)_{1\times 10},\\
 &u:=\left(\frac{1}{\chi^3},\mathbf 1, \frac{1}{\chi^3},\frac{1}{\chi^3};\mu,\frac{\mu}{\chi},\frac{1}{\mu\chi^2},\frac{1}{\chi^3} ;\frac{1}{\chi^2},\frac{1}{\chi^2} \right)^T_{1\times 10}.
 \end{aligned}
 \end{equation}
Then \eqref{eq: bd-x/2RI} gives
\begin{equation}\label{eqULI}\frac{1}{\chi}\left(\frac{\partial \ell_4}{\partial \mathcal V}\right)^R\Big|_{x_{4,\parallel}^R=-\frac{\chi}{2}}\lesssim l,\qquad \mathcal{F}^R\Big|_{x_{4,\parallel}^R=-\frac{\chi}{2}}\lesssim u. \end{equation}

Define
\begin{equation}\label{eqULIboth}
\begin{aligned}
&u_1^i=\mathcal{F}^R\Big|_{x_{4,\parallel}^R=-2},\quad l_1^i=\left(\frac{\partial \ell_4}{\partial \mathcal V}\right)^R\Big|_{x_{4,\parallel}^R=-2},\\ &u_1^f=(1+O(\mu))\mathcal{F}^R\Big|_{x_{4,\parallel}^R=-\frac{\chi}{2}}\lesssim u,\quad l_1^f=\frac{1}{\chi}\left(\frac{\partial \ell_4}{\partial \mathcal V}\right)^R\Big|_{x_{4,\parallel}^R=-\frac{\chi}{2}}\lesssim l,
\end{aligned} \end{equation}
where the inequlities follow from \eqref{eqULI} and \eqref{eq: bd-x/2RI}.
Then $(I)=(\mathrm{Id}-\mathcal{F}^R\otimes
\left(\frac{\partial \ell_4}{\partial\mathcal V}\right)^R)^{-1}N_1(\Id-u^i_1\otimes l_1^i)$ as claimed in Proposition \ref{Prop: main}. To invert $\mathrm{Id}-\mathcal{F}^R\otimes
\left(\frac{\partial \ell_4}{\partial\mathcal V}\right)^R$, we use \eqref{EqInversion} and verify that $\left(\frac{\partial \ell_4}{\partial\mathcal V}\right)^R\cdot \mathcal{F}^R=O(\mu)$. So we get $(I)$ as claimed in Proposition \ref{Prop: main}.

Finally, we show the $\beta$ rotation of the section $\{x^R_{4,\parallel}=-2, \ v^R_{4,\parallel}>0\}$ to  the section $\{(\mathrm{Rot}(-\beta)\cdot x_{4})^R_\parallel=-2, \ v^R_{4,\parallel}>0\}$ after applying $\cR$ in Definition \ref{def: sct} is negligible. Instead of \eqref{eq: section}, we need to use the expression $\cos\beta\cdot x^R_{4,\parallel}-\sin\beta\cdot x^R_{4,\perp}=-2$ and convert $x^R_4$ into Delaunay variables. Since we have $\ell^R_4=O(1)$ here, and $\beta=O(\mu/\chi)$ since $x_1\in\mathcal S_{\mu\hat C}$, we get a correction of order $O(\mu/\chi)\cdot \left(\frac{\partial \ell_4}{\partial \mathcal V}\right)^R$ to $\left(\frac{\partial \ell_4}{\partial \mathcal V}\right)^R$ in \eqref{eq: bd-2}, which is negligible.
\end{proof}
\subsection{Boundary contribution for $(III)$}\label{subsct: bdrIII}
\begin{proof}[Computation of matrix $(III)$ in Proposition \ref{Prop: main}]
For the matrix $(III)$, the solution for the variational equation is given by $N_3$. We only need to work out the two boundary terms on the sections $\left\{x^R_{4,\parallel}=-\frac{\chi}{2}, \ v^R_{4,\parallel}<0\right\},\ \left\{x^L_{4,\parallel}=\frac{\chi}{2},\ \ v^L_{4,\parallel}>0\right\}.$

In \eqref{eq: formald4}, the variables $\mathcal{V}^i,\ \mathcal{V}^f$ should carry superscript $L$ for matrix $(III)$ since we did not compose a coordinate change between the left and right variables in \eqref{eq: formald4}. However, the section $\left\{x^R_{4,\parallel}=-\frac{\chi}{2}, \ v^R_{4,\parallel}<0\right\}$ is defined using variables with superscript $R$, so we first need to express it using left variables.
We use the matrix $R\cdot L^{-1}$ to get $\mathcal{X}^R=R\cdot L^{-1}\mathcal{X}^L$. This implies
\beq
\beal
&x^R_{4,\parallel}=x^L_{1,\parallel}+\frac{1}{1+\mu}x^L_{4,\parallel}=x^L_{1,\parallel}+\frac{1}{1+\mu}(\cos g_4 L_4^2(\cosh u_4-e_4)-\sin g_4(L_4G_4\sinh u_4))=-\frac{\chi}{2}.
\enal
\eeq
So we get the following using the implicit function theorem and Appendix Lemma \ref{Lm: dx/dDe}.
\begin{equation}\label{eq: dell4}
\begin{aligned}
\left(\frac{\partial \ell_4}{\partial L_4}, \frac{\partial \ell_4}{\partial G_4}, \frac{\partial \ell_4}{\partial g_4},\frac{\partial \ell_4}{\partial x_{1,\parallel}}\right)^L &=-\frac{1}{\frac{\partial x^L_{4,\parallel}}{\partial \ell_4}}\left(\frac{\partial x_{4,\parallel}}{\partial L_4},\frac{\partial x_{4,\parallel}}{\partial G_4},\frac{\partial x_{4,\parallel}}{\partial g_4},-(1+\mu)\right)^L,\\
\left(\frac{\partial \ell_4}{\partial L_4}, \frac{\partial \ell_4}{\partial G_4}, \frac{\partial \ell_4}{\partial g_4},\frac{\partial \ell_4}{\partial x_{1,\parallel}}\right)^L\Big|_{x_{4,\parallel}^R =-\frac{\chi}{2}}&\lesssim(\chi, 1,1,1).\\
\end{aligned}\end{equation}
Using Corollary \ref{Cor: ham} and Lemma \ref{Lm: dL4}, we obtain
\begin{equation}\label{eq: bd-x/2RIII}
\begin{aligned}
\left(\frac{\partial \ell_4}{\partial \mathcal V}\right)^L\Big|_{x_{4,\parallel}^R=-\frac{\chi}{2}}&\lesssim\left(\boldsymbol \chi,\frac{1}{\chi^2},\frac{1}{\chi^2},\frac{1}{\chi^2};1,\frac{1}{\chi^2}, \mu\chi,\mu;1, 1\right)_{1\times 10},\\
\mathcal{F}^L\Big|_{x_{4,\parallel}^R=-\frac{\chi}{2}}&\lesssim\left(\frac{1}{\chi^3},\mathbf 1, \frac{1}{\chi^3},\frac{1}{\chi^3};\mu,\frac{\mu}{\chi},\frac{1}{\mu\chi^2},\frac{1}{\chi^3} ;\frac{1}{\chi^2},\frac{1}{\chi^2} \right)^T_{1\times 10},\\
\end{aligned}\end{equation}
where $\mathbf 1$ and $\boldsymbol\chi$ are estimates in the sense of $\sim$, having the same values as that in
\eqref{eq: bd-x/2RI}.
Denote
\begin{equation}\label{eqL'Def}
\begin{aligned}
&l':=\left(\mathbf 1, \frac{1}{\chi^3},\frac{1}{\chi^3},\frac{1}{\chi^3};\frac{1}{\chi},\frac{1}{\chi^3}, \mu,\frac{\mu}{\chi};\frac{1}{\chi}, \frac{1}{\chi}\right)_{1\times 10}, \end{aligned}
 \end{equation}
 which is different from $l$ in its fifth entry.
Then \eqref{eq: bd-x/2RIII} becomes
\begin{equation}\label{eqUL'}\frac{1}{\chi}\left(\frac{\partial \ell_4}{\partial \mathcal V}\right)^L\Big|_{x_{4,\parallel}^R=-\frac{\chi}{2}}\lesssim l',\qquad
\mathcal{F}^L\Big|_{x_{4,\parallel}^R=-\frac{\chi}{2}}\lesssim u.
\end{equation}

For the section $\left\{x^L_{4,\parallel}=\frac{\chi}{2},\ v^L_{4,\parallel}>0\right\}$, the estimate is exactly the same as the case $\left\{x^R_{4,\parallel}=-\frac{\chi}{2},\ v^L_{4,\parallel}<0\right\}$ in $(I)$, i.e. $u_1^f$ and $l_1^f$, and we get the same result
as \eqref{eqULIboth}
\begin{equation}\label{eq: ulIII}
\begin{aligned}
 u_3^f:=\mathcal{F}^L\Big|_{x_{4,\parallel}^L=\frac{\chi}{2}}\lesssim u,
\qquad l_3^f:=\frac{1}{\chi}\left(\frac{\partial \ell_4}{\partial \mathcal V}\right)^L\Big|_{x_{4,\parallel}^L=\frac{\chi}{2}} \lesssim l.
\end{aligned}
 \end{equation}
 However, we note that the $\chi$ entries in $\left(\frac{\partial \ell_4}{\partial \mathcal V}\right)^L\Big|_{x_{4,\parallel}^R=-\frac{\chi}{2}}$ and $\left(\frac{\partial \ell_4}{\partial \mathcal V}\right)^L\Big|_{x_{4,\parallel}^L=\frac{\chi}{2}}$ have the same expression $\frac{-2\ell_4}{L_4}$, by the same calculation as in Section \ref{SSbd1}. So we get that the former is actually $-\chi$ while the latter is $\chi.$ This proves the sign differences of $l_3^i$ and $l_3^f$ in Proposition 5.2(a.2).

We obtain the matrix $(III)=(\Id+\chi u^f_3\otimes l_3^f)N_3(\Id-\chi u^i_3\otimes l_3^i)$ in Proposition \ref{Prop: main}
 by defining
 \begin{equation}\label{eqUL'III} l_3^i:=\frac{1}{\chi}\left(\frac{\partial \ell_4}{\partial \mathcal V}\right)^L\Big|_{x_{4,\parallel}^R=-\frac{\chi}{2}}\lesssim l',\quad u_3^i:=\mathcal{F}^L\Big|_{x_{4,\parallel}^R=-\frac{\chi}{2}}\lesssim u,
 \end{equation}
where the inequalities follow from \eqref{eqUL'}.
\end{proof}
\subsection{Boundary contribution for $(V)$}\label{SSbd5}
\begin{proof}[Computation of matrix $(V)$ in Proposition \ref{Prop: main}]
For the matrix $(V)$, the solution of the variational equation is given by $N_5$. We only need to get two boundary contributions. Notice the section $\left\{x^L_{4,\parallel}=\frac{\chi}{2},\ v^L_{4,\parallel}>0\right\}$ is defined using left variables. However, we need to express the boundary contributions in \eqref{eq: formald4}. The estimate is exactly the same as that for the section $\left\{x^R_{4,\parallel}=-\frac{\chi}{2},\ v^R_{4,\parallel}<0\right\}$ of $(III)$, i.e. $u_3^i$ and $l_3^i,$ though this time we need to use $\mathcal{X}^L=L\cdot R^{-1}\mathcal{X}^R$. We get the same result as \eqref{eqUL'III}
\begin{equation}\label{eq: ulV}
\begin{aligned}
u_5^i:=\mathcal{F}^R\Big|_{x_{4,\parallel}^L=\frac{\chi}{2}}\lesssim u,\qquad
l_5^i:=\frac{1}{\chi}\left(\frac{\partial \ell_4}{\partial \mathcal V}\right)^R\Big|_{x_{4,\parallel}^L=\frac{\chi}{2}}\lesssim  l'.
\end{aligned}
\end{equation}

For the section $\left\{x^R_{4,\parallel}=-2\right\}$, the estimate is exactly the same as the estimate in the $\left\{x^R_{4,\parallel}=-2\right\}$ case of $(I)$, i.e. $u_1^i$ and $l_1^i$ in \eqref{eqULIboth}. Defining
\begin{equation}\label{eq: ulV-2}
\begin{aligned}
u_5^f:=\mathcal{F}^R\Big|_{x_{4,\parallel}^R=-2},\qquad l_5^f:=\left(\frac{\partial \ell_4}{\partial \mathcal V}\right)^R\Big|_{x_{4,\parallel}^R=-2}\end{aligned}
\end{equation}
we get $(V)=(\Id+\chi u^f_5\otimes l_5^f)N_5(\Id-u^i_1\otimes l_5^i)$ as claimed in Proposition \ref{Prop: main}.
The signs of $l_5^i$ and $l_5^f$ are analyzed in the same way as the cases $(I)$ and $(III)$.
\end{proof}

\section{Estimates of the matrices $(II),(IV)$ for switching foci}\label{sct: switch}
In this section, we study the matrices $(II)$ and $(IV)$ in Proposition \ref{Prop: main}.
\subsection{A simplifying computation} We start with a formal calculation, which liberates us from calculating the $\mathcal V_3$ part.
Both $R\cdot L^{-1}$ are $L\cdot R^{-1}$ can be represented as $\left[\begin{array}{cc}\mathrm{Id}_{4\times 4}&0\\
0& T_\mu\end{array}\right]$ for a $8\times 8$ matrix $T_\mu.$  We need to multiply $\frac{\partial \mathcal V^L}{\partial \mathcal X^L}$ on the left and $\frac{\partial \mathcal X^R}{\partial \mathcal V^R}$ on the right to get $(II)=\frac{\partial \mathcal V^L}{\partial \mathcal V^R}$ as follows
\begin{equation}\label{eq: II}
\begin{aligned}
&\frac{\partial \mathcal V^L}{\partial \mathcal X^L}L\cdot R^{-1}\frac{\partial \mathcal X^R}{\partial \mathcal V^R}\\
&=\left[\begin{array}{cc}\frac{\partial \mathcal V_3}{\partial \mathcal X_3}&0\\
0& \frac{\partial (\mathcal V_1,\mathcal V_4)}{\partial (\mathcal X_1,\mathcal X_4)}\end{array}\right]^L_{10\times 12}\left[\begin{array}{cc}\mathrm{Id}_{4\times 4}&0\\
0& T_\mu\end{array}\right]_{12\times 12} \left[\begin{array}{cc}\frac{\partial \mathcal X_3}{\partial \mathcal V_3}&0\\
\frac{\partial (\mathcal X_1,\mathcal X_4)}{\partial \mathcal V_3}& \frac{\partial (\mathcal X_1,\mathcal X_4)}{\partial (\mathcal V_1,\mathcal V_4)}\end{array}\right]^R_{12\times 10}\\
&=\left[\begin{array}{cc}\mathrm{Id}&0\\
0& \frac{\partial (\mathcal V_1,\mathcal V_4)}{\partial (\mathcal X_1,\mathcal X_4)}\end{array}\right]^L_{10\times 12}\left[\begin{array}{cc}\mathrm{Id}&0\\
0& T_\mu\end{array}\right]_{12\times 12} \left[\begin{array}{cc}\mathrm{Id}&0\\
\frac{\partial (\mathcal X_1,\mathcal X_4)}{\partial \mathcal V_3}& \frac{\partial (\mathcal X_1,\mathcal X_4)}{\partial (\mathcal V_1,\mathcal V_4)}\end{array}\right]_{12\times 10}^R.\\
\end{aligned}
\end{equation}
We have the same calculation for $(IV)=\frac{\partial \mathcal V^R}{\partial \mathcal V^L}=\frac{\partial \mathcal V^R}{\partial \mathcal X^R}R\cdot L^{-1}\frac{\partial \mathcal X^L}{\partial \mathcal V^L}$.
In the following, we only need to figure out the matrices $\frac{\partial (\mathcal X_1,\mathcal X_4)}{\partial (\mathcal V_3,\mathcal V_1,\mathcal V_4)}$ and $\frac{\partial (\mathcal V_1,\mathcal V_4)}{\partial (\mathcal X_1,\mathcal X_4)}$.

\subsection{From Delaunay to Cartesian coordinates}
In this section we compute
$$\frac{\partial (\mathcal X_1;\mathcal X_4)}{\partial (\mathcal V_3,\mathcal V_1,\mathcal V_4)}=\frac{\partial (x_1,v_1,x_4,v_4)}{\partial(L_3,\ell_3,G_3,g_3,x_1,v_1,G_4,g_4)}.$$
This computation is restricted to the section $\{x^R_{4,\parallel}=-\frac{\chi}{2}\}$ for matrix $(II)$ and to the section $\{x^L_{4,\parallel}=\frac{\chi}{2}\}$ for matrix $(IV)$. The key observation to obtain the tensor structure of the following sublemma is explained in Remark \ref{RkParallel} (2).
\begin{sublemma}\label{sublm: dX/dV}
Assume {\bf AG}, then
\begin{itemize}
\item[(a)] on the section $\{x_{4,\parallel}^R=-\frac{\chi}{2}\}$ the matrix $\frac{\partial(\mathcal X_1,\mathcal X_4)^R}{\partial \mathcal V^R}$ in \eqref{eq: II} is an $8\times 10$ matrix of the form
\begin{equation}
\begin{aligned}
&\frac{\partial(x_1,v_1,x_4,v_4)^R}{\partial (L_3,\ell_3,G_3,g_3,x_1,v_1,G_4,g_4)^R}=\chi u_i\otimes l_i+\left[\begin{array}{cc|cc}
0_{4\times 4}&\mathrm{Id}_{4\times 4} &0_{4\times 1}&0_{4\times 1}\\
\hline
0_{1\times 4}&0_{1\times 4} & 0& 0\\
0_{1\times 4}&0_{1\times 4} &O(1)&O(1)\\
\ \ \ \breve{l}_i&&O\left(\frac{1}{\chi}\right)&O\left(\frac{1}{\chi}\right)\\
0_{1\times 4}&0_{1\times 4} &0&0
\end{array}\right]_{8\times 10}\label{eq: Car14/De314}\end{aligned}\end{equation}
where we have the estimates \begin{equation}\nonumber\begin{aligned}&u_i=\left(0_{1\times 5},\frac{L_4}{2m_4^2k_4^2},0,\frac{1}{\chi}\right)^T_{1\times 8},\qquad \breve{l}_i\lesssim\(1,\frac{1}{\chi^3},\frac{1}{\chi^3},\frac{1}{\chi^3};\frac{1}{\mu\chi^2},\frac{1}{\chi^3},\mu,\frac{\mu}{\chi}\)_{1\times 8},\\
&l_{i}=\left(\frac{G_4k_4m_4}{L_4(G_4^2+L_4^2)},O\left(\frac{1}{\chi^3}\right)_{1\times 3};O\left(\frac{1}{\mu\chi^2},\frac{1}{\chi^3},\mu,\frac{\mu}{\chi}\right);\frac{-k_4m_4}{G_4^2+L_4^2},\frac{-k_4m_4}{L_4}\right)_{1\times 10}^R
\end{aligned}
\end{equation} and $l_i$ converges to $\hat\brlin$ defined in Lemma \ref{Lm: glob} as $1/\chi\ll\mu\to 0$.
\item[(b)] On the section $\{x_{4,\parallel}^L=\frac{\chi}{2}\}$ the matrix $\frac{\partial(\mathcal X_1,\mathcal X_4)^L}{\partial \mathcal V^L}$ for $(IV)$ has the same form with the same $u_i$, and $l_i$ replaced by \[l_{i'}=\(0_{1\times 8},\frac{-k_4m_4}{L_4^2},\frac{-k_4m_4}{L_4}\right)+O\left(\frac{1}{\chi}, \left(\frac{1}{\chi^4}\right)_{1\times 3};\frac{1}{\mu\chi^3},\frac{1}{\chi^4}, \frac{\mu }{\chi},\frac{\mu}{\chi^2}; \frac{1}{\chi^3},0\)_{1\times 10}.\]
\end{itemize}
\end{sublemma}
\begin{proof} We trivially have $\left(\frac{\partial \mathcal X_1}{\partial (\mathcal V_3,\mathcal V_4)}\right)^{R,L}=0$, and $\left(\frac{\partial \mathcal X_1}{\partial \mathcal V_1}\right)^{R,L}=\mathrm{Id}_4$ since the variables $\mathcal{X}_1=(x_1,v_1)$ are not transformed to Delaunay variables and they are independent of $\mathcal{V}_{3,4}$. It remains to obtain $\frac{\partial \mathcal X_4}{\partial \mathcal V}.$

{\bf Step 1, formal derivations.}

In the following calculation, we use \eqref{eq: dell4}. The formal calculation works for both cases, left and right, so we omit the superscripts. As before, we use $\partial$ to denote the partial derivative with respect to all 12 variables and $\nabla$ to denote the covariant derivative with respect to the 10 variables with $L_4$ and $\ell_4$ eliminated. Since we are restricted to the section $x_{4,\parallel}=\pm\chi/2$, we will solve $\ell_4$ as functions of $L_4$ and $G_4,g_4$ and we use $\frac{\dt \ell_4}{\dt L_4},\frac{\dt \ell_4}{\dt G_4},\frac{\dt \ell_4}{\dt g_4}$ for the corresponding partial derivatives.
\begin{equation}\label{eq: dX/dV}\begin{aligned}
\nabla_{\mathcal V} \mathcal X_4&=\frac{\partial \mathcal X_4}{\partial (L_4,\ell_4)}\nabla_{\mathcal V} (L_4,\ell_4)+\left(0_{4\times 8}\left|\frac{\partial \mathcal X_4}{\partial (G_4,g_4)}\right.\right)\\
&=\left(\frac{\partial \mathcal X_4}{\partial L_4}, \frac{\partial \mathcal X_4}{\partial \ell_4}\right)\left(\begin{array}{c}
\nabla_{\mathcal V}  L_4\\
\frac{\dt \ell_4}{\dt L_4}\nabla_{\mathcal V}   L_4+\left(0_{1\times 8};\frac{\dt \ell_4}{\dt G_4},\frac{\dt \ell_4}{\dt g_4}\right)
\end{array}\right)
+\left(0_{4\times 8}\left|\frac{\partial \mathcal X_4}{\partial G_4},\frac{\partial \mathcal X_4}{\partial g_4}\right.\right)\\
&=\left(\frac{\partial \mathcal X_4}{\partial L_4}+ \frac{\dt \ell_4}{\dt L_4}\frac{\partial \mathcal X_4}{\partial \ell_4}\right)\otimes \nabla_{\mathcal V}   L_4+
\frac{\partial \mathcal X_4}{\partial \ell_4}\otimes \left(0_{1\times 8};\frac{\dt \ell_4}{\dt G_4},\frac{\dt \ell_4}{\dt g_4}\right)
+\left(0_{4\times 8}\left|\frac{\partial \mathcal X_4}{\partial G_4},\frac{\partial \mathcal X_4}{\partial g_4}\right.\right)\\
&=\left(\frac{\partial \mathcal X_4}{\partial L_4}-\frac{\frac{\partial x_{4,\parallel}}{\partial L_4}}{\frac{\partial x_{4,\parallel}}{\partial \ell_4}}\frac{\partial \mathcal X_4}{\partial \ell_4}\right)\otimes \nabla_{\mathcal V}   L_4
-\frac{1}{\frac{\partial x_{4,\parallel}}{\partial \ell_4}}\frac{\partial \mathcal X_4}{\partial \ell_4}\otimes \left(0_{1\times 8};\frac{\partial x_{4,\parallel}}{\partial G_4},\frac{\partial x_{4,\parallel}}{\partial g_4}\right)+\left(0_{4\times 8}\left|\frac{\partial \mathcal X_4}{\partial G_4},\frac{\partial \mathcal X_4}{\partial g_4}\right.\right) \\
&=\left(\frac{\partial \mathcal X_4}{\partial L_4}-\frac{\frac{\partial x_{4,\parallel}}{\partial L_4}}{\frac{\partial x_{4,\parallel}}{\partial \ell_4}}\frac{\partial \mathcal X_4}{\partial \ell_4}\right)\otimes \nabla_{\mathcal V}   L_4
+\left(0_{4\times 8}\left|\frac{\partial \mathcal X_4}{\partial G_4}-\frac{\frac{\partial x_{4,\parallel}}{\partial G_4}}{\frac{\partial x_{4,\parallel}}{\partial \ell_4}}\frac{\partial \mathcal X_4}{\partial \ell_4},\frac{\partial \mathcal X_4}{\partial g_4}-\frac{\frac{\partial x_{4,\parallel}}{\partial g_4}}{\frac{\partial x_{4,\parallel}}{\partial \ell_4}}\frac{\partial \mathcal X_4}{\partial \ell_4}\right.\right)\\
&\sim \(0,-\frac{G_4L_4^2\ell_4}{m_4k_4(G_4^2+L_4^2)}+O(1),\frac{k_4m_4}{L_4^2},\frac{G_4k_4m_4}{L_4(G_4^2+L_4^2)}\)^T\otimes \nabla_{\mathcal V}  L_4 +\\
&\quad\(0_{4\times 8}\left|\begin{array}{cc}
0&0\\
\frac{L^3_4\ell_4/m_4k_4}{(G^2_4+L^2_4)}+O(1)&\frac{L_4^2\ell_4}{m_4k_4}+O(1)\\
\frac{1}{\chi}&\frac{1}{\chi}\\
-\frac{k_4m_4}{G_4^2+L_4^2}&-\frac{k_4m_4}{L_4}
\end{array}\right.\)_{4\times 10}
\end{aligned}
\end{equation}
where in the last step, we use Lemma \ref{Lm: dx/dDe},
and choose the sign $\sigma=$sign$(u)$ to be $+$ for both part (a) and (b) of the Lemma. Actually, the terms $\frac{\frac{\partial x_{4,\parallel}}{\partial (L_4,G_4,g_4)}}{\frac{\partial x_{4,\parallel}}{\partial \ell_4}}\frac{\partial \mathcal X_4}{\partial \ell_4}$ are small compared to the corresponding $\frac{\partial \mathcal X_4}{\partial (L_4,G_4,g_4)}$ due to the smallness of $\frac{\partial (x_{4,\perp},v_4)}{ \partial \ell_4}$ in \eqref{eq: d/dDe}.

 It is easy to see from the above calculation that the first row is zero since the first entry of $\mathcal X_4$
 is $x_{4,\parallel}$. This also follows from the fact that we are restricted on the sections $\{x^R_{4,\parallel}=-\frac{\chi}{2}\}$
 and $\{x^L_{4,\parallel}=\frac{\chi}{2}\}$ so that $x_{4,\parallel}$ is a constant. We already have the tensor structure in the first summand of the last line of \eqref{eq: dX/dV}. Next in the second summand of \eqref{eq: dX/dV}, the two nontrivial columns are nearly parallel.
The reason is that in equation \eqref{eq: d/dDe} the two vectors
 $\frac{\partial x_{4,\perp}}{\partial (L_4,G_4,g_4)}$ and $\frac{\partial v_{4,\perp}}{\partial (L_4,G_4,g_4)}$
 (the second and fourth rows in \eqref{eq: d/dDe}) are parallel with ratio of modulus $\frac{L_4^3\ell_4}{m_4^2k_4^2}$
 if we discard the $O(1)$ terms in the former (see Remark \ref{RkParallel}). The two terms $-\frac{\frac{\partial x_{4,\parallel}}{\partial G_4}}{\frac{\partial x_{4,\parallel}}{\partial \ell_4}}\frac{\partial \mathcal X_4}{\partial \ell_4}$ and $-\frac{\frac{\partial x_{4,\parallel}}{\partial g_4}}{\frac{\partial x_{4,\parallel}}{\partial \ell_4}}\frac{\partial \mathcal X_4}{\partial \ell_4}$ are obviously parallel.

 {\bf Step 2, the case $\frac{\partial(\mathcal X_1,\mathcal X_4)^R}{\partial \mathcal V^R}$ on the section $\{x_{4,\parallel}^R=-\frac{\chi}{2}\}$.}

 Orbit parameters in this step should carry a superscript $R$ which we omit for simplicity.
We define the last row of the above calculation \eqref{eq: dX/dV} as the vector $l_i$. That is,

\begin{equation}
\beal\label{eq: li}l_{i}&:=
\(\frac{G_4k_4m_4}{L_4(G_4^2+L_4^2)}\)\cdot \nabla_{\mathcal V}  L_4 +\(0_{1\times 8};
-\frac{k_4m_4}{G_4^2+L_4^2},-\frac{k_4m_4}{L_4}\).
\enal
\end{equation}

We get the estimate of $l_i$ stated in the lemma
using Lemma \ref{Lm: dL4} for the section $\{x^R_{4,\parallel}=-\frac{\chi}{2}\}$. Moreover, since the first entry in $\nabla_{\mathcal V}   L_4 $ is $\nabla_{L_3}  L_4=1+O(\mu)$ and the last two entries are $O(1/\chi^2)$, we
see that $l_i\to \hat\brlin$ defined in Lemma \ref{Lm: glob} when we take limit $1/\chi\ll\mu\to 0$.

Then the second row of the last line of \eqref{eq: dX/dV} is $(-\frac{L_4^3\ell_4}{m_4^2k_4^2} +O(1))l_i+(0_{1\times 8};O(1),O(1))$. For the third row, we define a vector $\breve{l}_i=\left(1,\frac{1}{\chi^3},\frac{1}{\chi^3},\frac{1}{\chi^3};\frac{1}{\mu\chi^2},\frac{1}{\chi^3},\mu,\frac{\mu}{\chi}\right)_{1\times 8}$ as the first 8 entries of $l_i$, so the third row is $(\breve{l}_i;1/\chi,1/\chi)_{1\times 10}$.

Finally, we define \[u_i=\left(0,0,0,0;0,\frac{L_4}{2m_4^2k_4^2}, 0,\frac{1}{\chi}\right)^T_{1\times 8},\]
which, after removing the first four zeros, is almost parallel to the last two columns in the last row of \eqref{eq: dX/dV},
where we used \eqref{eq: Q4}  to get that $L_4^2\ell_4\simeq -\frac{\chi}{2}$
when restricted to the section $\left\{x_{4,\parallel}^R=-\frac{\chi}{2}\right\}$.
This completes the proof of part (a).

{\bf Step 3, the case of $\frac{\partial(\mathcal X_1,\mathcal X_4)^L}{\partial \mathcal V^L}$ on the section $\{x_{4,\parallel}^L=\frac{\chi}{2}\}$.}
 This case follows from the same formal calculation \eqref{eq: dX/dV}. However, since the variables are to the left of the section $x^L_{4,\parallel}=\frac{\chi}{2}$, we have $G_4^L=O\left(\frac{1}{\chi}\right)$ in \eqref{eq: li} according to Lemma \ref{Lm: tilt}(b).
Thus we get an improved $l_{i'}$ in place of $l_i$ by applying Lemma \ref{Lm: dL4} to \eqref{eq: li}. We also have $L_4^2(-\ell_4)\simeq \frac{\chi}{2}$ using \eqref{eq: Q4l} for $u>0, \ell_4<0$ to the left of the section $x^L_{4,\parallel}=\frac{\chi}{2}$. So $u_i$ in the left case has the same expression as in the right case. This proves part (b).
\end{proof}
\subsection{From Cartesian to Delaunay coordinates}
In this section we compute  $\frac{\partial(\mathcal V_1;\mathcal V_4)}{\partial (\mathcal X_1,\mathcal X_4)}=\frac{\partial(x_1,v_1,G_4,g_4)}{\partial (x_1,v_1,x_4,v_4)}.$
The key observation to get the tensor structure is explained in Remark \ref{RkParallel} (3).
\begin{sublemma}\label{sublm: dV/dX}
Assume {\bf AG}, then
\begin{itemize}
\item[(a)]on the section $\{x_{4,\parallel}^R=-\frac{\chi}{2}\}$
the matrix $\frac{\partial(\mathcal V_1;\mathcal V_4)^L}{\partial (\mathcal X_1,\mathcal X_4)^L}$ in \eqref{eq: II} is a $6\times 8$ matrix of the following form
\begin{equation}
\begin{aligned}
&\frac{\partial (x_1,v_1,G_4, g_4)^L}{\partial(x_1,v_1,x_4,v_4)^L }=\chi u_{iii}\otimes l_{iii}+\left[\begin{array}{c|cccc}
\mathrm{Id}_{4\times 4}&0_{4\times 1}&0_{4\times 1}&0_{4\times 1}&0_{4\times 1}\\
\hline
0_{1\times 4}&0&0&0&0\\
0_{1\times 4}&O\left(\frac{1}{\chi^2}\right)&O\left(\frac{1}{\chi^2}\right)& O(1)&O(1)
\end{array}\right]_{6\times 8},
\end{aligned}
\end{equation}
where we have estimates {\small\[u_{iii}=\left(0,0,0,0; 1, \frac{1}{L_4}+O\(\frac{1}{\chi^2}\)\right)^T_{1\times 6},\ l_{iii}=\left(0_{1\times 4}; O\(\frac{1}{\chi^2}\), \frac{-m_4k_4}{\chi L_4}, O\(\frac{1}{\chi}\), -\frac{1}{2}\right)_{1\times 8}.\]}
\item[(b)] On the section $\{x_{4,\parallel}^L=\frac{\chi}{2}\}$ the matrix $\frac{\partial(\mathcal V_1;\mathcal V_4)^R}{\partial (\mathcal X_1,\mathcal X_4)^R}$ for $(IV)$ has the same form with $u_{iii}$ replaced by
\[u_{iii'}=\left(0,0,0,0; 1, \frac{L_4}{G_4^2+L_4^2}\right)^T_{1\times 6},\]
and $l_{iii}$ replaced by $l_{iii'}=-l_{iii}$.
\end{itemize}
\end{sublemma}
\begin{proof}The only nontrivial part of this matrix is $\frac{\partial(G_4,g_4)}{\partial (x_4,v_4)}$. We consider first part (a), to the left of the section $\{x_{4,\parallel}^R=-\frac{\chi}{2}\}$, where the variables should carry superscript $L$ that we omit.  It follows from Lemma \ref{Lm: De/Ca}(b) that
$$\frac{\partial g_4}{\partial (x_4,v_4)}=\frac{L_4}{G_4^2+L_4^2}\frac{\partial G_4}{\partial (x_4,v_4)}+O\left(\frac{1}{\chi^2}, \frac{1}{\chi^2}, 1,1\right).$$
 This implies that the two rows in $\frac{\partial(G_4,g_4)^L}{\partial (x_4,v_4)^L}$ are almost parallel up to the $O$ term. Therefore we have the tensor structure in the lemma.

Next we define
\begin{equation}
\begin{aligned}
u_{iii}&=\left(0_{1\times 4}; 1, \frac{L_4}{G_4^2+L_4^2}\right)^T=\left(0_{1\times 4}; 1, \frac{1}{L_4}+O\left(\frac{1}{\chi^2}\right)\right)^T,\\
 l_{iii}&=\frac{1}{\chi}\left(0_{1\times 4}; \frac{\partial G_4}{\partial (x_4,v_4)}\right)=\left(0_{1\times 4}; O\(\frac{1}{\chi^2}\), \frac{-m_4k_4}{\chi L_4}, O\(\frac{1}{\chi}\), -\frac{1}{2}\right),
 \end{aligned}
 \end{equation}
where the entry $\frac{-m_4k_4}{\chi L_4}$ is obtained using the following formulas
$$\frac{\partial G}{\partial Q_\perp}=P_\parallel \text{ (by Lemma \ref{Lm: De/Ca}), } E_4=\frac{|P|^2}{2m_4}-\frac{k_4}{|Q|}=\frac{m_4k_4^2}{2L_4^2}, \quad
|P|\simeq |P_\parallel | \text{ and } P_\parallel<0.$$
This gives the matrix stated in the sublemma. In part (a), all the Cartesian and Delaunay variables are immediately to the left of the section $\{x_{4,\parallel}^R=-\frac{\chi}{2}\}$, so we have $G^L_4=O(1/\chi)$ using Lemma \ref{Lm: tilt} and $x^L_{4,\parallel}\simeq \frac{\chi}{2}.$

Next we consider part (b). It follows from Lemma \ref{Lm: De/Ca} that to the right of the section $\{x^L_{4,\parallel}=\frac{\chi}{2}\}$ the matrix $\frac{\partial(\mathcal V_1;\mathcal V_4)^R}{\partial (\mathcal X_1,\mathcal X_4)^R}$ has the same estimates as in the
left case with
\begin{equation}
\begin{aligned}
u_{iii'}&=\left(0_{1\times 4}; 1, \frac{L_4}{G_4^2+L_4^2}\right)^T,\\
 l_{iii'}&=\frac{1}{\chi}\left(0_{1\times 4}; \frac{\partial G_4}{\partial (x_4,v_4)}\right)=\left(0_{1\times 4};
O\(\frac{1}{\chi^2}\), \frac{m_4k_4}{\chi L_4}, O\(\frac{1}{\chi}\), \frac{1}{2}\right).
 \end{aligned}
 \end{equation}
We see that  $l_{iii'}$ gets a ``-" sign compared to $l_{iii}$ since both $P_\parallel$ and $Q_\parallel$ get ``-" signs.
\end{proof}
With the two sublemmas, we can complete the computation of the matrices $(II)$ and $(IV)$.
\begin{proof}[Computation of matrices $(II)$ and $(IV)$ in Proposition \ref{Prop: main}]
To be compatible with the formal derivation in \eqref{eq: II}, we add four zeros to $u_i$ as the new first four entries. We still denote the new vector of 12 components by $u_i$ as stated in Proposition \ref{Prop: main}. We also define a $12\times 10$ matrix $C=\left[\begin{array}{c}
\mathrm{Id}_{4\times 4}, 0_{4\times 6}\\
*
\end{array}\right]$ where $*$ is the $O(1)$ matrix in Sublemma \ref{sublm: dX/dV}.

Then consider Sublemma \ref{sublm: dV/dX}. To be compatible with the formal derivation in \eqref{eq: II}, we enlarge $u_{iii}, u_{iii'}$ by adding four zeros as the new first four entries to get vectors in $\R^{10}$. We define a $10\times 12$ matrix $A=\left[\begin{array}{cc}
\mathrm{Id}_{4\times 4}& 0_{4\times 8}\\
0_{6\times 4}&*
\end{array}\right]$, where $*$ is the $O(1)$ matrix of Sublemma \ref{sublm: dV/dX}.

Fitting these manipulations into \eqref{eq: II} gives
\[(II)=(\chi u_{iii}\otimes l_{iii}+A)R\cdot L^{-1}(\chi u_{i}\otimes l_{i}+C ),\;
(IV)=(\chi u_{iii'}\otimes l_{iii'}+A)L\cdot R^{-1}(\chi u_{i}\otimes l_{i'}+C). \qedhere \]
\end{proof}

\section{The local map}\label{sct: loc}

The section $\{|q_3-q_4|=\mu^\kappa\}$ ($1/3< \kappa< 1/2$) cuts the orbit for the local map into three pieces: $\{x_{4,\parallel}^R=-2\}\to \{|q_3-q_4|=\mu^\kappa\}$, $\{|q_3-q_4|=\mu^\kappa\}\to \{|q_3-q_4|=\mu^\kappa\}$ and $\{|q_3-q_4|=\mu^\kappa\}\to \{x_{4,\parallel}^R=-2\}$. We define three maps $\Loc^-,\Loc^0,\Loc^+$ corresponding to the three pieces and we have $\Loc=\Loc^+\circ\Loc^0\circ\Loc^-$.
\begin{Not}
\begin{itemize}
\item We use the superscript $+$ $($or $-)$ to denote the value of the orbit parameters exiting $($or entering$)$ the circle $|q_3-q_4|=\mu^\kappa$.
\item Also recall the coordinates $q_-,\,p_-$ for the relative motion and $q_+,\ p_+$ for the motion of the center of mass of $Q_3$ and $Q_4$ in \eqref{eq: relcm}, i.e. $q_\pm=\frac{1}{2}(q_3\pm q_4)$ and $p_\pm=p_3\pm p_4.$
\item We introduce the notation
$\mathbf q=(q_+,q_1),\ \mathbf p=(p_+,p_1)$
to handle the center of mass and the remote body simultaneously.
\end{itemize}
\end{Not}
The Hamiltonian for $\Loc^\pm$ is the following given in Lemma \ref{LmRightOut}
\begin{equation}\label{eq: hamRRDel2}
\begin{aligned}
&H=\left(\frac{v_1^2}{2m_{1R}} -\frac{k_{1R}}{ |x_1|}\right)-\frac{m_{3R}k_{3R}^2}{2L_3^2}+\frac{m_{4R}k_{4R}^2}{2L_4^2}
-\frac{\mu}{\left|\frac{x_3}{1+\mu}-x_4\right|}+V_{out}
\end{aligned}
\end{equation}
where $V_{out}(x_3,x_1,x_4)=O(\mu+\frac{\mu}{\chi^2}+\frac{1}{\chi^3}) $ and the coordinates are Delaunay coordinates from \eqref{eq: right}, and the Hamiltonian for $\Loc^0$ is given in Lemma \ref{LmRightIn}
\begin{equation}\label{eq: hamrel2}
\begin{aligned}
H( q_1, p_1;q_-,p_-; q_+,p_+)&=\left(\mu p_1^2-\frac{1+2\mu}{\mu |q_1|}\right)+\left(\frac{1+2\mu}{4}p_+^2-\frac{2}{|q_+|}\right)+\left(\frac{1}{4}p_-^2-\frac{\mu}{2|q_-|}\right)+V_{in},
\end{aligned}
\end{equation}
where $V_{in}(q_1, p_1;q_-; q_+,p_+)=\mu\langle p_1,p_+\rangle- \frac{3\langle q_+,q_-\rangle^2}{2|q_+|^5}+
\frac{|q_-|^2}{|q_+|^3}-\frac{\langle q_1,q_+\rangle}{|q_1|^3}+O\left(|q_-|^3+\frac{1}{|q_1|^3}\right),$
and the coordinates are the relative motion and center of mass coordinates \eqref{eq: relcm}.
\subsection{$\mathscr{C}^0$ control of the local map, proof of Lemma \ref{LmLMC0}}\label{SSC0Loc}
In this section, we obtain the $\mathscr{C}^0$ estimate of the local map, based on which we prove Lemma \ref{LmLMC0}.

Suppose the assumption ${\bf AL}$ for the local map is satisfied. Then the orbit parameters $\bx^\pm=( x_1, v_1;L_3,\ell_3,G_3,g_3;L_4,\ell_4,G_4,g_4)^{R,\pm}$ evaluated on the section $\{x_{4,\parallel}^R=-2,\ \pm v_{4,\parallel}^R< 0 \}$, satisfy $|x_1|\geq \chi$ and all the other variables are $O(1)$ as $1/\chi\ll\mu\to 0$.

First it is clear that the time interval defining the maps $\Loc^\pm$ is $O(1)$. After integrating the Hamiltonian equations obtained from the Hamiltonian \eqref{eq: hamRRDel2}, we see that $V_{out}$ gives only a perturbation of order $ O(\mu+\frac{\mu}{\chi^2}+\frac{1}{\chi^3}) $ to the Kepler motion.  Next, the contribution of the term $-\frac{\mu}{\left|\frac{x_3}{1+\mu}-x_4\right|}$ to the Hamiltonian equation is estimated as $\frac{\mu}{\left|\frac{x_3}{1+\mu}-x_4\right|^2}$ and  after integration its perturbation to the Kepler motion is estimated as $\int_{-2}^{\mu^\kappa}\frac{\mu}{|t|^{2}}\,dt=O(\mu^{1-\kappa})$ since relative velocity is nonzero and the orbit approaches close encounter nearly linearly in $t$. So the orbit parameters are Kepler motions with $o(1)$ perturbations as $1/\chi\ll\mu\to 0$. In particular, except for the two variables $\ell_3,\ell_4$, all the other orbit parameters undergo only a $O(\mu^{1-\kappa})(\ll\mu^\kappa)$ perturbation.

Next we study the dynamics inside $\{|q_3-q_4|=\mu^\kappa\}$. The next lemma shows that the map $\Loc^0$ is close to elastic collision.
\begin{Lm} \label{Lm: landau}
Suppose the initial orbit parameters on the section $\{x_{4,\parallel}^R=-2,\  v_{4,\parallel}^R< 0 \}$ satisfy ${\bf AL}$.
\begin{itemize}
\item[(a)] We have the following equations for orbits crossing the section $\{|q_3-q_4|=\mu^\kappa\}$, $\frac13<\kappa<\frac12,$ as $\mu\to 0$
\begin{equation}
\begin{cases}
&p_3^+=\frac{1}{2}\mathrm{Rot}(\al)(p_3^--p_4^-)+\frac{1}{2}(p_3^-+p_4^-)+O(\mu^{2(1-\kappa)}+\mu^{3\kappa-1})
,\\
&p_4^+=-\frac{1}{2}\mathrm{Rot}(\al)(p_3^--p_4^-)+\frac{1}{2}(p_3^-+p_4^-)+O(\mu^{2(1-\kappa)}+\mu^{3\kappa-1}),\\
&(\mathbf q^+,\mathbf p^+)=(\mathbf q^-,\mathbf p^-)+O(\mu^k),\\
&|q_3^--q_4^-|=\mu^{\kappa},\quad |q_3^+-q_4^+|=\mu^{\kappa},\\
\end{cases}\label{eq: landau}
\end{equation}
where $\mathrm{Rot}(\al)=\left[\begin{array}{cc} \cos \al&  -\sin\al\\ \sin \al &\cos\al \end{array}\right]$, and
\begin{equation}
\label{Alpha}
\al=\pi+2\arctan\left(\frac{G_{in}}{\mu \mathcal{L}_{in}}\right), \text{ where }\ \frac{1}{4\mathcal{L}^{2}_{in}}=\frac{p_-^2}{4}-\frac{\mu}{2|q_-|},\quad G_{in}= 2p_-\times q_-.
\end{equation}
\item [(b)] We have $1/c<\mathcal{L}_{in}<c$ for some constant $c>1$. If $\al$ is bounded away from $0$ and $\pi$ by an angle independent of $\mu$ then $G_{in}=O(\mu)$
and the closest distance between $q_3$ and $q_4$ is bounded away from zero by $ \mu/c$ and from above by $\mu c$.

\item[(c)] If $\al$ is bounded away from $0$ and $\pi$ by an angle independent of $\mu$
then when measured on the boundary of the circle $|q_-|=2\mu^{\kappa}$,
the angle between $q_-$ and $p_-$ is $O(\mu^{1-\kappa})$.
\item[(d)] The time interval during which the orbit stays in the circle $|q_3-q_4|=\mu^{\kappa}$
is \[\Delta t =O(\mu^{\kappa}).\]
\end{itemize}
\end{Lm}

\begin{proof}

From \eqref{eq: hamrel2}, we decompose the Hamiltonian as $H=H_{rel}+\fh(\mathbf q,\mathbf p)$ where the Hamiltonian $H_{rel}$ governs the relative motion consisting of all the terms containing $q_-$ or $p_-$ in \eqref{eq: hamrel2}, as
\begin{equation}\label{eq: Hrel}H_{rel}=\frac{\mu^2}{4L_-^2}+\frac{|q_-|^2}{|q_+|^3}-\frac{3\langle q_+,q_-\rangle^2}{2|q_+|^5}+O(\mu^{3\kappa}+1/\chi^2),\end{equation}
where $\frac{\mu^2}{4L_-^2}=\frac{1}{4}p_-^2-\frac{\mu}{2|q_-|}$ is in Delaunay coordinates and $L_-=\mu\mathcal L_{in}$.

Fix a small number $\delta_1.$ Below we derive several estimates valid for the first $\delta_1$ units
of time the orbit spends in the set $|	q_-|\leq \frac12\mu^k.$ We then show that $\Delta t\ll \delta_1.$
It will be convenient to measure time from when the orbit enters the set $|q_-|<\frac12\mu^k.$

Note that $H$ is preserved and $\dot{\fh}=O(1)$, which implies that $\frac{L}{\mu}$ is $O(1)$ and
moreover that this ratio does not change much for $t\in [0, \delta_1].$
Using the identity
$\frac{\mu^2}{4L^2}=\frac{p_-^2}{4}-\frac{\mu}{2|q_-|}$ we see that initially
$\frac{L}{\mu}$ is uniformly bounded from below for the orbits by Lemma \ref{LmLMC0}.
Thus there is a constant $\delta_2$ such that for $t\in [0, \delta_1]$
we have $\delta_2\mu\leq L(t) \leq \frac{\mu}{\delta_2}. $

From  \eqref{eq: Ol4}, we obtain
\begin{equation}\label{eq: delaunayscattering}O(\mu^\kappa)=|q_-|=\dfrac{L^2}{\mu}(e\cosh u-1),\end{equation}
 with $\ell$ and $u$ related by $u-e\sinh u=\ell.$
This gives
\begin{equation}
\label{SectionEll}
\ell=O(\mu^{\kappa-1}).
\end{equation}
 Next
\begin{equation}\label{EqDotl} \dot{\ell}=\dfrac{\partial H}{\partial L}=-\dfrac{\mu^2}{2L^3}
-\dfrac{\partial H_{rel}}{\partial q_-}
\dfrac{\partial q_-}{\partial L}=
-\dfrac{\mu^2}{2L^3}+O(\mu^\kappa)O(\mu^{\kappa-1})=-\dfrac{\mu^2}{2L^3}+O(\mu^{2\kappa-1}). \end{equation}
Since the leading term here is at least
$\frac{\delta_2^3}{2\mu}$ while $\ell=O(\mu^{\kappa-1})$, we obtain part (d) of the lemma.
In particular the estimates derived above are valid for the time the orbit
spends in $|q_-|\leq \frac12\mu^\kappa.$

Next, without using any control on $G$ (using the inequality $\left|\frac{\partial e}{\partial G}\right|=\frac{1}{L}\frac{G/L}{e}\leq \frac{1}{L}$), we have
\begin{equation}
\begin{aligned}
\label{DerG}
\dot{L}&=-\dfrac{\partial H}{\partial q_-}\dfrac{\partial q_-}{\partial \ell}=O(\mu^{\kappa+1}),\\
\dot{G}&=-\dfrac{\partial H}{\partial q_-}\dfrac{\partial q_-}{\partial g}=O(|q_-|^2)=O(\mu^{2\kappa}), \\
 \dot{g}&=
\dfrac{\partial H}{\partial q_-} \dfrac{\partial q_-}{\partial G}=
O(\mu^\kappa)O(\mu^{\kappa-1})=O(\mu^{2\kappa-1}).
\end{aligned}
\end{equation}
Integrating over time $\Delta t=O(\mu^{\kappa})$ we get that
the oscillation of
$g$ and $\arctan\frac{G}{L}$ are $O(\mu^{3\kappa-1}).$

We are now ready to derive the first two equations of \eqref{eq: landau}. It is enough to show that $p_-^+=R(\al) p_-^-+O(\mu^{2(1-\kappa)}+\mu^{3\kappa-1})$, where $\al=2\arctan \frac{G}{L}$ is the angle formed by the two asymptotes of the Kepler hyperbolic motion. We first have $|p_-^+|=|p_-^-|+O(\mu^\kappa)$ using the total energy conservation. It remains to show the expression of $\al$.
Let us denote  $\phi=\arctan\frac{G}{L}$ and
    \begin{equation}\label{EqDelMom}p_-=(p_1,p_2),\quad (p_1,p_2)=R(g)\left(-\dfrac{\mu}{L}\dfrac{\sinh u}{1-e\cosh u},-\dfrac{ \mu G}{ L^{2}}\dfrac{\cosh u}{1-e\cosh u}\right)\end{equation}
     using \eqref{eq: delaunay4}. We have by \eqref{SectionEll} that
$e^{|u|}\sim \ell\sim \mu^{\kappa-1}$. Thus
\begin{equation}\label{Eqpg}
\begin{aligned}
\dfrac{p_2}{p_1}&=\dfrac{-\frac{\mu G}{L^2}\cosh u \cos g-\frac{\mu}{L}\sinh u\sin g}{\frac{\mu G}{L^2}\cosh u \sin g-\frac{\mu}{L}\sinh u\cos g}=\dfrac{\frac{G}{L}\pm \tan g}{\pm 1-\frac{G}{L}\tan g}+e^{-2|u|} E(G/L, g, u))\\
&=
\tan(g\pm \phi)+O(\mu^{2(1-\kappa)}),\end{aligned} \end{equation}
where the sign $\pm$ is taken as sign$(u)$ and $E$ is a $O(1)$ function as $|u|\to \infty$.
Since $\arctan$ is globally Lipschitz,
this completes the proof of part (a) by choosing $\al=2\phi$.

From the Hamiltonian equations for $\dot{\mathbf q}$, we obtain
\begin{equation}
\label{CmFixed}
\mathbf q^+=\mathbf q^-+O(\mu^{\kappa}).
\end{equation}
We also have $q_-^+=q_-^-+O(\mu^{\kappa})$ due to to the definition of the sections $\{|q_-^\pm|=\frac12\mu^{\kappa}\}$.
This proves the last two equation in \eqref{eq: landau}.  Plugging \eqref{CmFixed} into the Hamiltonian equation for $\dot{\mathbf p}$ we see that $ \mathbf p^+=\mathbf p^-+O(\mu^{\kappa})$.
This completes the proof of part (a).

The first claim of part (b) has already been established. The estimate of $G$ follows from the formula for $\al.$
The estimate of the closest distance follows from the fact that if $\alpha$ is bounded away from $0$ and $\pi$ then the $q_-(t)$ orbit is a small perturbation of Kepler motion and for Kepler motion the closest distance is of order $G.$ We integrate the $\dot G$ equation \eqref{DerG} over time $O(\mu^\kappa)$ to get that the total variation $\Delta G$ is at most $\mu^{3\kappa}$, which is much smaller than $\mu$. So $G$ is bounded away from 0 by a quantity of order $O(\mu)$.

Finally we get part (c) from the fact that $G=\mu^\kappa|v_-|\sin\measuredangle(p_-,q_-)=O(\mu).$
\end{proof}

Now we are ready to prove Lemma \ref{LmLMC0}.
\begin{proof}[Proof of Lemma \ref{LmLMC0}]

Since we assume the outgoing asymptote $\bar\theta^+$ is close to $\pi$, we get that the orbit under consideration has to intersect the section $|q_3-q_4|=\mu^{\kappa}$ and also achieve $|q_3-q_4|=O(\mu)$ by Lemma \ref{Lm: landau}(b). Indeed,  it is enough to show that  $p_-^+\times p_-^-$ is bounded away from 0 for both collisions. For the first collision $4(1+\varepsilon_0\varepsilon_1)(p_-^-\times p_-^+)=8(\sqrt{5+2\varepsilon_0\varepsilon_1} - 3(\varepsilon_1 - \varepsilon_0)> 8\sqrt{5} - 3$. For the second collision $\varepsilon_0^2(p_-^-\times p_-^+) = 2(\varepsilon_0+\varepsilon_1)(1+\sqrt{2})>2(1+\sqrt{2})$.

 With the same initial $E_3,e_3,g_3,e_4$, we determine a solution of the Gerver map. We have shown at the beginning of the section that the orbit parameters $(E_3,e_3,g_3,e_4)$ have an oscillation of order $O(\mu^{1-\kappa})$ and the $\ell_3$ variable is solved from the implicit function $|q_3-q_4|=\mu^{\kappa}$. We get that the $p^{-}_{3,4},q_{3,4}^-$ at collision in Gerver's case is close to those values measured on the section $|q_3-q_4|=\mu^{\kappa}$ in the $\mu>0$ case. Here we note that the coordinate change between Cartesian and Delaunay outside the section $|q_3-q_4|=\mu^{\kappa}$ is not singular.
Letting $\mu=0$ in the first two equations of \eqref{eq: landau} we obtain the equations
of elastic collisions. Namely, both the kinetic energy and momentum conservation hold:
$$|p_3^+|^2+|p_4^+|^2=|p_3^-|^2+|p_4^-|^2,\quad p_3^++p_4^+=p_3^-+p_4^- .$$
On the other hand, the Gerver map $\Ger$ in Lemma~\ref{LmLMC0} is also
defined through elastic collisions. If we could show that the rotation angle $\al$ in the $\mu>0$ case is close to Gerver's case, we then could show that the outgoing information $p^{+}_{3,4},q_{3,4}^+$ are close in both cases. We then complete the proof using the fact that the orbit outside $|q_3-q_4|=\mu^\kappa$ is an $O(\mu^{1-\kappa})$ small perturbation of the Kepler motion after running the orbit up to the section $\{x_4=-2\}$. By converting $p_4^+,q_4^+$ into Delaunay coordinates, we can express the outgoing asymptote $\bar\theta^+$ as a function of $p_4^+,q_4^+$, therefore a function of $\al,p_{3,4}^-,q_{3,4}^-$ using \eqref{eq: landau} where $\mu=0$ corresponds to Gerver's case. To compare the angle $\al$, it is enough to show that  the outgoing asymptote $\bar\theta^+$ as a function of $\al$ has non degenerate derivative so that we can apply the implicit function theorem to solve $\al$ as a function of $\bar\theta^+$ and the initial conditions. In fact we have $d\theta^+=c\brlin$ ($c\neq 0$) and from Corollary \ref{Cor} we have $\frac{\partial}{\partial \al}=\mathbf u$, hence $\frac{d\bar\theta^+}{d\al}=c \brlin \cdot \mathbf{u}$, which is non vanishing due to Lemma \ref{LmNonDeg}. Here the vectors $\brlin$ and $\mathbf{u}$ are in Lemma \ref{Lm: loc} and \ref{Lm: glob} with subscripts omitted.   So the assumption $|\bar\theta^+-\pi|\leq\tilde\theta$ implies that $\al$ in \eqref{eq: landau} is $\tilde\theta$-close to its value in Gerver's case. 

\end{proof}
\subsection{$\mathscr C^1$ control of the local map, proof of Lemma \ref{Lm: loc}}\label{SSLocC1}
To study the $\mathscr C^1$ estimate of the local map, we first show that $\Loc^+$ and $\Loc^-$ are negligible and we then focus on $\Loc^0$.
\begin{Lm} \label{Lm: 2pieces}Consider the maps $\Loc^\pm$ under the assumption of Lemma \ref{Lm: loc}.
Then the vectors $\brlin_j,\brrlin_j,\ j=1,2,$ are almost left invariant by $d\Loc^+ $and $\mathrm{span}\{\brv_{3-j},\brrv_{3-j}\}$ is almost right invariant by $d\Loc^-$ in the following sense: as $1/\chi\ll\mu\to 0$
\[d\Loc^-_2\cdot \mathrm{span}\{\brv_{1},\brrv_{1}\}=\mathrm{span}\{\brv_{1},\brrv_{1}\}+o(1),\]
\[d\Loc^-_1d\cR\cdot \mathrm{span}\{\brv_{2},\brrv_{2}\}=d\cR \cdot \mathrm{span}\{\brv_{2},\brrv_{2}\}+o(1),\]\[ \brlin_j\cdot d\Loc_j^+=\brlin_j+o(1),\ \brrlin_j\cdot d\Loc_j^+=\brrlin_j+o(1),\]
 \[d\tilde{\mathbb G}\cdot \mathrm{span}\{\brv_{2},\brrv_{2}\}=\mathrm{span}\{\brv_{2},\brrv_{2}\}+o(1).\]
\end{Lm}
\begin{proof} The proof is again to use \eqref{eq: formald4}  to reduce the proof into two boundary terms and the fundamental solution of the variational equation. We use the Hamiltonian \eqref{eq: hamRRDel2}. The potential $-\frac{\mu}{\left|\frac{x_3}{1+\mu}-x_4\right|}$ contributes a term of order $-\frac{\mu}{\left|\frac{x_3}{1+\mu}-x_4\right|^3}$ to the variational equation and its contribution to the fundamental solution is estimated as the integral $\int_{-2}^{\mu^\kappa}\frac{\mu}{|t|^{3}}\,dt=O(\mu^{1-2\kappa})$ since the relative velocity is of order 1 approaching close encounter. Moreover, the terms of $V_{out}$, all of which are also in $U^R$, contribute only $O(\mu)$ to the fundamental solution.  (If we let $\chi$ go to infinity and set $\ell_4 = O(1)$, then all the entries in the matrix in the statement of Lemma 7.3(a) become $\mu$, $\mu^2$, or 0.) So the fundamental solution is estimated as Id$+c e_{2,1}+O(\mu^{1-2\kappa})$ where $e_{2,1}$ is the matrix whose $(2,1)$ entry is 1 and 0 otherwise, and $c$ is a constant. For the boundary terms in \eqref{eq: formald4}, the estimate  $\mathcal F$ of the Hamiltonian equations is $(0,-1,0_{1\times 8})+O(\mu^{1-2\kappa})$, given in Section \ref{SSC0Loc} for $\{|q_3-q_4|=\mu^\kappa\}$ as well as Corollary \ref{Cor: ham} for the section $\{x^R_{4,\parallel}=-2\}$. It remains to estimate the term $\nabla_{\mathcal V}\ell_4 $ in \eqref{eq: formald4}. This estimate on the section $\{x^R_{4,\parallel}=-2\}$ is given in Section \ref{SSbd1} and \ref{SSbd5} as equation \eqref{eq: bd-2}. We next show the estimate of $\nabla_{\mathcal V}\ell_4$ on the section $\{|q_3-q_4|=\mu^\kappa\}$. We have
\begin{equation}\nabla_{\mathcal V}\ell_4=-\left(\dfrac{\partial |q_3-q_4|}{\partial \ell_4}\right)^{-1}\nabla_{\mathcal V} |q_3-q_4| =-\dfrac{(q_3-q_4)\cdot\nabla_{\mathcal V}(q_3-q_4) }{(q_3-q_4)\cdot\frac{\partial (q_3-q_4)}{\partial \ell_4}}\label{eq: l4}\end{equation}
By Lemma \ref{Lm: landau}(c) we know that the angle formed by $q_3-q_4$ and $p_3-p_4$ is $O\left(\mu^{1-\kappa}\right)$. Thus in \eqref{eq: l4} we can replace $q_3-q_4$ by $p_3-p_4$
making an error of $O\left(\mu^{1-\kappa}\right)$ error. Hence
\[\nabla_{\mathcal V}\ell_4=\dfrac{(p_3-p_4)\cdot \nabla_{\mathcal V}(q_3-q_4)}{(p_3-p_4)\cdot\frac{\partial q_4}{\partial \ell_4}}+O(\mu^{1-\kappa}).\]
Note that
$\frac{\partial q_4}{\partial \ell_4}$ is parallel to $p_4.$ Using the information about $v_3$ and $v_4$ from Appendix \ref{subsection: gerver}
we see that $\langle v_3, v_4\rangle\neq \langle v_4, v_4\rangle.$ Therefore
the denominator in \eqref{eq: l4} is bounded away from zero and so
\[\nabla_{\mathcal V}\ell_4=O(1,1,1,1;\mu^{1-\kappa},\mu^{1-\kappa},\mu^{1-\kappa},\mu^{1-\kappa};1,1).\]
We also need to make sure the second component $\nabla_{\ell_3}\ell_4 $ is not close to -1, so that $\mathrm{Id}- \mathcal{F} \otimes \nabla_{\mathcal V}\ell_4$ is invertible on the section $|q_3-q_4|=\mu^{\kappa}$ .
In fact, we have
$\nabla_{\ell_3} \ell_4\to -\frac{(p_3-p_4)\cdot p_3}{(p_3-p_4)\cdot p_4}$ as $\mu\to 0$. The fact that $\nabla_{\ell_3}\ell_4$ is not close to $-1$ is then verified using the information in Appendix \ref{subsection: gerver}.

For the derivative $d\tilde{\mathbb G}$, the fundamental solution is estimated as  Id$+c e_{2,1}+o(1)$ and the two boundary terms are both given by the estimates \eqref{eq: bd-2} and  Corollary \ref{Cor: ham} on the section $\{x^R_{4,\parallel}=-2\}$.  Now the statement of the lemma can be checked with this explicit information.

\end{proof}
This lemma shows that we can replace $\Glob$ by $\Loc^-\circ\Glob\circ \Loc^+$ and the estimate of Lemma \ref{Lm: glob} still holds. So in the following it is enough to study $\Loc^0$ instead of $\Loc$ to prove Lemma \ref{Lm: loc}.

\begin{proof}[Proof of Lemma \ref{Lm: loc}]
As before we use the formula~(\ref{eq: formald4}).
We need to consider the integration of the variational equations and also the boundary contribution.

{\it Recall that the subscripts $-$ and $+$ mean relative motion and center of mass motion respectively, and the superscripts $-$ and $+$ mean incoming and outgoing respectively. In the following, we are most interested in the relative motion, so we drop the \text{subscript} $-$ of $q_-,p_-,\mathcal L_-,G_-,g_-$ for simplicity without leading to confusion. Note that in the following discussion, the symbol $q_1$ has two different meanings, but it is always clear from context.  If $q_1$ and $q_2$ both appear in an equation, then they mean the horizontal and vertical components respectively of $q_-$. If $q_1$ appears in $\mathbf q$ or with $p_1$ without $q_2$, then it means $Q_1 - Q_2$, the position of body 1.
 }

{\bf Step 1, the Hamiltonian equations, the variational equations and the boundary contributions.}

It is convenient to use the variable $\mathcal L=L/\mu$. Lemma \ref{Lm: landau} says that $1/c<\mathcal L<c$ and $\mu/c\leq G\leq c\mu$ for some $c>1$ if the rotation angle $\al$ is bounded away from $0$ and $\pi$. We also have $g,\mathbf{q},\mathbf{p},p=O(1)$ and $q=O(\mu^\kappa)$.
From the Hamiltonian \eqref{eq: relcm}, we have $\dot\ell=-\frac{1}{2\mu \mathcal{L}^{3}}+O(\mu^{2\kappa-1})$ (see \eqref{EqDotl}). Using $\ell$ as the time variable we get from \eqref{eq: relcm} that the equations of motion take the following form
(recall that $\ell=O(\mu^{\kappa-1})$ due to \eqref{SectionEll}):
\begin{equation}\label{eq: Heqrel}
\begin{cases}
\frac{d\mathcal L}{d\ell}= -\mu^{-1}\frac{dt}{d\ell}\frac{\partial H}{\partial \ell}\sim \mu^{1+\kappa}\\
\frac{d G}{d\ell}=-\frac{dt}{d\ell}\frac{\partial H}{\partial g}\sim \mu^{1+2\kappa}\\
\frac{d g}{d\ell}=\frac{dt}{d\ell}\frac{\partial H}{\partial G}\sim \mu^{2\kappa}\\
\end{cases}
\begin{cases}
\frac{d q_+}{d\ell}=\frac{dt}{d\ell}\left(\frac{1+2\mu}{2}p_++\mu p_1\right)\sim\mu\\
\frac{d p_+}{d\ell}=\frac{dt}{d\ell}\left(\frac{2q_+}{|q_+|^3}+O(\mu^{2\kappa}+\frac{1}{\chi^2})\right)\sim \mu\\
\frac{d q_1}{d\ell}=\frac{dt}{d\ell}\left(\frac{1}{2}\mu p_1+\mu p_+\right) \sim \mu^2\\
\frac{d p_1}{d\ell}=\frac{dt}{d\ell}\left(\frac{1+2\mu q_1}{\mu|q_1|^3}+O(\frac{1}{\chi^3})\right)\sim \frac{1}{\chi^2},\\
\end{cases}
\end{equation}
and we have $\frac{d}{d\ell}({\bf q},{\bf p})=O(\mu)$. In the first three equations, the main contribution to $H$ comes from $|q|^2$ and $|q_+\cdot q|^2$, both of which are $O(\mu^{2\kappa})$. We have the estimate
$\left|\left(\frac{\partial}{\partial \mathcal L},\frac{\partial}{\partial \ell},\frac{\partial}{\partial G},\frac{\partial}{\partial g}\right) q\right|=O(\mu^\kappa, \mu, \mu^{\kappa-1},\mu^\kappa)$
using \eqref{eq: delaunayscattering} for $q=(q_1,q_2)$ up to a rotation by $g$.
 In fact, the $\frac{\partial}{\partial \ell}$ amounts to dividing by the scale of $\ell$, i.e. $\mu^{-1+\kappa}$. The derivatives $\frac{\partial}{\partial \mathcal L},\frac{\partial}{\partial g}$ do not change the order of magnitude. Finally since $G=O(\mu)$, the $\frac{\partial}{\partial G}$ amounts to dividing by $\mu.$

Next  we analyze the variational equations.
The same rules as those used to obtain \eqref{eq: Heqrel} apply here.
\begin{equation}
\dfrac{d}{d\ell}\left[\begin{array}{c}
\delta \mathcal L\\
\delta G\\
\delta g\\
\delta \mathbf{q}\\
\delta \mathbf{p}\\
\end{array}\right]=O\left(\begin{array}{ccccc}
\mu^{1+\kappa}& \mu^{\kappa}&\mu^{1+\kappa}&\mu^{1+\kappa}&0 \\
\mu^{1+2\kappa}&\mu^{2\kappa}&\mu^{1+2\kappa}&\mu^{1+2\kappa}&0\\
\mu^{2\kappa}&\mu^{2\kappa-1}&\mu^{2\kappa}& \mu^{2\kappa}&0\\
\mu&\mu^{2\kappa}&\mu^{2\kappa+1}&\mu^{2\kappa+1}&\mu\\
\mu&\mu^{2\kappa}&\mu^{2\kappa+1}&\mu&0\\
\end{array}\right)
\left[\begin{array}{c}
\delta \mathcal L\\
\delta G\\
\delta g\\
\delta \mathbf{q}\\
\delta \mathbf{p}\\
\end{array}\right].
\label{eq: varrel}
\end{equation}
We need to integrate this equation over time $\mu^{\kappa -1}$. 
Thus we compare  the solution to the variational equation with a constant linear
ODE of the form $X'=AX.$
Its solution has the form $X(\mu^{\kappa-1})=\sum_{n=0}^\infty \frac{(A\mu^{\kappa-1})^n}{n!}$.
We will show that
$(A\mu^{\kappa-1})^3\leq C_3((A\mu^{\kappa-1})+(A\mu^{\kappa-1})^2).$ Then we have
$$(A\mu^{\kappa-1})^n\leq C_n((A\mu^{\kappa-1})+(A\mu^{\kappa-1})^2), \quad C_n=C_3(1+C_3)^n.$$
Hence  $X(\mu^{\kappa-1})\leq \mathrm{Id}+C((A\mu^{\kappa-1})+(A\mu^{\kappa-1})^2)$. We next integrate the variational equations over time $O(\mu^{\kappa-1})$ to get the estimate of its fundamental solution  \begin{equation}
\Id_{11}+
O\left(\begin{array}{ccc|cc}
\mu^{2\kappa}& \mu^{2\kappa-1}&\mu^{2\kappa}&\mu^{2\kappa}&\mu^{3\kappa}\\
\mu^{3\kappa}&\mu^{3\kappa-1}&\mu^{3\kappa}&\mu^{3\kappa}&\mu^{4\kappa}\\
\mu^{3\kappa-1}&\mu^{3\kappa-2}&\mu^{3\kappa-1}&\mu^{3\kappa-1}&\mu^{4\kappa-1}\\
\hline
\mu^\kappa&\mu^{3\kappa-1}&\mu^{3\kappa}&\mu^{2\kappa}&\mu^{\kappa}\\
\mu^\kappa&\mu^{3\kappa-1}&\mu^{3\kappa}&\mu^{\kappa}&\mu^{2\kappa}\\
\end{array}\right).\label{eq: fundrel}
\end{equation}

Next, we compute the boundary contribution using the formula~(\ref{eq: formald4}).
In terms of the Delaunay variables inside the circle $|q|=\frac12\mu^{\kappa}$, we have
\begin{equation}\dfrac{\partial \ell}{\partial (\mathcal{L},G,g,\mathbf{q},\mathbf{p})}=-\left(\dfrac{\partial |q|}{\partial\ell}\right)^{-1}\dfrac{\partial |q|}{\partial (\mathcal{L},G,g,\mathbf{q},\mathbf{p})}=(O(\mu^{\kappa-1}),O(\mu^{\kappa-2}),0,0,0).\label{eq: boundaryrel}\end{equation}
Indeed, due to \eqref{eq: delaunayscattering}
we have
$\frac{\partial |q|}{\partial g}=0$, $\frac{\partial |q|}{\partial\ell}=O(\mu)$, $\frac{\partial |q|}{\partial \mathcal{L}}=O(\mu^{\kappa})$
and $\frac{\partial |q|}{\partial G}=O(\mu^{\kappa-1}).$ Combining this with \eqref{eq: Heqrel} we get
\begin{equation}
\label{EqBCLocRel}
\begin{aligned}
&\left(\dfrac{\partial }{\partial \ell}(\mathcal{L},G,g,\mathbf{q},\mathbf{p})\right)\otimes\dfrac{\partial \ell}{\partial (\mathcal{L},G,g,\mathbf{q},\mathbf{p})}\\
&=O(\mu^{1+\kappa},\mu^{1+2\kappa},\mu^{2\kappa},\mu,\mu)\otimes O(\mu^{\kappa-1},\mu^{\kappa-2},0,0,0).
\end{aligned}
\end{equation}

{\bf Step 2, the analysis of the relative motion part.}

The structure of $d\Loc_0$ comes mainly from the relative motion part, on which we now focus. We neglect the $\mathbf q,\mathbf{p}$ part and will study it in the last step.

{\it Substep 2.1, the strategy. }

Using \eqref{eq: formald4} we obtain the derivative matrix
\begin{equation}
\begin{aligned}
&\dfrac{\partial (\mathcal{L},G,g)^+}{\partial (\mathcal{L},G,g)^-}=
\left(\Id_3+O\left(\begin{array}{ccc}
\mu^{2\kappa}& \mu^{2\kappa-1}&0\\
\mu^{3\kappa}&\mu^{3\kappa-1}&0\\
\mu^{3\kappa-1}&\mu^{3\kappa-2}&0
\end{array}\right)\right)^{-1} \times \\
&\left(\Id_3+O\left(\begin{array}{ccc}
\mu^{2\kappa}& \mu^{2\kappa-1}&\mu^{2\kappa}\\
\mu^{3\kappa}&\mu^{3\kappa-1}&\mu^{3\kappa}\\
\mu^{3\kappa-1}&\mu^{3\kappa-2}&\mu^{3\kappa-1}
\end{array}\right)\right)
\left(\Id_3-O\left(\begin{array}{ccc}
\mu^{2\kappa}& \mu^{2\kappa-1}&0\\
\mu^{3\kappa}&\mu^{3\kappa-1}&0\\
\mu^{3\kappa-1}&\mu^{3\kappa-2}&0
\end{array}\right)\right)\\
&=\Id_3+O\left(\begin{array}{ccc}
\mu^{2\kappa}& \mu^{2\kappa-1}&\mu^{2\kappa}\\
\mu^{3\kappa}&\mu^{3\kappa-1}&\mu^{3\kappa}\\
\mu^{3\kappa-1}&\mu^{3\kappa-2}&\mu^{3\kappa-1}
\end{array}\right):=\Id_3+P.\label{eq: dervar}
\end{aligned}
\end{equation}
For the position variables $q$, we are only interested in the angle $\Theta:=\arctan\left(\frac{q_2}{q_1}\right)$ since the length $|(q_1,q_2)|=\frac12\mu^\kappa$ is fixed when restricted to the circle. We split the derivative matrix as follows:
\begin{equation}
\dfrac{\partial (\Theta,p)^+}{\partial (\Theta,p)^-}=\dfrac{\partial (\Theta,p)^+}{\partial (\mathcal{L},G,g)^+}\dfrac{\partial (\mathcal{L},G,g)^+}{\partial (\mathcal{L},G,g)^-}\dfrac{\partial (\mathcal{L},G,g)^-}{\partial (\Theta,p)^-}=
\label{eq: dersplitting3}
\end{equation}
$$ \dfrac{\partial (\Theta,p)^+}{\partial (\mathcal{L},G,g)^+}\dfrac{\partial (\mathcal{L},G,g)^-}{\partial (\Theta,p)^-}+
\dfrac{\partial (\Theta,p)^+}{\partial (\mathcal{L},G,g)^+} P \dfrac{\partial (\mathcal{L},G,g)^-}{\partial (\Theta,p)^-}=I+II. $$
In the following, we prove

{\bf Claim: }\begin{equation}\label{EqI&II}
I=\frac{1}{\mu}O(1)_{1\times 3}\otimes \frac{\partial G^-}{\partial (\Theta,p)^-}+O(1),\ II=\frac{1}{\mu}O(\mu^{3\kappa-1})_{1\times 3}\otimes \frac{\partial G^-}{\partial (\Theta,p)^-}+O(\mu^{3\kappa-1}).\end{equation}
We will give the expressions of $O(1)$ terms explicitly.

{\it Substep 2.2, the estimate of $I$ in the splitting \eqref{eq: dersplitting3}.}

Using equations \eqref{eq: delaunayscattering} and \eqref{EqDelMom}
we obtain
\begin{equation}
\label{LocRelCartDel}
\dfrac{\partial (\Theta,p)^+}{\partial (\mathcal{L},G,g)^+}=
O\left(\begin{array}{ccc}
1&\mu^{-1}&1\\
1&\mu^{-1}&1\\
1&\mu^{-1}&1
\end{array}\right).
\end{equation}
Next, we consider the first term in \eqref{eq: dersplitting3},
\begin{equation}
I =\dfrac{\partial (\Theta,p)^+}{\partial \mathcal{L}^{+}}\otimes \dfrac{\partial \mathcal{L}^{-}}{\partial (\Theta,p)^-}
+\dfrac{\partial (\Theta,p)^+}{\partial G^{+}}\otimes \dfrac{\partial G^{-}}{\partial (\Theta,p)^-}+\dfrac{\partial (\Theta,p)^+}{\partial g^{+}}\otimes \dfrac{\partial g^{-}}{\partial (\Theta,p)^-}.
\label{eq: dersplitingid}
\end{equation}
Using the expressions
$\frac{1}{4\mathcal{L}^{2}}=\frac{p^2}{4}-\frac{\mu}{2|q|},\  G=p\times q=|p|\cdot|q|\sin\measuredangle (p,q)$,
we see that
\begin{equation}
\label{RelLocDelCart}
\dfrac{\partial \mathcal{L}^{-}}{\partial (\Theta,p)^-}=O(1),\quad \dfrac{\partial G^{-}}{\partial (\Theta,p)^-}=(O(\mu^\kappa),O(\mu^\kappa)).
\end{equation}
It only remains to get the estimate of $\frac{\partial g^{-}}{\partial (\Theta,p)^-}.$ We claim that
 \begin{equation}
\dfrac{\partial g^-}{\partial (\Theta,p)^-}=
\left[\frac{\partial}{\partial G^-}\arctan\left(\frac{G^-}{\mu \mathcal{L}}\right)\right]
\dfrac{\partial G^-}{\partial (\Theta,p)^-}+O(1)=O(1/\mu)\dfrac{\partial G^-}{\partial (\Theta,p)^-}+O(1).\label{eq: derg}
\end{equation}
We use equation \eqref{Eqpg} to get
\[g=\arctan\left(\dfrac{p_2}{p_1}-e^{-2|u|} E(G/(\mu\mathcal L), g, u)\right)-\mathrm{sign}(u)\arctan\dfrac{G}{\mu \mathcal{L}}\ \mathrm{as}\ |u|\to\infty.\]
We have $e^{-2|u|}\sim (1+(G/\mu\cL)^2)/\ell^2\sim (1+(G/\mu\cL)^2)\mu^{2(1-\kappa)}\mathcal L^4$ using \eqref{eq: delaunayscattering} and \eqref{eq: hypul}, and $E(\cdot,\cdot,\cdot)$ has $O(1)$ derivatives as $|u|\to \infty$.
We choose the $\mathrm{sign}(u)=-$ for the incoming orbit parameters, thus we get
\begin{equation}\nonumber\begin{aligned}
\dfrac{\partial g}{\partial (\Theta,p)}\left(1+O(e^{-2|u|})\right)&=  \dfrac{\partial \arctan\frac{p_2}{p_1}}{\partial (\Theta,p)}+\left(\dfrac{\partial\arctan\frac{G}{\mu \mathcal{L}}}{\partial \mathcal{L}}+O(e^{-2|u|})\right)\dfrac{\partial \mathcal{L}}{\partial (\Theta,p)}\\
&+\left(\frac{\partial\arctan\frac{G}{\mu \mathcal{L}}}{\partial G}+O(e^{-2|u|}/\mu)\right)\dfrac{\partial G}{\partial (\Theta,p)} + O(e^{-2|u|}) \end{aligned}\end{equation}
proving \eqref{eq: derg}.

Plugging \eqref{LocRelCartDel}, \eqref{RelLocDelCart} and \eqref{eq: derg} back into \eqref{eq: dersplitingid} we get the estimate of $I$ in \eqref{EqI&II}. More explicitly,
$I=\frac{1}{\mu}\mathbf U\otimes \frac{\partial G^-}{\partial (\Theta,p)^-}+\mathbf B$, where
\begin{equation}
\begin{aligned}
\mathbf U=&\mu \dfrac{\partial (\Theta,p)^+}{\partial G^{+}}+\mu\dfrac{\partial\arctan\frac{G^-}{\mu \mathcal{L}^{-}}}{\partial G^-}\dfrac{\partial (\Theta,p)^+}{\partial g^{+}}+O(\mu^{1-\kappa})\\
\mathbf B=&\dfrac{\partial (\Theta,p)^+}{\partial \mathcal{L}^{+}}\otimes\dfrac{\partial \mathcal{L}^{-}}{\partial (\Theta,p)^-}+\\
&\dfrac{\partial (\Theta,p)^+}{\partial g^{+}}\otimes \left(\dfrac{\partial \arctan\frac{p_2^-}{p^-_1}}{\partial (\Theta,p)^-}+\dfrac{\partial\arctan\frac{G^-}{\mu \mathcal{L}^{-}}}{\partial \mathcal{L}^{-}}\dfrac{\partial \mathcal{L}^{-}}{\partial (\Theta,p)^-}\right)+O(\mu^{1-\kappa}).
\end{aligned}\label{eq: local2bp}
\end{equation}
{\it Substep 2.3, the estimate of $II$ in the splitting \eqref{eq: dersplitting3}.}

Now we study the second term in \eqref{eq: dersplitting3}
\begin{equation}
\begin{aligned}
II&=O\left(\begin{array}{ccc}
1&\mu^{-1}&1\\
1&\mu^{-1}&1\\
1&\mu^{-1}&1
\end{array}\right)\cdot O
\left(\begin{array}{ccc}
\mu^{2\kappa}& \mu^{2\kappa-1}&\mu^{2\kappa}\\
\mu^{3\kappa}&\mu^{3\kappa-1}&\mu^{3\kappa}\\
\mu^{3\kappa-1}&\mu^{3\kappa-2}&\mu^{3\kappa-1}
\end{array}\right)
\dfrac{\partial (\mathcal{L},G,g)^-}{\partial (\Theta,p)^-}\\
&=O\left(\begin{array}{ccc}
\mu^{3\kappa-1}& \mu^{3\kappa-2}&\mu^{3\kappa-1}\\
\mu^{3\kappa-1}&\mu^{3\kappa-2}&\mu^{3\kappa-1}\\
\mu^{3\kappa-1}&\mu^{3\kappa-2}&\mu^{3\kappa-1}
\end{array}\right)
\dfrac{\partial (\mathcal{L},G,g)^-}{\partial (\Theta,p)^-}\\
&=\mu^{3\kappa-1}\left[O(1)_{1\times 3}\otimes \dfrac{\partial \mathcal{L}^{-}}{\partial (\Theta,p)^-} +O(\mu^{-1})_{1\times 3}\otimes \dfrac{\partial G^{-}}{\partial (\Theta,p)^-}
+O(1)_{1\times 3}\otimes \dfrac{\partial g^{-}}{\partial (\Theta,p)^-}\right]
\end{aligned}\label{eq: addpert}
\end{equation}
where we use that $\mu^{2\kappa}<\mu^{3\kappa-1}$ and $\mu^{2\kappa-1}<\mu^{3\kappa-2}$,
since $\kappa<1/2$.
The first summand in \eqref{eq: addpert}
is $O(\mu^{3\kappa-1})$. Applying \eqref{eq: derg}, we get the estimate of $II$ in \eqref{EqI&II}.
We then obtain
$$I+II=\frac{1}{\mu}(\mathbf U+O(\mu^{3\kappa-1}))\otimes \frac{\partial G^-}{\partial (\Theta,p)^-}+\mathbf B+O(\mu^{3\kappa-1}).$$
{\it Substep 2.4, going from $\Theta$ to $q$.}

We use the variable $\Theta$ for the relative position $q$ and we have $\frac{\partial G^-}{\partial (\Theta,p)^-}=O(\mu^\kappa)$. 
To obtain $\frac{\partial (q,p)^+}{\partial (q,p)^-}$, we use
$q=\frac12\mu^\kappa(\cos\Theta,\sin\Theta)=(x,y),\ \Theta=\arctan\frac{y}{x},\ |q| d\Theta=-\sin\Theta dx+\cos\Theta dy.$
So we have the estimate $\frac{\partial q^+}{\partial (\mathcal L,G,g)^+}=O(\mu^\kappa)\frac{\partial \Theta^+}{\partial (\mathcal L,G,g)^+}=O(\mu^{\kappa-1})$. To get $\frac{\partial -}{\partial q^-},$ we transform polar coordinates to Cartesian, $\frac{\partial -}{\partial q^-}=\frac{\partial -}{\partial (r,\Theta)^-}\frac{\partial(r,\Theta)^-}{\partial q^-}$, where $r=|q^-|=\frac12\mu^\kappa$.
Therefore we have $\frac{\partial r^-}{\partial q^-}=0,\ \frac{\partial -}{\partial q^-}=\frac{2}{\mu^\kappa}\frac{\partial -}{\partial \Theta^-}(-\sin\Theta^-,\cos\Theta^-).$ So we have the estimate $\frac{\partial G^-}{\partial q^-}=O(1)$, and $\frac{\partial \mathcal{L}^-}{\partial q^-}=\frac{\partial \mathcal{L}^-}{\partial \Theta^-}=0$ since in the expression $\frac{1}{4\mathcal{L}^{2}}=\frac{p^2}{4}-\frac{\mu}{2|q|}$, the
angle $\Theta$ plays no role. Finally, we have $\frac{\partial }{\partial q^-}\arctan \frac{p_2^-}{p^-_1}=0$. Applying these estimates to \eqref{eq: local2bp} and \eqref{eq: addpert}  we get \begin{equation}\label{EqQvpm}\dfrac{\partial (q,p)^+}{\partial (q,p)^-}=\dfrac{1}{\mu}(O(\mu^\kappa)_{1\times 2}, O(1)_{1\times 2})\otimes (O(1)_{1\times 2},O(\mu^\kappa)_{1\times 2})+O(1)_{4\times 4}.\end{equation}

{\bf Step 3, the contribution from the motion of the center of mass.}

{\it Substep 3.1, the decomposition.}

Consider the following decomposition 
\begin{equation}
\begin{aligned}
\cD:=&\dfrac{\partial(\Theta,p,\mathbf q,\mathbf{p})^+}{\partial(\Theta,p,\mathbf{q},\mathbf{p})^-}=\dfrac{\partial(\Theta,p;\mathbf{q},\mathbf{p})^+}{\partial(\mathcal{L},G,g;\mathbf{q},\mathbf{p})^+}\dfrac{\partial(\mathcal{L},G,g;\mathbf{q},\mathbf{p})^+}{\partial (\mathcal{L},G,g;\mathbf{q},\mathbf{p})(\ell^f) }\\
&\dfrac{\partial (\mathcal{L},G,g;\mathbf{q},\mathbf{p})(\ell^f)}{\partial (\mathcal{L},G,g;\mathbf{q},\mathbf{p})(\ell^i)} \dfrac{\partial (\mathcal{L},G,g;\mathbf{q},\mathbf{p})(\ell^i)}{\partial (\mathcal{L},G,g;\mathbf{q},\mathbf{p})^-}\dfrac{\partial (\mathcal{L},G,g;\mathbf{q},\mathbf{p})^-}{\partial (\Theta,p;\mathbf{q},\mathbf{p})^-}\\
:=&\left[\begin{array}{cc}
M&0\\
0&\Id_8
\end{array}\right]
\left[\begin{array}{cc}
A&0\\
B&\Id_8
\end{array}\right]
\left[\begin{array}{cc}
C&D\\
E&F
\end{array}\right]
\left[\begin{array}{cc}
A'&0\\
B'&\Id_8
\end{array}\right]
\left[\begin{array}{cc}
N&0\\
0&\Id_8
\end{array}\right]\\
=&\left[\begin{array}{cc}
MACA'N+MADB'N&MAD\\
(BC+E)A'N+(BD+F)B'N&BD+F
\end{array}\right].
\end{aligned}\label{eq: matrices}
\end{equation}
Each of the above matrices is $11\times 11$.

{\it Substep 3.2, the estimate of each block. }

The matrix $M=\frac{\partial(\Theta,p)^+}{\partial(\mathcal L,G,g)^+}$ is given by \eqref{LocRelCartDel} and
$N=\frac{\partial(\mathcal L,G,g)^-}{\partial(\Theta,p)^-}$ by  \eqref{RelLocDelCart}, \eqref{eq: derg}  $$M=O\left(\begin{array}{ccc}
1& \mu^{-1}&1\\
1& \mu^{-1}&1\\
1& \mu^{-1}&1
\end{array}\right),\quad
N=\left(\begin{array}{cccc}
O(1)_{1\times 3}\\
\frac{\partial G^-}{\partial (\Theta,p)^-}\\
O(\frac{1}{\mu})\frac{\partial G^-}{\partial (\Theta,p)^-}+O(1)\\
\end{array}\right).$$
$C,D,E,F$ form the matrix \eqref{eq: fundrel}, the fundamental solution of the variational equation,
$$\left(\begin{array}{c|c}
C&D\\
\hline
E&F\end{array}\right)=	
\Id_{11}+
O\left(\begin{array}{ccc|cc}
\mu^{2\kappa}& \mu^{2\kappa-1}&\mu^{2\kappa}&(\mu^{2\kappa})_{1\times 4}&(\mu^{3\kappa})_{1\times 4}\\
\mu^{3\kappa}&\mu^{3\kappa-1}&\mu^{3\kappa}&(\mu^{3\kappa})_{1\times 4}&(\mu^{4\kappa})_{1\times 4}\\
\mu^{3\kappa-1}&\mu^{3\kappa-2}&\mu^{3\kappa-1}&(\mu^{3\kappa-1})_{1\times 4}&(\mu^{4\kappa-1})_{1\times 4}\\
\hline
(\mu^\kappa)_{4\times 1}&(\mu^{3\kappa-1})_{4\times 1}&(\mu^{3\kappa})_{4\times 1}&(\mu^{2\kappa})_{4\times 4}&(\mu^{\kappa})_{4\times 4}\\
(\mu^\kappa)_{4\times 1}&(\mu^{3\kappa-1})_{4\times 1}&(\mu^{3\kappa})_{4\times 1}&(\mu^{\kappa})_{4\times 4}&(\mu^{2\kappa})_{4\times 4}\\
\end{array}\right).
$$
  $A,B,A',B'$ are given by \eqref{EqBCLocRel}, boundary contributions,
	$$\left[\begin{array}{c|c}
A&0\\
\hline
B&\Id_8
\end{array}\right],\left[\begin{array}{c|c}
A'&0\\
\hline
B'&\Id_8
  \end{array}\right]=\mathrm{Id}_{11}+
O(\mu^{1+\kappa},\mu^{1+2\kappa},\mu^{2\kappa};\mu_{1\times8 })
\otimes O(\mu^{\kappa-1},\mu^{\kappa-2},0;0_{1\times 8}).
$$

{\it Substep 3.3, the estimate of the first block $MACA'N+MADB'N$ in $\mathcal D$.}
By \eqref{eq: dervar}
$$ACA'=\Id_3+P=\Id_3+O\left(\begin{array}{ccc}
\mu^{2\kappa}& \mu^{2\kappa-1}&\mu^{2\kappa}\\
\mu^{3\kappa}&\mu^{3\kappa-1}&\mu^{3\kappa}\\
\mu^{3\kappa-1}&\mu^{3\kappa-2}&\mu^{3\kappa-1}
\end{array}\right)$$
(Recall that \eqref{eq: dervar} is the part of $\frac{\partial (\mathcal{L},G,g)^+}{\partial (\mathcal{L},G,g)^-}$
without considering the motion of the center of mass),
and by \eqref{EqI&II} and \eqref{eq: local2bp}
\begin{equation}\label{CD-Main}
MACA'N=M(\Id_3+P)N=\frac{1}{\mu}\left(\mathbf{U}+O\left(\mu^{3\kappa-1}\right)\right)
\otimes \frac{\partial G^-}{\partial (\Theta,p)^-}
+\mathbf B+O\left(\mu^{3\kappa-1}\right).
\end{equation} 
Indeed, using the notation of \eqref{eq: dersplitting3}, we have $I=MN$ and $II=MPN$.
The estimates of $I$ and $II$ are given in  \eqref{EqI&II}.

Next we claim that
\begin{equation}\label{CD-Error}
  MADB'N=
  O\left(\mu^{3\kappa-2}\right)\frac{\partial G^-}{\partial (\Theta,p)^-}+
O\left(\mu^{3\kappa-1}\right)
\end{equation}
so it can be absorbed into the error terms of \eqref{CD-Main}.
To this end we split $N=N_1+N_2,$ $A=\mathrm{Id}+A_2$ where $A_2=
O(\mu^{1+\kappa},\mu^{1+2\kappa},\mu^{2\kappa})
\otimes O(\mu^{\kappa-1},\mu^{\kappa-2},0).
$ and
$$
N_1=\left(
\begin{array}{c} 0_{1\times 3} \\
\frac{\partial G^-}{\partial (\Theta,p)^-}\\
O(\frac{1}{\mu})\frac{\partial G^-}{\partial (\Theta,p)^-}
\end{array}\right), \quad
N_2=\left(\begin{array}{c} O(1)_{1\times 3} \\ 0_{1\times 3} \\
  O(1)_{1\times 3} \end{array}\right).$$Thus
$MADB'N=MDB'N+MA_2 DB'N.$ Let us work on the first term.
A direct computation shows that
$ DB'=O\left(\begin{array}{ccc} \mu^{3\kappa} & \mu^{3\kappa-1} & 0 \\
  \mu^{4\kappa} & \mu^{4\kappa-1} & 0 \\
  \mu^{4\kappa-1} & \mu^{4\kappa-2} & 0 \end{array} \right),
\quad MDB'=O\left(\mu^{4\kappa-1}_{3\times 1}, \mu^{4\kappa-2}_{3\times 1}, 0_{3\times 1}\right). $
Now it is easy to see that $MDB' N_1$ can be absorbed into the first term in \eqref{CD-Error}
and $MDB' N_2$ can be absorbed into the second term. The key is that $N_1$ has rank one and
the second row of $N_2$ is zero. The analysis of $MA_2DB'N$ is even easier since a direct
computation shows that $DB'$ dominates
$A_2 DB'$ componentwise.
This proves \eqref{CD-Error} and shows that
$MACA'N+MADB'N$ has the same asymptotics as \eqref{CD-Main}.


{\it Substep 3.4, estimate of the remaining blocks in $\mathcal D$.}

The following estimates are obtained by a direct computation
\begin{equation}\nonumber
\begin{aligned}
BD+F&=
O(\mu_{1\times 8})\otimes O(\mu^{\kappa-1},\mu^{\kappa-2},0)O\left(\begin{array}{cc}
(\mu^{2\kappa})_{1\times 4}&(\mu^{3\kappa})_{1\times 4}\\
(\mu^{3\kappa})_{1\times 4}&(\mu^{4\kappa})_{1\times 4}\\
(\mu^{3\kappa-1})_{1\times 4}&(\mu^{4\kappa-1})_{1\times 4}\\
\end{array}\right)
+\Id_8\\
&
+O\left(\begin{array}{cc}
(\mu^{2\kappa})_{4\times 4}&(\mu^{\kappa})_{4\times 4}\\
(\mu^{\kappa})_{4\times 4}&(\mu^{2\kappa})_{4\times 4}\\
\end{array}\right)=
\Id_8+O\left(\mu^{\kappa}\right)_{8\times 8}.\\
BC+E&=
O(\mu_{1\times 8})\otimes O(\mu^{\kappa-1},\mu^{\kappa-2},0)O\left(\begin{array}{ccc}
\mu^{2\kappa}& \mu^{2\kappa-1}&\mu^{2\kappa}\\
\mu^{3\kappa}&\mu^{3\kappa-1}&\mu^{3\kappa}\\
\mu^{3\kappa-1}&\mu^{3\kappa-2}&\mu^{3\kappa-1}
\end{array}\right)\\
&+\left(
(\mu^\kappa)_{8\times 1},(\mu^{3\kappa-1})_{8\times 1},(\mu^{3\kappa})_{8\times 1}
\right)=
O\left(
(\mu^{\kappa})_{8\times 1}, (\mu^{4\kappa-2})_{8\times 1},(\mu^{4\kappa-1})_{8\times 1}
\right).
\end{aligned}\end{equation}
Accordingly using \eqref{RelLocDelCart} and \eqref{eq: derg} for $N$, and arguing the same way
as substep 3.3 we get
\begin{equation}\label{eq: dercm}(BC+E)A'N+(BD+F)B'N=\dfrac{1}{\mu}[O(\mu^{\kappa})]_{1\times 8}\otimes \dfrac{\partial G^-}{\partial(\Theta,p)_-^-}+O(\mu^\kappa),\quad MAD=[O(\mu^{3\kappa-1})]_{3\times 8}.
 \end{equation}

{\it Substep 3.5, completing the asymptotics of $\cD$.}

Substeps 3.1--3.4 above can be summarized as follows:
\begin{equation}\label{eq: derloc}
\cD=\dfrac{1}{\mu}(\mathbf{U}+O(\mu^{3\kappa-1}); O(\mu^{\kappa})_{1\times 8})\otimes \left(\dfrac{\partial G^-}{\partial(\Theta,p)_-^-};0_{1\times 8}\right)+\left(\begin{array}{c|c}
\mathbf B&0\\
\hline
0&\Id_8
\end{array}\right)+O\left(\mu^{3\kappa-1}\right). \end{equation}
Finally, when we use the coordinates $(q_-,p_-)$ instead of $(\Theta_-,p_-)$ as we did in Substep 2.4, 
it follows from \eqref{EqQvpm} that we get $$\frac{\partial(q_-,p_-,\mathbf{q},\mathbf{p})^+}{\partial(q_-,p_-,\mathbf{q},\mathbf{p})^-}=\frac{1}{\mu}O(\mu^{\kappa}_{1\times 2}, 1_{1\times 2}; \mu^{\kappa}_{1\times 8})\otimes O\left( 1_{1\times 2}
,\mu^{\kappa}_{1\times 2};0_{1\times 8}\right)+O(1).$$
This is the structure of $d\Loc$ stated in the lemma. \end{proof}

It remains to obtain explicit asymptotics of the leading terms in Lemma \ref{Lm: loc}.
Below we use the Delaunay variables $(L_3,\ell_3,G_3,g_3; x_1,v_1; G_4,g_4)^\pm$ as the orbit parameters \textit{outside} the circle $|q_-|=2\mu^\kappa$ and add a subscript $in$ to the Delaunay variables \textit{inside} the circle. We relate $C^0$ estimates of Lemma~\ref{Lm: landau} to the $C^1$
estimates obtained above. Namely consider the following equation which is obtained by
discarding the $o(1)$ errors in~\eqref{eq: landau}
\begin{equation}
\label{LandauLead}
q_-^+=0,\ p_-^+=R(\al)p_-^-,\quad \mathbf{q}^+=\mathbf{q}^-,\ \mathbf{p}^+=\mathbf{p}^-,
\end{equation}
where $\al$ is given by \eqref{Alpha}. We have the following corollary saying that $d\Loc$ can be obtained by taking the derivative directly in \eqref{LandauLead}.
\begin{Cor}
\label{Cor} The vectors $\hat u_j,\ \hat\lin_j$ in Lemma \ref{Lm: loc} can be computed directly from \eqref{eq: landau} evaluated at the $j$-th Gerver collision point $j=1,2$ as follows:
\begin{equation}
\begin{aligned}
 \hat u_j&=\frac{\partial \mathcal V^+}{\partial\al}h=\frac{\partial \mathcal V^+}{\partial \mathcal X^+}\frac{\partial \mathcal X^+}
{\partial (q_-,p_-,\mathbf{q},\mathbf{p})^+} \frac{\partial (q_-,p_-,\mathbf{q},\mathbf{p})^+}{\partial \al}h, \\
\hat\lin_j&=\frac{\partial G_{in}}{\partial \mathcal V^-}=\frac{\partial G_{in}}{\partial (q_-,p_-,\mathbf{q},\mathbf{p})^-}\frac{\partial (q_-,p_-,\mathbf{q},\mathbf{p})^-}{\partial \mathcal X^-}\frac{\partial \mathcal X^-}{\partial \mathcal V^-},
\end{aligned}\label{Equl}
\end{equation}
where $h = \lim_{\mu \to 0} \mu \frac{\partial\alpha}{\partial G_{in}} = 2|v_3 - v_4|\sin^2(\alpha/2)$.  Here $v_3$ and $v_4$ are the velocities of bodies 3 and 4 at the j-th Gerver collision point (see Appendix B), where $|v_3 - v_4|$ is the same before and after the elastic collision.
See Notation \ref{NotVX} for the use of $\mathcal V$ and $\mathcal X$.
\end{Cor}

\begin{proof}
We begin by computing the rank 1 terms in the expression for $\cD.$
To get \eqref{Equl} we need to multiply the vector by
$\frac{\partial (L_3,\ell_3,G_3,g_3; q_1,p_1;G_4,g_4)^+}{\partial (q_3,p_3;q_1,p_1;q_4,p_4)^+}\frac{\partial (q_3,p_3;q_1,p_1;q_4,p_4)^+}
{\partial (q_-,p_-,\mathbf{q},\mathbf{p})^+}$ 
and the linear functional by $\frac{\partial (q_-,p_-,\mathbf{q},\mathbf{p})^-}{\partial (q_3,p_3;q_1,p_1;q_4,p_4)^-}\frac{\partial (q_3,p_3;q_1,p_1;q_4,p_4)^-}{\partial (L_3,\ell_3,G_3,g_3;q_1,p_1; G_4,g_4)^-}.$

For the map \eqref{LandauLead} we have
\[\dfrac{\partial (\mathbf{q},\mathbf{p})^+}{\partial (\mathbf{q},\mathbf{p})^-}=\Id_8,\ \dfrac{\partial (\mathbf{q},\mathbf{p})^+}{\partial (q_-,p_-)^-}=\dfrac{\partial (q_-,p_-)^+}{\partial (\mathbf{q},\mathbf{p})^-}=0, \ \dfrac{\partial(\mathbf{q},\mathbf{p})^+}{\partial \al}=\dfrac{\partial G_{in}}{\partial (\mathbf{q},\mathbf{p})^-}=0\]
which agrees with the corresponding blocks in \eqref{eq: derloc} up to an $o(1)$ error as $\mu\to 0$.

It remains to compare $\frac{\partial (q_-,p_-)^+}{\partial (q_-,p_-)^-}$.  Now the expression for $\lin_j$ follows from \eqref{CD-Main}.
Differentiating \eqref{LandauLead} we get
$\frac{\partial(q_-,p_-)^+}{\partial \al}=\left(0,\frac{\partial p_-^+}{\partial\al}\right).$
Thus to get the expression of $\hat{\mathbf u}$ in \eqref{Equl}, it is enough  to show
(cf.  \eqref{eq: local2bp}) that for the map \eqref{LandauLead} we have
\begin{equation}\label{EqCompareal}
\dfrac{\partial p_-^+}{\partial\al}\left(\dfrac{\partial \al}{\partial G_{in}}\right)=\left(\dfrac{\partial p_-^+}{\partial G^{+}}+\dfrac{\partial\arctan\frac{G^-}{\mu \mathcal{L}^{-}}}{\partial G^-}\dfrac{\partial p_-^+}{\partial g^{+}}\right),\quad G_{in}=G^-.
\end{equation}
Write $p_-^+=\bbV(G^+, \mu\cL, g^+)$ where $G^+$ and $g^+$ depend on $G^-$ as follows.
First, $G^+=G^-.$ Second, \eqref{eq: delaunay4}
gives
$ \arctan\left(\frac{p_2^\pm}{p_1^\pm}\right)\sim g^\pm-\arctan\left(\frac{G^\pm}{\mu\cL}\right), \quad \arctan\left(\frac{p_2^+}{p_1^+}\right)\sim \arctan\left(\frac{p_2^-}{p_1^-}\right)+\alpha,
$
where $\sim$ means that the difference between the LHS and the RHS is $O\left(e^{-2u}\right).$
Thus $g^+\sim g^-+\alpha$ and so
$ \frac{\partial p_-^+}{\partial G^-}=\frac{\partial \bbV}{\partial G^+}+
\frac{\partial \bbV}{\partial g^+} \frac{\partial g^+}{\partial G^-}\sim
\frac{\partial \bbV}{\partial G^+}+
\frac{\partial \bbV}{\partial g^+} \frac{\partial \alpha}{\partial G^-}
$
proving \eqref{EqCompareal}.


\end{proof}
The next corollary says that the small remainders in \eqref{eq: landau} are also $C^1$ small if the derivative is taken along the direction with small change of $G_{in}^-$.
\begin{Cor}\label{cor: 2}
Let $\gm(s):(-\eps,\eps)\to\R^{10} $ be a $C^1$ curve such that $\Gamma=\gm'(0),\ \|\Gamma\|=1$ and
$\frac{d(G^-_{in}\circ\gamma)(0)}{d s}=d G_{in}^-\cdot \Gamma=O(\mu)$ then when taking the derivative with respect to $s$ in the equations
\[\begin{cases}
&|p_3^+|^2+|p_4^+|^2=|p_3^-|^2+|p_4^-|^2+o(1),\\
&\mathbf p^+=\mathbf p^-+o(1),\\
&\mathbf q^+=\mathbf q^-+o(1),
\end{cases}
\]
obtained from equation \eqref{eq: landau},
the error terms are also $o(1)$  as $\mu\to 0$ after taking the directional derivative along the direction $\Gamma$.
\end{Cor}
\begin{proof}
For the motion of the mass center, it follows from Corollary \ref{Cor} and \eqref{eq: derloc} that
$$\dfrac{\partial (\mathbf{q},\mathbf{p})^+}{\partial (q_-,p_-,\mathbf{q},\mathbf{p})^-}=\dfrac{1}{\mu}\dfrac{\partial (\mathbf{q},\mathbf{p})^+}{\partial\al}\otimes \tilde\lin+(0_{8\times 4},\Id_{8})+o(1),\quad \mathrm{
where}\ \tilde\lin=\frac{\partial \al}{\partial G_{in}^-}\frac{\partial G^-_{in}}{\partial (q_-,p_-,\mathbf{q},\mathbf{p})^-}.$$
Here $\frac{\partial \al}{\partial G_{in}^-}=O(1)$. We already obtained in equation \eqref{eq: derloc} that $\frac{\partial (\mathbf{q},\mathbf{p})^+}{\partial\al}=O(\mu^{\kappa})$, so
our assumption $\frac{d (G^-_{in}\circ\gamma)(0)}{d s}=d G^-_{in}\cdot \Gamma=O(\mu)$ implies that
\begin{equation}
\label{SmallLin}
\tilde \lin \frac{\partial (q_-,p_-,\mathbf q,\mathbf p)^-}{\partial \mathcal V^-}\cdot\Gamma=O(1)d G^-_{in}\cdot \Gamma=O(\mu)
\end{equation}
which suppresses the $1/\mu$ term. Here $\frac{\partial (q_-,p_-,\mathbf q,\mathbf p)^-}{\partial \mathcal V^-}=O(1)$ by Lemma \ref{Lm: dx/dDe}.
This proves the last two identities of the corollary.

To derive the first equation it is enough to show $\frac{d}{ds}(|p_-^+|^2-|p_-^-|^2)=o(1)$
since we already have the required estimate for the velocity of the center of mass.
We use the fact that the RHS \eqref{eq: hamrel2} is the same
in incoming and outgoing variables (superscripts $+$ and $-$ respectively) by the total energy conservation.
In \eqref{eq: hamrel2}, the terms involving only $\mathbf{q},\mathbf{p}$ are handled using the result of the previous paragraph. The term $-\frac{\mu}{|q_-|}$ vanishes when taking the derivative since $|q_-|=\frac12\mu^\kappa$ is constant.  All the remaining terms have
$q_-$ to the power 2 or higher. We have $\frac{d q_-^-}{d s}=O(1)$ due to $\|\Gamma\|=1$ and $\frac{d q_-^+}{d s}=O(1)$ due to \eqref{SmallLin}.
Therefore after taking the derivative with respect to $s$, any term involving $q_-$ is of order $O(\mu^\kappa)$. This completes the proof of the energy conservation part.
\end{proof}

\subsection{Proof of Lemma~\ref{LmNonDeg}(c)}\label{SSTrans}
In this section we work out the $O(1/\mu)$ term in the local map and prove Lemma~\ref{LmNonDeg}(c).

\begin{proof}
\textbf{Before collision, $\hat\lin=\nabla_{\mathcal V^-} G^-_{in}$.}

In this calculation, every variable should carry a superscript $-$, and we omit it for simplicity. To verify $\hat\lin_i\cdot w_{3-i}\neq 0$ and $\hat\lin_i\cdot \tw\neq 0$ for $i=1,2$ in Lemma~\ref{LmNonDeg}(c), noting that $\tilde w=(0,1,0_{1\times 8})$ and $w=(0_{1\times 8},*,*)$, it is enough to work out the three entries $\nabla_{\ell_3}{G_{in}},\nabla_{G_4}{G_{in}},\nabla_{g_4}{G_{in}}$ in $\hat\lin$.
According to Corollary \ref{Cor} we can differentiate the asymptotic expression of Lemma \ref{Lm: landau}.
We have
$\left(\nabla_{G_4}{G_{in}},\nabla_{g_4}{G_{in}}\right)=$
\[-(p_3-p_4)\times\left(\frac{\partial}{\partial G_4},\frac{\partial}{\partial g_4}\right)q_4-(p_3-p_4)\times\left(\dfrac{\partial q_4}{\partial \ell_4}\right)\cdot\left(\nabla_{G_4}{\ell_4},\nabla_{g_4}{\ell_4}\right)
+O(\mu^{\kappa}),\] where $O(\mu^{\kappa})$ comes from $\left(\left(\frac{\partial}{\partial G_4},\frac{\partial}{\partial g_4}\right)(p_3-p_4)\right)\times (q_3-q_4)$ and dominates $\frac{\partial q_4}{\partial L_4}\cdot\nabla_{G_4,g_4} L_4=O(\mu)$. Indeed, we have $\frac{\partial q_4}{\partial L_4}=O(1)$ and $\frac{\mu}{|\frac{x_3}{1+\mu}-x_4|}=\frac{\mu}{|q_3-q_4|}=\mu^{1-\kappa}$, and using \eqref{eq: hamRR} with $H=0,$ the leading contribution to $\nabla_{G_4,g_4} L_4$ is given by $\nabla_{G_4,g_4}\left(\frac{1}{|x_4|}-\frac{1}{\left|x_4+\frac{\mu x_3}{1+\mu}\right|}\right)=O(\mu)$ where we use Newton-Leibniz $\frac{1}{|x_4|}-\frac{1}{\left|x_4+\frac{\mu x_3}{1+\mu}\right|}=\mu\int_0^1 \frac{(x_4+t\frac{\mu x_3}{1+\mu})}{\left|x_4+t\frac{\mu x_3}{1+\mu}\right|^3}\cdot \frac{x_3}{1+
\mu}dt$ and Lemma \ref{Lm: dx/dDe} for $\nabla_{G_4,g_4}x_4=O(1)$.

We next eliminate $\ell_4$ using the relation $|q_3-q_4|=\mu^{\kappa}$:
\[\left(\nabla_{ G_4}\ell_4,\nabla_{ g_4}\ell_4\right)=-\left(\dfrac{\partial |q_3-q_4|}{\partial \ell_4}\right)^{-1}\left[\left(\dfrac{\partial |q_3-q_4|}{\partial G_4},\dfrac{\partial |q_3-q_4|}{\partial g_4}\right)+\dfrac{\partial |q_3-q_4|}{\partial L_4}\left(\nabla_{G_4,g_4}L_4\right)\right]\]
\[=-\dfrac{(q_3-q_4)\cdot\left(\frac{\partial q_4}{\partial G_4},\frac{\partial q_4}{\partial g_4}\right)}{(q_3-q_4)\cdot\frac{\partial q_4}{\partial \ell_4}}+O(\mu)=-\dfrac{(p_3-p_4)\cdot\left(\frac{\partial q_4}{\partial G_4},\frac{\partial q_4}{\partial g_4}\right)}{(p_3-p_4)\cdot\frac{\partial q_4}{\partial \ell_4}}+O(\mu^{1-\kappa}).\]
Here we replaced $q_3-q_4$ by $p_3-p_4$, using the fact that the two vectors form
an angle of order $O(\mu^{1-\kappa})$ by Lemma \ref{Lm: landau}(c).
Therefore
\[\left(\nabla_{G_4}{G_{in}},\nabla_{g_4}{G_{in}}\right)=-(p_3-p_4)\times\left(\dfrac{\partial }{\partial G_4},\dfrac{\partial }{\partial g_4}\right)q_4\]\[+(p_3-p_4)\times\dfrac{\partial q_4}{\partial \ell_4}\left(\dfrac{(p_3-p_4)\cdot\left(\frac{\partial q_4}{\partial G_4},\frac{\partial q_4}{\partial g_4}\right)}{(p_3-p_4)\cdot\frac{\partial q_4}{\partial \ell_4}}\right)
+O(\mu^{\kappa}+\mu^{1-2\kappa}).\]
Similarly, we get
\[\nabla_{ \ell_3} G_{in}=(p_3-p_4)\times\dfrac{\partial q_3}{\partial \ell_3}+(p_3-p_4)\times\dfrac{\partial q_4}{\partial \ell_4}\left(\dfrac{(p_3-p_4)\cdot\frac{\partial q_3}{\partial \ell_3}}{(p_3-p_4)\cdot\frac{\partial q_4}{\partial \ell_4}}\right)
+O(\mu^{\kappa}+\mu^{1-2\kappa}).\]
We use {\sc Mathematica}
to work out the three entries  and check directly that $\hat\lin_i\cdot w_{3-i}\neq 0$ and $\hat\lin_i\cdot \tw\neq 0$ for $i=1,2$ using Lemma \ref{Lm: glob}.

\textbf{After collision, $\hat{\mathbf{u}}=\frac{\partial \mathcal V^+}{\partial \al}$.}
In equation~(\ref{eq: landau}), we let $\mu\to 0$.
Applying the implicit function theorem to \eqref{eq: landau} with $\mu=0$ we obtain
\begin{equation}
\begin{aligned}
\dfrac{\partial}{\partial \alpha}(q_3^+,p^+_3;q_1^+,p^+_1;q_4^+,p_4^+)
&=\dfrac{1}{2}\left(0,0,R\left(\frac{\pi}{2}+\al\right)(p^-_3-p^-_4);0,0,0,0;0,0,-R\left(\frac{\pi}{2}+\al\right)(p^-_3-p^-_4)\right)^T\\
&=\dfrac{1}{2}\left(0,0,R\left(\frac{\pi}{2}\right)(p^+_3-p^+_4);0,0,0,0;0,0,-R\left(\frac{\pi}{2}\right)(p^+_3-p^+_4)\right)^T.
\end{aligned}\nonumber
\end{equation}
where $R(\pi/2+\al)=\frac{dR(\al)}{d\al}$ and $\nabla_{\mathcal V^+} \ell_4^+ $ is given
by \eqref{eq: l4}.
Again we use {\sc Mathematica} to work out $\frac{\partial \mathcal V^+}{\partial \al}$ and check directly that $\brlin_i\cdot \hat{\mathbf{u}}_i\neq 0$ for $i=1,2$ using Lemma \ref{Lm: glob}.

To obtain a symbolic sequence with any order of symbols $3,4$ as claimed in the main theorem, we notice that the only difference is that the outgoing relative velocity changes sign $(p_3^+-p_4^+)\to -(p_3^+-p_4^+)$. So we only need to send $\hat{\mathbf{u}}\to -\hat{\mathbf{u}}$.
\end{proof}

\subsection{Proof of Lemma~\ref{LmNonDeg}(a)\&(b)}\label{subsection: local3}
In this section, we prove Lemma \ref{LmNonDeg}(a)\&(b). Since we have already obtained $\lin$ and $\mathbf u$ in $d\Loc$ and $\brlin,\brrlin,\brv, \bar{\brv}$ in $d\Glob$, one way to prove Lemma \ref{LmNonDeg} is to work out the matrix $B$ explicitly using an argument similar to that in Corollary  \ref{Cor}. In that case, the current section is not necessary. However, in this section, we use a different approach, which simplifies the computation and has clear physical and geometrical meaning. We first abbreviate $d\Loc$ in Lemma \ref{Lm: loc}(a) as  $d\Loc(\bx)=\frac{1}{\mu} \bv_{j,\mu}\otimes \lin_{j,\mu} +B_{j,\mu}$ using the subscript $\mu$ to absorb the $\mu$-depending $o(1)$s. Similarly, we write $d\Glob=\chi^2 \brv_{j,\mu}\otimes \brlin_{j,\mu} +4\chi \brrv_{j,\mu}\otimes \brrlin_{j,\mu} +O(\mu\chi)$.
\begin{Lm} Consider $\bx_\mu\in U_j(\dt)$, j=1,2 and $|\bar\theta_4^+-\pi|<\tilde\theta$ as in Lemma \ref{Lm: loc}.
Suppose the vector $\tilde{\Gamma}_{j,\mu}\in span\{\brv_{3-j}, \brrv_{3-j}\}\subset T_{\bx_\mu} U_j(\dt)$ for $j=2$, and $\tilde{\Gamma}_{j,\mu}\in span\{d\cR\brv_{3-j}, d\cR\brrv_{3-j}\}\subset T_{\bx_\mu} U_j(\dt)$ for $j=1$, satisfies
$\brlin_{j}(d\Loc \tilde{\Gamma}_{j,\mu})=0$ and $\Vert\tilde{\Gamma}_{j,\mu}\Vert_\infty=1.$
Then we have
\begin{enumerate}
\item[(a)]$\lin_{j,\mu}(\tilde\Gamma_{j,\mu})=O(\mu)$ as $\mu\to 0$,
\item[(b)] the limits $\lim_{\mu\to 0}\tilde{\Gamma}_{j,\mu}$ and $\lim_{\mu\to 0} d\Loc \tilde{\Gamma}_{j,\mu} $ exist, and $\lim_{\mu\to 0}\tilde{\Gamma}_{j,\mu}$ is continuous in $\hat\bx$ $($see Lemma \ref{Lm: loc}$)$ and $\lim_{\mu\to 0} d\Loc \tilde{\Gamma}_{j,\mu} $ is continuous in $\hat\bx$ and $\bar\theta_4^+$,
\item[(c)]$\displaystyle\hat\brlin_j(\lim_{\dt,\tilde\theta\to0}\lim_{\mu\to 0} d\Loc \tilde{\Gamma}_{j,\mu})=0$.
\end{enumerate}
\label{Lm: limit}
\end{Lm}
\begin{proof}
For simplicity, we give the proof in the case of $j=2$ without needing the renormalization. The other case $j=1$ is completely analogous.
Denote
$\Gamma'_{2,\mu}=\lin_{2,\mu}(\brv_{1,\mu})\brrv_{1,\mu}-\lin_{2,\mu}(\brrv_{1,\mu})\brv_{1,\mu}\in Ker \lin_{2,\mu}$ and let $v_\mu$ be a vector in $\Span(\brv_{1}, \brrv_{1})$ such that
$v_\mu\to v$ as $\mu\to 0$ and $\lin_{2,\mu}(v_\mu)=1.$
Suppose that
$\tilde\Gamma_{2,\mu}=a_\mu v_\mu+b_\mu \Gamma'_{2,\mu}$
then
\begin{equation}
\label{GammaBasis}
d\Loc(\tilde\Gamma_{2,\mu})=\dfrac{a_\mu}{\mu} \lin_{2,\mu}(v_\mu) \mathbf u_{2,\mu}+a_\mu B_{2,\mu}(v_\mu)+b_\mu B_{2,\mu} \Gamma'_{2,\mu}.
\end{equation}
So $\brlin_2(d\Loc(\tilde\Gamma_{2,\mu}))=0$ implies that
\begin{equation}
\label{EqAmu}
a_\mu=-\mu \dfrac{b_\mu \brlin_2(B_{2,\mu}\Gamma'_{2,\mu})}{\lin_{2,\mu}(v_\mu) \brlin_2(\mathbf u_{2,\mu})+\mu\brlin _2B_{2,\mu}(v_\mu)}.
\end{equation}
The denominator is not zero since $\lin_{2,\mu}(v_\mu)=1$ and $\brlin_2(\mathbf u_{2,\mu})\neq 0$ using Lemma \ref{LmNonDeg}(c).
Therefore $a_\mu=O(\mu)$ and $b_\mu=O(1)$ using $\|\Gamma_{2,\mu}\|_\infty=1$.
Hence $\tilde\Gamma_{2,\mu}=b_\mu \Gamma'_{2,\mu}+O(\mu)$ and $\lin_{2,\mu}(\tilde\Gamma_{2,\mu})=O(\mu).$ The continuous dependence on variables in part (b) follows from part (a) of Lemma \ref{Lm: loc} and \ref{Lm: glob}. Now the remaining statements of
the lemma follow from equations \eqref{GammaBasis} and \eqref{EqAmu}.
\end{proof}

To check the nondegeneracy condition, it is enough to know the following.
\begin{Lm} Let $\bx_\mu\in U_j(\delta)$ and $|\bar\theta_4^+-\pi|<\tilde\theta\ll 1$ be as in Lemma \ref{Lm: loc}.
If we take the directional derivative at $\bx_\mu$ of the local map along a direction
$\Gamma_{j,\mu}\in span\{\brv_{3-j},\brrv_{3-j}\}\subset T_{\bx_\mu} U_j(\dt)$ for $j=2$ and $\Gamma_{j,\mu}\in span\{d\cR\brv_{3-j},d\cR\brrv_{3-j}\}\subset T_{\bx_\mu} U_j(\dt)$ for $j=1$, such that $\brlin_j\cdot (d\Loc \Gamma_{j,\mu})=0,\ j=1,2,$ then
$\lim_{1/\chi\ll\mu \to 0}\frac{\partial E_3^+}{\partial \Gamma_{j,\mu}}$
is a continuous function of both $\bx$ and $\bar\theta_4^+$, where $E_3^+$ $($respectively $\bar\theta_4^+$$)$ is the energy of $Q_3$ $($respectively the outgoing asymptote of $Q_4)$ after the close encounter with $Q_4$.
If we take further limits $\dt\to 0$ and $\tilde\theta\to 0$, we have
$$\lim_{\dt,\tilde\theta\to 0}\;\;\;\lim_{1/\chi\ll\mu \to 0}\dfrac{\partial E_3^+}{\partial \Gamma_{j,\mu}}\neq 0,\quad j=1,2.$$
\label{Lm: local3}
\end{Lm}
Now we can check the nondegeneracy condition.
\begin{proof}[Proof of Lemma~\ref{LmNonDeg}(a)\&(b).]
We prove (b1) and (b2). The proofs of (a1) and (a2) are similar and are left to the reader.
To check (b2), $de_4$, we differentiate $e_4=\sqrt{1+(G_4/L_4)^2}$ to get
$de_4=\frac{1}{e_4}\left(\frac{G_4}{L^2_4}dG_4-\frac{G^2_4}{L_4^3}dL_4\right).$ Thus Lemma \ref{Lm: glob} gives
$de_4 w=\frac{G_4}{L^2_4}\neq 0$ as claimed.

Next we check (b1) which is equivalent to the following condition
\begin{equation} \det\left(\begin{array}{rr} \hat\brlin_{2}(\hat{\mathbf{u}}_{2}) & \hat\brlin_{2}(\hat B_2\Gamma'_2)
) \cr
                                            \hat\brrlin_{2}(\hat{\mathbf{u}}_{2}) & \hat\brrlin_{2}(\hat B_2 \Gamma'_2)
                                           \end{array}\right)\neq 0. \label{eq: nondeg}\end{equation}
where $\Gamma'_2=\hat\lin_{2}(\tw)w_{1}-\hat\lin_{2}(w_{1})\tw.$ The vector $\Gamma'_2\neq 0$ due to Lemma \ref{LmNonDeg}(c).

Let $\Gamma_2$ be a vector satisfying $\hat\brlin_2\cdot (d\Loc\Gamma_2)=0$ and
chosen as follows.
$d\Loc \Gamma_2$ is a vector in $span\{\hat{\mathbf{u}}_2,\hat B_2\Gamma'_2\}$, so it
can be represented as $d\Loc \Gamma_2=b \hat{\mathbf{u}}_2+b'\hat B_2\Gamma'_2.$
Thus we can take $b=-\hat\brlin_2\cdot \hat B_2\Gamma'_2$ and $b'=\hat\brlin_2(\hat{\mathbf{u}}_2)$
to ensure that $d\Loc \Gamma_2\in Ker \hat\brlin_2$.
Note that we have $b'\neq 0$ by Lemma \ref{LmNonDeg}(c).
Hence
\[\det\left(\begin{array}{rr} \hat\brlin_2(\hat{\mathbf{u}}_2) & \hat\brlin_2(\hat B_2\Gamma'_2) \cr
                                           \hat \brrlin_2(\mathbf u_2) & \hat\brrlin_2(\hat B_2\Gamma'_2)
                                            \end{array}\right)=\dfrac{1}{b'}\det\left(\begin{array}{rr} \hat\brlin_2(\hat{\mathbf{u}}_2) & \hat\brlin_2(d\Loc \Gamma_2) \cr
                                            \hat\brrlin_2(\hat{\mathbf{u}}_2) & \hat\brrlin_2(d\Loc \Gamma_2)
                                            \end{array}\right)=\hat\brrlin_2(d\Loc \Gamma_2)
                                            \]
where the last equality holds since $\hat\brlin_2(d\Loc \Gamma_2)=0.$
By Lemma \ref{Lm: glob} $\hat\brrlin_i=(1,0_{1\times 9})$. Therefore
$\hat\brrlin_2(d\Loc \Gamma_2)=\frac{\partial E_3^+}{\partial \Gamma_2}$
and so (b2) follows from Lemma~\ref{Lm: local3}.
\end{proof}
It remains to prove Lemma~\ref{Lm: local3}. It is more convenient for us to work with polar coordinates.
We need the following quantities.
\begin{Def}
 $\psi$: polar angle, related to $u$ by $\tan\frac{\psi}{2}=\sqrt{\frac{1+e}{1-e}}\tan\frac{u}{2}$ for an ellipse. We choose the positive $y$ axis as the axis $\psi=0$. $E$: energy; $e:$~eccentricity;
$G$: angular momentum, $g$: argument of apapsis.
\end{Def}
Recall the formula $r=\frac{G^2}{1-e\cos\psi}$ for conic sections in which the periapsis lies on the axis $\psi=\pi$.
In our case we have as $1/\chi\ll\mu\to 0$
\begin{equation}
\label{PolarGen}
\begin{cases}&r_3^\pm=\dfrac{(G^\pm_3)^2}{1-e^\pm_3\sin(\psi_3^\pm+g^\pm_3)}+o(1),\\
&r_4^\pm=\dfrac{(G^\pm_4)^2}{1-e_4^\pm\sin(\psi_4^\pm-g^\pm_4)}+o(1).
\end{cases}
\end{equation}
\label{Lm: polar}



\begin{Lm}\label{Lm: polarsection}
Under the assumptions of Corollary \ref{cor: 2} we have
\[\dfrac{d r_3^+}{d s}=\dfrac{d r_4^+}{d s}+o(1),\quad \dfrac{d r_3^-}{d s}=\dfrac{d r_4^-}{d s}+o(1),
\quad\dfrac{d \psi_3^+}{d s}=\dfrac{d \psi_4^+}{d s}+o(1),\quad\dfrac{d \psi_3^-}{d s}=\dfrac{d \psi_4^-}{d s}+o(1).\]
Moreover in \eqref{PolarGen} the $o(1)$ terms are also $C^1$ small when taking the derivative with respect to $s$.
\end{Lm}
\begin{proof}
To prove the statement about \eqref{PolarGen}, we use the Hamiltonian \eqref{eq: hamRR}.
We have seen in the beginning of the proof of Lemma \ref{Lm: 2pieces} that $\frac{-\mu}{|q_3-q_4|}$ gives an $O(\mu^{1-2\kappa})$ perturbation to the variational equations.  This shows that the perturbation to the Kepler motion is $C^1$ small.

Next we consider the derivatives $\frac{d  r_{3,4}^\pm}{d s}$.
We consider first the case of ``$-$". From the condition
$|\vec r_3-\vec r_4|=\mu^\kappa$, for the Poincar\'{e} section we get
$(\vec r_3-\vec r_4)\cdot \frac{d }{d s}(\vec r_3-\vec r_4)=0,$
hence $(\vec r_3-\vec r_4)\perp \frac{d }{d s}(\vec r_3-\vec r_4)$.
We also know the angular momentum for the relative motion is
$G_{in}=(\dot{\vec r}_3-\dot{\vec r}_4)\times (\vec r_3-\vec r_4)=O(\mu), $
which implies $\dot{\vec r}_3-\dot{\vec r}_4$ is almost parallel to $\vec r_3-\vec r_4$ by Lemma \ref{Lm: landau}(c), hence $\frac{d }{d s}(\vec r_3-\vec r_4)$ is almost perpendicular to $\dot{\vec r}_3-\dot{\vec r}_4$.
The condition $\frac{d G^-_{in}}{d s}=O(\mu)$ reads
\[\left(\dfrac{d }{d s}(\dot{\vec r}_3-\dot{\vec r}_4)\right)\times (\vec r_3-\vec r_4)+(\dot{\vec r}_3-\dot{\vec r}_4)\times \left(\dfrac{d }{d s}(\vec r_3-\vec r_4)\right)=O(\mu).\]
Since the first term is $O(\mu^\kappa)$ due to our choice of the Poincar\'e section we 
see $(\dot{\vec r}_3-\dot{\vec r}_4)\times \left(\frac{d }{d s}(\vec r_3-\vec r_4)\right)=o(1). $
Since $\frac{d }{d s}(\vec r_3-\vec r_4)$ is almost perpendicular to $(\dot{\vec r}_3-\dot{\vec r}_4)$ by the analysis above,
we get $\frac{d }{d s}(\vec r_3-\vec r_4)=o(1)$.
Taking the radial and angular part of this vector identity and using that $r_4=r_3+o(1),$ $\psi_4=\psi_3+o(1)$ we get the ``$-$" part of
the lemma.

To repeat the above argument for ``+" variables, we first need to establish that $\frac{d G^+_{in}}{ds}=O(\mu).$ Indeed, using equations \eqref{eq: dervar} and \eqref{eq: matrices} we get
\begin{align*}
\dfrac{\partial G^+_{in}}{\partial s}&=\dfrac{\partial G^+_{in}}{\partial(\mathcal{L},G_{in},g,\mathbf{q},\mathbf{p})^-}\dfrac{\partial(\mathcal{L},G_{in},g,\mathbf{q},\mathbf{p})^-}{\partial s}\\
&=O(\mu^{3\kappa},1, \mu^{3\kappa},\mu^{3\kappa}_{1\times 4},\mu^{3\kappa}_{1\times 4})\cdot O(1,\mu,1,1_{1\times 4},1_{1\times 4})=O(\mu).
\end{align*}

It remains to show $\left(\frac{d }{d s}(\dot{\vec r}_3-\dot{\vec r}_4)\right)=O(1)$ in the $``+"$ case. Since we know it is true in the ``$-$" case, the ``+" case follows, because the directional derivative of the local map $d\Loc\Gamma$ is bounded due to our choice of
$\Gamma$.
\end{proof}
We are now ready to describe the computation of Lemma \ref{Lm: local3}.
We will use the following set of equations which follows
from \eqref{LandauLead}.
\begin{equation}\label{eq: polarcollision1}
E_3^++E_4^+=E_3^-+E_4^-, \end{equation}
\begin{equation}\label{eq: polarcollision2}
G_3^++G_4^+=G_3^-+G_4^-, \end{equation}
\begin{equation}\label{eq: polarcollision3}
\dfrac{e_3^+}{G_3^+}\cos(\psi_3^++g_3^+)+\dfrac{e_4^+}{G_4^+}\cos(\psi_4^--g_4^{-})=\dfrac{e_3^-}{G_3^-}\cos(\psi_3^-+g_3^-)+\dfrac{e_4^-}{G_4^-}\cos(\psi_4^--g_4^{-}),
\end{equation}
\begin{equation}\label{eq: polarcollision4}
\dfrac{(G^+_3)^2}{1-e^+_3\sin(\psi_3^++g^+_3)}=\dfrac{(G_3^-)^2}{1-e_3^-\sin(\psi_3^-+g^-_3)},
\end{equation}
\begin{equation}\label{eq: polarcollision6}
\dfrac{(G^+_3)^2}{1-e^+_3\sin(\psi_3^++g^+_3)}=\dfrac{(G^+_4)^2}{1-e_4^+\sin(\psi_4^+-g^{+}_4)},
\end{equation}
\begin{equation}\label{eq: polarcollision7}
\dfrac{(G_3^-)^2}{1-e_3^-\sin(\psi_3^-+g^-_3)}=\dfrac{(G_4^-)^2}{1-e_4^-\sin(\psi_4^--g_4^{-})}, \end{equation}
\begin{equation}\label{eq: polarcollision9}
\psi^+_3=\psi^-_3,\quad \psi_4^-=\psi_3^-,\quad \psi_4^+=\psi_3^+,\quad x_1^+=x_1^-,\quad v_1^+=v_1^-.
\end{equation}
In the above equations we have dropped $o(1)$ terms for brevity.
We would like to emphasize that the above approximations hold not only in the $C^0$ sense but also
in the $C^1$ sense when we take the derivatives along directions satisfying the conditions of Corollary \ref{cor: 2}.
\eqref{eq: polarcollision1} is the approximate conservation of energy, \eqref{eq: polarcollision2} is the approximate
conservation of angular momentum and \eqref{eq: polarcollision3} follows from the approximate conservation of momentum as follows.
Represent the position vector as $\vec{r}= r \hat{e}_r$. Then the velocity is $\dot{\vec{r}}=\dot{r}\hat{e}_r+ r\dot{\psi} \hat{e}_\psi.$  Conservation of momentum gives
$(\dot{\vec{r}}_3)^-+(\dot{\vec{r}}_4)^-=(\dot{\vec{r}}_3)^++(\dot{\vec{r}}_4)^+.$
Taking the radial component and using the polar representation of the ellipse
$ r=\frac{G^2}{1-e\sin(\psi+g)},$ we get
\[\dot{r}=\frac{G^2 }{(1-e\sin(\psi+g))^2} e\cos (\psi+g) \dot{\psi}=\dfrac{r^2}{G^2} e\cos(\psi+g) \dfrac{G}{r^2}=
\dfrac{e}{G}\cos(\psi+g).\]

The possibility of differentiating these equations
is justified in Corollary \ref{cor: 2}. The remaining equations reflect the fact
that $Q_3^\pm$ and $Q_4^\pm$ are all close to each other. The possibility of differentiating these equations
is justified by Lemma \ref{Lm: polarsection}.
We set the total energy to be zero. So we get $E_4^\pm=-E_3^\pm$. This eliminates $E_4^\pm$. Then we also eliminate
$\psi_{4}^\pm$ by setting them equal to $\psi_3^\pm$.

\begin{proof}[Proof of the Lemma~\ref{Lm: local3}]
Lemma \ref{Lm: limit} and Corollary \ref{Cor} show that the assumption of Lemma \ref{Lm: local3} implies that the direction $\Gamma$ along which we take the directional derivative satisfies $\frac{\partial G_{in}}{\partial \Gamma}=O(\mu)$. So we can directly take derivatives in equations \eqref{eq: polarcollision1}-\eqref{eq: polarcollision7}.
Recall that we need to compute $dE_3^+(d\Loc \Gamma)$ where
$\Gamma\in Ker\mathbf l_j\cap$span$\{w_{3-j},\tilde w\}$.  Lemma \ref{Lm: glob} tells us that
in Delaunay coordinates we have
\begin{equation}
\label{wtw}
\tw=(0,1,0_{1\times 8}), \quad  w=(0_{1\times 8},1,a) \text{ where }a=\frac{-L_4^{-}}{(L^{-}_4)^2+ (G_4^{-})^2} .
\end{equation}
The formula $\tan\frac{\psi}{2}=\sqrt{\frac{1+e}{1-e}}\tan\frac{u}{2}$ which relates $\psi$ to $\ell$ through $u$
shows that \eqref{wtw} also holds if we use $(L_3, \psi_3, G_3, g_3;x_1,v_1; G_4, g_4)$ as coordinates.
Hence $\Gamma$ has the form $(0,1,0_{1\times 6},c,ca)$.
To find  the constant $c$ we use \eqref{eq: polarcollision7}.


Using \eqref{eq: polarcollision9}, we can replace $\psi_3^+=\psi_3^-=\psi_4^+=\psi_4^-$ by $\psi$, and get rid of $x_1$ and $v_1$.
Let $\bL$ denote the projection
of $\mathbb L$ to the variables $(E_3, G_3, g_3, G_4, g_4)$ and let $\bf \Gamma$ denote the projection of $\Gamma$ to the variables $(E_3,\psi, G_3, g_3, G_4, g_4)$. Thus we need to find
$dE_3^+(d\bL\bf \Gamma).$
To this end we write the remaining equations (\eqref{eq: polarcollision2}, \eqref{eq: polarcollision3}, \eqref{eq: polarcollision4}, and
\eqref{eq: polarcollision6})
formally as $\mathbf{F}(Z^+,Z^-)=0$, where $Z^+=(E_3^+,G_3^+,g_3^+,G_4^+,g_4^+)$ and $Z^-=(E_3^-,\psi,G_3^-,g_3^-,G_4^-,g_4^-)$. We have
$\dfrac{\partial\mathbf F}{\partial Z^+}
d\bL\mathbf{\Gamma}+\dfrac{\partial\mathbf F}{\partial Z^-}\mathbf{\Gamma}=0.$ However, $\frac{\partial\mathbf F}{\partial Z^+}$ is not invertible since $\mathbf F$ involves
only four equations while $Z^+$ has five variables.
To resolve this problem we notice that by definition of $\Gamma$ we have
  $\bar{\mathbf l}\cdot d\bL \bf\Gamma=0$, where $\bar{\mathbf l}=\left(\frac{ G_4^+/L_4^{+}}{(L_4^{+})^2+(G_4^{+})^2},0,0, \frac{-1}{(L_4^{+})^2+(G_4^{+})^2},\frac{1}{L_4^{+}}\right)$ by Lemma \ref{Lm: glob}.
Thus we get
$$\left[\begin{array}{c}\bar{\mathbf l}\\
\dfrac{\partial\mathbf F}{\partial Z^+}
\end{array}\right] d\bL\mathbf{\Gamma}
=-\left[\begin{array}{c}0\\
\dfrac{\partial\mathbf F}{\partial Z^-}\mathbf{\Gamma}
\end{array}\right],\quad d\bL\mathbf{\Gamma}
=-\left[\begin{array}{c}\bar{\mathbf l}\\
\dfrac{\partial\mathbf F}{\partial Z^+}
\end{array}\right]^{-1}\left[\begin{array}{c}0\\\dfrac{\partial\mathbf F}{\partial Z^-}\mathbf{\Gamma}
\end{array}\right].$$
We only need to show that the entry $dE_3^+ d\bL\bf\Gamma$ is nonvanishing to prove Lemma~\ref{Lm: local3}.
It turns out this number is 0.376322 for the first collision and -1.86463 for the second collision. \end{proof}

\appendix

\section{Delaunay coordinates}\label{appendixa}
\subsection{Elliptic motion}\label{subsection: ellip}
The material of this section could be found in \cite{F,W}.
Consider the two-body problem with Hamiltonian
$H(P,Q)=\frac{|P|^2}{2m}-\frac{k}{|Q|},\quad (P,Q)\in \R^4.$
This system is integrable in the Liouville-Arnold sense when $H<0$.
So we can introduce the action-angle variables $(L,\ell, G,g)$
in which the Hamiltonian can be written as
\[H(L,\ell, G,g)=-\frac{mk^2}{2L^2},\quad (L,\ell, G,g)\in T^*\T^2.\]
The Hamiltonian equations are $\dot{L}=\dot{G}=\dot{g}=0,\quad \dot{\ell}=\frac{mk^2}{L^3}.$
We introduce the following notations: $E$-energy, $M$-angular momentum, $e$-eccentricity, $a$-semimajor axis, $b$-semiminor axis.

Then we have the following relations which explain the physical and geometrical meaning of the Delaunay coordinates.
\[a=\frac{L^2}{mk}, \ b=\frac{LG}{mk},\ E=-\frac{k}{2a},\ -M=G,\ e=\sqrt{1-\left(\frac{G}{L}\right)^2}.\]
Moreover, $g$ is the argument of apapsis, $\ell$ is called the mean anomaly, and $\ell$ can be related to the polar angle $\psi$ through
the equations
\[    \tan\frac \psi 2 = \sqrt{\frac{1+e}{1-e}}\cdot\tan\frac u 2,\quad u-e\sin u=\ell.\]
We also have Kepler's law $\frac{a^3}{T^2}=\frac{1}{(2\pi)^2}$ which relates the semimajor axis
$a$ and the period $T$ of the ellipse.

Denoting the particle's position by $Q=(q_1, q_2)$ and its momentum by $P=(p_1,p_2)$ we have the following formulas
in the case $g=0$
\begin{equation}
\label{DelEll}
\begin{aligned}
q_1=\frac{L^2}{mk}\left(\cos u-\sqrt{1-\frac{G^2}{L^2}}\right), \quad &
q_2=\frac{LG}{mk} \sin u,
\\
p_1=-\frac{mk}{L}\frac{\sin u}{1-\sqrt{1-\frac{G^2}{L^2}} \cos u}, \quad &
p_2=\frac{mk}{L^2}\frac{G\cos u}{1-\sqrt{1-\frac{G^2}{L^2}}\cos u},
\end{aligned}
\end{equation}
where $u$ and $\ell$ are related by $u-e\sin u=\ell$. Here $g$ does not appear because the argument of apapsis is chosen to be zero. In the general case, we need to rotate the $(q_1,q_2)$ and
$(p_1,p_2)$ using the matrix
$\left[\begin{array}{cc}
\cos g& -\sin g\\
\sin g& \cos g
\end{array} \right].
$

Notice that the equation \eqref{DelEll} describes an ellipse with one focus at the origin and the other focus on the negative $x$-axis. We want to be consistent with \cite{G2}, i.e. we want $g=\pi/2$ to correspond to the
``vertical" ellipse with one focus at the origin and the other focus on the positive $y$-axis. Therefore we rotate the picture clockwise. So we use the Delaunay coordinates which are related to the Cartesian ones
through the equation
\begin{equation}
\begin{aligned}
Q_\parallel=&\frac{1}{mk}\left(L^2\left(\cos u-\sqrt{1-\frac{G^2}{L^2}}\right)\cos g+LG\sin u\sin g\right) , \\
Q_\perp=&\frac{1}{mk} \left(-L^2 \left(\cos u-\sqrt{1-\frac{G^2}{L^2}}\right)\sin g+LG\sin u\cos g\right).\\
\end{aligned}\label{eq: Q3}
\end{equation}
This is an ellipse focused at the origin with its other focus lying on the positive $y$ axis.
\subsection{Hyperbolic motion}\label{subsection: hyp}
The above formulas can also be used to describe hyperbolic motion, where we need
to replace ``$\sin\to\sinh,$ $\cos\to \cosh$" and change signs properly (c.f.\cite{F, W}). Namely, we have
$a=\frac{L^2}{mk}, \ b=\frac{LG}{mk},\ E=\frac{k}{2a},\ -M=G,\ e=\sqrt{1+\left(\frac{G}{L}\right)^2}.$
\begin{equation}
\begin{aligned}
q_1=\frac{L^2}{mk}
\left(\cosh u-\sqrt{1+\frac{G^2}{L^2}}\right), \quad
& q_2=\frac{LG}{mk} \sinh u, \\
p_1=-\frac{mk}{L}\frac{\sinh u}{1-\sqrt{1+\frac{G^2}{L^2}} \cosh u}, \quad
& p_2=-\frac{mk}{L^2}\frac{G\cosh u}{1-\sqrt{1+\frac{G^2}{L^2}}\cosh u}.
\end{aligned}\label{eq: delaunay4}
\end{equation}
where $u$ and $\ell$ are related by
\begin{equation}u-e\sinh u=\ell, \text{ where } e=\sqrt{1+\left(\frac{G}{L}\right)^2}.
\label{eq: hypul}
\end{equation}

This hyperbola is symmetric with respect to the $x$-axis, opens to the right, and the particle moves counterclockwise on it when $u$ increases ($\ell$ decreases) in the case when minus the angular momentum $G=P\times Q<0$. The angle $g$ is defined to be the angle measured from the positive $x$-axis to the symmetric axis. There are two such angles that differ by $\pi$ depending on the orientation of the symmetric axis. This $\pi$ difference disappears after taking $\tan$, or in the symplectic form and the Hamiltonian equation after taking derivative so it does not matter which angle we choose.

When the particle moves to the right of the sections $\{x_{4,\parallel}^R=-\frac{\chi}{2}\}$ and $\{x_{4,\parallel}^L=\frac{\chi}{2}\}$ (Definition \ref{DefSection} and Figure 3), we have a hyperbola opening to the left and the particle moves
counter-clockwise.
To achieve this, we rotate \eqref{eq: delaunay4} by an angle $\pi+g$. In this case, we choose $g$ to be the angle measured from the positive $x$-axis to the symmetric axis pointing to the perigee. 
\begin{equation}
\begin{aligned}
Q_\parallel=&\frac{1}{mk}\left(-\cos g L^2(\cosh u-e)+\sin g LG\sinh u\right), \\
Q_\perp=&\frac{1}{mk}\left(-\sin g L^2(\cosh u-e)-\cos g LG\sinh u \right).\\
P=&\frac{mk}{1-e\cosh u}\left(\frac{1}{L}\sinh u\cos g-\frac{G}{L^2}\sin g\cosh u, \right.\\
&\left.\frac{1}{L}\sinh u\sin g+\frac{G}{L^2}\cos g\cosh u\right).
\end{aligned}\label{eq: Q4}
\end{equation}
We see from \eqref{eq: hypul}, when $|u|$ is large, we have sign$(u)=-$sign$(\ell)$. We have three different choices of $g$ in this paper.
\begin{enumerate}
\item[(a)] When the particle $Q_4$ is moving to the right of the sections  $\{x_{4,\parallel}^R=-\frac{\chi}{2}\}$ and $\{x_{4,\parallel}^L=\frac{\chi}{2}\}$ and if  the incoming asymptote is horizontal, (see the arrows in Figure 1 and 2 for ``incoming" and ``outgoing"), then the particle comes from the left, and
as $u$ tends to $-\infty$,
the $y$-coordinate is bounded and $x$-coordinate is negative.
In this case we have $\tan g=\frac{G}{L}$, $g\in(-\pi/2,0).$ We use $u < 0$ to refer to this piece of orbit.
\item[(b)] When the particle $Q_4$ is moving to the right of the sections $\{x_{4,\parallel}^R=-\frac{\chi}{2}\}$ and $\{x_{4,\parallel}^L=\frac{\chi}{2}\}$ and if the outgoing asymptote is horizontal, then the particle escapes to the left, and
as $u$ tends to $+\infty$, the $y$-coordinate is bounded and $x$-coordinate is negative.
In this case we have $\tan g=-\frac{G}{L}, g\in(0,\pi/2)$. We use $u > 0$ to refer to this piece of orbit.
The above two cases can be unified as $$\tan g = -\mathrm{sign}(u)\frac{G}{L}, \ \mathrm{with\ } G < 0, L > 0.$$
\item[(c)] When the particle $Q_4$ is moving to the left of the sections $\{x_{4,\parallel}^R=-\frac{\chi}{2}\}$ and $\{x_{4,\parallel}^L=\frac{\chi}{2}\}$, we treat the motion as hyperbolic motion focused at $Q_1$.
We move the origin to $Q_1$. The hyperbola opens to the right. The particle $Q_4$ moves on the hyperbola counterclockwise with negative angular momentum $G$, we then rotate by angle $g$, and $g$ is the angle measured from the positive $x$-axis to the symmetric axis pointing to the opening of the hyperbola. The orbit has the following parametrization
\begin{equation}
\begin{aligned}
Q&=\frac{1}{mk}\left(\cos g L^2(\cosh u-e)-\sin g LG\sinh u,\right.\\
&\quad \left.\sin g L^2(\cosh u-e)+\cos g LG\sinh u\right).\\
P&=\frac{mk}{1-e\cosh u}\left(-\frac{1}{L}\sinh u\cos g+\frac{G}{L^2}\sin g\cosh u,\right.\\
&\quad \left. -\frac{1}{L}\sinh u\sin g-\frac{G}{L^2}\cos g\cosh u\right).
\end{aligned}\label{eq: Q4l}
\end{equation}
\end{enumerate}

We note that the Delaunay coordinates have some singular behavior near double collision.
When we set $e=1$ in \eqref{eq: hypul}, we find $\ell=u^3+h.o.t.$ Hence $u$ as a function $\ell$ in a neighborhood of $0$ is only $C^0$ but not $C^1$.  One can verify that for $G=0$ and $\ell\neq 0$ the hyperbolic Delaunay coordinates still give a symplectic transformation, so we only have singular behavior when $G$ and $\ell$ are both close to zero. To control this singular behavior, we need the following estimates in Lemma \ref{Lm: var}.

\begin{Lm}[Lemma A.1 of \cite{DX}] \label{LmSmallu}In the hyperbolic Delaunay coordinates, as $G\to 0$, $u\to 0$ and $L$ being close to $1$, we have the following estimates of the first order derivatives
\[\left|\dfrac{\partial u}{\partial G}\right|\leq 2,\quad \left|\dfrac{\partial u}{\partial L}\right|\leq 2|G|\]
and the second order derivatives
\[\left|\dfrac{\partial Q}{\partial u}\frac{\partial^2 u}{\partial G^2}\right|\leq 4,\quad \left|\dfrac{\partial Q}{\partial u}\frac{\partial^2 u}{\partial L^2}\right|\leq 4G^2,\quad \left|\dfrac{\partial Q}{\partial u}\frac{\partial^2 u}{\partial G\partial L}\right|\leq 4|G|.\]
\end{Lm}

We next cite Lemma A.2 of \cite{DX} to simplify our calculation. The lemma implies that we can replace $\pm u$ by $\ln(\mp \ell/e)$ when taking first and second order derivatives.
\begin{Lm}[Lemma A.2 of \cite{DX}] \label{Lm: simplify}
Let $u$ be the function of $\ell, G$ and $L$ given by \eqref{eq: hypul} and let $\sigma=sign(u)$ when $|u|$ is large.
Then we can approximate $u$
by $\ln (-\sigma\ell/e)$ in the following sense.
\[\sigma u-\ln\frac{-\sigma\ell}{e}=O(\ln|\ell|/\ell),\quad \frac{\partial u}{\partial\ell}=\sigma 1/\ell+O(1/\ell^2),\]
\[\left(\frac{\partial }{\partial L},\frac{\partial }{\partial G}\right)\left(u+\sigma\ln e\right)=O(1/|\ell|),\quad \left(\frac{\partial }{\partial L},\frac{\partial }{\partial G}\right)^2\left(u+\sigma\ln e\right)=O(1/|\ell|),\]
The estimates above are uniform as long as $|G|\leq C,$ $1/C\leq L \leq C,$ $\ell>\ell_0$ for some constant $C>1$
and the implied constants in $O(\cdot)$ depend only on $C$ and $\ell_0.$
\end{Lm}
\subsection{The derivative of Cartesian  with respect to Delaunay }
Next, we calculate the first order derivatives of the Cartesian variables with respect to the Delaunay variables. The assumption of the next lemma is met by Lemma \ref{Lm: tilt}.
\begin{Lm}\label{Lm: dx/dDe}
Assume that $|G|\leq C,$ $1/C\leq L \leq C$ for some $C>1$.
\begin{itemize}
\item[(a)] Assume further in the right case $g=-\sigma\arctan\frac{G}{L}+\eps$, where $\sigma=$sign$(u)$ and $\eps=O\left(\frac{\mu}{|\ell_4|^2+1}+\frac{1}{\chi}\right)$. Then we have the following estimate of the derivative of Cartesian coordinates with respect to the Delaunay coordinates as $\ell\to \infty$
\begin{equation}\label{eq: d/dDe}
\begin{aligned}&
\frac{\partial (Q_\parallel,Q_\perp,P_\parallel,P_\perp)}{\partial (L,\ell,G,g)}=
\cD
+\eps\cdot\left[\begin{array}{cc}
\mathrm{Rot}\(\frac{\pi}{2}\) &0\\
0&\mathrm{Rot}\(\frac{\pi}{2}\)
\end{array}\right]\cdot \cD +\left[\begin{array}{c}
O(1)_{2\times 4}\\
0_{2\times 4}
\end{array}\right]\\
&+O(\eps^2)\cdot \cD,
\mathrm{\ where\ }\cD=\left[\begin{array}{cccc}
\sigma\frac{2L \ell}{mk }&\sigma \frac{L^2}{ mk}& 0&0\\
  -\frac{G L^2 \ell}{mk (G^2 + L^2)}& 0& \frac{ L^3 \ell}{mk(G^2+  L^2)}& \sigma\frac{L^2\ell}{mk}\\
 \sigma \frac{k m }{L^2 }& -\frac{k m}{L \ell^2}&
0& 0\\
  \frac{G k m}{L (G^2 + L^2) }&0& -\frac{k m}{(G^2 + L^2) }& -\sigma\frac{k m }{L }
\end{array}\right] .
\end{aligned}\end{equation}
\item[(b)]In the left case, if we assume $g,G=O(1/\chi)$ and $L=O(1)$,
then the estimates of the derivative are obtained by setting $G=O(1/\chi)$ in the above matrix.
\item[(c)] We have $\left|\frac{\partial Q}{\partial \ell}\right|=O(1),\quad \left|\frac{\partial Q}{\partial (L,G,g)}\right|=O(\ell),
\quad\frac{\partial Q}{\partial g}\cdot Q=0,\quad\frac{\partial Q}{\partial G}\cdot Q=O_{C^2(L, G, g)}(\ell). $
\end{itemize}
\end{Lm}
\begin{proof}

First we drop $e$ in \eqref{eq: Q4}, since it will contribute a $O(1)$ term in \eqref{eq: d/dDe}.
To obtain the leading term we just need to calculate
$\frac{\partial (\tQ_\parallel,\tQ_\perp,P_\parallel,P_\perp)}{\partial (L,\ell,G,g)}(L, \ell, G, -\sigma \arctan (G/L))$ where $\tQ$ refers to the RHS of \eqref{eq: Q4} with the $e$ term discarded. This derivative is obtained by
a straightforward calculation using the formulas \eqref{eq: Q4}, \eqref{eq: Q4l} with the help of Lemma \ref{Lm: simplify}. The calculations of $\frac{\partial \tQ}{\partial G},\frac{\partial \tQ}{\partial L}$ are presented  in detail in Lemma A.3 of \cite{DX}
and the other derivatives are similar.
To get the first correction term, i.e. the $O\left(\frac{\mu}{|\ell_4|^2+1}+\frac{1}{\chi}\right)$ part, let $g_0=\pm \arctan(G/L),$
$\eps=g-g_0.$ We use the relations
$ (\tQ_\parallel,\tQ_\perp,P_\parallel,P_\perp)(L, \ell, G, g)=
\mathrm{Rot}(g) (\tQ_\parallel,\tQ_\perp,P_\parallel,P_\perp)(L, \ell, G, 0) $
and
$ \mathrm{Rot}(g_0+\eps)=\mathrm{Rot}(g_0)+\eps \mathrm{Rot}(\pi/2)\mathrm{Rot}(g_0)+O(\eps^2)$
and notice that rotation by $\pi/2$ has the effect of interchanging the roles of $\parallel$ and $\perp.$
This gives parts (a) and (b) of the lemma.

Part (c) follows by direct calculation from \eqref{eq: Q4} and Lemma \ref{Lm: simplify} (see Lemma A.3(a) of \cite{DX}).
\end{proof}
\begin{Rk}\label{RkParallel}
\begin{enumerate}
\item Part $(c)$ of Lemma \ref{Lm: dx/dDe} means $\frac{\partial Q}{\partial g}$ is almost parallel to $\frac{\partial Q}{\partial G}$. This plays an important role in our proof of Lemma \ref{Lm: ham} as well as in \cite{DX}. In fact, in \eqref{eq: d/dDe} the matrix has determinant 1 since it is symplectic. We look at the $\cD$ term in \eqref{eq: d/dDe}. The discussion remains true when the other terms are included. In $\cD$, the first, second and fourth columns have no obvious linear relations. However, the first and fourth columns have modulus $O(\ell)$ when $|\ell|$ is large. So the third column must be almost parallel to either the first or fourth column to get determinant one.
\item The second and fourth rows of $\mathcal D$ are almost parallel for similar reasons. This fact plays an important role in the proof of Sublemma \ref{sublm: dX/dV}.
\item The same argument can be applied to the inverse of the LHS of \eqref{eq: d/dDe}. We will see in Lemma \ref{Lm: De/Ca}
below that the two rows $\frac{\partial g}{\partial (Q,P)}$ and $\frac{\partial G}{\partial (Q,P)}$ are almost parallel, which is used in the proof of Sublemma \ref{sublm: dV/dX} to get the tensor structure.
\end{enumerate}
\end{Rk}
\subsection{The derivative of Delaunay with respect to Cartesian }
We could have inverted the matrix \eqref{eq: d/dDe} to get the result of this section. However, though the matrix \eqref{eq: d/dDe} is nonsingular, it is close to be singular since we have some large entries of $O(\chi)$. Therefore we calculate the derivatives $\frac{\partial (L,G,g)}{\partial (Q,P)}$ directly using known identities.
\begin{Lm}\label{Lm: De/Ca}
We have the following estimates about the derivatives of Delaunay variables with respect to the Cartesian variables.
\begin{equation}\begin{aligned}
&\frac{\partial L}{\partial (Q,P)}=-\frac{L^3}{mk^2}\left(\frac{kQ}{|Q|^3},\frac{P}{m}\right),\quad \frac{\partial G}{\partial(Q,P) }=(-P_{\perp},P_{\parallel},Q_{\perp},-Q_{\parallel}),\\
&\frac{\partial g}{\partial (Q,P)}=\frac{1}{|P|^2}(0,0,-P_\perp,P_\parallel)-\\
&\mathrm{sign}(u)\left(\frac{L}{G^2+L^2}\frac{\partial G}{\partial (Q,P)}-\frac{G}{G^2+L^2}\frac{\partial L}{\partial (Q,P)}\right)+O(1/\ell^2),\quad \mathrm{as}\ |\ell|\to\infty.
\end{aligned}
\end{equation}
\end{Lm}
\begin{proof}
From the relation $\frac{mk^2}{2L^2}=\frac{P^2}{2m}-\frac{k}{|Q|}$ we get
$\frac{\partial L}{\partial (Q,P)}=-\frac{L^3}{mk^2}\left(\frac{kQ}{|Q|^3},\frac{P}{m}\right).$
We also have
\[G=P\times Q,\quad \frac{\partial G}{\partial(Q,P) }=(-P_{\perp},P_{\parallel},Q_{\perp},-Q_{\parallel}).\]
To get the derivative $\frac{\partial g}{\partial (Q,P)}$, we take the quotient $P_\perp/P_\parallel$ in \eqref{eq: Q4} and \eqref{eq: Q4l}, then apply the formula $\tan(\al\pm\beta)=\frac{\tan\al\pm\tan\beta}{1\mp \tan\al\tan\beta}$ to get that as $|u|\to\infty$
\begin{equation}\label{Eqg}
g=\arctan \frac{P_\perp}{P_\parallel}-\mathrm{sign}(u)\arctan \frac{G}{L}-e^{-2|u|}E(G/L,g,u)+o(e^{-2|u|})\quad(\mathrm{mod}\ \pi).
\end{equation}
where $E$ is a smooth function.
Hence
\begin{equation}\begin{aligned}
\partial g&=\frac{P_{\parallel}\partial P_{\perp}-P_{\perp}\partial P_{\parallel}}{P_{\parallel}^2+P_{\perp}^2}-\mathrm{sign}(u)\frac{L\partial G-G\partial L}{G^2+L^2}+O\left(\frac{1}{\ell^2}\right)\\
&=\frac{1}{P^2}(0,0,-P_\perp,P_\parallel)-\mathrm{sign}(u)\frac{L\partial G-G\partial L}{G^2+L^2}
+O\left(\frac{1}{\ell^2}\right).
\end{aligned}
\end{equation}
\end{proof}
\subsection{Second order derivatives}\label{subsubsection: 2ndderivative}
The following estimates of the second order derivatives are used in integrating the variational equation. 
\begin{Lm}[Lemma A.5 of \cite{DX}]
Assume that $|G|\leq C,$ $1/C\leq L \leq C$ for some constant $C>1$.
\begin{itemize}
\item[(a)]  We have as $|\ell|\to \infty$
\begin{equation*}
\begin{aligned}\frac{\partial^2 Q}{\partial g^2}&=-Q,\quad \frac{\partial^2 Q}{\partial g\partial G}\perp \frac{\partial Q}{\partial G},
\quad \left(\frac{\partial}{\partial G},\frac{\partial}{\partial g}\right)\left(\frac{\partial |Q|^2}{\partial g}\right)=(0,0),\\
\frac{\partial^2Q}{\partial G^2}&=O(\ell),\quad \frac{\partial^2 Q}{\partial L^2}=O(\ell),\quad \frac{\partial^2 Q}{\partial L\partial G}=O(\ell).
\end{aligned}
\end{equation*}

\item[(b)] Under the conditions of Lemma \ref{Lm: dx/dDe}(a)
we have
\begin{equation}
\begin{aligned}
\frac{\partial^2 Q}{\partial G^2}&=\frac{L^2}{(L^2+G^2)^{3/2}}(L\cosh u, 2G \sinh u)+O(1),\\
\frac{\partial^2 Q}{\partial g\partial G}&=\left(\frac{L^2\sinh u}{\sqrt{L^2+G^2}},0\right)+O(1),\\
\frac{\partial^2 Q}{\partial g\partial L}&=\left(-\frac{GL\sinh u}{\sqrt{L^2+G^2}},-2 \sqrt{L^2+G^2}  \cosh u\right)+O(1),\\
\frac{\partial^2 Q}{\partial G\partial L}&=\frac{-L}{(L^2+G^2)^{3/2}}\left(LG \cosh u, (L^2+3G^2) \sinh u\right)+O(1).
\end{aligned}\nonumber\end{equation}
\item[(c)] Under the conditions of Lemma \ref{Lm: dx/dDe}(b) we have
\begin{equation}
\begin{aligned}
&\frac{\partial^2 Q}{\partial G^2}=-\cosh u(1,0)+O(1), \quad
\frac{\partial^2 Q}{\partial g\partial G}=-L\sinh u(1,0)+O(1), \\
& \frac{\partial^2 Q}{\partial g\partial L}=L\sinh u(0,2)+O(1), \quad
\frac{\partial^2 Q}{\partial G\partial L}=\cosh u(0,1)+O(1).\\
\end{aligned}\nonumber\end{equation}
\end{itemize}
\label{Lm: 2ndderivative}
\end{Lm}

\section{Gerver's mechanism}\label{section: gerver}
\subsection{Gerver's result in \cite{G2}}\label{subsection: gerver}
We summarize the result of \cite{G2} in the following table. Recall that the Gerver scenario deals with the limiting case
$\chi\to\infty, \mu\to 0$. Accordingly $Q_1$ disappears at infinity and there is no interaction between $Q_3$ and $Q_4.$
Hence both particles perform Kepler motions.
The shape of each Kepler orbit is characterized by energy, angular momentum and the argument of apapsis.
In Gerver's scenario, the incoming and outgoing asymptotes of the hyperbola are always horizontal and the semimajor of the ellipse is always vertical. So we only need to describe on the energy and angular momentum.
\begin{center}
\begin{tabular}{|c|c|c|c|c|}
\hline
\  &\small{1st collision}  &\small{@$(-\eps_0\eps_1,\eps_0+\eps_1)$}  & \small{2nd collision}  &\small{$@(\eps_0^2,0)$}\\
\hline
 \  & $Q_3$ & $Q_4$ & $Q_3$ & $Q_4$ \\
\hline
energy  & $-\frac{1}{2}$ & $\frac{1}{2}$ & $-\frac{1}{2}\rightarrow -\frac{\eps_1^2}{2\eps_0^2}$ & $\frac{1}{2}\rightarrow \frac{\eps^2_1}{2\eps_0^2}$\\
\hline
{\small angular\,momentum}& $\eps_1\rightarrow -\eps_0$ & $p_1\rightarrow -p_2$ & $-\eps_0$ & $\sqrt{2}\eps_0$ \\
\hline
eccentricity & $\eps_0\rightarrow \eps_1$ &  & $\eps_1\rightarrow \eps_0$& \\
\hline
semimajor & $1$ & $-1$ & $1\rightarrow \left(\frac{\eps_0}{\eps_1}\right)^2$ &  $1\rightarrow -\frac{\eps^2_1}{\eps_0^2}$\\
\hline
semiminor & $\eps_1\rightarrow \eps_0$ & $p_1\rightarrow p_2$ & $\eps_0\rightarrow \frac{\eps_0^2}{\eps_1}$ & $\sqrt{2}\eps_0 \rightarrow \sqrt{2}\eps_1$\\
\hline
\end{tabular}
\end{center}
Here $p_{1,2}=\frac{-Y\pm \sqrt{Y^2+4(X+R)}}{2},\quad R=\sqrt{X^2+Y^2},$
and $(X,Y)$ stands for the point where collision occurs (the parenthesis after $@$ in the table). We will call the two points the Gerver collision points. In the above table $\eps_0$ is a free parameter and $\eps_1=\sqrt{1-\eps_0^2}.$ At the collision points, the velocities of the particles are the following. For the first collision,
\[v_3^-=\left(\frac{-\eps_1^2}{\eps_0\eps_1+1}, \frac{-\eps_0}{\eps_0\eps_1+1}\right),\quad v_4^-=\left(1-\frac{Y}{Rp_1}, \frac{1}{Rp_1}\right).\]
\[v_3^+=\left(\frac{\eps_0^2}{\eps_0\eps_1+1}, \frac{\eps_1}{\eps_0\eps_1+1}\right),\quad v_4^+=\left(-1+\frac{Y}{Rp_2}, -\frac{1}{Rp_2}\right).\]

For the second collision,
\[v_3^-=\left(\frac{-\eps_1}{\eps_0}, \frac{-1}{\eps_0}\right),\ v_4^-=\left(1, \frac{\sqrt{2}}{\eps_0}\right),\quad v_3^+=\left(1, \frac{-1}{\eps_0}\right),\ v_4^+=\left(\frac{-\eps_1}{\eps_0}, \frac{\sqrt{2}}{\eps_0}\right).\]

\section{$\mathscr C^1$ control of the global map, proof of Lemma \ref{Lm: glob}}\label{appendixb}
In this Appendix, we derive Lemma \ref{Lm: glob} from Proposition \ref{Prop: main}. We split the proof into six steps.

\textbf{STEP 0:} {\it Preparations. Definitions of auxiliary vectors and simplification of the five matrices. }

We define some new auxiliary vectors. Recall that in the paragraph before Proposition \ref{Prop: main}, we introduced a convention to use {\bf bold} font to indicate that the estimate of the corresponding entry is actually $\sim$, not only $\lesssim.$

Below we use the following notational convention to make it easier for the reader to keep track of the computations.
\begin{Not}
A vector with $tilde,\ hat,\ bar$ means a $O(1/\chi), O(\mu), O(1)$ perturbation to the vector respectively.\end{Not}
\begin{Def}We define the following list of vectors and the matrix $S$.
\begin{equation}\nonumber
\begin{aligned}
\bullet\ \ \ \tilde u:=&N_3u=u+O\(\frac{1}{\chi}\)\lesssim\left(\frac{1}{\chi^2},-\mathbf 1,\frac{1}{\chi^2},\frac{1}{\chi^2};\mu,\frac{\mu}{\chi},\frac{1}{\mu\chi^2},\frac{\mu}{\chi^2};\frac{\mu}{\chi},\frac{\mu}{\chi}\right)^T,\\
\tilde l:=&lN_3=l+O\(\frac{1}{\chi}\)\lesssim\left(\mathbf 1,\frac{1}{\chi^2},\frac{1}{\chi^2},\frac{1}{\chi^2};\frac{1}{\mu\chi^2},\frac{1}{\chi^2},\mu,\frac{\mu}{\chi};\frac{1}{\chi},\frac{1}{\chi}\right),\\
\bullet\ \ \ \hat u:=&N_5u=u+O(\mu)\lesssim \left(\mu,-\mathbf 1,\mu,\mu;\mu,\frac{\mu}{\chi},\frac{\mu}{\chi},\frac{\mu}{\chi};\mu,\mu\right),\\
\hat l:=&lN_1=l+O(\mu)\lesssim \left(\mathbf 1,\mu,\mu,\mu;\frac{\mu}{\chi},\frac{\mu}{\chi},\mu,\frac{\mu}{\chi};\mu,\mu\right),\\
\bullet\ \ \dt u:=&AL\cdot R^{-1}u_i\sim AR\cdot L^{-1}u_{i'}\lesssim\left(0,0,0,0;0,\mu,0, \frac{1}{\chi};0,\frac{1}{\chi}\right)^T=O(\mu),\\
\dt l:=&l_{iii}L\cdot R^{-1}C\sim l_{iii'}R\cdot L^{-1}C\lesssim\left(\frac{1}{\chi},\frac{1}{\chi^4},\frac{1}{\chi^4},\frac{1}{\chi^4};\frac{1}{\chi^2},\frac{1}{\chi}, \frac{\mu}{\chi},\mu;\frac{1}{\chi},\frac{1}{\chi}\right)=O(\mu),\\
\end{aligned}
\end{equation}
\begin{equation}\nonumber
\begin{aligned}
\bullet\ \ \hat u_{iii}:=&u_{iii}+\dt u\lesssim\left(0,0,0,0;0,\mu,0,\frac{1}{\chi};1, 1\right)^T,\\
\hat l_i:=&l_i+\dt l\lesssim\left(1,\frac{1}{\chi^3},\frac{1}{\chi^3},\frac{1}{\chi^3};\frac{1}{\mu\chi^2},\frac{1}{\chi}, \mu,\mu; 1, 1\right),\\
\hat u_{iii'}:=&u_{iii'}+\dt u\lesssim\left(0,0,0,0;0,\mu,0,\frac{1}{\chi}; 1, 1\right)^T,\\
\hat l_{i'}:=&l_{i'}+\dt l\lesssim\left(\frac{1}{\chi},\frac{1}{\chi^4},\frac{1}{\chi^4},\frac{1}{\chi^4};\frac{1}{\chi^2},\frac{1}{\chi}, \frac{\mu}{\chi},\mu; 1, 1\right),\\
\bullet AL\cdot R^{-1}&C=S_2,\quad AR\cdot L^{-1}C= S_4,\quad \mathrm{\ where\ }S_2, S_4\lesssim S:=\\
&\left[\begin{array}{cc|cccc|c}
\mathrm{Id}_{4\times 4}&\ &0_{4\times 1}&0_{4\times 1}&0_{4\times 1}&0_{4\times 1}&0_{4\times 2}\\
\hline
0&0_{1\times 3}&1+O(\mu)&0&0&0&0_{1\times 2}\\
0&0_{1\times 3}&0&1+O(\mu)&0&0&O(\mu)_{1\times 2}\\
O(1)&O\left(\frac{1}{\chi^3}\right)_{1\times 3}&O\left(\frac{1}{\mu\chi^2}\right)&O\left(\frac{1}{\mu\chi^3}\right)&1+O(\mu)&O\left(\frac{\mu}{\chi}\right)&O\left(\frac{1}{\chi}\right)_{1\times 2}\\
0&0_{1\times 3}&0&0&0&1+O(\mu)&0_{1\times 2}\\
\hline
0&0_{1\times 3}&0&0&0&0&0_{1\times 2}\\
1+O(\mu)&O\left(\frac{1}{\chi^3}\right)_{1\times 3}&O\left(\frac{1}{\mu\chi^2}\right)&O\left(\frac{1}{\chi^2}\right)&O(\mu)&O(\mu)&O\left(\frac{1}{\chi}\right)_{1\times 2}
\end{array}\right].
\end{aligned}
\end{equation}\label{Def: ul}
\end{Def}

\begin{sublemma}\label{sublm: ul}
\begin{enumerate}
\item $l\cdot u,\ \tilde l\cdot u,\  l\cdot \tilde u\lesssim\frac{1}{\chi^2}$,
\item $l_{iii}\cdot L\cdot R^{-1}u_{i}= \frac{1+O(\mu)}{\chi}$,\quad  $l_{iii'}\cdot R\cdot L^{-1}u_{i'}=-\frac{1+O(\mu)}{\chi}$.
\end{enumerate}
\end{sublemma}
\begin{proof}
All of these estimates come from straightforward calculations using Proposition \ref{Prop: main} and Definition \ref{Def: ul}.
Item (2) is exact. It uses Proposition \ref{Prop: main}(b2).
\end{proof}
Using the Definition \ref{Def: ul} and Sublemma \ref{sublm: ul}, we simplify the five matrices into sums as follows. Notice that
the factors $(\Id_{10}+u_1^i\otimes l_1^i)$ in $(I)$ and $(\Id_{10}+u_5^f\otimes l_5^f)$ in $(V)$ are both $O(1)$, so we do not include them in the following calculation until the final step for simplicity. We shall write $(\bar{I})$ and $(\bar{V})$ for the modified
matrices.
\begin{equation*}
\begin{aligned}
(III)&\lesssim(\Id_{10}+\chi u\otimes l)N_3(\Id_{10}+\chi u\otimes l')=(\Id_{10}+\chi u\otimes l)(N_3+\chi N_3u\otimes l')\\
&=(\Id_{10}+\chi u\otimes l)(N_3+\chi \tilde u\otimes l')=(N_3+\chi \tilde u\otimes l')+\chi u\otimes l(N_3+\chi \tilde u\otimes l')\\
&\sim N_3+\chi \tilde u\otimes l'+\chi u\otimes \tilde l,\\
\end{aligned}
\end{equation*}
\begin{equation}\label{eq: 5matrices}
\begin{aligned}
(II)&=(\chi u_{iii}\otimes l_{iii}+A)L\cdot R^{-1}(\chi u_{i}\otimes l_{i}+C )\\
&=\chi^2 u_{iii}\otimes l_{iii}L\cdot R^{-1}u_{i}\otimes l_{i}+\chi \dt u\otimes l_{i}+\chi u_{iii}\otimes \dt l+S \\
&\sim\chi \hat u_{iii}\otimes \hat l_{i}-\chi \dt u\otimes \dt l+S,\\
(IV)&\sim\chi \hat u_{iii'}\otimes \hat l_{i'}-\chi \dt u\otimes \dt l+S,\\
(\bar{I})&\lesssim (\Id_{10}+\chi u\otimes l)N_1=N_1+\chi u\otimes lN_1= N_1+\chi u\otimes \hat l,\\
(\bar{V})&\lesssim N_5(\Id_{10}+\chi u\otimes l')=N_5+N_5\chi u\otimes l'= N_5+\chi \hat u\otimes l',\\
\end{aligned}
\end{equation}
where for $(III)$, we used that $l(N_3+\chi \tilde u\otimes l')\lesssim\tilde l+\frac{1}{\chi}l'\sim \tilde l$ by Definition \ref{Def: ul}(first bullet point) and Sublemma \ref{sublm: ul}(1).

\textbf{STEP 1:} {\it Decomposing $(IV)(III)(II)$ into three summands.}

We start with an auxillary estimate.
\begin{sublemma}
We have the following estimates as $1/\chi\ll\mu\to 0$.
\begin{enumerate}
\item $(III)\hat u_{iii}\lesssim\left(
\frac{1}{\chi^2}, 1,\frac{1}{\chi^2}, \frac{1}{\chi^2};\mu,\mu, \frac{1}{\chi}, \frac{1}{\chi}; \mathbf{1},
\mathbf{1}\right)^T=O(1),$\\
    $\hat l_{i'}(III)\lesssim\left(1, \frac{1}{\chi^2},\frac{1}{\chi^2},\frac{1}{\chi^2};
    \frac{1}{\chi}, \frac{1}{\chi},\mu,\mu;  \mathbf{1}, \mathbf{1}\right)=O(1),$
\item $\hat l_{i'}(III)\hat u_{iii}\to -2$.
\end{enumerate}\label{sublm:l(III)u}
\end{sublemma}
\begin{Cor}\label{CorB1}
\begin{enumerate}
\item $\dt l(III) \hat u_{iii}\lesssim\frac{1}{\chi}$,
\item $\hat l_{i'}(III)\dt u\lesssim\frac{1}{\chi}$.
\end{enumerate}
\end{Cor}
\begin{proof}All of these are done by straightforward calculation using the information obtained in Proposition \ref{Prop: main} together with the calculation of $(III)$ in \eqref{eq: 5matrices}.
The $\mathbf 1$ entries in item (1) are actually $(N_3)_{44}(u_{iii}(9),u_{iii}(10))$ and $(l_{i'}(9),l_{i'}(10))(N_3)_{44}$ up to a $O(\mu)$ error. Item (2) is in fact $(l_{i'}(9),l_{i'}(10))(N_3)_{44}(u_{iii}(9),u_{iii}(10))$ up to a $O(\mu)$ error.
These terms can be calculated explicitly using part (b1), (b2), (b3) of Proposition \ref{Prop: main}.
\end{proof}


Then we consider
\begin{equation}\label{eq: (IV)(III)(II)}
\begin{aligned}
&(IV)(III)(II)\sim
(\chi \hat u_{iii'}\otimes \hat l_{i'}-\chi \dt u\otimes \dt l+S)(III)(\chi \hat u_{iii}\otimes \hat l_{i}-\chi \dt u\otimes \dt l+S)\\
&=\chi^2 \hat u_{iii'}\otimes \hat l_{i'}(III) \hat u_{iii}\otimes \hat l_{i}+(-\chi \dt u\otimes \dt l+S)(III)(\chi \hat u_{iii}\otimes \hat l_{i})\\
&+(\chi \hat u_{iii'}\otimes \hat l_{i'})(III)(-\chi \dt u\otimes \dt l+S)+(-\chi \dt u\otimes \dt l+S)(III)(-\chi \dt u\otimes \dt l+S).
\end{aligned}
\end{equation}
Define $v=\hat l_{i'}(III)(-\chi \dt u\otimes \dt l+S),\quad v'=(-\chi \dt u\otimes \dt l+S)(III) \hat u_{iii}.$
Both are of order $1$ by Corollary \ref{CorB1} and Sublemma \ref{sublm:l(III)u} (1).
From Sublemma \ref{sublm:l(III)u}(2) we get
\begin{equation}
\begin{aligned}
&\eqref{eq: (IV)(III)(II)}\sim\chi^2 \hat u_{iii'}\otimes \hat l_{i}+\chi v'\otimes \hat l_{i}+\chi \hat u_{iii'}\otimes v+(\chi \dt u\otimes \dt l-S)(III)(\chi \dt u\otimes \dt l-S)\\
&=\chi^2 \left(\hat u_{iii'}+\frac{1}{\chi} v'\right)\otimes \left(\hat l_{i}+\frac{1}{\chi}v\right)- v'\otimes v+(\chi \dt u\otimes \dt l-S)(III)(\chi \dt u\otimes \dt l-S)\\
&:=\chi^2 \tilde{\hat u}_{iii'}\otimes \tilde{\hat l}_{i}-v'\otimes v+(\chi \dt u\otimes \dt l-S)(III)(\chi \dt u\otimes \dt l-S),
\label{eq: (IV)(III)(II)new}
\end{aligned}
\end{equation}

where we have defined
$\tilde{\hat u}_{iii'}=u_{iii'}+\dt u+\frac{1}{\chi}v',\quad \tilde{\hat l}_{i}=l_i+\dt l+\frac{1}{\chi}v.$ It is important to stress that the coefficient of the $
\chi^2$ term is nonzero.

Next we consider $(V)(IV)(III)(II)(I)=(V)\eqref{eq: (IV)(III)(II)new}(I)$.

{\it In the following, we are going to show that $\chi^2 (V)\tilde{\hat u}_{iii'}\otimes \tilde{\hat l}_{i}(I)$ gives rise to the $\chi^2$ part of the main lemma \ref{Lm: glob}. The $(V)v'\otimes v(I)$ part will be absorbed into $O(\mu\chi)$ part. The last summand in \eqref{eq: (IV)(III)(II)new} will give rise to the $O(\chi)$ part together with a perturbation of order $O(\mu\chi)$, where the $O(\chi)$ part comes from $(V)S(III)S(I)$.}

\textbf{STEP 2:} {\it The first summand in \eqref{eq: (IV)(III)(II)new} gives the $O(\chi^2)$ contribution in Lemma \ref{Lm: glob}.}

The following sub lemma is needed for this step.
\begin{sublemma}\label{sublm: VX2I}
\begin{enumerate}
\item $l'\cdot \hat u_{iii'}\lesssim\frac{1}{\chi}$,
\quad $\hat l_i\cdot u\lesssim \frac{1}{\chi^2}.$
\item $l'\cdot v'\lesssim \frac{\mu}{\chi},\quad v\cdot u\lesssim \frac{\mu}{\chi}$.
\end{enumerate}
\end{sublemma}
We consider first the term $(\bar{V})(\chi^2 \tilde{\hat u}_{iii'}\otimes \tilde{\hat l}_{i})(\bar{I})$. We keep in mind that $N_1,N_5=O(\mu\chi).$ Define
\begin{equation}\label{eq: chi^2}\begin{aligned}
&\bar{ u}':=(\bar{V})\tilde{\hat u}_{iii'}=N_5\tilde{\hat u}_{iii'}+\chi \hat u\otimes l'\cdot\tilde{\hat u}_{iii'}
=(N_5\hat u_{iii'}+O(\mu))+ \hat u\(\chi l'\cdot\hat u_{iii'}+l'\cdot v'\)\\
&= N_5\hat u_{iii'}+O(1)\hat u+O(\mu),\\
&\bar{ l}':=\tilde{\hat l}_{i}(\bar{I})=\tilde{\hat l}_{i}N_1+\chi \tilde{\hat l}_{i}\cdot u\otimes \hat l=(\hat l_{i}N_1+O(\mu))+(\chi\hat l_{i}\cdot u+v\cdot u)\hat l =\hat l_{i}N_1+O(\mu).
\end{aligned}\end{equation}
We will analyze $\bar{ u}'$ and $\bar{ l}'$ in more detail in the final step.

\textbf{STEP 3:} {\it The second summand $v'\otimes v$ in \eqref{eq: (IV)(III)(II)new} gives  $(V)v'\otimes v(I)=O(\mu\chi)$.}

The following sub lemma is needed in this step.
\begin{sublemma}\label{sublmB5}
We have the following estimates.
\begin{enumerate}
\item $ N_5\dt u\lesssim\left(\frac{\mu}{\chi},\frac{1}{\chi},\frac{\mu}{\chi},\frac{\mu}{\chi};\frac{\mu}{\chi}, \mu,\frac{1}{\chi^2},\frac{1}{\chi};\frac{1}{\chi},\frac{1}{\chi}\right)^T=O(\mu)$,\\
$\dt l N_1\lesssim\left(\frac{1}{\chi},\frac{\mu}{\chi},\frac{\mu}{\chi},\frac{\mu}{\chi};\frac{1}{\chi^2}, \frac{1}{\chi},\frac{\mu}{\chi},\mu;\frac{1}{\chi},\frac{1}{\chi}\right)=O(\mu)$,\\
\item  $l'\cdot\dt u\lesssim\frac{1}{\chi^2},\quad \dt l \cdot u\lesssim\frac{\mu}{\chi^2}.$
\end{enumerate}
\end{sublemma}
Before considering $(\bar{V})v'\otimes v(\bar{I})$, we perform the following calculation.
\begin{equation}\begin{aligned}
&(\bar{V})\chi \dt u\otimes \dt l=(N_5+\chi \hat u\otimes l')\chi \dt u\otimes \dt l= \chi( N_5\dt u+\chi \hat u\otimes l'\cdot \dt u)\otimes \dt l:=\chi \hat{\dt u}\otimes \dt l,\\
&\chi \dt u\otimes \dt l(\bar{I})=\chi \dt u\otimes \dt l(N_1+\chi u\otimes \hat l)=\chi \dt u\otimes (\dt lN_1+\chi \dt l\cdot u\otimes \hat l):=\chi \dt u\otimes \hat{\dt l},
\end{aligned}\label{eq: VulI}
\end{equation}
We use Sublemma \ref{sublmB5} to conclude that $\hat{\dt u},\hat{\dt l}=O(\mu)$.

Next we consider $(\bar{V})v'\otimes v(\bar{I})$.
\begin{equation}
\begin{aligned}
&(\bar{V})v'\otimes v(\bar{I})=
(\bar{V})(-\chi \dt u\otimes \dt l+S)(III) \hat u_{iii}\otimes\hat l_{i'}(III)(-\chi \dt u\otimes \dt l+S)(\bar{I})\\
&=(-\chi \hat{\dt u}\otimes \dt l+(\bar{V})S)(III) \hat u_{iii}\otimes\hat l_{i'}(III)(-\chi \dt u\otimes \hat {\dt l}+S(\bar{I}))\\
&=\chi^2\hat{\dt u}\otimes \dt l[(III)\hat u_{iii}\otimes\hat l_{i'}(III)]\dt u\otimes \hat {\dt l}-\chi \hat{\dt u}\otimes \dt l[(III)\hat u_{iii}\otimes\hat l_{i'}(III)]S(\bar{I})\\
&-\chi (\bar{V})S[(III) \hat u_{iii}\otimes\hat l_{i'}(III)]\dt u\otimes \hat {\dt l}+(\bar{V})S[(III) \hat u_{iii}\otimes\hat l_{i'}(III)]S(\bar{I})\\
&\lesssim \hat{\dt u}\otimes \hat{\dt l}+ \hat{\dt u}\otimes\hat l_{i'}(III)S(\bar{I})+ (\bar{V})S(III) \hat u_{iii}\otimes \hat {\dt l}+
(\bar{V})S(III) \hat u_{iii}\otimes\hat l_{i'}(III)S(\bar{I})\\
\end{aligned}\label{eq: wtensorw}
\end{equation}
where in the last step we use Corollary \ref{CorB1}. The first term above is $O(\mu^2)$. To study the remaining three terms, we
continue the calculation in Sublemma \ref{sublm:l(III)u} to get 
\begin{sublemma}We have the following estimates.
\begin{enumerate}
\item $\hat l_{i'}(III)S\lesssim\left(1,\frac{1}{\chi^2},\frac{1}{\chi^2},\frac{1}{\chi^2};\frac{1}{\chi},\frac{1}{\chi},\mu,\mu;\frac{1}{\chi},\frac{1}{\chi} \right)=O(1)$,\\
 $\hat l_{i'}(III)SN_1\lesssim\left(1,\mu,\mu,\mu;\frac{1}{\chi},\frac{1}{\chi},\mu,\mu;\mu,\mu \right)=O(1),$
\item $S(III)\hat u_{iii}\lesssim\left(\frac{1}{\chi^2},1,\frac{1}{\chi^2},\frac{1}{\chi^2};\mu,\mu,\frac{1}{\chi},\frac{1}{\chi};0,\frac{1}{\chi}\right)=O(1)$, \\
$N_5S(III)\hat u_{iii}\lesssim\left(\mu,1,\mu,\mu;\mu,\mu,\frac{1}{\chi},\frac{1}{\chi};\mu,\mu\right)=O(1)$.
\end{enumerate}
\label{sublm: l(III)BC}
\end{sublemma}
\begin{Cor}\label{CorB2}
 $\hat l_{i'}(III)S\cdot u\lesssim \frac{\mu}{\chi}$,\quad  $l'\cdot S(III) \hat u_{iii}\lesssim \frac{\mu}{\chi}.$
\end{Cor}
Using the Sublemma \ref{sublm: l(III)BC} and Corollary \ref{CorB2}, we get
\begin{equation}\begin{aligned}\nonumber
&\hat l_{i'}(III)S(\bar{I})=\hat l_{i'}(III)SN_1+\chi \hat l_{i'}(III)S\cdot u\otimes \hat l=O(1),\\
&(\bar{V})S(III) \hat u_{iii}=N_5S(III) \hat u_{iii}+\chi \hat u\otimes l'\cdot S(III) \hat u_{iii}=O(1).
\end{aligned}
\end{equation}
Accordingly the fourth term in \eqref{eq: wtensorw} is $O(1)$ and the other terms are even smaller. Hence
$(\bar{V})v'\otimes v(\bar{I})=O(1)\ll O(\mu\chi)$.


\textbf{STEP 4:} {\it The last summand in \eqref{eq: (IV)(III)(II)new} gives the $O(\chi)$ contribution in Lemma \ref{Lm: glob} and a $O(\mu\chi)$ perturbation. }

To proceed, the following calculation is needed.
\begin{sublemma}\label{sublm: dtl(III)BC}
We have the following estimates.
\begin{enumerate}
\item $(III)\dt u\lesssim\left(\frac{1}{\chi^3}, \frac{1}{\chi}, \frac{1}{\chi^3},\frac{1}{\chi^3};\frac{\mu}{\chi},\mu,\frac{1}{\chi^2},\frac{1}{\chi};\frac{1}{\chi},\frac{1}{\chi}\right)^T=O(\mu),$\\
$\dt l(III)\lesssim\left(\frac{1}{\chi}, \frac{1}{\chi^3},\frac{1}{\chi^3},\frac{1}{\chi^3};\frac{1}{\chi^2},\frac{1}{\chi},\frac{\mu}{\chi},\mu;\frac{1}{\chi},\frac{1}{\chi}\right)=O(\mu)$,
\item$\dt l(III)S\lesssim\left(\frac{1}{\chi},\frac{1}{\chi^3},\frac{1}{\chi^3},\frac{1}{\chi^3};\frac{1}{\chi^2},\frac{1}{\chi},\frac{\mu}{\chi},\mu;\frac{\mu}{\chi},\frac{\mu}{\chi}\right)=O(\mu)$,\\
$\dt l(III)SN_1\lesssim\left(\frac{1}{\chi},\frac{\mu}{\chi},\frac{\mu}{\chi},\frac{\mu}{\chi};\frac{1}{\chi^2},\frac{1}{\chi},\frac{\mu}{\chi},\mu;\frac{\mu}{\chi},\frac{\mu}{\chi}\right)=O(\mu)$,\\
\item
$S(III)\dt u\lesssim\left(\frac{1}{\chi^3},\frac{1}{\chi},\frac{1}{\chi^3},\frac{1}{\chi^3};\frac{\mu}{\chi},\mu,\frac{1}{\chi^2},\frac{1}{\chi};0,\frac{\mu}{\chi}\right)=O(\mu),$\\
$N_5S(III)\dt u\lesssim\left(\frac{\mu}{\chi},\frac{1}{\chi},\frac{\mu}{\chi},\frac{\mu}{\chi};\frac{\mu}{\chi},\mu,\frac{1}{\chi^2},\frac{1}{\chi};\frac{\mu}{\chi},\frac{\mu}{\chi}\right)=O(\mu).$
\end{enumerate}
\end{sublemma}
\begin{Cor}\label{CorB3}
\begin{enumerate}
\item $\dt l (III)\dt u\lesssim\frac{\mu}{\chi},$
\item $\dt l(III)S\cdot u\lesssim\frac{\mu}{\chi^2}$,\quad $l'\cdot S(III) \dt u\lesssim\frac{\mu}{\chi^2}$.
\end{enumerate}
\end{Cor}
We are now ready to consider the last summand in \eqref{eq: (IV)(III)(II)new}. Using \eqref{eq: VulI}, we get
\begin{equation}
\begin{aligned}\label{eq: O(chi)}
&(\bar{V})(-\chi \dt u\otimes \dt l+S)(III)(-\chi \dt u\otimes \dt l+S)(\bar{I})\\
&=
(-\chi \hat{\dt u}\otimes \dt l+(\bar{V})S)(III) (-\chi \dt u\otimes \hat {\dt l}+S(\bar{I}))\\
&=\chi^2 \hat{\dt u}\otimes \dt l(III) \dt u\otimes \hat {\dt l}-\chi \hat{\dt u}\otimes \dt l(III)S(\bar{I})-\chi(\bar{V})S(III)  \dt u\otimes \hat {\dt l}\\
&+(\bar{V})S(III) S(\bar{I}).\\
\end{aligned}
\end{equation}
The first term in the RHS of \eqref{eq: O(chi)} is $O(\mu^3\chi)$ using Corollary \ref{CorB3}(1). Next,
\[\dt l(III)S(\bar{I})=\dt l(III)SN_1+\chi\dt l(III)S u\otimes\hat l=O(\mu).\]
This implies the second term in the RHS of \eqref{eq: O(chi)} is $O(\mu\chi)$. To consider the third term in the RHS of
\eqref{eq: O(chi)}, we note that
\[(\bar{V})S(III) \dt u=N_5S(III) \dt u+\chi \hat u\otimes l'\cdot S(III) \dt u=O(\mu).\]
So the third term is also $O(\mu\chi)$. Thus we get
\[\eqref{eq: O(chi)}=(\bar{V})S(III) S(\bar{I})+O(\mu\chi).\]
We need the following calculations.
\begin{sublemma}\label{sublm: last}
We have the following estimates as $1/\chi\ll \mu\to 0$.
\begin{enumerate}
\item $ l'SN_1, l'N_3SN_1, l'SN_3SN_1= (1,0_{1\times 9})+O(\mu)\to \hat\brrlin_j$,
\item $ N_5Su,  N_5N_3Su, N_5SN_3Su= (0,1,0_{1\times 8})^T+O(\mu)\to \tw,$
\item $l'Su,\ l'S\tilde u,\ l' S
N_3Su \lesssim\frac{\mu}{\chi},\quad \tilde lSu=O\(\frac{1}{\chi^2}\),$
\item $N_5SN_3SN_1=O(\mu\chi)$.
\end{enumerate}
\end{sublemma}
\begin{proof}
Items (1) and (2) can be obtained by taking the limit
$\displaystyle \lim_{\mu\to 0,\chi\to\infty}$ using Mathematica.
In item (4), we use Mathematica to get $\displaystyle \lim_{\mu\to 0,\chi\to\infty}N_5SN_3SN_1/\chi=0.$
\end{proof}

To understand (1) and (2) heuristically, we notice that all the entries of $l'$ are small except the first one, so
multiplying $l'$ by a matrix corresponds to picking out the first row.
Though $N_1,N_3$ have some large entries of order $O(\mu\chi)$, the corresponding entries of $l'$ are small enough to suppress them. The first rows of the matrices $S,N_3,N_1$ all have a similar structure to $l'$. Therefore, we may think of $l'$ as a left eigenvector of the matrices. The same heuristic argument applies to $u$.
To see where (4) comes from we may think of $S$ as the identity. The big entries of $O(\mu\chi)$ in $N_1,N_3,N_5$ are off-diagonal. It is not hard to keep track of these $O(\mu\chi)$ entries to see that we do not get terms greater than $O(\mu\chi)$.

Next, we multiply $(\bar{V})S(III)S(\bar{I})$ to get
\begin{equation}\label{eq: 8terms}
\begin{aligned}
&[N_5+\chi \hat u\otimes  l']S[N_3+\chi \tilde u\otimes l'+\chi u\otimes \tilde
l]S[N_1+\chi u\otimes \hat l]\\
&=[N_5+\chi \hat u\otimes  l'][SN_3+\chi S\tilde u\otimes l'+\chi Su\otimes \tilde l][SN_1+\chi Su\otimes \hat l]\\
&=[N_5SN_3+N_5(\chi S\tilde u\otimes l'+\chi Su\otimes \tilde l)+\chi \hat u\otimes  l'SN_3+O(\mu\chi)\hat u\otimes l'+O(\mu)]\cdot\\
&\ (SN_1+\chi Su\otimes \hat l)\\
&=N_5SN_3SN_1+\chi N_5 S\tilde u\otimes l'SN_1+\chi N_5Su\otimes \tilde lSN_1+\chi \hat u\otimes  l'SN_3SN_1+\\
&\ O(\mu\chi)\hat u\otimes l'SN_1
+\ \chi N_5SN_3Su\otimes \hat l+N_5(\chi S\tilde u\otimes l'+\chi Su\otimes \tilde l)(\chi Su\otimes \hat l)+\\
&\ (\chi \hat u\otimes  l') S
N_3(\chi Su\otimes \hat l)+O(\mu\chi^2)\hat u\otimes l'\cdot Su\otimes \hat l+O(\mu\chi),
\end{aligned}
\end{equation}
where in the second $=$, we use $l'Su\sim l'S\tilde u\lesssim \mu/\chi$ by Sublemma \ref{sublm: last}(3) and $\tilde l=l'+O(1/\chi)$ by their definitions in Definition \ref{Def: ul} and Proposition \ref{Prop: main}(a.2).

The first term in  \eqref{eq: 8terms} is $O(\mu\chi)$ by Sublemma \ref{sublm: last}(4). The ninth term $\mu\chi^2\hat u\otimes l'\cdot Su\otimes \hat l=O(\mu^2\chi)$ since $ l'\cdot Su=O\left(\frac{\mu}{\chi}\right)$ by
Sublemma \ref{sublm: last}(3). The seventh term
$N_5(\chi S\tilde u\otimes l'+\chi Su\otimes \tilde l)(\chi S u\otimes \hat l)
=O(\mu\chi)$
using that $\tilde u=u+O\left(\frac{1}{\chi}\right)$ and 
Sublemma \ref{sublm: last}(3). The fifth term has the estimate $\mu\chi \hat u\otimes l'SN_1=O(\mu\chi)$ by Sublemma \ref{sublm: last}(1). The eighth term
$(\chi \hat u\otimes  l') S
N_3(\chi Su\otimes \hat l)=O(\mu\chi),$ since $l' S
N_3Su\lesssim\frac{\mu}{\chi}$ by Sublemma~\ref{sublm: last}(3).

We are left with four terms, the second, third, fourth and sixth terms, written together as
\begin{equation}\label{eq: OK}
\chi\left[N_5 S\tilde u\otimes l'SN_1+N_5Su\otimes \tilde lSN_1+ \hat u\otimes  l'SN_3SN_1+N_5SN_3 Su\otimes \hat l\right].\end{equation}
We first use the fact that \[\tilde u=u+O\(\frac{1}{\chi}\),\ \tilde l=l+O\(\frac{1}{\chi}\),\ l'=l+O\(\frac{1}{\chi}\),\ \hat l=l+O\(\mu\)\] and  $N_1,N_5=O(\mu\chi)$ to reduce the four terms to
\begin{equation}\nonumber
\chi\left[2 (N_5 S u+O\(\mu\))\otimes (l'SN_1+O(\mu))+\hat u\otimes  l'SN_3SN_1+ N_5SN_3 Su\otimes ( l'+O(\mu))\right].\end{equation}

Using parts (1) and (2) of Sublemma \ref{sublm: last}, we find that each term in expression \eqref{eq: OK} has the form of $\chi ( u+O(\mu))\otimes ( l'+O(\mu))=\chi  u\otimes l'+O(\mu\chi).$
Up to now, we have successfully separated the $O(\chi^2), O(\chi)$ and $O(\mu\chi)$ parts in the global map.

\textbf{Step 5.} {\it Completing the proof.}

Remember we have dropped the two $O(1)$ matrix $(1+u_1^i\otimes l_1^i)$ in $(I)$ and the matrix $(1+u_5^f\otimes l_5^f)$ in $(V)$ in Step 0. We summarize the results of Steps 2 and 4 as follows
\[d\Glob=(\Id_{10}+u_5^f\otimes l_5^f)(\chi^2\bar{ u}'\otimes \bar{ l}'+\chi u\otimes l'+O(\mu\chi))(\Id_{10}+u_1^i\otimes l_1^i).\]
 To complete the proof of the lemma, it is enough to define
 \begin{equation}\label{eq: bar ul}\brv=(\Id_{10}+u_5^f\otimes l_5^f)\bar{ u}',\quad \brrv=(\Id_{10}+u_5^f\otimes l_5^f) u,\quad \brlin=\bar{ l}'(\Id_{10}+u_1^i\otimes l_1^i),\quad  \brrlin=l'(\Id_{10}+u_1^i\otimes l_1^i).\end{equation}
We obtain the structure of $d\Glob$ stated in Lemma \ref{Lm: glob}. It remains to work out the vectors $\brrv,\brv,\brrlin,\brlin$.
We have \[\brrv= u+O(\mu)\to (0,1,0_{1\times 8})^T,\quad \brrlin =l'+O(\mu)\to (1,0_{1\times 9}) \mathrm{\ as\ } 1/\chi\ll\mu\to 0\] using Sublemma \ref{sublm: last} for $ u, l'$ and Proposition \ref{Prop: main} for $u_1^i$ and $l_5^f$. According to \eqref{eq: chi^2} in Step 2, we have $\bar{ u}'=N_5\hat u_{iii'}+O(1) \hat u,\
\bar{ l}'=\hat l_iN_1.$
We neglect the term $O(1)\hat u$ since it is enough to consider the span$\{N_5\hat u_{iii'},\hat u\}$ and $\hat u=u+O(\mu)$ is already provided by the $O(\chi)$ part of $d\Glob$.
Using $\hat u_{iii'}$ in Definition \ref{Def: ul} and $N_1$ in Proposition \ref{Prop: main}, we find in $\bar{ u}'$, we have $N_5\hat u_{iii'}\to (0,O(1), 0_{1\times 6},O(1),O(1))$ as $1/\chi\ll\mu\to 0$, where the last two $O(1)$ entries are \begin{equation}\label{EqEigen}(N_5)_{44}.(u_{iii'}(9),u_{iii'}(10))^T=(u_{iii'}(9),u_{iii'}(10))^T=\left(1,\frac{\hat L_4}{\hat L_4^2+\hat G_4^2}\right),\end{equation}
($(u_{iii'}(9),u_{iii'}(10))$ is an eigenvector of $(N_5)_{44}$  with eigenvalue 1)
and in $\bar{ l}',$ we have
$$\hat l_iN_1\to \lim l_i=\left(\frac{\tilde G_4}{\tilde L_4(\tilde L_4^2+\tilde G_4^2)},0_{1\times 7},-\frac{1}{\tilde L_4^2+\tilde G_4^2},-\frac{1}{\tilde L_4}\right).$$
It is easy to see that $u\to \brrv$ using the definition of $u$ in Proposition \ref{Prop: main}. We substitute these calculations back to \eqref{eq: bar ul} to get $\brv \to u_{iii'}+c\brrv$ for some constant $c$.

\section*{Acknowledgement}
The author would like to thank his thesis advisor Prof. Dmitry Dolgopyat who spent one year checking all the details of the paper. Without his  enormous intelligence and time input, constant encouragement and financial support, the work could not have been completed. The author would also like to express his deep gratitude to Professor Joseph Gerver who also carefully checked all the details of the paper and gave numerous suggestions which significantly improve the readability of the paper. The author is supported by NSFC (Significant project No.11790273) in China and Beijing Natural Science Foundation (Z180003).

\end{document}